\documentclass[dvipsnames]{article}

\usepackage{xcolor}
\definecolor{darkred}{rgb}{0.45,0,0}
\usepackage[pagebackref,colorlinks,linkcolor=darkred,citecolor=darkred]{hyperref}
\renewcommand*{\backref}[1]{}
\renewcommand*{\backrefalt}[4]{({%
    \ifcase #1 Not cited.%
    \or Cited on p.~#2%
    \else Cited on pp.~#2%
    \fi%
  })}
\usepackage{amsmath,amssymb,latexsym,mathpartir,fontawesome,amsthm,cleveref,url,
rotating,multirow,stmaryrd,eulervm,charter,xcolor,mathtools,graphicx,stackengine,pifont,tikz-cd,bbold,bbm,scalerel}

\usetikzlibrary{decorations.pathreplacing, decorations.markings, shapes, decorations.pathmorphing, patterns}

\newsavebox{\pullback}
\sbox\pullback{%
\begin{tikzpicture}%
\draw (0,0) -- (1ex,0ex);%
\draw (1ex,0ex) -- (1ex,1ex);%
\end{tikzpicture}}

\usepackage[draft]{fixme}

\usepackage{geometry}
\usepackage[title]{appendix}

\usepackage{sectsty}
\sectionfont{\centering \large} \subsectionfont{\centering \normalfont \textsc}

\title{Displayed Type Theory and Semi-Simplicial Types}

\author{Astra Kolomatskaia\thanks{We acknowledge the support of the Natural
Sciences and Engineering Research Council of Canada (NSERC). Cette recherche a
\'et\'e financ\'ee par le Conseil de recherches en sciences naturelles et en
g\'enie du Canada (CRSNG). [funding reference number CGSD3-545891-2020] \\ This material is based upon work supported by the National Science Foundation under Grant No. DMS-2204304.} \and
Michael Shulman\thanks{This material is based upon work supported by the Air
Force Office of Scientific Research under award number FA9550-21-1-0009.}}

\date{\vspace{-5ex}}

\definecolor{codegreen}{rgb}{0,0.6,0}
\definecolor{codegray}{rgb}{0.5,0.5,0.5}
\definecolor{codepurple}{rgb}{0.58,0,0.82}
\definecolor{backcolour}{rgb}{0.95,0.95,0.92}

\usepackage{listings} 
\usepackage[X2,OT1]{fontenc}

\newcommand{\sectend}{\triangleleft}
\newenvironment{sectendproof}{\begingroup\begin{proof}}{\end{proof}\endgroup}

\def\subell{_{\;\!\ell}}
\def\subellp{_{\;\!\ell\cblu{'}}}
\def\sell#1{_{\;\!\ell_{\cblu{#1}}}}
\def\sellj#1#2{_{\;\!\ell_{\cblu{#1}}\join \ell_{\cblu{#2}}}}

\def\M{\mathcal{\cred{M}}}
\def\MC{\mathcal{\cblu{C}}}
\def\MB{\mathcal{\cblu{B}}}
\def\op{{^{\mathsf{\cred{op}}}}}
\def\coop{^{\mathsf{\cred{coop}}}}
\def\Cat{\cred{\mathcal{C}\!\mathit{at}}}
\def\pr{\mathsf{\cred{pr}}}
\def\prl{{\mathsf{\cred{pr}}\subell}}
\def\Set{\mathsf{\cred{Set}}}
\def\Ty{\mathsf{\cred{Ty}}}
\def\Tyl{\Ty\subell}
\def\Tm{\mathsf{\cred{Tm}}}
\def\Tml{\Tm\subell}

\def\tpr{\mathsf{\cred{tpr}}}
\def\Tel{\mathsf{\cred{Tel}}}
\def\Telover{\mathsf{\cred{Tel}}\sslash}
\def\bTel{\mathbf{\cred{Tel}}}
\def\PSub{\mathsf{\cred{PSub}}}
\def\ce{\centerdot}
\newcommand{\leibhom}[2]{\cred{\widehat{\hom}}(#1,~#2)}
\newcommand{\Ftil}{\cred{\widetilde{F}}}

\def\dCoind{\mathsf{\cred{dCoind}}}
\def\head{\mathsf{\cred{head}}}
\def\tail{\mathsf{\cred{tail}}}
\def\corec{\mathsf{\cred{corec}}}

\DeclareFontFamily{U}{min}{}
\DeclareFontShape{U}{min}{m}{n}{<-> udmj30}{}
\newcommand{\yon}{\cred{\text{\usefont{U}{min}{m}{n}\symbol{'210}}}}

\newcommand{\drpullback}[1][dr]{\ar[#1,phantom,near start,"\lrcorner"]}

\DeclareSymbolFont{bbold}{U}{bbold}{m}{n}
\DeclareSymbolFontAlphabet{\mathbbb}{bbold}

\let\oldflat\flat

\def\dia{{\cred{\lozenge}}}
\def\undia{{\cred{\blacklozenge}}}
\def\tri{{\cred{\triangle}}}
\def\untri{{\cred{\blacktriangle}}}
\def\sq{{\cred{\square}}}
\def\unsq{{\cred{\blacksquare}}}
\def\tridia{{\cred{\triangle\!\!\lozenge}}}

\def\trisq{{\cred{\triangle\!\square}}}

\def\lock{\text{\faLock}}
\def\key{\text{\faKey}}

\def\ec{{\bkt{}}}
\def\esub{{[\;]}}
\def\locks{\mathsf{locks}}

\def\mode{\;\mathsf{mode}}
\def\flat{\;\mathsf{flat}}

\def\ctx{\;\mathsf{ctx}}

\def\tel{\;\mathsf{tel}}
\def\stel{\;\mathsf{stel}}
\def\inftel{\;\mathsf{stel}^\infty}

\def\ext{\mid}
\def\ob{\;\mathsf{ob}}

\def\type{\;\mathsf{type}}
\def\level{\;\mathsf{level}}
\def\lzero{\mathsf{\cred{lzero}}}
\def\lsuc{\mathsf{\cred{lsuc}}}

\def\eval{\mathsf{\cred{app}}}

\def\Type{\mathsf{\cred{Type}}}

\def\Disc{\mathsf{\cred{Disc}}}
\def\SST{\mathsf{\cred{SST}}}
\def\ASST{\mathsf{\cred{ASST}}}
\def\El{\mathsf{\cred{El}}}
\def\Code{\mathsf{\cred{Code}}}

\def\weak{\mathord{\uparrow}}
\def\d{{^{\mathsf{\cred{d}}}}}
\def\D{{^{\mathsf{\cred{D}}}}}
\def\dm{\mathsf{\cred{dm}}}
\def\sm{\mathsf{\cred{sm}}}
\def\inf{\mathsf{\cred{in_{dm}}}}
\def\ins{\mathsf{\cred{in_{sm}}}}
\def\inc{\mathsf{\cred{in_{fl}}}}

\def\lim{\mathsf{\cred{lim}}}
\def\res{\mathsf{\cred{res}}}

\def\nat{\mathsf{\cred{act}}}

\def\u!{\mathsf{!}}
\def\ft{\mathsf{\cred{pt}}}
\def\zv{\mathsf{\cred{zv}}}

\def\join{\cred{\sqcup}}

\def\evens{\mathsf{\cred{evens}}}
\def\ev{{^\mathsf{\cred{ev}}}}
\def\od{{^\mathsf{\cred{od}}}}

\def\smsq{\cblu{\scalebox{0.82}{$\square$}}}
\def\subsq{\cred{\scalebox{0.68}{$\square$}}}

\newcommand{\plus}[1]{#1^{\cblu{+}}}

\tikzset{
  on each segment/.style={
    decorate,
    decoration={
      show path construction,
      moveto code={},
      lineto code={
        \path [#1]
        (\tikzinputsegmentfirst) -- (\tikzinputsegmentlast);
      },
      curveto code={
        \path [#1] (\tikzinputsegmentfirst)
        .. controls
        (\tikzinputsegmentsupporta) and (\tikzinputsegmentsupportb)
        ..
        (\tikzinputsegmentlast);
      },
      closepath code={
        \path [#1]
        (\tikzinputsegmentfirst) -- (\tikzinputsegmentlast);
      },
    },
  },
  mid arrow/.style={postaction={decorate,decoration={
        markings,
        mark=at position .5 with {\arrow[scale=1.5, #1]{to}}
      }}},
}

\DeclareTextSymbolDefault{\cyrabhdze }{X2}
\newcommand{\oldze}{\text{\cyrabhdze}}
\newcommand{\ze}{\cblu{\oldze}}

\usepackage[most]{tcolorbox}
\tcbset{on line,
        boxsep=4pt, left=0pt,right=0pt,top=0pt,bottom=0pt,
        colframe=black,colback=white,boxrule=0.2mm,
        highlight math style={enhanced}
        }

\def\id{\cred{1}}

\let\oldtop\top
\renewcommand{\top}{\cred{\oldtop}}

\let\oldell\ell
\renewcommand{\ell}{\cblu{\oldell}}

\let\oldalpha\alpha
\renewcommand{\alpha}{\cblu{\oldalpha}}

\let\oldmu\mu
\renewcommand{\mu}{\cblu{\oldmu}}

\let\oldnu\nu
\renewcommand{\nu}{\cblu{\oldnu}}

\let\oldrho\rho
\renewcommand{\rho}{\cred{\oldrho}}

\let\oldlambda\lambda
\renewcommand{\lambda}{\cred{\oldlambda}}

\let\oldtheta\theta
\renewcommand{\theta}{\cblu{\oldtheta}}

\let\olddelta\delta
\renewcommand{\delta}{\cblu{\olddelta}}

\let\oldDelta\Delta
\renewcommand{\Delta}{\cblu{\oldDelta}}

\let\oldsigma\sigma
\renewcommand{\sigma}{\cblu{\oldsigma}}

\let\oldtau\tau
\renewcommand{\tau}{\cblu{\oldtau}}

\let\oldzeta\zeta
\renewcommand{\zeta}{\cblu{\oldzeta}}

\let\oldphi\phi
\renewcommand{\phi}{\cblu{\oldphi}}

\let\oldvarphi\varphi
\renewcommand{\varphi}{\cblu{\oldvarphi}}

\let\oldSigma\Sigma
\renewcommand{\Sigma}{\cred{\oldSigma}}

\let\oldint\int
\renewcommand{\int}{\cred{\oldint}}

\let\oldGamma\Gamma
\renewcommand{\Gamma}{\cblu{\oldGamma}}

\let\oldPhi\Phi
\renewcommand{\Phi}{\cblu{\oldPhi}}

\let\oldPsi\Psi
\renewcommand{\Psi}{\cblu{\oldPsi}}

\let\oldOmega\Omega
\renewcommand{\Omega}{\cblu{\oldOmega}}

\let\oldomega\omega
\renewcommand{\omega}{\cblu{\oldomega}}

\let\oldgamma\gamma
\renewcommand{\gamma}{\cblu{\oldgamma}}

\let\oldxi\xi
\renewcommand{\xi}{\cblu{\oldxi}}

\let\oldepsilon\epsilon
\renewcommand{\epsilon}{\cblu{\oldepsilon}}

\let\oldpi\pi
\renewcommand{\pi}{\cred{\oldpi}}

\let\oldTheta\Theta
\renewcommand{\Theta}{\cblu{\oldTheta}}

\let\oldUpsilon\Upsilon
\renewcommand{\Upsilon}{\cblu{\oldUpsilon}}

\let\oldupsilon\upsilon
\renewcommand{\upsilon}{\cblu{\oldupsilon}}

\let\oldparallel\parallel
\renewcommand{\parallel}{~\cred{\oldparallel}~}

\def\n{\cblu{n}}
\def\h{\cblu{h}}
\def\m{\cblu{m}}
\def\k{\cblu{k}}
\def\l{\cblu{l}}
\def\bB{\cred{\mathbb{B}}}
\def\One{\cred{\mathbb{1}}}
\def\Zero{\cred{\mathbb{0}}}
\def\es{\cred{\varnothing}}
\def\p{\cblu{p}}
\def\P{\cred{P}}
\def\q{\cblu{q}}

\def\x{\cblu{x}}
\def\xp{\cblu{x'}}
\def\y{\cblu{y}}
\def\yp{\cblu{y'}}
\def\z{\cblu{z}}
\def\zv{\cred{\mathsf{zv}}}

\def\u{\cblu{u}}
\def\t{\cblu{t}}
\def\s{\cblu{s}}

\def\ct{\cblu{\mathfrak{t}}}
\def\cs{\cblu{\mathfrak{s}}}
\def\f{\cblu{f}}
\def\g{\cblu{g}}

\def\a{\cblu{a}}
\def\b{\cblu{b}}
\def\c{\cblu{c}}

\def\A{\cblu{A}}
\def\cA{\cblu{\mathcal{A}}}
\def\B{\cblu{B}}
\def\cB{\cblu{\mathcal{B}}}
\def\C{\cblu{C}}
\def\cC{\cblu{\mathcal{C}}}

\def\X{\cblu{X}}
\def\Xtil{\widetilde{X}}
\def\Y{\cblu{Y}}
\def\E{\cblu{E}}
\def\F{\cred{F}}
\def\G{\cblu{G}}
\def\Z{\cred{\mathsf{Z}}}
\def\S{\cred{\mathsf{S}}}
\def\R{\cred{\mathsf{R}}}
\def\cB{\cblu{\mathcal{B}}}

\def\I{\cblu{I}}
\def\J{\cblu{J}}

\def\H{\cblu{H}}
\def\K{\cblu{K}}

\def\pair#1#2{\cred{\boldsymbol\langle}\; #1\;\cred{\boldsymbol,}\;#2 \;\cred{\boldsymbol\rangle}}
\def\smpair#1#2{\cred{\boldsymbol\langle}\, #1\,\cred{\boldsymbol,}\,#2 \,\cred{\boldsymbol\rangle}}
\def\angle#1{\cred{\boldsymbol\langle}#1\cred{\boldsymbol\rangle}}

\def\bt{\scalebox{0.9}{$\cred{\bigtimes}$}}

\newcommand{\bbox}{\rule{0em}{1ex}\hfill\ensuremath{\triangleleft}}

\newcommand{\set}[1]{\left\{#1\right\}}
\newcommand{\bkt}[1]{\left(#1\right)}
\newcommand{\sqbkt}[1]{[~#1~]}
\newcommand{\smsqbkt}[1]{[\,#1\,]}

\def\ddia{{\cred{\stackinset{c}{0pt}{c}{0em}{\scalebox{0.5}{\ensuremath{\blacklozenge}}}{{\ensuremath{\lozenge}}}}}}

\def\unddia{{\stackinset{c}{0pt}{c}{0em}{\scalebox{0.4}{\ensuremath{\cwhite{\blacklozenge}}}}{{\ensuremath{\cred{\blacklozenge}}}}}}

\makeatletter
\newsavebox{\@brx}
\newcommand{\llangle}[1][]{\savebox{\@brx}{\(\m@th{#1\cred{\boldsymbol\langle}}\)}%
  \mathopen{\copy\@brx\kern-0.5\wd\@brx\usebox{\@brx}}}
\newcommand{\rrangle}[1][]{\savebox{\@brx}{\(\m@th{#1\cred{\boldsymbol\rangle}}\)}%
  \mathclose{\copy\@brx\kern-0.5\wd\@brx\usebox{\@brx}}}
\makeatother

\makeatletter
\newsavebox{\@bry}
\newcommand{\newllparenthesis}[1][]{\savebox{\@bry}{\(\m@th{#1\cred{\boldsymbol(}}\)}%
  \mathopen{\copy\@bry\kern-0.5\wd\@bry\usebox{\@bry}}}
\newcommand{\newrrparenthesis}[1][]{\savebox{\@bry}{\(\m@th{#1\cred{\boldsymbol)}}\)}%
  \mathclose{\copy\@bry\kern-0.5\wd\@bry\usebox{\@bry}}}
\makeatother

\newcommand{\dbkt}[1]{\scaleleftright{\newllparenthesis}{~\rule[-2.5pt]{0pt}{1em}#1~}{\newrrparenthesis}}
\newcommand{\smdbkt}[1]{\scaleleftright{\newllparenthesis}{\,\rule[-2.5pt]{0pt}{1em}#1\,}{\newrrparenthesis}}

\makeatletter
\newsavebox{\@brz}
\newcommand{\newllbracket}[1][]{\savebox{\@brz}{\(\m@th{#1\cred{\boldsymbol[}}\)}%
  \mathopen{\copy\@brz\kern-0.5\wd\@brz\usebox{\@brz}}}
\newcommand{\newrrbracket}[1][]{\savebox{\@brz}{\(\m@th{#1\cred{\boldsymbol]}}\)}%
  \mathclose{\copy\@brz\kern-0.5\wd\@brz\usebox{\@brz}}}
\makeatother

\newcommand{\dsqbkt}[1]{\newllbracket~#1~\newrrbracket}

\newcommand{\blank}{\mathord{\hspace{1pt}\text{--}\hspace{1pt}}}
\newcommand{\pblank}{(\blank)}

\newcommand{\slfrac}[2]{\left.#1\middle/#2\right.}

\newcommand{\sub}[1]{_{\scalebox{0.55}[1.0]{$\scriptscriptstyle #1$}}}
\newcommand{\ff}{\mathfrak{f}}

\newcommand{\minus}[1]{{\mathchoice%
  {\scalebox{0.6}[1.0]{$-$}#1}%
  {\scalebox{0.6}[1.0]{$-$}#1}%
  {\scalebox{0.6}[1.0]{$\scriptstyle -$}#1}%
  {\scalebox{0.6}[1.0]{$\scriptscriptstyle-$}#1}%
}}
\newcommand{\minusone}{\cred{\minus{1}}}
\newcommand{\minustwo}{\cred{\minus{2}}}
\newcommand{\bminusone}{\cblu{\minus{1}}}

\definecolor{darkgreen}{rgb}{0,0.4,0}
\newcommand{\cblu}[1]{{\color{blue}#1}}
\newcommand{\cred}[1]{{\color{darkgreen}#1}}
\newcommand{\cprime}[1]{{\color{red}#1}}
\newcommand{\cwhite}[1]{{\color{white}#1}}

\numberwithin{equation}{section}
\newtheorem{theorem}[equation]{Theorem}
\theoremstyle{definition}
\newtheorem{idea}[equation]{Idea}
\newtheorem{remark}[equation]{Remark}
\newtheorem{definition}[equation]{Definition}
\newtheorem{lemma}[equation]{Lemma}
\newtheorem{example}[equation]{Example}

\crefformat{paragraph}{section #2#1#3}
\Crefformat{paragraph}{Section #2#1#3}

\begin{document}

\maketitle

\begin{abstract}
We introduce \emph{Displayed Type Theory (dTT)}, a multi-modal homotopy type
theory with \emph{discrete} and \emph{simplicial} modes. In the intended
semantics, the discrete mode is interpreted by a model for an arbitrary
$\infty$-topos, while the simplicial mode is interpreted by Reedy fibrant
augmented semi-simplicial diagrams in that model. This simplicial structure is
represented inside the theory by a primitive notion of \emph{display} or
\emph{dependency}, guarded by modalities, yielding a partially-internal form of
unary parametricity.

Using the display primitive, we then give a coinductive definition, at
the simplicial mode, of a type $\SST$ of semi-simplicial
types. Roughly speaking, a semi-simplicial type $\X$ consists of a type
$\X_\cred{0}$ together with, for each $\x:\X_\cred{0}$, a displayed semi-simplicial
type over $\X$.  This mimics how simplices can be generated
geometrically through repeated cones, and is made possible by the
display primitive at the simplicial mode. The discrete part of $\SST$
then yields the usual infinite indexed definition of semi-simplicial
types, both semantically and syntactically. Thus, dTT enables working
with semi-simplicial types in full semantic generality.
\end{abstract}

\setcounter{tocdepth}{2}
\tableofcontents

\section{Introduction}
\label{sec:intro}

\paragraph{Semi-simplicial types.}

\emph{Homotopy Type Theory (HoTT)}~\cite{hottbook} is a perspective on intensional dependent type theory that regards types as homotopical spaces or $\infty$-groupoids.
It has proven remarkably successful as a synthetic context in which to do homotopy theory and algebraic topology, and as an internal language for $(\infty,1)$-toposes~\cite{shulman:univinj}.
However, an enduring frustration has been its apparent inability to define general homotopy-coherent structures.
Some infinite structures can be defined in HoTT, such as globular types and spectra; but others, such as $A_\infty$-spaces or $(\infty,1)$-categories, have so far resisted all attempts at definition.
We know no convincing explanation for why they should be impossible, but the fact that all attempts appear to fail in a similar way suggests the operation of an as-yet-unarticulated principle.

Specifically, stating an \emph{`infinite coherence'} property generally seems to require an infinite structure within which to assemble the coherences, while defining such a structure itself seems to require infinite coherence, leading to an infinite regress.
This is in contrast to the situation in classical homotopy theory where the infinite structures to describe coherence, such as operads and simplicial diagrams, can themselves be defined using \emph{strict} point-set-level equalities, which are then automatically fully coherent.
It is tempting to try to mimic this in homotopy type theory using \emph{definitional} equalities in place of point-set ones, but this is difficult because definitional equality is not reified in the theory and we have limited tools for forcing it to hold.

One of the more flexible ways to enforce definitional equalities is to use type dependency, moving from a \emph{fibred} perspective to an \emph{indexed} one.
In the simplest case, this means replacing a function $\p : \E \to \B$ with a type family $\cblu{\widetilde{E}} : \B \to \Type$.
This has a corresponding projection function $\pi_\cred{1} : \Sigma~\cblu{\widetilde{E}} \to \B$, but we can suppose a point $\cblu{\widetilde{e}} : \Sigma~\cblu{\widetilde{E}}$ with a definitional equality $\pi_\cred{1}~\cblu{\widetilde{e}} \equiv \cblu{b}$ by supposing a point $\cblu{e}: \cblu{\widetilde{E}}~\cblu{b}$ and letting $\cblu{\widetilde{e}} = \bkt{\cblu{b},~\cblu{e}}$.

Thus, it is natural to try to define infinitely coherent structures that can be expressed in a purely indexed way.
The example of this sort which has attracted the most attention is that of \emph{semi-simplicial types}, because they are well-known within homotopy theory and could be used to encode many, if not all, other infinitely coherent structures.
In indexed style, a semi-simplicial type consists of type families $\cred{A_0}$, $\cred{A_1}$, $\cred{A_2}$, and so on, having types that start out as follows:
\begin{align*}
  \cred{A_0} :~&\Type \\
  \cred{A_1} :~&{\normalcolor (\cblu{x\sub{01}} :
  \cred{A_0})}~{\normalcolor (\cblu{x\sub{10}} :
  \cred{A_0})} \to \Type \\
  \cred{A_2} :~&{\normalcolor (\cblu{x\sub{001}} :
  \cred{A_0})}~{\normalcolor (\cblu{x\sub{010}} :
  \cred{A_0})}
  ~{(\cblu{\beta\sub{011}} : \cred{A_1}~\cblu{x\sub{001}~x\sub{010}})}
  ~{\normalcolor (\cblu{x\sub{100}} : \cred{A_0})}
  ~{(\cblu{\beta\sub{101}} : \cred{A_1}~\cblu{x\sub{001}~x\sub{100}})}
  ~{(\cblu{\beta\sub{110}} : \cred{A_1}~\cblu{x\sub{010}~x\sub{100}})} \to
  \Type \\
  \cred{A_3} :~&{\normalcolor (\cblu{x\sub{0001}} :
  \cred{A_0})}~{\normalcolor (\cblu{x\sub{0010}} :
  \cred{A_0})}
  ~{(\cblu{\beta\sub{0011}} : \cred{A_1}~\cblu{x\sub{0001}~x\sub{0010}})}
  ~{\normalcolor (\cblu{x\sub{0100}} : \cred{A_0})}
  ~{(\cblu{\beta\sub{0101}} : \cred{A_1}~\cblu{x\sub{0001}~x\sub{0100}})}
  ~{(\cblu{\beta\sub{0110}} : \cred{A_1}~\cblu{x\sub{0010}~x\sub{0100}})} \\
  &{(\cblu{\ff\sub{0111}} :
  \cred{A_2}~\cblu{x\sub{0001}~x\sub{0010}~\beta\sub{0011}~x\sub{0100}~\beta\sub{0101}~\beta\sub{0110}})}
  ~{\normalcolor (\cblu{x\sub{1000}} : \cred{A_0})}
  ~{(\cblu{\beta\sub{1001}} : \cred{A_1}~\cblu{x\sub{0001}~x\sub{1000}})}
  ~{(\cblu{\beta\sub{1010}} : \cred{A_1}~\cblu{x\sub{0010}~x\sub{1000}})}\\
  &{(\cblu{\ff\sub{1011}} :
  \cred{A_2}~\cblu{x\sub{0001}~x\sub{0010}~\beta\sub{0011}~x\sub{1000}~\beta\sub{1001}~\beta\sub{1010}})}
  ~{(\cblu{\beta\sub{1100}} : \cred{A_1}~\cblu{x\sub{0100}~x\sub{1000}})}
  ~{(\cblu{\ff\sub{1101}} :
  \cred{A_2}~\cblu{x\sub{0001}~x\sub{0100}~\beta\sub{0101}~x\sub{1000}~\beta\sub{1001}~\beta\sub{1100}})}\\
  &{(\cblu{\ff\sub{1110}} :
  \cred{A_2}~\cblu{x\sub{0010}~x\sub{0100}~\beta\sub{0110}~x\sub{1000}~\beta\sub{1010}~\beta\sub{1100}})} \to
  \Type
\end{align*}
First, we have a type $\cred{A_0}$ of \emph{points}. Second, we
have for every two points $\cblu{x\sub{01}},\cblu{x\sub{10}} : \cred{A_0}$, a type
$\cred{A_1}~\cblu{x\sub{01}~x\sub{10}}$ of \emph{lines} joining $\cblu{x\sub{01}}$ and $\cblu{x\sub{10}}$.
Third, for every three points $\cblu{x\sub{001}},\cblu{x\sub{010}},\cblu{x\sub{100}} : \cred{A_0}$ and
three lines $\cblu{\beta\sub{011}} : \cred{A_0}~\cblu{x\sub{001}~x\sub{010}}$, $\cblu{\beta\sub{101}} :
\cred{A_0}~\cblu{x\sub{001}~x\sub{100}}$, and $\cblu{\beta\sub{110}} : \cred{A_0}~\cblu{x\sub{010}~x\sub{100}}$, a type
$\cred{A_1}~\cblu{x\sub{001}~x\sub{010}~\beta\sub{011}~x\sub{100}~\beta\sub{101}~\beta\sub{110}}$ of \emph{triangles} with the given
boundary. The pattern continues with \emph{tetrahedra}, which we may also, more
technically, call $\cred{3}$-simplices. In general, $\cred{A}_\n$ describes the type of
$\n$-simplices indexed by their boundaries.

\begin{remark}
  The binary subscripts on variables above follow a scheme that we learned from Tim Campion, although related schemas have been rediscovered many times.
  The $\cred{2}^{\n+\cred{1}}-\cred{1}$ simplices constituting an $\n$-simplex and its boundary are labeled by the numbers from $\cred{1}$ to $\cred{2}^{\n+\cred{1}}-\cred{1}$ written in binary, where the number of $\cred{1}$s in a binary number corresponds to the dimension of the simplex, and the binary numbers corresponding to the boundary of some simplex are obtained by replacing one or more of the $\cred{1}$s in its binary number by $\cred{0}$s.
  (As we will see later, the binary number $\cred{0}$ can also be regarded as denoting the unique $(\minusone)$-simplex in the boundary of an $\n$-simplex in an \emph{augmented} semi-simplicial set.)

  We then list these simplices in the order given by these numbers.
  This ordering may seem somewhat curious, compared to the more na\"ive approach of listing all the $\cred{0}$-simplices, then all the $\cred{1}$-simplices, and so on; but it does retain the important property that all the simplices in the boundary of some simplex are listed before it (thus, for instance, it makes the semi-simplex category into an \emph{`ordered direct category'}; see \cref{sec:correctness-sst}).
  We will see later that this ordering is what arises most naturally in (co)inductive constructions of semi-simplicial sets (which was also Campion's motivation).

  In addition to subscripting simplex variables by binary numbers according to this scheme, in this paper we will use different base \emph{letters} to indicate the \emph{dimension} of each variable.
  Thus for instance $\cblu{x\sub{001}}$, $\cblu{x\sub{010}}$, and $\cblu{x\sub{100}}$ are all $\cred{0}$-simplices, while $\cblu{\beta\sub{011}}$, $\cblu{\beta\sub{101}}$, and $\cblu{\beta\sub{110}}$ are all $\cred{1}$-simplices, and $\cblu{\ff\sub{111}}$ is a $\cred{2}$-simplex.
  We will not have much occasion to denote $\cred{3}$-simplices; for $(\minusone)$-simplices we use the Cyrillic letter $\ze$ \emph{(ze)}.
  We may also use different letters from the same alphabets for simplices in different semi-simplicial types, e.g.\ $\cblu{\gamma\sub{11}} : \cred{B_1}~\cblu{y\sub{01}}~\cblu{y\sub{10}}$ and $\cblu{\delta\sub{11}} : \cred{C_1}~\cblu{z\sub{01}}~\cblu{z\sub{10}}$.
  \bbox
\end{remark}

One may think of the terms $\cred{A_0}$, $\cred{A_1}$, $\cred{A_2}$, and so on as defining the fields of an infinite record type.
Terms of this infinite record type $\SST$ are known as \emph{semi-simplicial types}.
Thus, the problem is to define a type of semi-simplicial types within homotopy type theory, continuing the above pattern.
As a correctness criterion, one would expect that when interpreted in any $(\infty,1)$-topos the type $\SST$ becomes a classifier of semi-simplicial objects.
However, even this problem is still unsolved: every attempt to internally encode the combinatorics that generate the type of $\cred{A}_\n$, as a function of $\n$, seems to lead once again to an infinite regress.

In light of this situation, an alternative approach is to formulate more expressive type theories that can solve the problem of infinitely coherent objects.
One such proposal is \emph{Two-Level Type Theory (2LTT)}~\cite{ack:2ltt}, which introduces an \emph{`outer level'} of \emph{`exo-types'} that are not homotopy-invariant.
The exo-types admit a strict exo-equality type, essentially reifying definitional equality, which can then be used analogously to classical point-set equality to define infinitely coherent structures.
And by the results of~\cite{uskuplu:thesis}, 2LTT can be interpreted in any $(\infty,1)$-topos, so its semantics are not significantly less general than ordinary HoTT, and the type $\SST$ defined in 2LTT does interpret to the correct classifier.
However, although the exo-equality is assumed to satisfy Uniqueness of Identity Proofs, to keep type-checking decidable it cannot satisfy a reflection rule making exo-equalities into definitional equalities.
Thus, it can be quite cumbersome to work with in practice.

Another proposal is \emph{Simplicial Type Theory (STT)}~\cite{rs:stt}, which changes perspective to view individual types as \emph{simplicial} spaces, with additional primitives for manipulating the simplicial structure.
One can then simply impose conditions on one of these \emph{`simplicial types'} to make it represent (for instance) an $(\infty,1)$-category.
This suggests a `synthetic' approach to higher category theory analogous to ordinary HoTT's synthetic approach to homotopy theory, which is potentially quite powerful; and the results of~\cite{rs:stt,weinberger:stab-ext} imply that it can be interpreted in the category of simplicial objects in any $(\infty,1)$-topos.
However, the strength of the synthetic approach is also its weakness: because simplicial types are postulated rather than defined, what we can do with them is limited to what is expressed by the axiomatization.

\paragraph{A coinductive definition of semi-simplicial types.}

In this paper we propose a third enhancement of homotopy type theory, called \emph{Displayed Type Theory (dTT)}, in which it is possible to define and work with semi-simplicial types (and many other things).
This type theory is inspired by the following idea for a \emph{coinductive definition} of a type $\SST$ of semi-simplicial types:
\begin{idea}
  A semi-simplicial type $\A$ consists of
  \begin{itemize}
  \item a type $\Z~\A$, and
  \item for each $\x : \Z~\A$, a semi-simplicial type $\S~\A~\x$ \emph{over $\A$}.
  \end{itemize}
\end{idea}
\noindent
It may not be at all obvious why this should be a definition of semi-simplicial types, so let us unravel it a few steps:
\begin{enumerate}\setcounter{enumi}{-1}
\item Every semi-simplicial type $\A$ has a type $\Z~\A$, whose points we call \emph{$\cred{0}$-simplices} of $\A$.
  Thus we may also write $\A_\cred{0} = \Z~\A$.
\item Every $\cred{0}$-simplex $\x : \A_\cred{0}$ gives rise to a semi-simplicial type $\S~\A~\x$ over $\A$, called the \emph{slice} of $\A$ over $\x$.
  Of course, if we don't know what a semi-simplicial type is, we can't be expected to know what one semi-simplicial type over another one is --- at least, not completely.
  But we do know that every semi-simplicial type $\A$ has an underlying type $\Z~\A$, so it stands to reason that a semi-simplicial type $\B$ \emph{over} $\A$ should in particular have an underlying type $\Z\d~\B$ \emph{over $\Z~\A$}, i.e.\ a type family $\Z\d~\B: \Z~\A \to \Type$.
  Thus, in particular for every $\x : \A_\cred{0}$ we have $\Z\d~(\S~\A~\x) : \A_\cred{0} \to \Type$, hence for every additional $\y: \A_\cred{0}$ we have a type $\Z\d~(\S~\A~\x)~\y$.
  We call this the type $\A_\cred{1}~\x~\y$ of \emph{$\cred{1}$-simplices} from $\x$ to $\y$.
\item Now we know that every semi-simplicial type $\A$ has not only an underlying type of $\cred{0}$-simplices $\A_\cred{0}$, but for every $\cblu{x\sub{01}},\cblu{x\sub{10}} : \A_\cred{0}$ a type $\A_\cred{1}~\cblu{x\sub{01}}~\cblu{x\sub{10}}$ of $\cred{1}$-simplices.
  Therefore, it stands to reason that any semi-simplicial type $\B$ \emph{over} $\A$ should also have, not only a type family $\B_\cred{0}$ over $\A_\cred{0}$, but a type family $\B_\cred{1}$ over $\A_\cred{1}$.
  Thinking of $\A_\cred{1}$ as an indexed representation of a span $\A_\cred{0} \leftarrow \int~\A_\cred{1} \to \A_\cred{0}$, we deduce that $\B_\cred{1}$ should be an indexed representation of a span morphism
  \[
    \begin{tikzcd}
      \int~\B_\cred{0} \ar[d] & \int\int~\B_\cred{1} \ar[l] \ar[r] \ar[d] & \int~\B_\cred{0} \ar[d] \\
      \A_\cred{0} & \int~\A_\cred{1} \ar[l] \ar[r] & \A_\cred{0}
    \end{tikzcd}
  \]
  and therefore we should have
  \[\B_\cred{1} : (\cblu{y\sub{01}} : \A_{\cred{0}})~(\cblu{z\sub{01}} : \B_{\cred{0}}~\cblu{y\sub{01}})~(\cblu{y\sub{10}} : \A_{\cred{0}})~(\cblu{z\sub{10}} : \B_{\cred{0}}~\cblu{y\sub{10}})~(\cblu{\gamma\sub{11}} : \A_{\cred{1}}~\cblu{y\sub{01}}~\cblu{y\sub{10}}) \to \Type.\]

  More precisely, since every $\cred{0}$-simplex $\cblu{y\sub{01}}$ of $\A$ gives rise to a semi-simplicial type $\S~\A~\cblu{y\sub{01}}$ over $\A$, any $\cred{0}$-simplex $\cblu{z\sub{01}}$ of $\B$ over $\cblu{y\sub{01}}$ should give rise to a semi-simplicial type $\S\d~\B~\cblu{y\sub{01}}~\cblu{z\sub{01}}$ over both $\B$ and $\S~\A~\cblu{y\sub{01}}$.
  But the common dependence on $\A$ should be shared, so $\S\d~\B~\cblu{y\sub{01}}~\cblu{z\sub{01}}$ should live over the cospan $\B \rightarrow \A \leftarrow \S~\A~\cblu{y\sub{01}}$:
  \[
    \begin{tikzcd}
      \B \ar[d] & \S\d~\B~\cblu{y\sub{01}}~\cblu{z\sub{01}} \ar[l] \ar[d] \\
      \A & \S~\A~\cblu{y\sub{01}} \ar[l]
    \end{tikzcd}
  \]
  Passing to $\cred{0}$-simplices, this means that $\Z\d\d\bkt{\S\d~\B~\cblu{y\sub{01}}~\cblu{z\sub{01}}}$ should be a type dependent on $\cblu{y\sub{10}}:\A_{\cred{0}}$, $\cblu{z\sub{10}} : \B_{\cred{0}}~\cblu{y\sub{10}}$, and $\cblu{y_{11}}:(\S~\A~\cblu{y\sub{01}})_\cred{0}~\cblu{y\sub{10}} \equiv \A_{\cred{1}}~\cblu{y\sub{01}}~\cblu{y\sub{10}}$.
  Thus we can define the 1-simplices of $\B$ as $\B_{\cred{1}}~\cblu{y\sub{01}}~\cblu{z\sub{01}}~\cblu{y\sub{10}}~\cblu{z\sub{10}}~\cblu{\gamma\sub{11}} \equiv \Z\d\d\bkt{\S\d~\B~\cblu{y\sub{01}}~\cblu{z\sub{01}}}~\cblu{y\sub{10}}~\cblu{z\sub{10}}~\cblu{\gamma\sub{11}}$.

  In particular, therefore, since for any $\x : \A_\cred{0}$ we have a semi-simplicial type $\S~\A~\x$ over $\A$, we have
  \[ \bkt{\S~\A~\x}_\cred{1} : (\cblu{y\sub{01}} : \A_\cred{0})~(\cblu{z\sub{01}} : \bkt{\S~\A~\x}_\cred{0}~\cblu{y\sub{01}})~(\cblu{y\sub{10}} : \A_\cred{0})~(\cblu{\cblu{z\sub{10}}} : \bkt{\S~\A~\x}_\cred{0}~\cblu{y\sub{10}})~(\cblu{\cblu{\gamma\sub{11}}} : \A_\cred{1}~\cblu{y\sub{01}}~\cblu{y\sub{10}}) \to \Type. \]
  Since $\bkt{\S~\A~\x}_\cred{0}~\cblu{y\sub{01}} \equiv \A_\cred{1}~\x~\cblu{y\sub{01}}$ by definition, this is equivalently
  \[ \bkt{\S~\A~\x}_\cred{1} : (\cblu{y\sub{01}}:\A_\cred{0})~(\cblu{z\sub{01}} : \A_\cred{1}~\x~\cblu{y\sub{01}})~(\cblu{y\sub{10}}:\A_\cred{0})~(\cblu{\cblu{z\sub{10}}} : \A_\cred{1}~\x~\cblu{y\sub{10}})~(\cblu{\cblu{\gamma\sub{11}}} : \A_\cred{1}~\cblu{y\sub{01}}~\cblu{y\sub{10}}) \to \Type. \]
  Renaming the variables as $\cblu{y\sub{ab}}\equiv \cblu{x\sub{ab0}}$, $\cblu{\gamma\sub{ab}}\equiv \cblu{\beta\sub{ab0}}$, and $\cblu{z\sub{ab}}\equiv \cblu{\beta\sub{ab1}}$, and writing $\x\equiv \cblu{x\sub{001}}$, this becomes
  \[ \bkt{\S~\A~\cblu{x\sub{001}}}_\cred{1} : (\cblu{x\sub{010}}:\A_\cred{0})~(\cblu{\beta\sub{011}} : \A_\cred{1}~\cblu{x\sub{001}}~\cblu{x\sub{010}})~(\cblu{x\sub{100}}:\A_\cred{0})~(\cblu{\cblu{\beta\sub{101}}} : \A_\cred{1}~\cblu{x\sub{001}}~\cblu{x\sub{100}})~(\cblu{\cblu{\beta\sub{110}}} : \A_\cred{1}~\cblu{x\sub{010}}~\cblu{x\sub{100}}) \to \Type. \]
  Thus, this is precisely correct to be a type of $\cred{2}$-simplices:
  \[ \A_\cred{2}~\cblu{x\sub{001}}~\cblu{x\sub{010}}~\cblu{\beta\sub{011}}~\cblu{x\sub{100}}~\cblu{\cblu{\beta\sub{101}}}~\cblu{\cblu{\beta\sub{110}}} \equiv \bkt{\S~\A~\cblu{x\sub{001}}}_\cred{1}~\cblu{x\sub{010}}~\cblu{\beta\sub{011}}~\cblu{x\sub{100}}~\cblu{\cblu{\beta\sub{101}}}~\cblu{\cblu{\beta\sub{110}}}. \]
\end{enumerate}
This pattern may be visualised as follows:
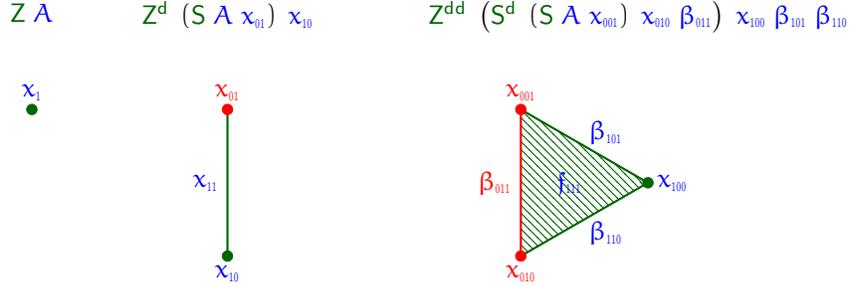
\begin{center}
\begin{tikzpicture}[scale=1.3]
\draw (0,-0.2) node[anchor=south]{$\Z~\A$};
\filldraw[color=darkgreen] (0,-1) circle (1.5pt);
\draw (0,-1) node[anchor=south]{$\cblu{x\sub{1}}$};
\draw (2,-0.282) node[anchor=south]{$\Z\d~\bkt{\S~\A~\cblu{x\sub{01}}}~\cblu{x\sub{10}}$};
\filldraw[color=darkgreen, thick] (2,-1)--(2,-2.5);
\filldraw[color=red] (2,-1) circle (1.5pt);
\draw (2,-1) node[anchor=south]{$\cprime{x\sub{01}}$};
\filldraw[color=darkgreen] (2,-2.5) circle (1.5pt);
\draw (2,-2.5) node[anchor=north]{$\cblu{x\sub{10}}$};
\draw (2,-1.75) node[anchor=east]{$\cblu{x\sub{11}}$};
\draw (6.2,-0.307) node[anchor=south]{$\Z\d\d~\bkt{\S\d~\bkt{\S~\A~\cblu{x\sub{001}}}~\cblu{x\sub{010}}~\cblu{\beta\sub{011}}}~\cblu{x\sub{100}}~\cblu{\beta\sub{101}}~\cblu{\beta\sub{110}}$};
\filldraw[color=darkgreen, thick] (5,-1)--(6.3,-1.75);
\filldraw[color=darkgreen, thick] (5,-2.5)--(6.3,-1.75);
\filldraw[color=darkgreen] (6.3,-1.75) circle (1.5pt);
\fill[pattern=north west lines, pattern color=darkgreen] (5,-1)--(6.3,-1.75)--(5,-2.5);
\filldraw[color=red, thick] (5,-1)--(5,-2.5);
\filldraw[color=red] (5,-1) circle (1.5pt);
\draw (5,-1) node[anchor=south]{$\cprime{x\sub{001}}$};
\filldraw[color=red] (5,-2.5) circle (1.5pt);
\draw (5,-2.5) node[anchor=north]{$\cprime{x\sub{010}}$};
\draw (5,-1.75) node[anchor=east]{$\cprime{\beta\sub{011}}$};
\draw (6.3,-1.75) node[anchor=west]{$\cblu{x\sub{100}}$};
\draw (5.5,-1.75) node{$\cblu{\ff\sub{111}}$};
\draw (5.6,-1.45) node[anchor=south west]{$\cblu{\beta\sub{101}}$};
\draw (5.6,-2.05) node[anchor=north west]{$\cblu{\beta\sub{110}}$};
\end{tikzpicture}
\end{center}

An alternative, and more geometrical, viewpoint, is that the $\n$-simplex is the \emph{cone} of the $(\n-\cred{1})$-simplex.
Thus, if we already know that every semi-simplicial type has a suitably indexed type of $(\n-\cred{1})$-simplices, we can conclude the same about the $n$-simplices as follows.
For every $\cred{0}$-simplex $\x$, the dependent semi-simplicial type $\S~\A~\x$ has a type of \emph{`dependent $(\n-\cred{1})$-simplices'} indexed by the type of $(\n-\cred{1})$-simplices of $\A$.
Thus, an element of this type depends on $\x$ (the cone vertex) as well as an $(\n-\cred{1})$-simplex of $\A$ (the base face, opposite the cone vertex), and its boundary consisting of dependent $\k$-simplices for $\k<\n-\cred{1}$ that form cones over all the faces of the base $(\n-\cred{1})$-simplex sharing the same vertex $\x$.
Together, these form the boundary of an $\n$-simplex.
(In the ordering of variables induced by the above presentation, the simplices in the base face are interspersed with their dependent versions: thus in the case $\n\equiv\cred{2}$ the faces $\cblu{x\sub{010}},\cblu{x\sub{100}},\cblu{\beta\sub{110}}$ form the base $\cred{1}$-simplex with $\cblu{\beta\sub{011}},\cblu{\beta\sub{101}}$ the dependent $\cred{0}$-simplices (i.e. $\cred{1}$-simplices) forming a cone over the boundary $\cblu{x\sub{010}},\cblu{x\sub{100}}$.)
Hopefully this is sufficiently convincing for now; later we will give a precise justification.

As simple and appealing as this \emph{`definition'} is, it is not meaningful in ordinary dependent type theory.
The intuitive claim is that it defines a type $\SST$ by coinduction, with $\Z$ and $\S$ as destructors.
For $\Z$ this is unproblematic (it is not even corecursive).
However, the output of $\S$ is not an element of the type $\SST$ being defined, as would be usual for a corecursive destructor of a coinductive type, but a \emph{`dependent element'}, or \emph{`displayed element'}, of $\SST$ over the input of $\S$.
If we write $\SST\d$ for this putative family of \emph{`displayed elements'}, the types of $\Z$ and $\S$ are
\begin{equation}
  \label{eq:ZS-intro}
  \begin{aligned}
    \Z &: \SST \to \Type \\
    \S &: \bkt{\X :\SST} \to \Z~\X \to \SST\d~\X.
  \end{aligned}
\end{equation}
We would like to regard this as a sort of \emph{`higher coinductive type'}.
Just as a higher \emph{inductive} type can have constructors involving not just elements of the type being defined but also \emph{paths} therein, here we have a putative coinductive type whose destructors involve not just elements of the type being defined but also \emph{`displayed elements'} thereof.
Thus, to make sense of this we need a type theory with a primitive operation $\pblank\d$ associating to a type its family of \emph{`displayed elements'}.
As it turns out, the precise notion of $\pblank\d$ that we require is a variant of \emph{unary internal parametricity}.

\paragraph{External and internal parametricity.}

In general, by `parametricity' we mean a statement that every type (perhaps subject to contextual restrictions; see below) is equipped with a relation (of some arity), and every function (subject to the same restrictions) preserves those relations.
The original form of parametricity, such as in~\cite{wadler:free-theorems}, is a meta-theoretic statement \emph{about} type theory, in which the relations are meta-theoretic, and the contextual restriction is that it applies only to \emph{closed} types and terms (those defined in the empty context).
That is, in this \emph{`external'} parametricity, every closed type is given a relation on its closed terms, which is preserved by every closed function.
For instance, in the unary case, given a closed function $\ec \vdash \f : \A \to \B$, if the $\A$-relation holds of a closed term $\ec\vdash \a:\A$, then the $\B$-relation holds of the closed term $\ec\vdash \f~\a:\B$.
Indeed, the fact that $f$ satisfies this condition is exactly the statement that the $(\A\to\B)$-relation holds of it, and is thus a special case of the `fundamental theorem of logical relations' that \emph{every} closed term satisfies the relation on its type.
Semantically, external parametricity is obtained by interpreting type theory in a `gluing' or `relational' model.

By contrast, in type theories with fully \emph{internal} parametricity such as~\cite{bm:parametricity,bcm:pshf-parametric,moulin2016internalizing}, there is no contextual restriction, and the relations are internal (i.e.\ type families).
In the unary case, this means for any type in any context, say $\Gamma \vdash \A \type$, there is a type family that we will denote $\Gamma,~\x:\A \vdash \A\d~\x \type$, and every function $\Gamma \vdash \f : \A \to \B$ in any context lifts to a function $\Gamma \vdash \f\d : (\x:\A)(\x\cblu{'}:\A\d~\x) \to \B\d~(\f~\x)$ between these type families.
As in the external case, there is a formula for the relation of a function-type:
\begin{equation}
  \label{eq:d-pi-intro}
  \bkt{\bkt{\a : \A} \to \B~\sqbkt{\a}}\d~\f \equiv \bkt{\a : \A} \bkt{\a\cblu{'} : \A\d~\a} \to \B\d~\sqbkt{\a~, \a\cblu{'}}~\bkt{\f~\a}
\end{equation}
which says that so-called \emph{`computability witnesses'} for $\f$ take computability witnesses of $\a$ to computability witnesses of $\f~\a$.
Thus the statement that $\f$ preserves computability witnesses is equivalent to saying it lifts to an element $\f\d$ of this type, and is a special case of a general rule that \emph{every} term $\Gamma \vdash \a:\A$ lifts to a computability witness $\Gamma \vdash \a\d : \A\d~\a$.\footnote{The punning of notation is intentional, and indeed consistent, as we will see: the type $\pblank\d$ coincides with the term $\pblank\d$ applied to elements of the universe.}

Such an internalization of parametricity introduces the new possibility of \emph{iterating} it, leading to $\A\d\d$, $\A\d\d\d$, and so on.
Semantically, this means that each type must be interpreted by a \emph{cubical type} (or set), where the arity of the relations determines the number of `boundary points' on each side of the cubes, and the dependencies of each iterated relation on the previous ones supplies the \emph{faces} of a cube.
Our primary interest is in unary parametricity, where $\A\d$ depends only on one copy of $\A$; in this case the semantics involves \emph{`unary cubes'} where each edge has only one vertex.\footnote{Geometrically, these can be thought of as powers of a half-open interval $[0,1)^n$, or closed octants $[0,\infty)^n$ in $n$-dimensional Euclidean space.}

The rule lifting any $\a$ to $\a\d$ then implies that these cubical types must have \emph{degeneracies}, taking any cube to a higher-dimensional one with some boundaries trivial.
More surprisingly, it seems that to have good computational behavior and semantic models, the cubical types must also include \emph{symmetries} (transpositions): from $\cblu{d} \equiv \cred{\lambda}~\x.~\x\d : \bkt{\x : \X} \to \X\d~\x$, we can either directly obtain the type of $\cblu{d}\d$ as $\bkt{\x : \X} \bkt{\x\cblu{'} : \X\d~\x} \to \X\d\d~\x~\x\cblu{'}~\x\d$ or first perform the computation $\cblu{d}\d \equiv \cblu{\lambda} ~\x~\x\cblu{'}.~\x\cblu{'}\d$ and obtain its type as $\bkt{\x : \X} \bkt{\x\cblu{'} : \X\d~\x} \to \X\d\d~\x~\x\d~\x\cblu{'}$, and it these must be related by a symmetry operation.

The general advantages of internal parametricity over external are clear: we can reason about computability witnesses while staying within a single type theory, using a single proof assistant.
Moreover, internal parametricity allows us to at least try to make sense of a type $\SST$ with the destructors~\eqref{eq:ZS-intro}.
However, fully internal parametricity is not conservative over ordinary type theory: it is highly nonclassical, incompatible with axioms such as the law of excluded middle; and its semantics is not as general as we would like.
We would like our type theory to be interpretable in any $(\infty,1)$-topos (generalizing~\cite{shulman:univinj}), in such a way that our type $\SST$ is interpreted by a category-theoretic \emph{`classifier'} of semi-simplicial objects; but as we have just observed, the semantics of internal parametricity seems to live only in a category of cubical objects.
Now the category of cubical objects in an $(\infty,1)$-topos is again an $(\infty,1)$-topos, so (using~\cite{shulman:univinj}) we can expect to interpret an internally parametric type theory in the latter; but the connection of this interpretation to the \emph{original} $(\infty,1)$-topos is not clear.

\paragraph{Displayed Type Theory.}

Our solution is to use a \emph{less internally} parametric type theory.
We can think of internal parametricity as arising in three stages.
The first stage is an \emph{external} parametricity model, where types are interpreted by pairs $\bkt{\cblu{A_0}:\Type, ~\cblu{A_1} :\cblu{A_0} \to \Type}$.
In this case $\pblank\d$ maps one model to a \emph{different} model, interpreting an ordinary type by a pair of types; this yields parametricity results as metatheorems.
In the second stage, to make $\A\d$ live in the same model as $\A$, we iterate this construction infinitely many times; now types are interpreted by \emph{semi-cubical types}, with faces but not degeneracies.
Finally, in the third stage we add a \emph{`degeneracy'} operation $\bkt{\x:\A} \to \A\d~\x$ making every \emph{term} parametric.
This allows proving parametricity theorems \emph{inside} the type theory, and semantically moves us from semi-cubical types to cubical ones.

The solution to our problem is to stop after the \emph{second} stage.
This is advantageous semantically because semi-cubical sets are presheaves on a \emph{direct} category (i.e.\ covariant diagrams on an \emph{inverse} category), and in this case the model construction is much more concrete.
Specifically, in \cite{shulman15} it was shown that from any model of univalent dependent type theory inside of a type theoretic fibration category, one may form a derived model of Reedy fibrant presheaves on any direct category, with the type formers in the presheaf model constructed inductively in terms of those in the original model.
In particular, in degree $\cred{0}$ all the type-formers act exactly as they do in the original model.
Thus, semantically we can be sure that all our constructions, including $\SST$, specialise to something meaningful in an arbitrary $(\infty,1)$-topos.

dTT is a syntax corresponding to this model, which is likewise intermediate between external parametricity and fully internal parametricity: its parametricity primitive $\pblank\d$ has a contextual restriction that is weaker than the \emph{`only closed terms'} requirement of external parametricity, but stronger than the \emph{`any context goes'} laxity of internal parametricity.
We start by observing that in either cubical \emph{or} semi-cubical sets, the semantically fundamental parametricity operation actually changes the context: given $\cblu{\Gamma} \vdash \t : \A$, one has $\cblu{\Gamma}\D \vdash \t\d : \A\d~\t$, where $\cblu{\Gamma}\D$ augments $\Gamma$ by computability witnesses of all its variables.
For cubical sets with degeneracies, we can deduce a version of $\pblank\d$ that doesn't change the context by substituting along a degeneracy map $\Gamma \to \Gamma\D$ (e.g.\ this is the isomorphism between the \emph{`global'} and \emph{`local'} models of~\cite{acks:iparam-noint}).
But for semi-cubical sets this is impossible, so we have to bite the bullet and deal with context-modifying operations.

The notion of a non-binding operation that changes contexts is familiar from the realm of modal logic, where, to first approximation, a proof of necessity of some proposition, i.e. of $\sq~\A$, may only use necessary assumptions.
Modalities in dependent type theory have previously been used to internalise meta-theoretic operations that don't make sense in arbitrary contexts, such as the right adjoint to a $\Pi$-type in~\cite{lops:internal-universes}, and we use them similarly here.
Specifically, in dTT we have a modality $\trisq$ that partially internalises the notion of \emph{`closed term'} appearing in external parametricity, and which restricts the domain of $\pblank\d$.
Thus the only analogue of the above $\cblu{d} \equiv \cred{\lambda}~\x.~\x\d : \bkt{\x : \X} \to \X\d~\x$ in dTT is $\cblu{\widetilde{d}} \equiv \cred{\lambda}~\x.~\x\d : \bkt{\x :^\trisq \X} \to \X\d~\x$.
In particular, modal variables are protected from alteration \emph{by} $\pblank\d$, so that we have $\cblu{\widetilde{d}}\d \equiv \cred{\lambda}~\x.~\x\d\d : \bkt{\x :^\trisq \X} \to \X\d\d~\x~\x\d~\x\d$, thereby avoiding the need for symmetry.

In fact, to emphasise further that dTT retains general semantics over an arbitrary $(\infty,1)$-topos, we will use a multimodal type theory~\cite{gknb:mtt,gckgb:fitchtt} with two modes, one for the original topos and the other for the topos of semi-cubical sets.
These modes are related by modalities $\tri$ (the constant semi-cubical type), $\dia$ (the $\cred{0}$-cubes of a cubical type), and $\sq$ (the limit of a cubical type), and $\trisq = \tri\circ\sq$ is a composite endo-modality.
Only $\trisq$ is necessary to formulate display, but the other modalities are also useful to have around: in particular, $\dia$ internalises the process of passing from the model in semi-cubical types to the original model in $\MC$.
For instance, $\dia~\SST$ is what corresponds semantically to the classifier of semi-simplicial objects in $\MC$.

Furthermore, display itself may be thought of as a modality, albeit one that is indexed over the original type.
Display falls into into the new and yet underdeveloped framework of \emph{indexed modalities}, such as the path types of cubical type theory (treated modally in~\cite{gckgb:fitchtt}) and the identity types of the forthcoming \emph{Higher Observational Type Theory (HOTT)}~\cite{HOTT,acks:iparam-noint}.
Moreover, analogously to those cases, we could formulate display either as an inert type-former defined by abstraction over an \emph{`interval'} (like path types in cubical type theory), or as an operation that computes on most other canonical type-formers (like identity types in HOTT).
In this paper we make the latter choice, so that rules like~\cref{eq:d-pi-intro} are actually definitional equalities.
Formulating such rules computationally is actually easier for dTT than for HOTT, due mainly to the lack of symmetry (although there is a tradeoff, since the presence of modalities is an extra complicating factor).

Until now we have been talking about semi-cubical types to make the connection with parametricity clear, but in the unary case there is an intriging coincidence: the unary semi-cube category is isomorphic to the \emph{augmented} semi-\emph{simplex} category,\footnote{To see this geometrically, note that the $(\n+\cred{1})$-dimensional octant $[\cred{0},\cred{\infty})^{\n+\cred{1}}$ contains a standard face-preserving embedding of the $\n$-simplex, $\cred{\oldDelta}^\n = \{ (\x_\cred{0},\dots,\x_\n) \in [\cred{0},\cred{\infty})^{\n+\cred{1}} \mid \x_\cred{0}+\cdots +\x_\n = \cred{1}\}$, including the augmentation case $\cred{\oldDelta}^\minusone = \cred{\emptyset}$.} with a dimension shift: the $\n$-cube corresponds to the $(\n-\cred{1})$-simplex.
For this reason we refer to the two modes in our theory as the \emph{discrete mode} $\dm$ and the \emph{simplicial mode} $\sm$.
Thus, dTT can actually internalise semi-simplicial types in \emph{two} ways: as the coinductive type $\SST$ mentioned above, and as the \emph{universe} of types at the simplicial mode.
The latter suggests that dTT could also be used similarly to simplicial type theory, with types at the simplicial mode treated as synthetic (augmented semi-) simplicial types.

(Note that an augmented semi-simplicial type can be viewed as a family of ordinary semi-simplicial sets indexed by the type of $(\minusone)$-simplices.
Thus, our observation that augmented semi-simplicial types support a better internal language than ordinary ones is analogous to the observation of~\cite{rfl:synthetic-spectra} that \emph{parametrised} spectra are likewise preferable to unparametrised ones.)

\paragraph{Displayed structures.}

Our terminology \emph{display} for the operation $\d$ is inspired by the fact that when applied to record types whose elements are algebraic structures, it produces \emph{displayed structures} of the corresponding sort.
Here a \emph{`displayed structure'} over a structure $\B$ is a structure $\A$ of the same kind with a structure map $\A \to \B$, but reformulated in terms of the corresponding family of fibres $\B \to \Type$.
Working with displayed structures rather than morphisms is a technique for enforcing definitional equalities on images in $\B$.

The most common displayed structure is a \emph{displayed category}; here the terminology was introduced by~\cite{al:dispcat}.
This arises from the record type of categories, defined in the usual dependently typed way (where we omit the axioms for concision):
\begin{lstlisting}
(*\textbf{record}*) (*$\cred{\mathsf{Cat}}$*) : (*$\Type$*) (*\textbf{where}*)
  (*$\cred{\mathsf{ob}}$*) : (*$\Type$*)
  (*$\cred{\mathsf{hom}}$*) : (*$\cred{\mathsf{ob}}$*) (*$\to$*) (*$\cred{\mathsf{ob}}$*) (*$\to$*) (*$\Type$*)
  (*$\cred{\mathsf{id}}$*) : ((*$\x$*) : (*$\cred{\mathsf{ob}}$*)) (*$\to$*) (*$\cred{\mathsf{hom}}$*) (*$\x$*) (*$\x$*)
  (*$\cred{\mathsf{comp}}$*) : {(*$\x$*) (*$\y$*) (*$\z$*) : (*$\cred{\mathsf{ob}}$*)} (*$\to$*) (*$\cred{\mathsf{hom}}$*) (*$\y$*) (*$\z$*) (*$\to$*) (*$\cred{\mathsf{hom}}$*) (*$\x$*) (*$\y$*) (*$\to$*) (*$\cred{\mathsf{hom}}$*) (*$\x$*) (*$\z$*)
  (*$\cred{\textbf{\ldots}}$*)
\end{lstlisting}
We do not discuss record types (including $\Sigma$-types) in this paper, but the extension of $\d$ to them produces another record type whose fields have $\d$ applied to them.
For instance, from a $\Sigma$-type:\\
\begin{minipage}{\textwidth}
\begin{lstlisting}
(*\textbf{record}*) (*$\Sigma$*) ((*$\A$*) : (*$\Type$*)) ((*$\B$*) : (*$\A$*) (*$\to$*) (*$\Type$*)) : (*$\Type$*) (*\textbf{where}*)
  (*$\cred{\mathsf{fst}}$*) : (*$\A$*)
  (*$\cred{\mathsf{snd}}$*) : (*$\B$*) (*$\cred{\mathsf{fst}}$*)
\end{lstlisting}
\end{minipage}
We obtain:
\begin{lstlisting}
(*\textbf{record}*) (*$\Sigma\d$*) ((*$\A$*) : (*$\Type$*)) ((*$\A\cblu{'}$*) : (*$\Type\d$*) (*$\A$*)) ((*$\B$*) : (*$\A$*) (*$\to$*) (*$\Type$*))
         ((*$\B\cblu{'}$*) : ((*$\A$*) (*$\to$*) (*$\Type$*))(*$\d$*) (*$\B$*)) : (*$\Type\d$*) ((*$\Sigma$*) (*$\A$*) (*$\B$*)) (*\textbf{where}*)
  (*$\cred{\mathsf{fst}}\d$*) : (*$\A\d$*) (*$\cred{\mathsf{fst}}$*)
  (*$\cred{\mathsf{snd}}\d$*) : ((*$\B$*) (*$\cred{\mathsf{fst}}$*))(*$\d$*) (*$\cred{\mathsf{snd}}$*)
\end{lstlisting}
Applying~\eqref{eq:d-pi-intro} and the similar rule $\Type\d~\A \equiv \A \to \Type$, this becomes:
\begin{lstlisting}
(*\textbf{record}*) (*$\Sigma\d$*) ((*$\A$*) : (*$\Type$*)) ((*$\A\cblu{'}$*) : (*$\A$*) (*$\to$*) (*$\Type$*)) ((*$\B$*) : (*$\A$*) (*$\to$*) (*$\Type$*))
         ((*$\B\cblu{'}$*) : ((*$\x$*) : (*$\A$*)) (*$\to$*) (*$\A\cblu{'}$*) (*$\x$*) (*$\to$*) (*$\B$*) (*$\x$*) (*$\to$*) (*$\Type$*)) ((*$\s$*) : (*$\Sigma$*) (*$\A$*) (*$\B$*)) : (*$\Type$*) (*\textbf{where}*)
  (*$\cred{\mathsf{fst}}\d$*) : (*$\A\cblu{'}$*) ((*$\cred{\mathsf{fst}}$*) (*$\s$*))
  (*$\cred{\mathsf{snd}}\d$*) : (*$\B\cblu{'}$*) ((*$\cred{\mathsf{fst}}$*) (*$\s$*)) (*$\cred{\mathsf{fst}}\d$*) ((*$\cred{\mathsf{snd}}$*) (*$\s$*))
\end{lstlisting}
In a similar way, the above definition of the record type of categories yields:
\begin{lstlisting}
(*\textbf{record}*) (*$\cred{\mathsf{Cat}}\d$*) ((*$\MC$*) : (*$\cred{\mathsf{Cat}}$*)) : (*$\Type$*) (*\textbf{where}*)
  (*$\cred{\mathsf{ob}}\d$*) : (*$\cred{\mathsf{ob}}$*) (*$\MC$*) (*$\to$*) (*$\Type$*)
  (*$\cred{\mathsf{hom}}\d$*) : ((*$\x$*) : (*$\cred{\mathsf{ob}}$*) (*$\MC$*)) ((*$\x\cblu{'}$*) : (*$\cred{\mathsf{ob}}\d$*) (*$\x$*)) ((*$\y$*) : (*$\cred{\mathsf{ob}}$*) (*$\MC$*)) ((*$\y\cblu{'}$*) : (*$\cred{\mathsf{ob}}\d$*) (*$\y$*)) (*$\to$*) (*$\cred{\mathsf{hom}}$*) (*$\MC$*) (*$\x$*) (*$\y$*) (*$\to$*) (*$\Type$*)
  (*$\cred{\mathsf{id}}\d$*) : ((*$\x$*) : (*$\cred{\mathsf{ob}}$*) (*$\MC$*)) ((*$\x\cblu{'}$*) : (*$\cred{\mathsf{ob}}\d$*) (*$\x$*)) (*$\to$*) (*$\cred{\mathsf{hom}}\d$*) (*$\x$*) (*$\x\cblu{'}$*) (*$\x$*) (*$\x\cblu{'}$*) ((*$\cred{\mathsf{id}}$*) (*$\MC$*) (*$\x$*))
  (*$\cred{\mathsf{comp}}\d$*) : {(*$\x$*) : (*$\cred{\mathsf{ob}}$*) (*$\MC$*)} {(*$\x\cblu{'}$*) : (*$\cred{\mathsf{ob}}\d$*) (*$\x$*)} {(*$\y$*) : (*$\cred{\mathsf{ob}}$*) (*$\MC$*)} {(*$\y\cblu{'}$*) : (*$\cred{\mathsf{ob}}\d$*) (*$\y$*)} {(*$\z$*) : (*$\cred{\mathsf{ob}}$*) (*$\MC$*)}
    {(*$\z\cblu{'}$*) : (*$\cred{\mathsf{ob}}\d$*) (*$\z$*)} ((*$\alpha$*) : (*$\cred{\mathsf{hom}}$*) (*$\MC$*) (*$\y$*) (*$\z$*)) ((*$\alpha\cblu{'}$*) : (*$\cred{\mathsf{hom}}\d$*) (*$\y$*) (*$\y\cblu{'}$*) (*$\z$*) (*$\z\cblu{'}$*) (*$\alpha$*)) ((*$\cblu{\beta}$*) : (*$\cred{\mathsf{hom}}$*) (*$\MC$*) (*$\x$*) (*$\y$*))
    ((*$\cblu{\beta'}$*) : (*$\cred{\mathsf{hom}}\d$*) (*$\x$*) (*$\x\cblu{'}$*) (*$\y$*) (*$\y\cblu{'}$*) (*$\cblu{\beta}$*)) (*$\to$*) (*$\cred{\mathsf{hom}}\d$*) (*$\x$*) (*$\x\cblu{'}$*) (*$\z$*) (*$\z\cblu{'}$*) ((*$\cred{\mathsf{comp}}$*) (*$\MC$*) (*$\alpha$*) (*$\cblu{\beta}$*))
  (*$\cred{\textbf{\ldots}}$*)
\end{lstlisting}
Thus a displayed category over $\MC$ has a type of objects indexed by those of $\MC$, types of morphisms indexed by pairs of objects-over-objects and by a morphism of $\MC$, identity and composition operations on displayed objects and morphisms that lie strictly over those in $\MC$, and similarly for the axioms.

As observed in~\cite{al:dispcat}, one use of displayed categories is to state definitions such as Grothendieck fibrations in terms of the existence of cartesian liftings strictly over any morphism in $\MC$, without internalizing definitional equality.
Another is to construct categories and prove their properties in a modular way out of dependent pieces, just as we do for types using $\Sigma$-types and more general records.
It is \emph{`well-known'} by now that any sort of algebro-categorical structure has a \emph{`displayed version'} --- for instance, displayed bicategories were used in~\cite{afmvv:bicat-unimath} to prove univalence modularly --- but to our knowledge this has not previously been formalised.
Our \emph{Displayed Type Theory (dTT)} automatically generates the displayed version of any notion definable in type theory; hence the name.

\paragraph{Outline of the paper.}

The rest of this paper has three parts.
In \cref{sec:syntax} we describe the general syntax of dTT, including the modalities, the operation of display (in various different forms), and how they compute.
We do not prove any canonicity or normalization results, but we conjecture that they hold.

In \cref{sec:ssts} we extend the syntax of dTT to define a type $\SST$ of semi-simplicial types.
In fact, we obtain this as a special case of a general notion of \emph{`displayed coinductive type'}, which is easier to work with abstractly, and also includes other important examples such as the type of semi-simplicial morphisms between two semi-simplicial types.%
\footnote{One might hope that it would also include the displayed versions $\SST\d$, $\SST\d\d$, etc., but this does not seem to be the case unless we add symmetry to our theory.}
Then we explore a few applications, to make the point that this coinductive notion of semi-simplicial type is useful and practical.

Finally, in \cref{sec:semantics} we consider the semantics of dTT.
In particular, we will show that from any model $\MC$ of ordinary dependent type theory with countable inverse limits (roughly as considered in~\cite{kraus:proptrunc}), we can construct a model of dTT whose discrete mode is $\MC$ and whose simplicial mode is the category of Reedy fibrant augmented semi-simplicial diagrams in $\MC$, and that this model supports displayed coinductive types including a type $\SST$ of semi-simplicial types.
The underlying ordinary type theory of this model at the simplicial mode is an instance of the inverse diagram models of~\cite{shulman15,kl:hoinvdia}, but we construct it more explicitly by hand so as to be able to verify the needed formulas for the additional operations of dTT.

Thus, although dTT is (apparently) not conservative over ordinary dependent type theory, we can isolate exactly a kind of extra infinitary structure that yields a well-behaved theory for working with semi-simplicial types, which precisely includes the original model at one mode.
In particular, by~\cite{shulman:univinj} any $(\infty,1)$-topos can be presented by a type-theoretic model topos, which is a model of type theory with countable inverse limits, and thus also yields a model of dTT.
However, an object with the \emph{internal} universal property of $\SST$ expressable in dTT has the potential to exist even in models that lack such infinitary limits, which may have implications for a notion of elementary $(\infty,1)$-topos.

\paragraph{Acknowledgements.} Both authors would like to thank Tim Campion for bringing the binary ordering to their attention, via a talk given by Emily Riehl on her joint work with Tim. Astra would also like to thank Emily for for many discussions in the course of weekly advising meetings. Further, Astra is grateful to Steve Awodey for hosting her during the months of March and April 2023 at Carnegie Mellon University, and Emily for hosting her during the Fall 2023 semester at Johns Hopkins Univeristy. Many of the initial ideas regarding the semantics of dTT were developed during the CMU visit, and our construction of the simplicial model was worked out during the JHU visit. Mike is grateful to Thorsten Altenkirch and Ambrus Kaposi for many useful conversations while developing Higher Observational Type Theory that have also informed this work.

\setcounter{secnumdepth}{5}

\clearpage

\section{Syntax}
\label{sec:syntax}

As suggested in the introduction, dTT is based on a modal type theory roughly in the style of~\cite{gknb:mtt,gckgb:fitchtt}, with two modes, one for discrete types and one for (augmented semi-)simplicial types.
It then adds a notion of \emph{`display'} at the simplicial mode that partially internalises unary parametricity.

In addition, the general form of display, which is needed to state the computation rules for simple display, incorporates dependence on an arbitrary telescope (i.e.\ context extension).
Thus, we also have to include a calculus of telescopes in the theory.\footnote{It would probably be possible to collapse this to dependence on a single type, using $\Sigma$-types instead of telescope extension, as in~\cite{acks:iparam-noint}, but this would be unaesthetic and less practical for implementation.}
The fully general calculus of telescopes and display involves a lot of operations, but in syntax and in most models they are all definable from a smaller number of primitives.

This section is organised as follows.
In \cref{subsec:mode-theory} we define the mode theory, which is a 2-category describing the structure of the modal operators.
Then in \cref{subsec:modal-type-theory} we give the rules for the underlying modal type theory, with modalities but not display.

In \cref{sec:tel} we introduce the most basic notions of the telescope calculus: telescopes, partial substitutions (elements of telescopes), and types and terms dependent on a specified telescope (which we call \emph{`meta-abstractions'}).
These basic notions suffice to give the rules for display, defined mutually with a similar but non-indexed operation on telescopes that we call d\'ecalage, in \cref{sec:dd}.

The remaining two sections introduce further operations that are all essentially \emph{`definable'} in terms of the previous ones.
This is not strictly true at the level of algebraic syntax, where telescopes are just an additional sort of a generalised algebraic theory.
But in a model where telescopes are defined to be finite lists of types --- which is an option in any model, both the free syntactic model and in semantic models arising from categories --- the laws satisfied by these operations characterise them uniquely.
Specifically, in \cref{sec:tma2} we introduce meta-abstracted telescopes, telescope concatenation, and $\Pi$-telescopes, and then in \cref{sec:dtel} we introduce display for telescopes, and d\'ecalage for dependent telescopes.
These operations will be used in \cref{sec:ssts} to formulate displayed coinductive types, including the type of semi-simplicial types.

\subsection{The mode theory} \label{subsec:mode-theory}

We begin with a modal type theory based on the following 2-category $\M$:
\begin{itemize}
\item there are two modes (objects), $\dm$ for \emph{discrete} and $\sm$ for
\emph{simplicial}
\item there are five nonidentity morphisms, forming hom-posets:
  \begin{alignat*}{2}
    \M(\dm,\dm) &= \set{\id_\dm}&\qquad
    \M(\dm,\sm) &= \set{\tri}\\
    \M(\sm,\dm) &= \set{ \sq \le \dia }&\qquad
    \M(\sm,\sm) &= \set{ \trisq \le \id_\sm \le \tridia }
  \end{alignat*}
\item composition is defined by the following tables (plus identity laws)
  \[
    \begin{array}{cc|cccc}
      && \multicolumn{3}{c}{\cblu{\oldrho}} \\
      & \nu \circ \cblu{\oldrho} &  \tri & \tridia & \trisq\\\hline
      \multirow{4}{*}{$\nu$}& \dia & \id_\dm  & \dia & \sq \\
      & \sq  & \id_\dm  & \dia & \sq \\
      & \tridia   & \tri  & \tridia & \trisq \\
      & \trisq  & \tri  & \tridia & \trisq
    \end{array}
    \qquad
    \begin{array}{cc|cc}
      && \multicolumn{2}{c}{\cblu{\oldrho}} \\
      &\nu \circ \cblu{\oldrho} & \dia & \sq \\\hline
      \nu &\tri & \tridia & \trisq
    \end{array}
  \]
\end{itemize}
Intuitively, $\tri$ takes a discrete type and forms the constant (augmented semi-)simplicial type, while $\dia$ takes the $(\minusone)$-simplices of an (augmented semi-)simplicial type and $\sq$ takes the limit of an (augmented semi-)simplicial diagram.

One verifies that the following adjunctions hold in $\M$:
\begin{mathpar}
  \dia \dashv \tri \dashv \sq
  \and
  \tridia \dashv \trisq
  \and
  \id_\p \dashv \id_\p
\end{mathpar}
Thus every morphism in $\M$, except for $\dia$ and $\tridia$, has a right
adjoint. We refer to the morphisms $\dia$ and $\tridia$ as \emph{hazardous}, and
the others \emph{safe}. \bbox

\subsection{The modal type theory}
\label{subsec:modal-type-theory}

The basic syntactic structure of dTT follows MTT~\cite{gknb:mtt}.
Following Coquand, we parametrize the type judgment by a universe level; we assume these form a linear hierarchy generated by $\lzero$ and $\lsuc$, with a join operation $\join$, giving the judgement $\ell \level$.
Each mode has
its contexts, substitutions, types, and terms, so we have the following
judgements where $\p$ denotes an arbitrary mode ($\dm$ or $\sm$).
\begin{mathpar}
  \Gamma \ctx_\p
  \and
  \Gamma \vdash_\p \A \type_{\;\!\ell}
  \and
  \Gamma \vdash_\p \t: \A
  \and
  \theta : \Gamma \Rightarrow_\p \Theta
\end{mathpar}
Formally speaking, the inference rules we will give below for these judgments should be interpreted as defining a Generalised Algebraic Theory with these four generating sorts.
Later, we will also introduce some additional sorts.

\subsubsection{Contexts}
\label{sec:contexts}

Contexts are built up from the empty contexts by extending with modally
annotated variables and applying locks associated to modalities:
\begin{mathpar}
  \infer{\p \mode}{\ec_\p \ctx_\p}
  \and
  \infer{\mu:\p\to \q \\
    \Gamma \ctx_\q}{\bkt{\Gamma,~\lock_\mu} \ctx_\p}
  \and
  \infer{\mu:\p\to \q \\
    \Gamma \ctx_\q \\ \Gamma,~\lock_\mu \vdash_\p \A \type_{\;\!\ell}}{\bkt{\Gamma,~\x
    :^\mu \A} \ctx_\q}
\end{mathpar}
We additionally enforce the functoriality of locking, and the fact that some
locks preserve empty contexts.
\begin{mathpar}
  \bkt{\Gamma,~\lock_{\id_\p}} \equiv \Gamma
  \and
  \bkt{\Gamma,~\lock_\mu,~\lock_\nu} \equiv \bkt{\Gamma,~\lock_{\mu\circ\nu}}
  \and
  \bkt{\ec_\dm,~\lock_\sq} \equiv \ec_\sm
  \and
  \bkt{\ec_\sm,~\lock_\tri} \equiv \ec_\dm
\end{mathpar}
In fact, the last equality follows from the other three, since $\bkt{\ec_\sm,~\lock_\tri} = \bkt{\ec_\dm,~\lock_\sq,~\lock_\tri} = \bkt{\ec_\dm,~\lock_{\sq\circ\tri}} = \bkt{\ec_\dm,~\lock_{\id_\dm}} = \ec_{\dm}$.
Note that $\lock_\dia$ does \emph{not} preserve empty contexts.

These equalities mean that contexts no longer have a unique presentation using
the above rules. However, there are ways to select a canonical presentation for
any context. One is to interpret the above equalities as directed rewrites and
work with context presentations that are normal for this rewriting system; thus
there are no identity locks, no repeated locks, and $\lock_\sq$ and $\lock_\tri$ never occur immediately after an empty context.
Another way is to require that exactly one lock appears in between
any two variables.

Semantically, each lock is left adjoint to its corresponding modality.
Thus, $\lock_\sq$ is semantically the same as $\tri$, while $\lock_\tri$ is semantically the same as $\dia$.
The other lock, $\lock_\dia$, is not reified internally by a modality: intuitively, it takes a discrete type and makes it an (augmented semi-)simplicial type that is \emph{empty} at all dimensions $\n>\minusone$.

In particular, $\lock_\sq$ and $\lock_\tri$ have the further left adjoints $\lock_\tri$ and $\lock_\dia$, respectively, so that we will be able to represent their modalities $\sq$ and $\tri$ in Fitch-style as in~\cite{gckgb:fitchtt,shulman:matt}.

As far as $\lock_\dia$ goes, we can say that it is fully faithful and hence an equivalence onto its image, which consists of the simplicial types that are empty in dimensions $\n>\minusone$.
Moreover, since the initial object is strict, this subcategory is a \emph{sieve}: if we have a morphism $\Gamma\to \Delta$ and $\Delta$ lies in this subcategory, then so does $\Gamma$.
Therefore, while $\lock_\dia$ is not a right adjoint, it is a \emph{parametric} right adjoint, so we can still use the method of of~\cite{gckgb:fitchtt} for $\dia$.

However, rather than postulating an uninterpreted parametric left adjoint of $\lock_\dia$ as in~\cite{gckgb:fitchtt}, we can use our knowledge about how this left adjoint is defined semantically to give more specific rules.
Specifically, we can identify its domain as the subcategory of contexts in the essential image of $\lock_\dia$, and on that subcategory the left adjoint actually coincides with $\lock_\tri$.
To represent this semantically, we say that an $\sm$-context is \emph{flat} if, intuitively, it contains a $\lock_\dia$ which is not to the left of any $\lock_\tri$.
Formally, flatness is a predicate on contexts (an additional sort of the GAT) characterised by the following rules:
\begin{mathpar}
  \infer{\Gamma \ctx_\dm}{(\Gamma,~\lock_\dia) \flat}
  \and
  \infer{\Gamma \ctx_\sm \\ \Gamma \flat \\ \mu : \p \to \sm \\
  \Gamma,~\lock_\mu \vdash_\p \A \type_{\;\!\ell}}{(\Gamma,~\x :^\mu \A) \flat}
\end{mathpar}
Semantically, we think of these as the $\sm$-contexts that are empty above dimension $\n>\minusone$, on which the parametric left adjoint of $\lock_\dia$ will act as $\lock_\tri$.
(In our actual model we will do something more clever to avoid assuming the existence of strict initial objects.) \bbox

\subsubsection{Substitutions}
\label{sec:substitutions}

The judgment $\theta: \Gamma \Rightarrow_\p \Theta$ says that $\theta$ is a substitution from context $\Gamma$ to context $\Theta$ at mode $\p$.
It is generated by the following generating rules, which are the same as those given for MTT in~\cite{gknb:mtt}
\begin{mathpar}
  \infer{\Gamma \ctx_\p}{\id_\Gamma : \Gamma \Rightarrow_\p \Gamma}
  \and
  \infer{\Gamma \ctx_\p}{\esub_\p : \Gamma \Rightarrow_\p \ec_\p}
  \and
  \infer{\mu:\p\to \q \\ \Gamma \ctx_\q \\ \Gamma,~\lock_\mu \vdash_\p \A \type_{\;\!\ell}}{\weak^{\x:^\mu \A} : (\Gamma,~\x:^\mu \A) \Rightarrow_\q \Gamma}
  \and
  \infer{\theta : \Gamma \Rightarrow_\p \Theta \\ \nu : \Theta \Rightarrow_\p \Upsilon}{\nu \circ \theta : \Gamma \Rightarrow_\p \Upsilon}
  \and
  \infer{\theta : \Gamma \Rightarrow_\q \Theta \\ \Gamma,~\lock_\mu \vdash_\p \t
  : \A~\sqbkt{\theta,~\lock_\mu}} {\sqbkt{\theta,~\t} : \Gamma \Rightarrow_\q
  \bkt{\Theta,~\x:^\mu \A}}
  \and
  \infer{\theta : \Gamma \Rightarrow_\q \Theta}{\sqbkt{\theta,~\lock_\mu} : \bkt{\Gamma,~\lock_\mu} \Rightarrow_\p \bkt{\Theta,~\lock_\mu}}
  \and
  \infer{\mu \leq \nu}{\key^{\mu \leq \nu} : \bkt{\Gamma,~\lock_\nu} \Rightarrow_\p \bkt{\Gamma,~\lock_\mu}}
\end{mathpar}
With the the exceptional rule $\key^{\tridia \geq \id_\sm}$, which represents the fact that semantically $\lock_{\tridia}$ acts as the identity on flat contexts:
\[\infer{\Gamma \flat}{\key^{\tridia \geq \id_\sm} : \Gamma \Rightarrow_\sm \bkt{\Gamma,~\lock_\tridia}}\]
In practice, it is useful to iterate the weakening rule and combine it with the lock and key rules to obtain the following rule:
\[
  \infer{\theta : \Gamma \Rightarrow_\q \Theta \\ \mu \leq \locks(\Upsilon)}{\sqbkt{ \theta,~\weak^{\Upsilon}_\mu}: \bkt{\Gamma,~\Upsilon} \Rightarrow_\p \bkt{\Theta,~\lock_\mu}}
\]
In fact, we will generally use named variables and leave weakening implicit. \bbox

\subsubsection{Types}
\label{sec:types}

These are defined by several classes of type formers, including, at the
most basic level: \emph{$\Pi$-types} (parametrised by a modality, as in MTT), \emph{universes} (at each mode), and \emph{modal operators}.
\begin{mathpar}
  \infer{\Gamma,~\lock_\mu \vdash_\p \A \type_{\;\!\ell_\cblu{1}} \\ \Gamma,~\x :^\mu \A \vdash_\q
  \B \type_{\;\!\ell_\cblu{2}}}{\Gamma \vdash_\q \bkt{\x :^\mu \A} \to \B \type_{\;\!\ell_\cblu{1}\;\join\;\ell_\cblu{2}}}
  \\
  \infer{\ell \level}{\Gamma \vdash_\dm \Disc_{\;\!\ell} \type_{\;\!\;\!\lsuc\;\ell}}
  \and
  \infer{\ell \level}{\Gamma \vdash_\sm \Type_{\;\!\ell} \type_{\;\!\lsuc\;\ell}}
  \\
  \infer{\Gamma,~\lock_\sq \vdash_\sm \A \type_{\;\!\ell}}{\Gamma \vdash_\dm \sq~\A
  \type_{\;\!\ell}}
  \and
  \infer{\Gamma,~\lock_\tri \vdash_\dm \A \type_{\;\!\ell}}{\Gamma \vdash_\sm \tri~\A
  \type_{\;\!\ell}}
  \and
  \infer{\Gamma,~\lock_\dia \vdash_\sm \A \type_{\;\!\ell}}{\Gamma \vdash_\dm \dia~\A
  \type_{\;\!\ell}}
\end{mathpar}
We don't bother with primitive modal operators $\tridia$ or $\trisq$, since they can be obtained up to isomorphism by composing the others.

We will work with Tarski style universes, and thus require a decoding
operation:
\begin{mathpar}
  \infer{\Gamma \vdash_\dm \A : \Disc_{\;\!\ell}}{\Gamma \vdash_\dm \El~\A \type_{\;\!\ell}}
  \and
  \infer{\Gamma \vdash_\sm \A : \Type_{\;\!\ell}}{\Gamma \vdash_\sm \El~\A \type_{\;\!\ell}}
\end{mathpar}
Finally, we also have types that arise from substitution:
\begin{mathpar}
  \infer{\theta : \Gamma \Rightarrow_\p \Theta \\ \Theta \vdash_\p \A \type_{\;\!\ell}}{\Gamma \vdash_\p \A~\sqbkt{\theta} \type_{\;\!\ell}}
\end{mathpar}
As usual, substitution will be \emph{`eliminable'} in that $\A~\sqbkt{\theta}$ is always equal to something not involving $\sqbkt{\theta}$, but in the GAT presentation it is one of the generating rules like the others. \bbox

\subsubsection{Terms}
\label{sec:terms}

Terms are defined for each class of type former through introduction
and elimination rules. But first, we have variables.
There are two rules for variables: the ordinary one from MTT, and an \emph{`exceptional'} one arising, like the exceptional key $\key^{\tridia \geq \id_\sm}$, from the fact that $\lock_\tridia$ acts as the identity on flat contexts.
\begin{mathpar}
  \infer{\mu \leq \locks(\Theta)}{\Gamma,~\x:^\mu \A,~\Theta \vdash_\q \x :
  \A~\sqbkt{\id_\Gamma,~\weak^{\x:^\mu \A,\;\Theta}_\mu}}
  \and
  \infer{\Gamma \flat \\ \locks(\Theta) = \id_\sm}{\Gamma,~\x:^\tridia
  \A,~\Theta \vdash_\sm \x : \A~\sqbkt{\id_\Gamma,~\weak^{\x:^\tridia \A,\;\Theta}_\tridia}~\sqbkt{\key^{\tridia \geq \id_\sm}}}
\end{mathpar}
For $\Pi$-types, we have (as in MTT):
\begin{mathpar}
  \infer{\Gamma,~\x:^\mu \A \vdash_\q \t : \B}{\Gamma \vdash_\q \lambda~\x.~\t :
  \bkt{\x :^\mu \A} \to \B}
  \and
  \infer{\Gamma \vdash_\q \f : \bkt{\x :^\mu \A} \to \B \\ \Gamma,~\lock_\mu
  \vdash_\p \a : \A}{\Gamma \vdash_\q \f~\a : \B~\sqbkt{\a / \x}}
\end{mathpar}
For universes, we have a coding function:
\begin{mathpar}
  \infer{\Gamma \vdash_\dm \A \type_{\;\!\ell}}{\Gamma \vdash_\dm \Code~\A : \Disc_{\;\!\ell}}
  \and
  \infer{\Gamma \vdash_\sm \A \type_{\;\!\ell}}{\Gamma \vdash_\sm \Code~\A : \Type_{\;\!\ell}}
\end{mathpar}
For the modal operators, we have an introduction rule and negative \emph{`Fitch-style'}
elimination rules.
Following~\cite{gckgb:fitchtt}, we formulate these using parametric adjoints in the mode theory.
As noted in \cref{subsec:mode-theory}, the safe modalities have actual left adjoints, so their rules simplify as in~\cite{shulman:matt}.
And for $\lock_\dia$, we have observed that its parametric left adjoint is defined on the flat contexts, and on those it coincides with $\lock_\tri$.

\begin{mathpar}
  \infer{\Gamma,~\lock_\sq \vdash_\sm \t : \A}{\Gamma \vdash_\dm \sq~\t :
  \sq~\A}
  \and
  \infer{\Gamma,~\lock_\tri \vdash_\dm \t : \A}{\Gamma \vdash_\sm \tri~\t :
  \tri~\A}
  \and
  \infer{\Gamma,~\lock_\dia \vdash_\sm \t : \A}{\Gamma \vdash_\dm \dia~\t :
  \dia~\A}
  \\
  \infer{\Gamma,~\lock_\tri \vdash_\dm \t : \sq~\A}{\Gamma \vdash_\sm \unsq^\A~\t : \A~\sqbkt{\key^{\trisq \leq \id_\sm}}}
  \and
  \infer{\Gamma,~\lock_\dia \vdash_\sm \t : \tri~\A}{\Gamma \vdash_\dm \untri^\A~\t : \A}
  \and
  \infer{\Gamma \flat \\ \Gamma,~\lock_\tri \vdash_\dm \t : \dia~\A}{\Gamma \vdash_\sm \undia^\A~\t : \A~\sqbkt{\key^{\tridia \geq \id_\sm}}}
\end{mathpar}
Finally, we have terms that arise from substitution:
\begin{equation*}
  \infer{\theta : \Gamma \Rightarrow_\p \Theta \\ \Theta \vdash_\p \t : \A}{\Gamma \vdash_\p \t~\sqbkt{\theta} : \A~\sqbkt{\theta}}
  \eqno\sectend
\end{equation*}

\subsection{Telescopes and meta-abstractions, I}
\label{sec:tel}

\subsubsection{Telescopes}
\label{sec:telescopes}

Telescopes are suffixes of contexts, with the restriction that they may not contain locks.
The judgement $\Gamma \vdash_\p \Theta \tel_{\;\!\ell}$ denotes that $\Theta$ is a telescope in context $\Gamma$ of `level $\ell$', where the latter means that $\ell$ is greater than or equal to the level of the types occurring in $\Theta$.
We allow it to be strictly greater, and in particular allow an empty telescope to exist at all universe levels, for a reason to be explained in \cref{sec:teldisp}.
Formally, telescopes are an additional level-indexed sort of the GAT, with formation rules saying that there is an empty one and they can be built by concatenating types.
\begin{mathpar}
  \infer{\Gamma \ctx_\p}{\Gamma \vdash_\p \ec_\p \tel_{\;\!\ell}}
  \and
  \infer{\mu:\p\to \q \\
    \Gamma \vdash_\q \Theta \tel_{\;\!\ell} \\ \Gamma\ext\Theta,~\lock_\mu \vdash_\p \A \type_{\;\!\ell\cblu{'}} \\ \ell\cblu{'}\le\ell}{\Gamma \vdash_\q \bkt{\Theta,~\x:^\mu \A} \tel_{\;\!\ell}}
\end{mathpar}
Just as with ordinary contexts, we regard these rules as generating telescopes \emph{`inductively'}, although this is not formally the case syntactically.
We thus regard some other operation on telescopes as \emph{`defined'} when we specify rules for how it computes on these forms.
This is justified in most models, where we do actually define the judgment of telescopes inductively by the above rules.

For example, the premise of the second rule above requires knowing how to extend a context by a telescope.
We write this with a distinctive notation as $\Gamma\ext\Theta$, from which the reader can infer that $\Theta$ is a telescope.
Since $\ext$ is an operation on contexts, not a constructor of contexts, it computes on the constructors of telescopes:
\begin{mathpar}
  \infer{\Gamma \vdash_\p \Theta\tel_{\;\!\ell}}{(\Gamma\ext\Theta) \ctx}
  \and
  (\Gamma\ext\ec_\p) \equiv  \Gamma
  \and
  (\Gamma\ext(\Theta,~\x:^\mu \A)) \equiv  ((\Gamma\ext\Theta),~\x:^\mu \A)
\end{mathpar}
We consider the operation $\ext$ to be left-associative with the comma.
Thus, for instance, the context $\Gamma\ext\Theta,~\lock_\mu$ in the rule for extending a telescope by a type means $(\Gamma\ext\Theta),~\lock_\mu$.

By a \emph{strict telescope} we mean a telescope without any nontrivially modal variables.
\begin{mathpar}
  \infer{\Gamma \ctx_\p}{\Gamma \vdash_\p \ec_\p \stel_{\;\!\ell}}
  \and
  \infer{\Gamma \vdash_\p \Theta \stel_{\;\!\ell} \\ \Gamma\ext\Theta \vdash_\p \A \type_{\;\!\ell\cblu{'}} \\ \ell\cblu{'}\le\ell}{\Gamma \vdash_\p \bkt{\Theta,~\x:^{\id_\p} \A} \stel_{\;\!\ell}}
\end{mathpar}
As is evident, we do not distinguish syntactically between general telescopes and strict ones.
That is, we consider strictness to be a mere \emph{property} of a telescope, or alternatively we treat the obvious map from strict telescopes to telescopes as an implicit coercion.

Similarly, we can introduce a `lifting' operation taking a telescope to any higher level.
\[ \infer{\Gamma \vdash_\p \Theta \tel_{\;\!\ell} \\ \ell\le\ell\cblu{'}}{\Gamma \vdash_\p \Theta \tel_{\;\!\ell\cblu{'}} }
\]
As is evident from the notation, we also treat this as an implicit coercion.
Thus, when we define it recursively on the structure of a telescope, the rules look trivial unless we annotate them somehow with levels:
\begin{mathpar}
  \ec_\p \equiv \ec_\p
  \and
  \bkt{\Theta,~\x:^\mu \A} \equiv\bkt{\Theta,~\x:^\mu \A}
  \eqno\sectend
\end{mathpar}

\subsubsection{Partial substitutions}
\label{sec:part-subst}

If $\Gamma\vdash \Upsilon\tel_{\;\!\ell}$ there is a judgment $\Gamma \vdash \sigma : \Upsilon$ for the \emph{`elements'} of $\Upsilon$.
We call such $\sigma$ a `partial substitution', thinking of it as a substitution $\Gamma \Rightarrow (\Gamma\ext \Upsilon)$ that is the identity on $\Gamma$.
Formally, we specify that they can be built out of terms:
\begin{mathpar}
  \infer{\ }{\Gamma \vdash_\p~\esub_\p~: \ec_\p}
  \and
  \infer{\Gamma \vdash_\q \sigma : \Upsilon \\ \Gamma \ext \Upsilon,~\lock_\mu \vdash_\p \A \type_{\;\!\ell} \\  \Gamma,~\lock_\mu \vdash_\p \t : \A~\sqbkt{\id_\Gamma\ext\sigma,~\lock_\mu}}{\Gamma \vdash_\q \sqbkt{\sigma,~\t} : \bkt{\Upsilon,~\x :^\mu \A}}
\end{mathpar}
The second rule involves a notion of extending an ordinary substitution by a partial one.
\begin{mathpar}
  \infer{\theta : \Gamma \Rightarrow_\p \Delta \\ \Delta \vdash_\p \Upsilon \tel_{\;\!\ell} \\ \Gamma \vdash_\p \sigma : \Upsilon~\sqbkt{\theta}}
  {\sqbkt{\theta\ext\sigma} : \Gamma \Rightarrow_\p (\Delta \ext \Upsilon)}
  \and
  \sqbkt{\theta\ext\esub_\p} \equiv  \theta
  \and
  \sqbkt{\theta\ext\sqbkt{\sigma,~\t}} \equiv  \sqbkt{\sqbkt{\theta\ext\sigma},~\t}
\end{mathpar}
This gives a simple way to ensure that partial substitutions are uniquely determined by their components: their equality is detected by equality of the induced substitutions.
\begin{mathpar}
  \infer{\Gamma \vdash_\p \sigma : \Upsilon \\ \Gamma \vdash_\p \tau : \Upsilon \\ \sqbkt{\id_\Gamma\ext\sigma} \equiv  \sqbkt{\id_\Gamma\ext\tau} : \Gamma \Rightarrow_\p \bkt{\Gamma \ext \Upsilon}}{\Gamma \vdash_\p \sigma \equiv  \tau : \Upsilon}
\end{mathpar}
Note that a partial substitution \emph{does}, in fact, \emph{have} \emph{`components'}: given $\Gamma \vdash_\q \sigma : \bkt{\Upsilon,~\x :^\mu \A}$ we have
\begin{gather*}
 \sqbkt{\id_\Gamma\ext\sigma,~\x,~\lock_\mu} : (\Gamma,~\lock_\mu) \to (\Gamma\ext\Upsilon,~\x:^\mu \A,~\lock_\mu) \\
 \Gamma\ext\Upsilon,~\x:^\mu \A,~\lock_\mu \vdash \x : \A\\
 \Gamma,~\lock_\mu \vdash \x~\sqbkt{\id_\Gamma\ext\sigma,~\x,~\lock_\mu} : \A
\end{gather*}
We also have a notion of weakening for telescopes.  As before, we omit the equations that this must satisfy.
\begin{mathpar}
  \infer{\Gamma\vdash \Theta \tel_{\;\!\ell}}{\weak^\Theta : (\Gamma\ext \Theta) \Rightarrow \Gamma}
  \and
  \weak^{\ec_\p} \equiv  \id_\Gamma
  \and
  \weak^{\Theta,\;\x:^\mu \A} \equiv  \weak^\Theta \circ \weak^{\x:^\mu \A}
  \eqno\sectend
\end{mathpar}

\subsubsection{Meta-abstracted types and terms}
\label{sec:meta-abstractions}

We now introduce a new judgement form $\Gamma \vdash_\p \A \slfrac{\type_{\;\!\ell_\cblu{1}}}{_{\upsilon\;:\;\Upsilon}}$, where $\Gamma \vdash \Upsilon \tel_{\;\!\ell_\cblu{0}}$.
This should be thought of saying that $\A$ is a type depending on the variables $\upsilon$ in $\Upsilon$, i.e.\ belonging to a \emph{`framework-level $\Pi$-type'} $\A : (\upsilon:\Upsilon) \to \mathsf{type}_{\;\!\ell_\cblu{1}}$.
Accordingly, elements of this judgment are introduced by binding and eliminated by application, with a $\beta$ and $\eta$-rule.
\begin{mathpar}
  \infer{\Gamma \ext \bkt{\upsilon:\Upsilon} \vdash_\p \A \type_{\;\!\ell_\cblu{1}}}{\Gamma \vdash_\p \dbkt{\A}_{\upsilon\;:\;\Upsilon} \slfrac{\type_{\;\!\ell_\cblu{1}}}{_{\upsilon\;:\;\Upsilon}}}
  \and
  \infer{\Gamma \vdash_\p \cA \slfrac{\type_{\;\!\ell_\cblu{1}}}{_{\upsilon\;:\;\Upsilon}} \\ \Gamma \vdash_\p \sigma : \Upsilon}{\Gamma \vdash_\p \cA~\sigma \type_{\;\!\ell_\cblu{1}}}
  \and
  \infer{\Gamma \ext \bkt{\upsilon:\Upsilon} \vdash_\p \A \type_{\;\!\ell_\cblu{1}} \\ \Gamma \vdash_\p \sigma : \Upsilon}{\Gamma \vdash_\p \dbkt{\A}_{\upsilon\;:\;\Upsilon} ~\sigma \equiv  \A~\sqbkt{\id_\Gamma\ext \sigma}}
  \and
  \infer{\Gamma \vdash_\p \cA \slfrac{\type_{\;\!\ell_\cblu{1}}}{_{\upsilon\;:\;\Upsilon}} \\ \Gamma \vdash_\p \cB \slfrac{\type_{\;\!\ell_\cblu{1}}}{_{\upsilon\;:\;\Upsilon}} \\ \Gamma \ext \bkt{\upsilon:\Upsilon} \vdash_\p \cA~\upsilon \equiv  \cB~\upsilon }{\Gamma \vdash_\p \cA \equiv  \cB}
\end{mathpar}
We also regard $\cA \slfrac{\type_{\;\!\ell_\cblu{1}}}{_{\upsilon\;:\;\Upsilon}}$ as standing in for its own $\Pi$-type \emph{`$(\upsilon:\Upsilon) \to \cA~\upsilon$'}.
Thus, such an $\cA$ can have its own terms belonging to it, which are also introduced by binding and eliminated by application, with a $\beta$ and $\eta$-rule.
\begin{mathpar}
  \infer{\Gamma \vdash_\p \cA \slfrac{\type_{\;\!\ell_\cblu{1}}}{_{\upsilon\;:\;\Upsilon}} \\ \Gamma \ext \bkt{\upsilon:\Upsilon} \vdash_\p \t : \cA~\upsilon}{\Gamma \vdash_\p \dsqbkt{\t}_{\;\upsilon\;:\;\Upsilon} : \dbkt{\A}_{\upsilon\;:\;\Upsilon}}
  \and
  \infer{\Gamma \vdash_\p \cA \slfrac{\type_{\;\!\ell_\cblu{1}}}{_{\upsilon\;:\;\Upsilon}} \\ \Gamma \vdash_\p \ct : \cA \\ \Gamma \vdash_\p \sigma : \Upsilon}{\Gamma \vdash_\p \ct~\sigma~:~\cA~\sigma \type_{\;\!\ell_\cblu{1}}}
  \and
  \infer{\Gamma \vdash_\p \cA \slfrac{\type_{\;\!\ell_\cblu{1}}}{_{\upsilon\;:\;\Upsilon}} \\ \Gamma \ext \bkt{\upsilon:\Upsilon} \vdash_\p \t : \cA ~\upsilon\\  \Gamma \vdash_\p \sigma : \Upsilon}
  { \Gamma \vdash_\p \dsqbkt{\t}_{\;\upsilon\;:\;\Upsilon} ~ \sigma \equiv  \t~\sqbkt{\id_\Gamma\ext\sigma}}
  \and
  \infer{\Gamma \vdash_\p \cA \slfrac{\type_{\;\!\ell_\cblu{1}}}{_{\upsilon\;:\;\Upsilon}} \\\Gamma \vdash_\p \ct : \cA \\ \Gamma \vdash_\p \cs : \cA \\ \Gamma\ext (\upsilon : \Upsilon) \vdash_\p \ct~\upsilon \equiv  \cs~\upsilon}{\Gamma \vdash_\p \ct \equiv  \cs}
  \eqno\sectend
\end{mathpar}

\subsection{D\'ecalage and displayed types}
\label{sec:dd}

Semantically, the fundamental operation is shifting the dimensions of a simplicial type.
In classical simplicial homotopy theory, this is called \emph{d\'ecalage}:
\[ \bkt{\A\D}_\n = \A_{\n+\cred{1}} \]
The simplicial structure maps of $\A\D$ are a subset of those of $\A$, while the unused ones assemble into a simplicial map $\A\D \to \A$.
When $\A$ is a type at mode $\sm$, we will regard $\A\D$ as the projection from a type $\A\d$ dependent on $\A$; thus we have
\[ \A\D = (\x:\A,~\xp:\A\d~\x) \]
(Semantically, this is validated by the fact that if $\A$ is Reedy fibrant, then the map $\A\D \to \A$ is a Reedy fibration.)
These dependent types $\A\d$, which we call \emph{display}, are our version of the \emph{`logical relations'} assigned to every type by an internal parametricity theory.

\subsubsection{Display for types}
\label{sec:disp-ty}

In contrast to fully internal parametricity theories, because we don't have degeneracies in our cube category, d\'ecalage and display can only be applied in restricted contexts.
In external parametricity, the logical relations apply only to types in the empty context; but our modalities allow us to say more generally that they apply to any \emph{`boxed'} type.
Here by \emph{`box'} we mean not $\sq$ but the corresponding \emph{endofunctor} of the simplicial mode, namely $\trisq$.
Thus, informally display should have the type $\cred{d} : (\A :^\trisq \Type_{\;\!\ell}) \to \A \to \Type_{\;\!\ell}$, with computability witnesses being assigned by a function $\cred{d} : (\A :^\trisq \Type_{\;\!\ell})( \x :^\trisq \A) \to \A\d~\x$.
If we reformulate these without referring to $\Pi$-types, we obtain the following rules for our basic notion of \emph{displayed type}:
\begin{mathpar}
  \infer{\Gamma,~\lock_{\trisq} \vdash_\sm \A \type_{\;\!\ell} \\ \Gamma \vdash_\sm \t :
    \A~\sqbkt{\key^{\trisq \leq \id_\sm}}}{\Gamma \vdash_\sm \A\d~\t \type_{\;\!\ell}}
  \and
  \infer{\Gamma,~\lock_\trisq \vdash_\sm \t : \A}{\Gamma \vdash \t\d :
  \A\d~\big(\t~\sqbkt{\key^{\trisq \leq \id_\sm}}\big)}
\end{mathpar}

However, in order to compute with this, we need a version of it that incorporates dependence on a telescope to the right of the lock.
The corresponding action on that telescope is called d\'ecalage. \bbox

\subsubsection{Telescope d\'ecalage}
\label{sec:telescope-decalage}

As noted above, with display $\A\d$ defined as dependent on $\A$, d\'ecalage $\A\D$ is naturally not a single type but a telescope.
It is therefore natural to generalise its \emph{input} to be a telescope also.
This yields an operation that doubles the variables and groups each type with its displayed version, e.g.\
\[(\x:\A,~ \y:\B)\D \equiv  (\x:\A,~ \x\cblu{'}:\A\d~ \x,~ \y:\B,~ \y\cblu{'}:\B\d~ \y).\]
The classical projection from d\'ecalage to the identity, composed of the leftover face maps, becomes an \emph{`evens'} substitution $\Upsilon\D \to \Upsilon$ that throws away the elements of the displayed types (the primed variables in the above example).
(The corresponding `odds' substitution must wait until we introduce telescope display in \cref{sec:teldisp}.)
\begin{mathpar}
    \infer{\Gamma,~\lock_\trisq \vdash_\sm \Upsilon \tel_{\;\!\ell}}{\Gamma \vdash_\sm \Upsilon\D \tel_{\;\!\ell}}
  \and
  \infer{\Gamma,~\lock_\trisq \vdash_\sm \Upsilon \tel_{\;\!\ell} \\ \Gamma,~\lock_\trisq \vdash_\sm \sigma : \Upsilon }{\Gamma \vdash_\sm \sigma\D : \Upsilon\D}
  \\
  \infer{\Gamma \vdash_\sm \plus{\sigma} : \Upsilon\D}{\Gamma \vdash_\sm \plus{\sigma}\ev : \Upsilon~\sqbkt{\key^{\trisq\le \id_\sm}}}
  \and
  \sigma^{\mathsf{\cred{D}}\;\mathsf{\cred{ev}}} \equiv \sigma~\sqbkt{\key^{\trisq\le \id_\sm}}
\end{mathpar}
Notationally, we put a superscript `\cblu{+}' on variables and metavariables belonging to d\'ecalaged telescopes, and a prime on variables belonging to displayed types and telescopes.
These symbols are part of the variable name, e.g.\ $\plus{\sigma}$ above is a single variable that just happens to be named mnemonically.

At this point we can assert that d\'ecalage preserves empty telescopes.
\begin{mathpar}
  \ec^{\mathsf{\cred{D}}}_\sm \equiv \ec_\sm\and
  \esub^{\mathsf{\cred{D}}}_\sm \equiv \esub_\sm\and
  \esub^{\mathsf{\cred{ev}}}_\sm \equiv \esub_\sm\and
\end{mathpar}
D\'ecalage will also compute on telescopes extended by a type, but we wait to give these rules in \cref{sec:comp-tel-dec}, since they require more structure. \bbox

\subsubsection{Display for meta-abstractions}
\label{sec:displ-meta-abstr}

The more general version of display alluded to above can informally be thought of as having the following rule:
\begin{mathpar}
  \text{\cprime{\textquestiondown}}\quad\infer{\Gamma,~\lock_\trisq\ext \Upsilon \vdash_\sm \A \type_{\;\!\ell}}{\Gamma\ext \Upsilon\D, \a:\A \vdash_\sm \A\d~\a \type_{\;\!\ell}} \quad \cprime{?}
\end{mathpar}
However, this is not a well-behaved rule because the context of the conclusion is not fully general.
There are multiple ways to solve this problem; we will solve it by saying that general display acts on a meta-abstracted type.
\begin{mathpar}
  \infer{\Gamma,~\lock_\trisq \vdash_\sm \cA \slfrac{\type_{\;\!\ell_\cblu{1}}}{_{\upsilon\;:\;\Upsilon}}}{\Gamma \vdash_\sm \cA\d \slfrac{\type_{\;\!\ell_\cblu{1}}}{_{\upsilon^\cblu{+}\;:\;\Upsilon\D,\; \a\;:\;\cA\;[\;\key^{\trisq\le \id_\sm}\;]\;\upsilon^\cblu{+}\ev}}} \\
  \infer{\Gamma,~\lock_\trisq \vdash_\sm \cA \slfrac{\type_{\;\!\ell_\cblu{1}}}{_{\upsilon\;:\;\Upsilon}} \\ \Gamma,~\lock_\trisq \vdash_\sm \ct : \cA}{\Gamma \vdash_\sm \ct\d : \dbkt{\cA\d~\upsilon^\cblu{+}~\bkt{\ct~\upsilon^\cblu{+}\ev}}_{\;\upsilon^\cblu{+}\;:\;\Upsilon\D}}
\end{mathpar}
In general, this does not reduce to ordinary display, but it does when applied to a \emph{d\'ecalaged} partial substitution.
\begin{mathpar}
  \infer{\Gamma,~\lock_\trisq \vdash_\sm \cA \slfrac{\type_{\;\!\ell_\cblu{1}}}{_{\upsilon\;:\;\Upsilon}} \\ \Gamma,~\lock_\trisq \vdash_\sm \sigma : \Upsilon \\ \Gamma \vdash \t : (\cA~\sigma)~\sqbkt{\key^{\trisq\le \id_\sm}}}
  {\Gamma \vdash \cA\d~\sigma\D~\t \equiv (\cA~\sigma)\d~\t}
  \and
  \infer{\Gamma,~\lock_\trisq \vdash_\sm \cA \slfrac{\type_{\;\!\ell_\cblu{1}}}{_{\upsilon\;:\;\Upsilon}} \\ \Gamma,~\lock_\trisq \vdash_\sm \sigma : \Upsilon \\ \Gamma,~\lock_\trisq \vdash_\sm \ct : \cA}
  {\Gamma \vdash \ct\d~\sigma\D \equiv (\t~\sigma)\d}
\end{mathpar}
In particular, when $\Upsilon\equiv \ec_\sm$ these rules say that display for trivial meta-abstractions is equivalent to ordinary display. \bbox

\subsubsection{Computing d\'ecalage}
\label{sec:comp-tel-dec}

Now we can give the rules \emph{`defining'} telescope d\'ecalage on telescopes extended by a variable.
Specifically, when extending by a non-modal variable, we also extend by its displayed version.
But that displayed version needs to depend on $\Theta\D$, so we define it in terms of display for meta-abstractions.
Note that the well-typedness of $\t\d$ in these rules depends on the reduction of meta-abstraction display on displayed partial substitutions.
\begin{align*}
  \bkt{\theta:\Theta,~\x : \A}\D
  &\equiv \big(\theta^\cblu{+}:\Theta\D,~\x : \A~\sqbkt{\key^{\trisq\le \id_\sm} \ext \theta^\cblu{+}\ev},~\xp : \dbkt{\A}_{\theta\;:\;\Theta}\d~\theta^\cblu{+}~\x\big) \\
  \sqbkt{\sigma,~\t}\D
  &\equiv \sqbkt{\sigma\D,~\t,~\t\d}\\
  \sqbkt{\plus{\sigma},~\t,~\t\cblu{'}}\ev
  &\equiv \sqbkt{\plus{\sigma}\ev,~\t}
\end{align*}

The case of a nontrivially modal variable is actually simpler.
Note that in this case, the modality must be of the form $\tri \circ \mu$.
Recalling that semantically, $\tri$ constructs a constant simplicial type, we should have informally $(\tri~\A)\D = \tri~\A$, and therefore $(\tri~\A)\d$ is trivial.
For an action on types, this would mean that $(\tri~\A)\d~\x$ is the unit type; for our current action on telescopes, it means we can just omit the displayed variables.
\begin{align*}
  \bkt{\theta:\Theta,~\x :^{\tri \circ \mu} \A}\D
  &\equiv \big(\theta^\cblu{+}:\Theta\D,~\x :^{\tri\circ\mu} \A~\sqbkt{\key^{\trisq\le \id_\sm} \ext \theta^\cblu{+}\ev,~\lock_{\tri\circ\mu}}\big)\\
  \sqbkt{\sigma,~\t}\D
  &\equiv \sqbkt{\sigma\D,~\t}\\
  \sqbkt{\plus{\sigma},~\t}\ev
  &\equiv \sqbkt{\plus{\sigma}\ev,~\t}
  \tag*{\bbox}
\end{align*}

\subsubsection{Computing display}
\label{sec:comp-telesc-displ}

Recall from \cref{sec:intro} that we treat display computationally like the identity types of HOTT, so that it computes on the basic type-formers.
Note that the abstracting telescope changes as we compute, so these rules could not be stated for ordinary display alone.

\paragraph{Non-modal $\Pi$-Types}

This rule represents the traditional behavior of parametricity and logical relations on functions: a computability witness for a function says that it preserves computability witnesses.
\begin{align*}
    \dbkt{\big(\x : \A\big) \to \B}_{\upsilon\;:\;\Upsilon}\d &\equiv \scaleto{\newllparenthesis}{1.2em}\;\big(\x : \A~\sqbkt{\key^{\trisq\le \id_\sm} \ext \upsilon^\cblu{+}\ev}\big)\; \big(\xp : \dbkt{\A}_{\upsilon\;:\;\Upsilon}\d~\upsilon^\cblu{+}~\x\big) \to \\
    &\quad\quad\ \ \dbkt{\B}_{\upsilon\;:\;\Upsilon,\; \x\;:\;\A}\d~\upsilon^\cblu{+}~\x~\xp~\bkt{\f~\x}\;\scaleto{\newrrparenthesis}{1.2em}_{\upsilon^\cblu{+}\;:\;\Upsilon\D,\;\f\;:\;\bkt{\x\;:\;\A}\;\to\;\B} \\
    \dsqbkt{\lambda~\x.~\t}_{\;\upsilon\;:\;\Upsilon}\d &\equiv \dsqbkt{\lambda~\x~\xp.~\dsqbkt{\t}_{\;\upsilon\;:\;\Upsilon,\; \x\;:\;\A}\d~\upsilon^\cblu{+}~\x~\xp}_{\;\upsilon^\cblu{+}\;:\;\Upsilon\D} \\
    \dsqbkt{\f~\a}_{\;\upsilon\;:\;\Upsilon}\d &\equiv \dsqbkt{\big(\dsqbkt{\f}_{\;\upsilon\;:\;\Upsilon}\d~\upsilon^\cblu{+}\big)~\a~\big(\dsqbkt{\a}_{\;\upsilon\;:\;\Upsilon}\d~\upsilon^\cblu{+}\big)}_{\;\upsilon^\cblu{+}\;:\;\Upsilon\D}
\end{align*}

\paragraph{Nontrivially modal $\Pi$-Types}

As with d\'ecalage, here we use the fact that display of $\tri$ is trivial.
Note also that here a modal variable appears in the domain of a meta-abstraction.
This is the reason that we cannot restrict to strict telescopes in general.
\begin{align*}
    \dbkt{\big(\x :^{\tri\circ\mu} \A\big) \to \B}_{\upsilon\;:\;\Upsilon}\d &\equiv \scaleto{\newllparenthesis}{1.2em}\;\big(\x :^{\tri\circ\mu} \A ~\sqbkt{\key^{\trisq\le \id_\sm} \ext \upsilon^\cblu{+}\ev,~\lock_{\tri\circ\mu}}\big) \to \\
    &\quad\quad\ \ \dbkt{\B}_{\upsilon\;:\;\Upsilon,\; \x\;:^{\tri\circ\mu}\A}\d~\upsilon^\cblu{+}~\x~\bkt{\f~\x}\;\scaleto{\newrrparenthesis}{1.2em}_{\upsilon^\cblu{+}\;:\;\Upsilon\D,\;\f\;:\;\bkt{\x\;:^{\tri\circ\mu}\;\A}\;\to\;\B} \\
    \dsqbkt{\lambda~\x.~\t}_{\;\upsilon\;:\;\Upsilon}\d &\equiv \dsqbkt{\lambda~\x.~\dsqbkt{\t}_{\;\upsilon\;:\;\Upsilon,\;\x\;:^{\tri\circ\mu}\A}\d~\upsilon^\cblu{+}~\x}_{\;\upsilon^\cblu{+}\;:\;\Upsilon\D} \\
    \dsqbkt{\f~\a}_{\;\upsilon\;:\;\Upsilon}\d &\equiv \dsqbkt{\dsqbkt{\f}_{\;\upsilon\;:\;\Upsilon}\d~\upsilon^\cblu{+}~\a}_{\;\upsilon^\cblu{+}\;:\;\Upsilon\D}
\end{align*}

\paragraph{Universes}

As with $\Pi$-types, this rule represents the traditional behavior of parametricity and logical relations on universes: a computability witness for a type is a relation on that type.
\begin{align*}
    \dbkt{\Type_{\;\!\ell}}_{\upsilon\;:\;\Upsilon}\d &\equiv \dbkt{\El~\A \to \Type_{\;\!\ell}}_{\upsilon^\cblu{+}\;:\;\Upsilon\D,\;\A\;:\;\Type_{\;\!\ell}} \\
    \dsqbkt{\Code~\A}_{\;\upsilon\;:\;\Upsilon}\d &\equiv \dsqbkt{\lambda~\a.~\Code~\big(\dbkt{\A}_{\upsilon\;:\;\Upsilon}\d~\upsilon^\cblu{+}~\a\big)}_{\;\upsilon^\cblu{+}\;:\;\Upsilon\D} \\
    \dbkt{\El~\A}_{\upsilon\;:\;\Upsilon}\d &\equiv \dbkt{\El~\big(\dsqbkt{\A}_{\;\upsilon\;:\;\Upsilon}\d~\upsilon^\cblu{+}~\a\big)}_{\upsilon^\cblu{+}\;:\;\Upsilon\D,\;\a\;:\;\El\;\A}
    \tag*{\bbox}
\end{align*}

\subsection{Telescopes and meta-abstractions, II}
\label{sec:tma2}

The rules given so far essentially suffice to characterise the basic theory of dTT.
However, in order to formulate our definition of semi-simplicial types, we need a bit more structure.
To this end, in this section we introduce some more operations on telescopes that can be \emph{`defined'} in terms of those already given.

\subsubsection{Meta-abstracted telescopes}
\label{sec:meta-abstr-telesc}

We start with another judgement form $\Gamma \vdash_\p \Phi \slfrac{\tel_{\;\!\ell_\cblu{1}}}{_{\upsilon\;:\;\Upsilon}}$ for a telescope dependent on a telescope, with rules entirely analogous to those for types and terms in \cref{sec:meta-abstractions}.
\begin{mathpar}
  \infer{\Gamma \ext \bkt{\upsilon:\Upsilon} \vdash_\p \Phi \tel_{\;\!\ell_\cblu{1}}}{\Gamma \vdash_\p \dbkt{\Phi}_{\upsilon\;:\;\Upsilon} \slfrac{\tel_{\;\!\ell_\cblu{1}}}{_{\upsilon\;:\;\Upsilon}}}
  \and
  \infer{\Gamma \vdash_\p \Phi \slfrac{\tel_{\;\!\ell_\cblu{1}}}{_{\upsilon\;:\;\Upsilon}} \\ \Gamma \vdash_\p \sigma : \Upsilon}{\Gamma \vdash_\p \Phi~\sigma \tel_{\;\!\ell_\cblu{1}}}
  \and
  \infer{\Gamma \ext \bkt{\upsilon:\Upsilon} \vdash_\p \Phi \tel_{\;\!\ell_\cblu{1}} \\ \Gamma \vdash_\p \sigma : \Upsilon}{\Gamma \vdash_\p \dbkt{\Phi}_{\upsilon\;:\;\Upsilon} ~\sigma \equiv  \Phi~\sqbkt{\id_\Gamma\ext \sigma}}
  \and
  \infer{\Gamma \vdash_\p \Phi \slfrac{\tel_{\;\!\ell_\cblu{1}}}{_{\upsilon\;:\;\Upsilon}} \\ \Gamma \vdash_\p \Psi \slfrac{\tel_{\;\!\ell_\cblu{1}}}{_{\upsilon\;:\;\Upsilon}} \\ \Gamma \ext \bkt{\upsilon:\Upsilon} \vdash_\p \Phi~\upsilon \equiv  \Psi~\upsilon }{\Gamma \vdash_\p \Phi \equiv  \Psi}
  \and
  \infer{\Gamma \vdash_\p \Phi \slfrac{\tel_{\;\!\ell_\cblu{1}}}{_{\upsilon\;:\;\Upsilon}} \\ \Gamma \ext \bkt{\upsilon:\Upsilon} \vdash_\p \t : \Phi~\upsilon}{\Gamma \vdash_\p \dsqbkt{\t}_{\;\upsilon\;:\;\Upsilon} : \dbkt{\Phi}_{\upsilon\;:\;\Upsilon}}
  \and
  \infer{\Gamma \vdash_\p \Phi \slfrac{\tel_{\;\!\ell_\cblu{1}}}{_{\upsilon\;:\;\Upsilon}} \\ \Gamma \vdash_\p \ct : \Phi \\ \Gamma \vdash_\p \sigma : \Upsilon}{\Gamma \vdash_\p \ct~\sigma~:~\Phi~\sigma \tel_{\;\!\ell_\cblu{1}}}
  \and
  \infer{\Gamma \vdash_\p \Phi \slfrac{\tel_{\;\!\ell_\cblu{1}}}{_{\upsilon\;:\;\Upsilon}} \\ \Gamma \ext \bkt{\upsilon:\Upsilon} \vdash_\p \t : \Phi ~\upsilon\\  \Gamma \vdash_\p \sigma : \Upsilon}
  { \Gamma \vdash_\p \dsqbkt{\t}_{\;\upsilon\;:\;\Upsilon} ~ \sigma \equiv  \t~\sqbkt{\id_\Gamma\ext\sigma}}
  \and
  \infer{\Gamma \vdash_\p \Phi \slfrac{\tel_{\;\!\ell_\cblu{1}}}{_{\upsilon\;:\;\Upsilon}} \\\Gamma \vdash_\p \ct : \Phi \\ \Gamma \vdash_\p \cs : \Phi \\ \Gamma\ext (\upsilon : \Upsilon) \vdash_\p \ct~\upsilon \equiv  \cs~\upsilon}{\Gamma \vdash_\p \ct \equiv  \cs}
  \eqno\sectend
\end{mathpar}

\subsubsection{Telescope concatenation}
\label{sec:telesc-conc}

Telescope concatenation is not necessary for the syntactic definition of SSTs, but seems to be required for a clean description of the semantics.
It is essentially a $\Sigma$-type for telescopes, which is definitionally associative with context and telescope extension.

\begin{mathpar}
  \infer{\Gamma \vdash \Upsilon \tel_{\;\!\ell_\cblu{0}} \\ \Gamma\ext \Upsilon \vdash \Phi \tel_{\;\!\ell_\cblu{1}}}{\Gamma \vdash (\Upsilon\ext\Phi) \tel_{\;\!\ell_\cblu{0}\;\join\;\ell_\cblu{1}}}
  \and
  \infer{\Gamma \vdash \sigma : \Upsilon \\ \Gamma \vdash \delta : \Phi~\sqbkt{\id_\Gamma\ext\sigma}}
  {\Gamma \vdash \sqbkt{\sigma\ext\delta} : (\Upsilon\ext\Phi)}
  \and
  \infer{\Gamma \vdash \theta : (\Upsilon\ext\Phi)}{\Gamma \vdash \theta_\cred{0} : \Upsilon}
  \and
  \infer{\Gamma \vdash \theta : (\Upsilon\ext\Phi)}{\Gamma \vdash \theta_\cred{1} : \Phi~\theta_\cred{0}}
  \and
  \sqbkt{\sigma\ext\delta}_\cred{0} \equiv \sigma
  \and
  \sqbkt{\sigma\ext\delta}_\cred{1} \equiv \delta
  \and
  \sqbkt{\theta_\cred{0}\ext \theta_\cred{1}} \equiv \theta
  \and
  (\Gamma \ext (\Upsilon\ext \Phi)) \equiv ((\Gamma \ext \Upsilon)\ext \Phi)
  \and
  (\Upsilon \ext \ec_\p) \equiv \Upsilon
  \and
  (\Upsilon \ext \bkt{\Phi,~\x:^\mu \A}) \equiv ((\Upsilon\ext \Phi),~\x :^\mu \A)
\end{mathpar}
Note that to be at the right universe level, in the rule $(\Upsilon \ext \ec_\p) \equiv \Upsilon$, the right-hand side `$\Upsilon$' must be implicitly lifted to $\ell_\cblu{0}\;\join\;\ell_\cblu{1}$.\bbox

\subsubsection{$\Pi$-telescopes}
\label{sec:pi-telescopes}

To define the copointed endofunctors whose coalgebras are display inductive types, we will need $\Pi$-telescopes.
For simplicity, we require that the codomain be a strict telescope; this suffices for our application.
The basic rules are just like those for $\Pi$-types.

\begin{mathpar}
  \infer{\Gamma \vdash \Upsilon\tel_{\;\!\ell_\cblu{0}} \\ \Gamma\ext(\upsilon:\Upsilon) \vdash \Theta \stel_{\;\!\ell_\cblu{1}}}{\Gamma \vdash (\upsilon:\Upsilon) \to \Theta \stel_{\;\!\ell_\cblu{0}\;\join\;\ell_\cblu{1}}}
  \and
  \infer{\Gamma\ext(\upsilon:\Upsilon) \vdash \theta:\Theta}{\Gamma \vdash \lambda~\upsilon.~\theta : (\upsilon:\Upsilon) \to \Theta}
  \and
  \infer{\Gamma\vdash \delta : (\upsilon:\Upsilon) \to \Theta \\\Gamma \vdash \sigma:\Upsilon}{\Gamma \vdash \delta~\sigma : \Theta~\sqbkt{\id_\Gamma\ext\sigma}}
  \and
  \infer{\Gamma\ext(\upsilon:\Upsilon) \vdash \theta:\Theta\\\Gamma \vdash \sigma:\Upsilon}{\Gamma\vdash  (\lambda~\upsilon.~\theta)~\sigma \equiv \theta~\sqbkt{\id_\Gamma\ext\sigma}}
  \and
  \infer{\Gamma\vdash \delta : (\upsilon:\Upsilon) \to \Theta \\ \Gamma\vdash \delta\cblu{'} : (\upsilon:\Upsilon) \to \Theta \\ \Gamma\ext (\upsilon:\Upsilon) \vdash \delta~\upsilon \equiv \delta\cblu{'}~\upsilon}{\Gamma \vdash \delta\equiv \delta\cblu{'}}
\end{mathpar}
In addition, we assert computation laws for $\Pi$-telescopes when the domain or codomain is an extension.
\begin{mathpar}
  \infer{\Gamma \vdash \Upsilon\tel_{\;\!\ell}}{((\upsilon:\Upsilon)\to \ec_\p) \equiv \ec_\p}
  \and
  \infer{\Gamma \vdash \Theta\stel_{\;\!\ell}}{((\xi:\ec_\p)\to \Theta) \equiv \Theta}
  \\
  \lambda~\upsilon.~\esub_\p \equiv \esub_\p
  \and
  \lambda~(\upsilon:\ec_\p).~\theta \equiv \theta
  \and
  \infer{\Gamma \vdash \Upsilon\tel_{\;\!\ell_\cblu{0}} \\ \Gamma\ext(\upsilon:\Upsilon) \vdash \Theta \stel_{\;\!\ell_\cblu{1}} \\ \Gamma\ext(\upsilon:\Upsilon)\ext(\theta:\Theta)\vdash \B \type_{\;\!\ell_\cblu{2}}}{((\upsilon:\Upsilon) \to (\Theta,~ \y : \B)) \equiv (\delta:(\upsilon:\Upsilon) \to \Theta,~\epsilon:(\upsilon:\Upsilon) \to  \B~\sqbkt{\id_\Gamma\ext\upsilon\ext\delta~\upsilon})}
  \and
  \lambda~\upsilon.~\sqbkt{\theta,~\b}
  \equiv \sqbkt{\lambda~\upsilon.~\theta,~\lambda~\upsilon.~\b}
  \and
  \infer{\Gamma \vdash_\q \Upsilon\tel_{\;\!\ell_\cblu{0}} \\ \Gamma\ext(\upsilon:\Upsilon),~\lock_\mu \vdash_\p \A \type_{\;\!\ell_\cblu{1}} \\ \Gamma\ext(\upsilon:\Upsilon),~\x:^\mu \A \vdash_\q \Theta \tel_{\;\!\ell_\cblu{2}}}{((\upsilon:\Upsilon,~\x:^\mu \A) \to \Theta) \equiv ((\upsilon:\Upsilon) \to (\x:^\mu \A) \to \Theta)}
  \and
  \lambda~(\upsilon : (\Upsilon\ext\A)).~\theta
  \equiv \lambda~\upsilon.~\lambda~\x.~\theta
\end{mathpar}
\begin{mathpar}
  \infer{\Gamma \vdash \Upsilon\tel_{\;\!\ell_\cblu{0}} \\ \Gamma\ext(\upsilon:\Upsilon) \vdash \Theta \tel_{\;\!\ell_\cblu{1}} \\ \Gamma\ext(\upsilon:\Upsilon)\ext(\theta:\Theta)\vdash \Phi \tel_{\;\!\ell_\cblu{2}}}{((\upsilon:\Upsilon) \to (\Theta\ext\Phi)) \equiv (\delta:(\upsilon:\Upsilon) \to \Theta~\ext~\epsilon:(\upsilon:\Upsilon) \to (\Phi~\sqbkt{\id_\Gamma\ext\upsilon\ext\delta~\upsilon}))}
  \and
    \lambda~\upsilon.~\sqbkt{\theta\ext\phi}
  \equiv \sqbkt{\lambda~\upsilon.~\theta\ext \lambda~\upsilon.~\phi}
\end{mathpar}
When telescopes are definitionally lists of types, these rules suffice to compute any $\Pi$-telescope in terms of $\Pi$-types.
Note that in the rule $((\xi:\ec_\p)\to \Theta)\equiv \Theta$, the `$\Theta$' on the right-hand side must be implicitly lifted to the maximum of the levels of $\ec_\p$ and $\Theta$.\bbox

\subsection{Displayed telescopes}
\label{sec:dtel}

The structure in \cref{sec:dd}, with display acting on types and d\'ecalage on telescopes, is sufficient to determine the behavior of display.
However, in practice we will also need a notion of dependent d\'ecalage and display for telescopes.
When telescopes are lists of types, this is determined (like ordinary d\'ecalage) by display for types.

\subsubsection{Meta-abstracted d\'ecalage}
\label{sec:metaabsdec}

\begin{spreadlines}{10pt}
\begin{gather}
  \infer{\Gamma,~\lock_{\trisq} \vdash_\sm \Phi \slfrac{\tel_{\;\!\ell_\cblu{1}}}{_{\upsilon\;:\;\Upsilon}}}{\Gamma \vdash_\sm \Phi\D \slfrac{\tel_{\;\!\ell_\cblu{1}}}{_{\upsilon^\cblu{+}\;:\;\Upsilon\D}}}
  \hspace{1.2cm}
  \infer{\Gamma,~\lock_{\trisq} \vdash_\sm \Phi \slfrac{\tel_{\;\!\ell_\cblu{1}}}{_{\upsilon\;:\;\Upsilon}} \\ \Gamma,~\lock_{\trisq} \vdash_\sm \tau : \Phi}{\Gamma \vdash_\sm \tau\D : \Phi\D}
  \notag\\
  \infer{\Gamma,~\lock_{\trisq} \vdash_\sm \Phi \slfrac{\tel_{\;\!\ell_\cblu{1}}}{_{\upsilon\;:\;\Upsilon}} \\ \Gamma \vdash_\sm \plus{\tau} : \Phi\D}{\Gamma \vdash_\sm \plus{\tau}\ev : \dbkt{\Phi~\sqbkt{\key^{\trisq\le \id_\sm}}~\upsilon^\cblu{+}\ev}_{\upsilon^\cblu{+}\;:\;\Upsilon\D}}
  \notag\\
  \infer{\Gamma,~\lock_{\trisq} \vdash_\sm \Phi \slfrac{\tel_{\;\!\ell_\cblu{1}}}{_{\upsilon\;:\;\Upsilon}} \\ \Gamma,~\lock_{\trisq} \vdash_\sm \sigma : \Upsilon}{\Gamma \vdash_\sm \Phi\D~\sigma\D \equiv (\Phi~\sigma)\D}
  \label{eq:D-app}\\
  \infer{\Gamma,~\lock_{\trisq} \vdash_\sm \Phi \slfrac{\tel_{\;\!\ell_\cblu{1}}}{_{\upsilon\;:\;\Upsilon}} \\ \Gamma,~\lock_{\trisq} \vdash_\sm \sigma : \Upsilon \\ \Gamma,~\lock_{\trisq} \vdash_\sm \tau : \Phi}{\Gamma \vdash_\sm \tau\D~\sigma\D \equiv (\tau~\sigma)\D}
  \notag\\
  \infer{\Gamma,~\lock_{\trisq} \vdash_\sm \Phi \slfrac{\tel_{\;\!\ell_\cblu{1}}}{_{\upsilon\;:\;\Upsilon}} \\ \Gamma \vdash_\sm \plus{\sigma} : \Upsilon\D \\ \Gamma \vdash_\sm \tau : \Phi}{\Gamma \vdash_\sm \tau\D\;\ev~\plus{\sigma} \equiv \tau~\plus{\sigma}\ev}
  \notag
\end{gather}
\end{spreadlines}
We also require that this operation reduce to ordinary d\'ecalage on constant meta-abstractions, and commute appropriately with telescope concatenation, both globally and inside further meta-abstractions.
\begin{spreadlines}{10pt}
\begin{gather}
  \infer{\Gamma,~\lock_\trisq \vdash_\sm \Upsilon\tel_{\;\!\ell_\cblu{0}} \\ \Gamma,~\lock_\trisq \vdash_\sm \Phi\tel_{\;\!\ell_\cblu{1}} \\ \Gamma \vdash_\sm \plus{\sigma} : \Upsilon\D}{\Gamma \vdash_\sm \dbkt{\Phi}_{\upsilon\;:\;\Upsilon}\D~\plus{\sigma} \equiv \Phi\D}
  \notag\\
  \infer{\Gamma,~\lock_\trisq \vdash_\sm \Upsilon\tel_{\;\!\ell_\cblu{0}} \\ \Gamma,~\lock_\trisq \vdash_\sm \Phi\tel_{\;\!\ell_\cblu{1}} \\ \Gamma \vdash_\sm \plus{\sigma} : \Upsilon\D \\ \Gamma \vdash_\sm \delta : \Phi}{\Gamma \vdash_\sm \dsqbkt{\delta}_{\;\upsilon\;:\;\Upsilon}\D~\plus{\sigma} \equiv \delta\D}
  \notag\\
  \infer{\Gamma,~\lock_\trisq \vdash_\sm \Upsilon\tel_{\;\!\ell_\cblu{0}} \\ \Gamma,~\lock_\trisq\ext \Upsilon \vdash_\sm \Phi\tel_{\;\!\ell_\cblu{1}}}
  {(\Upsilon\ext\Phi)\D \equiv \big((\upsilon^\cblu{+}:\Upsilon\D) \ext \dbkt{\Phi}_{\upsilon\;:\;\Upsilon}\D~\upsilon^\cblu{+}\big)}
  \label{eq:D-tel-concat}
\end{gather}
\begin{gather}
  \infer{
    \Gamma,~\lock_\trisq\ext \Upsilon \vdash_\sm \Phi\tel_{\;\!\ell_\cblu{1}}\\\Gamma,~\lock_\trisq \vdash_\sm \sigma : \Upsilon \\ \Gamma,~\lock_\trisq \vdash_\sm \delta : \Phi~\sqbkt{\id_\Gamma,~\lock_\trisq\ext \sigma}}
  {\sqbkt{\sigma \ext \delta}\D \equiv \sqbkt{\sigma\D\ext \delta\D}}
  \label{eq:D-psub-concat}
  \\
  \infer{\Gamma,~\lock_\trisq\ext\Theta \vdash_\sm \Upsilon\tel_{\;\!\ell_\cblu{1}} \\ \Gamma,~\lock_\trisq\ext\Theta\ext \Upsilon \vdash_\sm \Phi\tel_{\;\!\ell_\cblu{2}}}
  {\dbkt{(\Upsilon\ext\Phi)}_{\theta\;:\;\Theta}\D \equiv \dbkt{(\upsilon^\cblu{+}:\dbkt{\Upsilon}_{\theta\;:\;\Theta}\D~\theta^\cblu{+}) \ext \dbkt{\Phi}_{\theta\;:\;\Theta,\;\upsilon\;:\;\Upsilon}\D~\theta^\cblu{+}~\upsilon^\cblu{+}}_{\theta^\cblu{+}\;:\;\Theta\D}}
  \notag\\
  \infer{\Gamma,~\lock_\trisq\ext\Theta \vdash_\sm \sigma : \Upsilon \\ \Gamma,~ \lock_\trisq\ext\Theta \vdash_\sm \delta : \Phi~\sqbkt{\id_\Gamma,~\lock_\trisq\ext\id_\Theta\ext \sigma}}
  {\dsqbkt{\sigma \ext \delta}_{\;\theta\;:\;\Theta}\D \equiv \dsqbkt{\dsqbkt{\sigma}_{\;\theta\;:\;\Theta}\D~\theta^\cblu{+} \ext \dsqbkt{\delta}_{\;\theta\;:\;\Theta}\D~\theta^\cblu{+}}_{\;\theta^\cblu{+}\;:\;\Theta\D}}
  \notag
\end{gather}
\end{spreadlines}
Note that in~\eqref{eq:D-tel-concat} we have to meta-abstract $\Phi$ in order to apply d\'ecalage, since $\Phi$ itself is not in a $\trisq$-locked context.
In~\eqref{eq:D-psub-concat}, $\delta\D$ is supposed to inhabit $\dbkt{\Phi}^{\mathsf{\cred{D}}}_{\upsilon\;:\;\Upsilon}~\sigma\D$, but by~\eqref{eq:D-app} and $\beta$-reduction for meta-abstractions this is equal to $(\Phi~\sqbkt{\id_\Gamma\ext \sigma})\D$, which is the natural type of $\delta\D$.
The last two equations are similar.

\begin{remark}\label{rmk:acks}
  The rules for telescope d\'ecalage from \cref{sec:telescope-decalage} and this section should be compared with the \emph{`local theory'} from~\cite{acks:iparam-noint}, with telescopes and telescope concatenation replacing types and $\Sigma$-types.
  The $\Upsilon\D$ from \cref{sec:telescope-decalage} corresponds to their $\forall\A$, while the $\dbkt{\Phi}_{\Upsilon}\D$ from this section corresponds to their $\forall\!\cblu{d} (\x.\B)$.
  The $\sigma\D$ from \cref{sec:telescope-decalage} corresponds to their $R$ (although with an added modal lock), while the $\dsqbkt{\delta}_{\;\upsilon\;:\;\Upsilon}\D$ from this section corresponds to their $\mathrm{\cred{apd}}$ (with modal lock).
  We don't have their $\mathrm{\cred{ap}}$ represented explicitly, but one of their rules says it is equivalent to $\mathrm{\cred{apd}}$ in a constant family.
  Our $\ev$ is their $\cblu{k}$ (in the unary case, so $\cblu{k}=\cred{0}$), and we do not have their $\S$; as discussed before, the modal guards make symmetry unnecessary.
  And, of course, we don't have $\Pi$-telescopes or a universe; we have those only for display, which is indexed. \bbox
\end{remark}

\subsubsection{Computing meta-abstracted d\'ecalage}
\label{sec:comp-metaabsdec}

Like ordinary d\'ecalage, meta-abstracted d\'ecalage computes on telescopes that are made out of types.
For brevity we omit the typing premises of these equalities, but we emphasize that all the meta-variables on the left-hand sides such as $\Theta,\A,\sigma,\t$ can depend nontrivially on the abstraction variables $\upsilon$ or $\plus{\upsilon}$.

\begin{align*}
  \dbkt{\theta:\Theta,~\x : \A}_{\upsilon\;:\;\Upsilon}\D
  &\equiv \scaleto{\newllparenthesis}{1.2em}\;\theta^\cblu{+}:\dbkt{\Theta}_{\upsilon\;:\;\Upsilon}\D~\upsilon^\cblu{+},~\x : \A~\sqbkt{\key^{\trisq\le \id_\sm} \ext\upsilon^\cblu{+}\ev\ext \theta^\cblu{+}\ev~\upsilon^\cblu{+}},\\
  &\quad\quad\quad\quad\quad\quad\quad\quad\quad\quad\quad\ \ \  \xp : \dbkt{\A}_{\upsilon\;:\;\Upsilon\;\ext\;\theta\;:\;\Theta}\d~\upsilon^\cblu{+}~\theta^\cblu{+}~\x\;\scaleto{\newrrparenthesis}{1.2em}_{\upsilon^\cblu{+}\;:\;\Upsilon\D}
  \\
  \dsqbkt{\sigma,~\t}_{\;\upsilon \;:\; \Upsilon}\D
  &\equiv \dsqbkt{\dsqbkt{\sigma}_{\;\upsilon\;:\;\Upsilon}\D~\upsilon^\cblu{+},~\dsqbkt{\t}_{\;\upsilon\;:\;\Upsilon}~\upsilon^\cblu{+}\ev,~\dsqbkt{\t}_{\;\upsilon\;:\;\Upsilon}\d~\upsilon^\cblu{+}}_{\;\upsilon^\cblu{+}\;:\;\Upsilon\D}\\
  \dsqbkt{\plus{\sigma},~\t,~\t\cblu{'}}_{\;\plus{\upsilon}\;:\;\Upsilon\D}\ev
  &\equiv \dsqbkt{\plus{\sigma}\ev,~\t}_{\;\plus{\upsilon}\;:\;\Upsilon\D}\\
  \dbkt{\theta:\Theta,~\x :^{\tri \circ \mu} \A}_{\upsilon\;:\;\Upsilon}\D
  &\equiv \scaleto{\newllparenthesis}{1.2em}\;\theta^\cblu{+} : \dbkt{\Theta}_{\upsilon\;:\;\Upsilon}\D~\upsilon^\cblu{+}, \\
  & \quad\quad\ \; \x :^{\tri\circ\mu} \A~\sqbkt{\key^{\trisq\le \id_\sm} \ext\upsilon^\cblu{+}\ev \ext \theta^\cblu{+}\ev~\upsilon^\cblu{+},~\lock_{\tri\circ\mu}}\;\scaleto{\newrrparenthesis}{1.2em}_{\upsilon^\cblu{+}\;:\;\Upsilon\D}\\
  \dsqbkt{\sigma,~\t}_{\;\upsilon \;:\; \Upsilon}\D
  &\equiv \dsqbkt{\dsqbkt{\sigma}_{\;\upsilon\;:\;\Upsilon}\D~\upsilon^\cblu{+},~\dsqbkt{\t}_{\;\upsilon\;:\;\Upsilon}~\upsilon^\cblu{+}\ev}_{\;\upsilon^\cblu{+}\;:\;\Upsilon\D}\\
  \dsqbkt{\plus{\sigma},~\t}_{\;\plus{\upsilon}\;:\;\Upsilon\D}\ev
  &\equiv \dsqbkt{\plus{\sigma}\ev,~\t}_{\;\plus{\upsilon}\;:\;\Upsilon\D}
  \tag*{\bbox}
\end{align*}

\subsubsection{Telescope display}
\label{sec:teldisp}

We also need a notion of indexed display for telescopes.
Note that this always gives a \emph{strict} telescope, even if its input is not.
\begin{mathpar}
  \infer{\Gamma,~\lock_\trisq \vdash_\sm \Upsilon \tel_{\;\!\ell} \\ \Gamma \vdash_\sm \sigma : \Upsilon~\sqbkt{\key^{\trisq\le \id_\sm}}}{\Gamma \vdash_\sm (\Upsilon\d~\sigma) \stel_{\;\!\ell}}
  \and
  \infer{\Gamma,~\lock_\trisq \vdash_\sm \Upsilon \tel_{\;\!\ell} \\ \Gamma,~\lock_\trisq \vdash_\sm \sigma : \Upsilon}{\Gamma \vdash_\sm \sigma\d : (\Upsilon\d~(\sigma~\sqbkt{\key^{\trisq\le \id_\sm}}))}
  \and
\end{mathpar}
Like d\'ecalage, telescope display computes on empty telescopes, and on telescopes extended by a type:
\begin{alignat*}{2}
  \ec^{\mathsf{\cred{d}}}_\sm~\esub_\sm &\equiv \ec_\sm\\
  \esub^{\mathsf{\cred{d}}}_{\sm} &\equiv \esub_\sm \\
  \bkt{\theta:\Theta,~\x : \A}\d~\sqbkt{\sigma,\t}
  &\equiv \mathrlap{\bkt{\theta\cblu{'}:\Theta\d~\sigma,~\xp : \dbkt{\A}_{\theta\;:\;\Theta}\d~\pair{\sigma}{\theta\cblu{'}}~\t}} \\
  \sqbkt{\sigma,~\t}\d
  &\equiv \sqbkt{\sigma\d,~\t\d} && \qquad \text{(for a non-modal variable)}\\
  \bkt{\theta:\Theta,~\x :^{\tri \circ \mu} \A}\d~\sqbkt{\sigma,\t}
  &\equiv \Theta\d~\sigma\\
  \sqbkt{\sigma,~\t}\d
  &\equiv \sigma\d&& \qquad \text{(for a modal variable)}
\end{alignat*}
As promised, this is the reason that the empty telescope must exist at all levels: if $\Upsilon$ consists only of modal variables, then $\Upsilon\d$ is empty, but it must be at the same level as $\Upsilon$.

Note that compared to d\'ecalage, telescope display reorders the variables.
For instance, we have
\begin{align*}
  (\x:\A,~ \y:\B)\D &\equiv  (\x:\A,~ \x\cblu{'}:\A\d~ \x,~ \y:\B,~ \y\cblu{'}:\B\d~ \y)\\
  (\x:\A,~ \y:\B)\d &\equiv \dbkt{\x\cblu{'}:\A\d~ \x,~ \y\cblu{'}:\B\d~ \y }_{\x\;:\;\A,\;\y\;:\;\B}\\
  (\x:\A,~ \y:\B) \ext (\x:\A,~ \y:\B)\d &\equiv (\x:\A,~\y:\b,~ \x\cblu{'}:\A\d~ \x,~ \y\cblu{'}:\B\d~ \y)\\
  &\not\equiv (\x:\A,~ \y:\B)\D.
\end{align*}
Thus, we instead relate telescope display to d\'ecalage by an \emph{`odds'} operation that picks out the elements of displayed types, and a \emph{`pairing'} operation that interleaves them together, such that evens and odds together form an isomorphism with pairing as inverse.
\begin{mathpar}
  \infer{\Gamma \vdash_\sm \plus{\sigma} : \Upsilon\D}{\Gamma \vdash_\sm \plus{\sigma}\od : \Upsilon\d~\sigma\ev}
  \and
  \infer{\Gamma \vdash_\sm \sigma : \Upsilon~\sqbkt{\key^{\trisq\le \id_\sm}} \\ \Gamma \vdash_\sm \sigma\cblu{'} : \Upsilon\d~\sigma}{\Gamma \vdash_\sm \pair{\sigma}{\sigma\cblu{'}} : \Upsilon\D}
  \\
  \plus{\sigma} \equiv \pair{\plus{\sigma}\ev}{\plus{\sigma}\od}
  \and
  \pair{\sigma}{\sigma\cblu{'}}\ev \equiv \sigma
  \and
  \pair{\sigma}{\sigma\cblu{'}}\od \equiv \sigma\cblu{'}
  \\
  \sigma^{\mathsf{\cred{D}}\;\mathsf{\cred{od}}} \equiv \sigma\d
  \and
  \pair{\sigma~\sqbkt{\key^{\trisq\le \id_\sm}}}{\sigma\d} \equiv \sigma\D
\end{mathpar}
These operations also compute on empty telescopes and on telescopes extended by a type:
\begin{alignat*}{2}
  \esub^{\mathsf{\cred{od}}}_\sm &\equiv \esub_\sm \\
  \pair{\esub_\sm}{\esub_\sm} &\equiv \esub_\sm\\
  \sqbkt{\plus{\sigma},~\t,~\t\cblu{'}}\od
  &\equiv \sqbkt{\plus{\sigma}\od,~\t\cblu{'}} && \quad\text{(for a non-modal variable)}\\
  \pair{\sqbkt{\sigma,~\t}}{\sqbkt{\sigma\cblu{'},~\t\cblu{'}}}
  &\equiv \sqbkt{\pair{\sigma}{\sigma\cblu{'}},~\t,~\t\cblu{'}} && \quad\text{(for a non-modal variable)}\\
  \sqbkt{\plus{\sigma},~\t}\od
  &\equiv \plus{\sigma}\od && \quad\text{(for a modal variable)}\\
  \pair{\sqbkt{\sigma,~\t}}{\sigma\cblu{'}}
  &\equiv \sqbkt{\pair{\sigma}{\sigma\cblu{'}},~\t} && \quad\text{(for a modal variable)}
\end{alignat*}

\subsubsection{Meta-abstracted telescope display}
\label{sec:meta-abstr-displ}

Unsurprisingly, we generalise telescope display to apply to meta-abstracted telescopes as well, with rules combining those of \cref{sec:displ-meta-abstr,sec:teldisp}.
First we have the basic rules:
\begin{mathpar}
  \infer{\Gamma,~\lock_\trisq \vdash_\sm \Phi \slfrac{\tel_{\;\!\ell_\cblu{1}}}{_{\upsilon\;:\;\Upsilon}}}{\Gamma \vdash_\sm \Phi\d \slfrac{\stel_{\;\!\ell_\cblu{1}}}{_{\plus{\upsilon}\;:\;\Upsilon\D,\;\phi\;:\;\Phi\;\sqbkt{\key^{\trisq\le \id_\sm}}\;\plus{\upsilon}\ev}}}
  \and
  \infer{\Gamma,~\lock_\trisq \vdash_\sm \Phi \slfrac{\tel_{\;\!\ell_\cblu{1}}}{_{\upsilon\;:\;\Upsilon}} \\ \Gamma,~\lock_\trisq \vdash_\sm \delta : \Phi}{\Gamma \vdash_\sm \delta\d: \dbkt{\Phi\d~\plus{\upsilon}~(\delta~\plus{\upsilon}\ev)}_{\plus{\upsilon}\;:\;\Upsilon\D}}
  \and
  \infer{\Gamma,~\lock_\trisq \vdash_\sm \Phi \slfrac{\tel_{\;\!\ell_\cblu{1}}}{_{\upsilon\;:\;\Upsilon}} \\ \Gamma,~\lock_\trisq \vdash_\sm \sigma : \Upsilon \\ \Gamma \vdash_\sm \t : (\Phi~\sigma)~\sqbkt{\key^{\trisq\le \id_\sm}}}
  {\Gamma \vdash_\sm \Phi\d~\sigma\D~\t \equiv (\Phi~\sigma)\d~\t}
  \and
  \infer{\Gamma,~\lock_\trisq \vdash_\sm \Phi \slfrac{\tel_{\;\!\ell_\cblu{1}}}{_{\upsilon\;:\;\Upsilon}} \\ \Gamma,~\lock_\trisq \vdash_\sm \sigma : \Upsilon \\ \Gamma,~\lock_\trisq \vdash_\sm \ct : \Phi}
  {\Gamma \vdash_\sm \ct\d~\sigma\D \equiv (\t~\sigma)\d}
\end{mathpar}
Then we have the odds/pairing isomorphism:
\begin{mathpar}
  \infer{\Gamma,~\lock_\trisq \vdash_\sm \Phi \slfrac{\tel_{\;\!\ell_\cblu{1}}}{_{\upsilon\;:\;\Upsilon}}\\ \Gamma \vdash_\sm \delta^\cblu{+} : \Phi\D}{\Gamma \vdash_\sm \delta^\cblu{+}\od : \dbkt{\Phi\d~\plus{\upsilon}~  (\delta^\cblu{+}\ev~\plus{\upsilon})}_{\plus{\upsilon}\;:\;\Upsilon\D}}
  \and
  \infer{\Gamma,~\lock_\trisq \vdash_\sm \Phi \slfrac{\tel_{\;\!\ell_\cblu{1}}}{_{\upsilon\;:\;\Upsilon}} \\
    \Gamma \vdash_\sm \delta : \dbkt{\Phi~\sqbkt{\key^{\trisq\le \id_\sm}}~\plus{\upsilon}\ev}_{\plus{\upsilon}\;:\;\Upsilon\D} \\
    \Gamma \vdash_\sm \delta\cblu{'} : \dbkt{\Phi\d~\plus{\upsilon}~(\delta~\plus{\upsilon})}_{\plus{\upsilon}\;:\;\Upsilon\D}}
  {\Gamma \vdash_\sm \pair{\delta}{\delta\cblu{'}} : \Phi\D}
  \\
  \delta \equiv \pair{\delta\ev}{\delta\od}
  \and
  \pair{\delta}{\delta\cblu{'}}\ev \equiv \delta
  \and
  \pair{\delta}{\delta\cblu{'}}\od \equiv \delta\cblu{'}
  \\
  \delta^{\mathsf{\cred{D}}\;\mathsf{\cred{od}}} \equiv \delta\d
  \and
  \pair{\delta~\sqbkt{\key^{\trisq\le \id_\sm}}}{\delta\d} \equiv \delta\D
\end{mathpar}
On constant meta-abstractions, this reduces to ordinary telescope display:
\begin{mathpar}
  \infer{\Gamma,~\lock_\trisq \vdash_\sm \Upsilon\tel_{\;\!\ell_\cblu{0}} \\ \Gamma,~\lock_\trisq \vdash_\sm \Phi\tel_{\;\!\ell_\cblu{1}} \\ \Gamma \vdash_\sm \plus{\sigma} : \Upsilon\D \\ \Gamma \vdash_\sm \delta : \Phi~\sqbkt{\key^{\trisq\le \id_\sm}\ext \plus{\sigma}}}{\Gamma \vdash_\sm \dbkt{\Phi}_{\upsilon\;:\;\Upsilon}\d~\plus{\sigma}~\delta \equiv \Phi\d~\delta}
  \and
  \infer{\Gamma,~\lock_\trisq \vdash_\sm \Upsilon\tel_{\;\!\ell_\cblu{0}} \\ \Gamma,~\lock_\trisq \vdash_\sm \Phi\tel_{\;\!\ell_\cblu{1}} \\ \Gamma \vdash_\sm \plus{\sigma} : \Upsilon\D \\ \Gamma \vdash_\sm \delta : \Phi~\sqbkt{\key^{\trisq\le \id_\sm}\ext \plus{\sigma}}}{\Gamma \vdash_\sm \dsqbkt{\delta}_{\;\upsilon\;:\;\Upsilon}\d~\plus{\sigma} \equiv \delta\d}
\end{mathpar}
And we have computation rules for telescope extensions:
\begin{mathpar}
  \infer{\Gamma,~\lock_\trisq \vdash_\sm \Upsilon\tel_{\;\!\ell_\cblu{0}} \\ \Gamma,~\lock_\trisq\ext \Upsilon \vdash_\sm \Phi\tel_{\;\!\ell_\cblu{1}}\\
  \Gamma \vdash_\sm \sigma : \Upsilon~\sqbkt{\key^{\trisq\le \id_\sm}} \\ \Gamma \vdash_\sm \delta : \Phi~\sqbkt{\key^{\trisq\le \id_\sm}\ext \sigma}}
  {(\Upsilon\ext\Phi)\d~\sigma~\delta \equiv \left((\upsilon\cblu{'}:\Upsilon\d~\sigma) \ext \dbkt{\Phi}_{\upsilon\;:\;\Upsilon}\d~\pair{\sigma}{\upsilon\cblu{'}}~\delta\right)}
  \and
  \infer{\Gamma,~\lock_\trisq \vdash_\sm \Upsilon\tel_{\;\!\ell_\cblu{0}} \\ \Gamma,~\lock_\trisq\ext \Upsilon \vdash_\sm \Phi\tel_{\;\!\ell_\cblu{1}}\\
         \Gamma,~\lock_\trisq \vdash_\sm \sigma : \Upsilon \\ \Gamma,~\lock_\trisq \vdash_\sm \delta : \Phi~\sqbkt{\id_\Gamma,~\lock_\trisq\ext \sigma}}
  {\sqbkt{\sigma \ext \delta}\d \equiv \sqbkt{\sigma\d\ext \delta\d}}
  \\
  \infer{\Gamma,~\lock_\trisq\ext\Theta \vdash_\sm \Upsilon\tel_{\;\!\ell_\cblu{0}} \\ \Gamma,~\lock_\trisq\ext\Theta\ext \Upsilon \vdash_\sm \Phi\tel_{\;\!\ell_\cblu{1}}}
  {\dbkt{(\Upsilon\ext\Phi)}_{\theta\;:\;\Theta}\d \equiv \dbkt{(\upsilon\cblu{'}:\dbkt{\Upsilon}_{\theta\;:\;\Theta}\d~\theta^\cblu{+}~\upsilon) \ext \dbkt{\Phi}_{\theta\;:\;\Theta,\;\upsilon\;:\;\Upsilon}\d~\theta^\cblu{+}~\pair{\upsilon}{\upsilon\cblu{'}}~\phi}_{\theta^\cblu{+}\;:\;\Theta\D,\;\upsilon\;:\;\Upsilon,\;\phi\;:\;\Phi}}
  \and
  \infer{\Gamma,~\lock_\trisq\ext\Theta \vdash_\sm \sigma : \Upsilon \\ \Gamma,~\lock_\trisq\ext\Theta \vdash_\sm \delta : \Phi~\sqbkt{\id_\Gamma,~\lock_\trisq\ext\id_\Theta\ext \sigma}}
  {\dsqbkt{\sigma \ext \delta}_{\;\theta\;:\;\Theta}\d \equiv \dsqbkt{\dsqbkt{\sigma}_{\;\theta\;:\;\Theta}\d~\theta^\cblu{+} \ext \dsqbkt{\delta}_{\;\theta\;:\;\Theta}\d~\theta^\cblu{+}}_{\;\theta^\cblu{+}\;:\;\Theta\D}}
  \eqno\sectend
\end{mathpar}

\subsubsection{Computing meta-abstracted telescope display}
\label{sec:comp-matd}

These computation rules are analogous to the previous ones, first in the non-modal case:
\begin{align*}
  \dbkt{\theta:\Theta,~\x : \A}_{\upsilon\;:\;\Upsilon}\d~\sqbkt{\plus{\sigma}\ext\delta,~\t}
  &
  \begin{multlined}[t][0.5\linewidth]
    \equiv \left(\theta\cblu{'}:\dbkt{\Theta}_{\upsilon\;:\;\Upsilon}\d~\plus{\sigma}~\delta,\right.\\
    \left.\xp : \dbkt{\A}_{\upsilon\;:\;\Upsilon,\;\theta\;:\;\Theta}\d~\sigma~\pair{\delta}{\theta\cblu{'}}~\t\right)
  \end{multlined}
  \\
  \dsqbkt{\delta,~\t}_{\;\upsilon\;:\;\Upsilon}\d~\plus{\sigma}
  &\equiv \sqbkt{\dsqbkt{\delta}_{\;\upsilon\;:\;\Upsilon}\d~\plus{\sigma},~\dsqbkt{\t}_{\;\upsilon\;:\;\Upsilon}\d~\plus{\sigma}}\\
  \dsqbkt{\plus{\delta},~\t,~\t\cblu{'}}_{\;\plus\upsilon\;:\;\Upsilon\D}\od~\plus{\sigma}
  &\equiv \sqbkt{\dsqbkt{\plus{\delta}}_{\;\plus\upsilon\;:\;\Upsilon\D}\od~\plus{\sigma},~\dsqbkt{\t\cblu{'}}_{\;\plus\upsilon\;:\;\Upsilon\D}~\plus{\sigma}}\\
  \pair{\dsqbkt{\delta,~\t}_{\;\plus\upsilon\;:\;\Upsilon\D}}{\dsqbkt{\delta\cblu{'},~\t\cblu{'}}_{\;\plus\upsilon\;:\;\Upsilon\D}}~\plus{\sigma}
  &\equiv
  \begin{multlined}[t][0.5\linewidth]
    [\;\pair{\dsqbkt{\delta}_{\;\plus\upsilon\;:\;\Upsilon\D}}{\dsqbkt{\delta\cblu{'}}_{\;\plus\upsilon\;:\;\Upsilon\D}}~\plus{\sigma},\\
    \dsqbkt{\t}_{\;\plus\upsilon\;:\;\Upsilon\D}~\plus{\sigma},~\dsqbkt{\t\cblu{'}}_{\;\plus\upsilon\;:\;\Upsilon\D}~\plus{\sigma}\;]
  \end{multlined}
\end{align*}
and then the modal case:
\begin{align*}
  \dbkt{\theta:\Theta,~\x :^{\tri \circ \mu} \A}_{\upsilon\;:\;\Upsilon}\d~\sqbkt{\plus{\sigma}\ext\delta,~\t}
  &\equiv \dbkt{\Theta}_{\upsilon\;:\;\Upsilon}\d~\plus{\sigma}~\delta\\
  \dsqbkt{\delta,~\t}_{\;\upsilon\;:\;\Upsilon}\d~\plus{\sigma}
  &\equiv \dsqbkt{\delta}_{\;\upsilon\;:\;\Upsilon}\d~\plus{\sigma}\\
  \dsqbkt{\delta,~\t}_{\;\plus\upsilon\;:\;\Upsilon\D}\od~\plus{\sigma}
  &\equiv \dsqbkt{\delta}_{\;\plus\upsilon\;:\;\Upsilon\D}\od~\plus{\sigma}\\
  \pair{\dsqbkt{\delta,~\t}_{\;\plus\upsilon\;:\;\Upsilon\D}}{\dsqbkt{\delta\cblu{'}}_{\;\plus\upsilon\;:\;\Upsilon\D}}~\plus{\sigma}
  &\equiv
  \begin{multlined}[t][0.4\linewidth]
    [\;\pair{\dsqbkt{\delta}_{\;\plus\upsilon\;:\;\Upsilon\D}}{\dsqbkt{\delta\cblu{'}}_{\;\plus\upsilon\;:\;\Upsilon\D}}~\plus{\sigma},\\
    \dsqbkt{\t}_{\;\plus\upsilon\;:\;\Upsilon\D}~\plus{\sigma}\;]
    \tag*{\bbox}
  \end{multlined}
\end{align*}

\subsubsection{Computing display on $\Pi$-telescopes}
\label{sec:computing-display-pi}

Finally, we give the following rules for computing display and meta-abstracted display of a $\Pi$-telescope.
These are consistent with the rules for computing display on telescopes extended by a type.
\begin{gather*}
  ((\phi:\Phi) \to \Theta~\phi)\d~\delta
  \equiv (\plus{\phi} : \Phi\D) \to \Theta\d~\plus{\phi}\ev~(\delta~\plus{\phi}\ev)\\[10pt]
  \begin{multlined}[][0.9\linewidth]
    \dbkt{(\phi:\Phi~\upsilon) \to \Theta~\upsilon~\phi}_{\upsilon \;:\; \Upsilon}\d\\
    \equiv \dbkt{(\plus\phi:\Phi\D~\plus\upsilon) \to \Theta\d~\plus\upsilon~\plus\phi~(\delta~\plus\phi\ev)}_{\plus{\upsilon}\;:\;\Upsilon\D, \;\delta\;:\;(\phi\;:\;\Phi\;\plus{\upsilon}\ev)\;\to\; \Theta\;\plus\upsilon\ev\;\phi}
  \end{multlined}
\end{gather*}
However, it does not seem possible to give rules for computing \emph{d\'ecalage} on a $\Pi$-telescope that are similarly consistent.
Fortunately, we will not need such rules.
Of course, since d\'ecalage and $\Pi$-telescopes both compute independently on telescopes extended by a type, so does their combination.

This concludes our description of the ambient syntax of dTT. \bbox

\section{Semi-Simplicial and Displayed Coinductive Types}
\label{sec:ssts}

Recall from the introduction that our primary goal in formulating dTT (at the moment) is to have a type theory in which we can make precise our coinductive definition of the type $\SST$ of semi-simplicial types.
In this section we give that definition, making use of the \emph{`display'} primitives of dTT that was introduced in \cref{sec:syntax}.
The basic definition is contained in \cref{subsec:semi-simplicial-types}, followed by an exploration of some examples in \cref{sec:eg-sst}.
Then in \cref{sec:dcoind} we describe a more general notion of \emph{`displayed coinductive type'} that has $\SST$ as a special case, and in \cref{sec:exampl-displ-coind} we explore a few other examples of the general notion.

\subsection{Semi-simplicial types}
\label{subsec:semi-simplicial-types}

In an ideal version of displayed type theory, one could define semi-simplicial types as an instance of a general \textbf{codata} decleration. We would expect to write this in a proof assistant with a syntax like the following, which generalises Agda-like syntax for records by allowing the coinductive input of each destructor to be specified explicitly and referred to in its type:
\begin{lstlisting}
(*\textbf{codata}*) (*$\SST$*) : (*$\Type$*) (*\textbf{where}*)
  (*$\Z$*) : (*$\SST$*) (*$\to$*) (*$\Type$*)
  (*$\S$*) : ((*$\A$*) : (*$\SST$*)) (*$\to$*) (*$\Z$*) (*$\A$*) (*$\to$*) (*$\SST\d$*) (*$\A$*)
\end{lstlisting}
It is beyond the scope of this paper to give a sufficiently broad framework to generally encompass such definitions, but we will describe one general paradigmatic class of them, analogous to W-types as paradigmatic inductive types and M-types as paradigmatic coinductive types.
However, we begin by discussing the concrete example of $\SST$ in more detail, to help motivate the general case.

\subsubsection{SST basics}
\label{sec:sst-basics}

We begin by giving the type formation law and destructors.
Of course, since $\SST$ is a sort of \emph{`universe'}, its elements consisting of types, it must also be parametrized by a level.

\vspace{-0.4cm}
\[\infer{\hspace{0.1cm}}{\Shortstack{{$\Gamma \vdash_\sm \SST_{\;\!\ell} \type_{\;\!\lsuc\;\ell}$} \hspace{0.1cm} {$\Gamma \vdash_\sm \Z : \dbkt{\Type_{\;\!\ell}}_{\A\;:\;\SST_{\ell}}$} \hspace{0.1cm} {$\Gamma \vdash_\sm \S: \dbkt{\SST^{\mathsf{\cred{d}}}_{\ell}~\A}_{\A\;:\;\SST_{\ell},\;\a\;:\;\El\;\bkt{\Z\;\A}}$}}}\]
Note that the destructors are defined as terms belonging to `meta-abstractions' as introduced in \cref{sec:meta-abstractions}.
We have chosen this over the more common method of supplying the arguments in premises, e.g.
\[ \infer{\Gamma \vdash_\sm \A : \SST_{\;\!\ell}}{\Gamma \vdash_\sm \Z~\A : \Type_{\;\!\ell}}, \]
because it makes it easier to compute $\d$ of them:
\begin{gather*}
\Gamma \vdash_\sm \SST^{\mathsf{\cred{d}}}_{\ell} \slfrac{\type_{\;\!\lsuc\;\ell}}{_{\A_\ze\;:\;\SST_{\ell}}}
\\
\Gamma \vdash_\sm \Z\d : \dbkt{\El~\bkt{\Z~\A_\ze} \to \Type_{\;\!\ell}}_{\set{\A_\ze\;:\;\SST_{\ell}},\;\A_\x\;:\;\SST^{\mathsf{\cred{d}}}_{\ell}~\A_\ze}
\\
\Gamma \vdash_\sm \S\d : \dbkt{\SST^{\mathsf{\cred{dd}}}_{\ell}\;\A_\ze\;\A_\x\;\bkt{\S\;\A_\ze\;\a_\ze}}_{\set{\A_\ze\,:\,\SST_{\ell}},\,\A_\x\,:\,\SST^{\mathsf{\cred{d}}}_{\ell}\,\A_\ze,\,\a_\ze\,:\,\El\,\bkt{\Z\,\A_\ze},\,\a_\x\,:\,\El\,\bkt{\Z\d\,\A_\x\,\a_\ze}}
\end{gather*}
What this calculation suggests is that the family $\SST\d$ should behave as though defined by computing $\d$ on all of the destructors in the code block above:
\begin{lstlisting}
(*\textbf{codata}*) (*$\SST\d$*) ((*$\A_\ze$*) : (*$\SST$*)) : (*$\Type$*) (*\textbf{where}*)
  (*$\Z\d$*) : (*$\SST\d$*) (*$\A$*) (*$\to$*) (*$\Z$*) (*$\A$*) (*$\to$*) (*$\Type$*)
  (*$\S\d$*) : ((*$\A_\x$*) : (*$\SST\d$*) (*$\A$*)) ((*$\a_{\ze}$*) : (*$\Z$*) (*$\A_\ze$*)) (*$\to$*) (*$\Z\d$*) (*$\A_\x$*) (*$\a_{\ze}$*) (*$\to$*) (*$\SST\d\d$*) (*$\A_\ze$*) (*$\A_\x$*) ((*$\S$*) (*$\A_\ze$*) (*$\a_{\ze}$*))
\end{lstlisting}
Unfortunately, as we will see this is not actually possible in our theory, but it is a useful intuition.
In general, the types obtained by iterating $\d$ $\n$-times on $\Z$ and $\S$ will begin by taking a $\n$-fold dependent $\SST$ in a generic augmented simplicial context of $\SST$s of lower dependency. This context can be generally inferred from the type of $\n$-fold dependent $\SST$, and we have thus chosen to make those arguments implicit, which aligns with the syntactic presentation in the introduction. In particular, the formula for $\A_\cred{2}$ is given by:
\[\Z\d\d~\bkt{\S\d~\bkt{\S~\A~\cblu{x\sub{001}}}~\cblu{x\sub{010}}~\cblu{\beta\sub{011}}} ~\cblu{x\sub{100}}~\cblu{\beta\sub{101}}~\cblu{\beta\sub{110}}\]
as opposed to:
\[\Z\d\d~\A~(\S~\A~\cblu{x\sub{001}})~(\S~\A~\cblu{x\sub{010}})~ \bkt{\S\d~\A~\bkt{\S~\A~\cblu{x\sub{001}}}~\cblu{x\sub{010}}~\cblu{\beta\sub{011}}} ~\cblu{x\sub{100}}~\cblu{\beta\sub{101}}~\cblu{\beta\sub{110}}. \eqno\bbox\]

\subsubsection{The coinduction principle}
\label{sec:sst-coinduction}

Suppose that we want to construct a function mapping into $\SST$ from a telescope of arbitrary length. We first think purely in terms of code, written in the style of \texttt{Agda}-esque copattern matching, with the goal of writing down something that can conceivably be justified:
\begin{lstlisting}
(*$\f$*) : (*$\X$*) (*$\to$*) (*$\SST$*)
(*$\Z$*) ((*$\f$*) (*$\t$*)) = ((*\cprime{?$_{\mathsf{Z}_0}$}*) : (*$\Type$*))
(*$\S$*) ((*$\f$*) (*$\t$*)) (*$\a$*) = (*$\f\d$*) (*$\t$*) ((*\cprime{?$_{\mathsf{S}_0}$}*) : (*$\X\d$*) (*$\t$*))

(*$\g$*) : ((*$\t$*) : (*$\X$*)) (*$\to$*) (*$\Y$*) (*$\t$*) (*$\to$*) (*$\SST$*)
(*$\Z$*) ((*$\g$*) (*$\t$*) (*$\s$*)) = ((*\cprime{?$_{\mathsf{Z}_1}$}*) : (*$\Type$*))
(*$\S$*) ((*$\g$*) (*$\t$*) (*$\s$*)) (*$\a$*) = (*$\g\d$*) (*$\t$*) ((*\cprime{?$_{\mathsf{S}_1}$}*) : (*$\X\d$*) (*$\t$*)) (*$\s$*) ((*\cprime{?$_{\mathsf{S}_2}$}*) : (*$\Y\d$*) (*$\t$*) (*\cprime{?$_{\mathsf{S}_1}$}*) (*$\s$*))
\end{lstlisting}
Here, suppose that $\Gamma,~\lock_\trisq \vdash_\sm \Upsilon \tel_{\;\!\ell^\cblu{'}}$. If we think of $\Upsilon$ as a \emph{state space} and $\sqbkt{\sigma : \Upsilon}$ as a \emph{state}. Then the above definition suggests that we are able to define $\f : \bkt{\upsilon : \Upsilon} \to \SST_{\ell}$ provided that we are able to provide two ingredients. First, we need a way of extracting $\sqbkt{\cred{\widetilde{\Z}}~\sigma : \Type_{\;\!\ell}}$, a type of $\cred{0}$-simplices, from a state $\sigma$. Second, we need a way of extracting $\sqbkt{\cred{\widetilde{\S}}~\sigma~\a : \Upsilon\d~\sigma}$, a dependent section of $\Upsilon$ over $\sigma$, from a state $\sigma$ and a $\cred{0}$-simplex $\sqbkt{\a : \cred{\widetilde{\Z}}~\sigma}$. This suggests that a reasonable coinduction principle for $\SST_{\ell}$ is the following:
\begin{mathpar}
\infer{\Gamma,~\lock_\trisq \vdash_\sm \Upsilon \ctx_{\;\!\ell^\cblu{'}} \\ \Gamma,~\lock_\trisq \vdash_\sm \cred{\widetilde{\Z}} : \dbkt{\Type_{\;\!\ell}}_{\delta\;:\;\Upsilon} \\ \Gamma,~\lock_\trisq \vdash_\sm \cred{\widetilde{\S}} :  \dbkt{\Upsilon\d~\delta}_{\;\delta\;:\;\Upsilon,\;\a\;:\;\El\;\bkt{\cred{\widetilde{\Z}}\;\delta}}}{\Gamma \vdash_\sm \mathsf{\R_\Upsilon}~\cred{\widetilde{\Z}} ~\cred{\widetilde{\S}} : \dbkt{\SST_{\ell}}_{\delta\;:\;\Upsilon}}
\end{mathpar}
and that its computation rules should be:
\begin{align*}
\Z~\big(\mathsf{\R_\Upsilon}~\cred{\widetilde{\Z}} ~\cred{\widetilde{\S}}~\sigma\big) &\equiv \cred{\widetilde{\Z}}~\sigma \\
\S~\big(\mathsf{\R_\Upsilon}~\cred{\widetilde{\Z}} ~\cred{\widetilde{\S}}~\sigma\big)~\a_\ze &\equiv \big(\mathsf{\R_\Upsilon}~\cred{\widetilde{\Z}} ~\cred{\widetilde{\S}}\big)\d~\pair{\sigma}{\cred{\widetilde{\S}}~\sigma~\a_\ze}.
\end{align*}
Now, the expression $\big(\mathsf{\R_\Upsilon}~\cred{\widetilde{\Z}} ~\cred{\widetilde{\S}}\big)\d$ defines a meta abstracted-term of meta-abstracted type $\dbkt{\SST^{\mathsf{\cred{d}}}_{\ell}~\big(\mathsf{\R_\Upsilon}~\cred{\widetilde{\Z}} ~\cred{\widetilde{\S}}~\upsilon^\cblu{+}\ev\big)}_{\upsilon^\cblu{+}\;:\;\Upsilon\D}$. One reasonable hope is that the display in the above line could be computed in terms of a corecursor for $\SST^{\mathsf{\cred{d}}}_{\ell}$. However, this approach runs into issues. Towards this aim, let us more generally try to work out the coinduction principle that would let us define $\f : \bkt{\x : \X} \to \SST^{\mathsf{\cred{d}}}_{\ell}~\bkt{\A~\x}$ for $\Gamma,~\lock_\trisq \vdash_\sm \A : \X \to \SST^{\mathsf{\cred{d}}}_{\ell}$. We apply the same methodology as before, and start by writing down reasonable looking code:
\begin{lstlisting}
(*$\f$*) : ((*$\t$*) : (*$\X$*)) (*$\to$*) (*$\SST\d$*) ((*$\A$*) (*$\t$*))
(*$\Z\d$*) ((*$\f$*) (*$\t$*)) (*$\a$*) = ((*\cprime{?$_{\mathsf{Z}_2}$}*) : (*$\Type$*))
(*$\S\d$*) ((*$\f$*) (*$\t$*)) (*$\a$*) (*$\b$*) = (*$\f\d$*) (*$\t$*) ((*\cprime{?$_{\mathsf{S}_3}$}*) : (*$\X\d$*) (*$\t$*))
\end{lstlisting}
However, we then have:
\begin{align*}
\Gamma,~\lock_\trisq,~\t : \X,~\a : \El~\bkt{\Z~\bkt{\A~\t}}, ~\b : \El~\cprime{?_{\mathsf{Z}_2}} &\vdash_\sm \S\d~\bkt{\f~\t}~\a~\b : \SST^{\mathsf{\cred{dd}}}_{\ell} ~\bkt{\A~\t}~\bkt{\f~\t}~\bkt{\S~\bkt{\A~\t}~\a} \\
\Gamma,~\lock_\trisq,~\t : \X,~\a : \El~\bkt{\Z~\bkt{\A~\t}}, ~\b : \El~\cprime{?_{\mathsf{Z}_2}} &\vdash_\sm \f\d~\t~\cprime{?_{\mathsf{S}_3}} : \SST^{\mathsf{\cred{dd}}}_{\ell}~\bkt{\A~\t}~\bkt{\A\d~\t~\cprime{?_{\mathsf{S}_3}}}~\bkt{\f~\t}
\end{align*}
We see then that there is an index ordering mismatch that seems to prevent us from writing down a coinduction principle for $\SST^{\mathsf{\cred{d}}}_{\ell}$ corresponding to a simple class of syntactic tricks as above. If dTT were extended to have \emph{symmetries}, then we could make progress here by lining up the $\f~\t$ indices and imposing the definitional equality $\S~\bkt{\A~\t}~\a \equiv \A\d~\t~\cprime{?_{\mathsf{S}_3}}$ as a corecursor premise. On the other hand, without the ability to line up the two $\f~\t$ indicies, trying to instead impose definitional equalities involving $\f$, the very term being defined, creates a vicious cycle, since whether or not a definition of $\f$ is well-typed would depend on checking a definitional equality with $\f$, which presupposes that $\f$ is well-typed. Since, for the present, we have chosen to develop a theory without symmetries, we must abandon this approach.

To salvage this, we will leave $\big(\mathsf{\R_\Upsilon}~\cred{\widetilde{\Z}} ~\cred{\widetilde{\S}}\big)\d$ as a stuck form in the theory, but will specify how to compute $\Z\d$ and $\S\d$ on this normal form. The main idea is that if we define:
\begin{lstlisting}
(*$\f$*) : (*$\X$*) (*$\to$*) (*$\SST$*)
(*$\Z$*) ((*$\f$*) (*$\t$*)) = (*$\cblu{\mathfrak{z}}$*) (*$\t$*)
(*$\S$*) ((*$\f$*) (*$\t$*)) (*$\a$*) = (*$\f\d$*) (*$\t$*) ((*$\cblu{\mathfrak{s}}$*) (*$\t$*) (*$\a$*))
\end{lstlisting}
then we can compute display on each line of this definition to obtain:
\begin{lstlisting}
(*$\f\d$*) : ((*$\t$*) : (*$\X$*)) (*$\to$*) (*$\X\d$*) (*$\t$*) (*$\to$*) (*$\SST\d$*) ((*$\f$*) (*$\t$*))
(*$\Z\d$*) ((*$\f\d$*) (*$\t$*) (*$\s$*)) = (*$\lambda$*) (*$\a$*) (*$\to$*) (*$\cblu{\mathfrak{z}}\d$*) (*$\t$*) (*$\s$*) (*$\a$*)
(*$\S\d$*) ((*$\f\d$*) (*$\t$*) (*$\s$*)) (*$\cblu{a\sub{0}}$*) (*$\cblu{a\sub{1}}$*) = (*$\f\d\d$*) (*$\t$*) (*$\s$*) ((*$\cblu{\mathfrak{s}}$*) (*$\t$*) (*$\cblu{a\sub{0}}$*)) ((*$\cblu{\mathfrak{s}}\d$*) (*$\t$*) (*$\s$*) (*$\cblu{a\sub{0}}$*) (*$\cblu{a\sub{1}}$*))
\end{lstlisting}
Thus we obtain the computation laws:
\begin{align*}
\Z\d~\big(\big(\mathsf{\R_\Upsilon}~\cred{\widetilde{\Z}} ~\cred{\widetilde{\S}}\big)\d~\sigma^\cblu{+}\big)~\cblu{a} &\equiv \cred{\widetilde{\Z}}\d~\sigma^\cblu{+}~\cblu{a} \\
\S\d~\big(\big(\mathsf{\R_\Upsilon}~\cred{\widetilde{\Z}} ~\cred{\widetilde{\S}}\big)\d~\sigma^\cblu{+}\big)~\cblu{a\sub{0}}~\cblu{a\sub{1}} &\equiv \big(\mathsf{\R_\Upsilon}~\cred{\widetilde{\Z}} ~\cred{\widetilde{\S}}\big)\d\d~\pair{\sigma^\cblu{+}}{\cred{\widetilde{\S}}\D~\sigma^\cblu{+}~\cblu{a\sub{0}}~\cblu{a\sub{1}}}
\end{align*}
Note that these computation rules were exactly obtained by applying display to both sides of the equation in the initial computation rules. We can iterate this to obtain:
\begin{align*}
\Z^{\cred{\mathsf{d}}^\n} ~\big(\big(\mathsf{\R_\Upsilon}~\cred{\widetilde{\Z}} ~\cred{\widetilde{\S}}\big)^{\cred{\mathsf{d}}^\n}~\sigma^\n\big)~\cblu{\partial}\a &\equiv \cred{\widetilde{\Z}}^{\cred{\mathsf{d}}^\n}~\sigma^\n~\cblu{\partial}\a \\
\S^{\cred{\mathsf{d}}^\n}~\big(\big(\mathsf{\R_\Upsilon}~\cred{\widetilde{\Z}} ~\cred{\widetilde{\S}}\big)^{\cred{\mathsf{d}}^\n}~\sigma^\n\big)~\cblu{\partial}\a~\a &\equiv \big(\mathsf{\R_\Upsilon}~\cred{\widetilde{\Z}} ~\cred{\widetilde{\S}}\big)^{\cred{\mathsf{d}}^{\n+\cred{1}}}~\pair{\sigma^\n}{\cred{\widetilde{\S}}^{\cred{\mathsf{D}}^\n}~\sigma^\n~\cblu{\partial}\a ~\a}
\end{align*}
The situation on our hands is not unlike that of \texttt{Agda}, where a definition of $\f$ made by (co)pattern matching defines a new normal form and does not expand to a first class intro or elim form when normalised\footnote{The culprit here is not a lack of first-class forms, since \texttt{Agda} has pattern matching lambdas. Rather, the restriction is made primarily to control such runaway unfolding that would substantially affect the performance of type-checking and normalisation. As a consequence, two structurally identical top-level definitions of functions $\f$ and $\g$ made by pattern matching are not definitionally equal.}; such names only reduce when their defining patterns occur. This specific point does not itself inhibit $\textsf{\cred{Nat}}$ canonicity (which \texttt{Cubical Agda} otherwise currently lacks due to its treatment of transport in indexed inductives).

We conjecture that dTT, including its treatment of $\SST$, is fully computational in the sense of $\textsf{\cred{Nat}}$ canonicity, normalization, and decidable typechecking.
More precisely, although this may very well not hold \emph{verbatim} of the theory as written down in this paper, we expect it to hold of a modified presentation fitting within the general framework of ideas.
In particular, while we have presented dTT only as a Generalised Algebraic Theory, all the equations have a clear direction and there are no obvious stuck terms. \bbox

\subsection{Examples of semi-simplicial types}
\label{sec:eg-sst}

Of course, simply \emph{defining} a type of semi-simplicial types is only the first step: we also want to be able to work with such things conveniently.
Developing a full theory of semi-simplicial types is beyond the scope of this paper, but in this section we will give a few examples to suggest that this at least may be possible with our definition of $\SST$ and its corecursion principle.
We will use Agda-esque copattern-matching, and assume that our type theory has plenty of other structure rather than the bare-bones version of dTT that we have studied formally in this paper.

\subsubsection{The singular semi-simplicial types}
\label{subsec:sing}

Thus far we have not discussed propositional equality at all, and the reason for this is that the implementation of display is independent from any implementation of equality, whether that be Martin-L\"of, cubical, or observational. However, we now want to define a semi-simplicial type that arises from the $\infty$-groupoid structure of a type in HoTT.
For concreteness we will do this using a cubical notion of equality, with notation that aligns with \texttt{Cubical Agda}.

When dTT is combined with cubical type theory, we expect display on cubical path types should work as follows.
We have:
\begin{align*}
    \A : \Type_{\;\!\ell},~\x:\A,~\y:\A &\vdash_\sm \cred{\mathsf{Path}}~\A~\x~\y \type_{\;\!\ell} \\
    \A : \Type_{\;\!\ell},~\cblu{P} : \A \to \Type_{\;\!\ell},~\x:\A,~\xp:\cblu{P}~\x, \quad\quad\quad\quad\quad\quad\;\;\;\!& \\ \y:\A,~\yp:\cblu{P}~\y,~\p:\cred{\mathsf{Path}}~\A~\x~\y &\vdash_\sm \cred{\mathsf{PathP}} \bkt{\lambda~\cblu{i}.~ \cblu{P}~\bkt{\p~\cblu{i}}}~\xp~\yp \type_{\;\!\ell},
\end{align*}
so the latter has the right type to be the display of the former.
Thus we expect:
\begin{multline*}
\dbkt{\cred{\mathsf{Path}}~\A~\x~\y}_{\A \;:\; \Type_{\;\!\ell},\;\x\;:\;\A,\;\y\;:\;\A}\d
  \equiv\\
  \dbkt{\cred{\mathsf{PathP}} \bkt{\lambda~\cblu{i}.~ \cblu{P}~\bkt{\p~\cblu{i}}}~\xp~\yp}_{\A \;:\; \Type_{\;\!\ell},\;\cblu{P} \;:\; \A \;\to\; \Type_{\;\!\ell},\;\x\;:\;\A,\;\xp\;:\;\cblu{P}~\x,\;\y\;:\;\A,\;\yp\;:\;\cblu{P}\;\y,\; \p\;:\;\cred{\mathsf{Path}}\;\A\;\x\;\y}.
\end{multline*}

With this given, the singular semi-simplicial types are defined by corecursion.
Rather than write this explicitly using the corecursor from \cref{subsec:semi-simplicial-types}, we use a copattern-matching syntax, including a \emph{`displayed corecursive call'} $\cred{\mathsf{Sing}}\d$.
\begin{lstlisting}
(*$\cred{\mathsf{Sing}}$*) : (*$\Type$*) (*$\to$*) (*$\SST$*)
(*$\Z$*) ((*$\cred{\mathsf{Sing}}$*) (*$\A$*)) = (*$\A$*)
(*$\S$*) ((*$\cred{\mathsf{Sing}}$*) (*$\A$*)) (*$\x$*) = (*$\cred{\mathsf{Sing}}\d$*) (*$\A$*) ((*$\lambda$*) (*$\y$*) (*$\to$*) (*$\cred{\mathsf{Path}}$*) (*$\A$*) (*$\x$*) (*$\y$*))
\end{lstlisting}
A calculation then yields:
\begin{lstlisting}
(*$\Z$*) ((*$\cred{\mathsf{Sing}}$*) (*$\A$*)) = (*$\A$*)
(*$\Z\d$*) ((*$\S$*) ((*$\cred{\mathsf{Sing}}$*) (*$\A$*)) (*$\cblu{x\sub{01}}$*)) (*$\cblu{x\sub{10}}$*) = (*$\cred{\mathsf{Path}}$*) (*$\A$*) (*$\cblu{x\sub{01}}$*) (*$\cblu{x\sub{10}}$*)
(*$\Z\d\d$*) ((*$\S\d$*) ((*$\S$*) ((*$\cred{\mathsf{Sing}}$*) (*$\A$*)) (*$\cblu{x\sub{001}}$*)) (*$\cblu{x\sub{010}}$*) (*$\cblu{\beta\sub{011}}$*)) (*$\cblu{x\sub{100}}$*) (*$\cblu{\beta\sub{101}}$*) (*$\cblu{\beta\sub{110}}$*)
    = (*$\cred{\mathsf{PathP}}$*) ((*$\lambda$*) (*$\cblu{i}$*) (*$\to$*) (*$\cred{\mathsf{Path}}$*) (*$\A$*) (*$\cblu{x\sub{001}}$*) (*$\cblu{\beta\sub{110}}$*) (*$\cblu{i}$*)) (*$\cblu{\beta\sub{011}}$*) (*$\cblu{\beta\sub{101}}$*)
(*$\Z\d\d\d$*) ((*$\S\d\d$*) ((*$\S\d$*) ((*$\S$*) ((*$\cred{\mathsf{Sing}}$*) (*$\A$*)) (*$\cblu{x\sub{0001}}$*)) (*$\cblu{x\sub{0010}}$*) (*$\cblu{\beta\sub{0011}}$*)) (*$\cblu{x\sub{0100}}$*) (*$\cblu{\beta\sub{0101}}$*) (*$\cblu{\beta\sub{0110}}$*) (*$\cblu{\ff\sub{0111}}$*)) (*$\cblu{x\sub{1000}}$*) (*$\cblu{\beta\sub{1001}}$*) (*$\cblu{\beta\sub{1010}}$*) (*$\cblu{\ff\sub{1011}}$*) (*$\cblu{\beta\sub{1100}}$*) (*$\cblu{\ff\sub{1101}}$*) (*$\cblu{\ff\sub{1110}}$*)
    = (*$\cred{\mathsf{PathP}}$*) ((*$\lambda$*) (*$\cblu{i}$*) (*$\to$*) (*$\cred{\mathsf{PathP}}$*) ((*$\lambda$*) (*$\cblu{j}$*) (*$\to$*) (*$\cred{\mathsf{Path}}$*) (*$\A$*) (*$\cblu{x\sub{0001}}$*) (*$\cblu{\ff\sub{1110}}$*) (*$\cblu{i}$*) (*$\cblu{j}$*)) (*$\cblu{\beta\sub{0011}}$*) ((*$\cblu{\ff\sub{1101}}$*) (*$\cblu{i}$*))) (*$\cblu{\ff\sub{0111}}$*) (*$\cblu{\ff\sub{1011}}$*)
\end{lstlisting}

\subsubsection{Nerves of categories}
\label{sec:nerves}

The semi-simplicial nerve of a 1-category can also be defined by corecursion.
Let $\cred{\mathsf{Cat}}$ denote the type of 1-categories, defined as a record inside dTT (extended by record types), and recall that in \cref{sec:intro} we computed $\cred{\mathsf{Cat}}\d$ to consist of \emph{`displayed categories'} in the usual sense~\cite{al:dispcat}.
Thus we can define:
\begin{lstlisting}
(*$\cred{\mathsf{Nerve}}$*) : (*$\cred{\mathsf{Cat}}$*) (*$\to$*) (*$\SST$*)
(*$\Z$*) ((*$\cred{\mathsf{Nerve}}$*) (*$\MC$*)) = (*$\cred{\mathsf{ob}}$*) (*$\MC$*)
(*$\S$*) ((*$\cred{\mathsf{Nerve}}$*) (*$\MC$*)) (*$\x$*) = (*$\cred{\mathsf{Nerve}}\d$*) (*$\MC$*) ((*$\cred{\mathsf{coslice}}$*) (*$\MC$*) (*$\x$*))
\end{lstlisting}
Here for a category $\MC$ and object $\x : \cred{\ob}~\MC$, by $\cred{\mathsf{coslice}}~\MC~\x$ we mean the coslice category $\x/\MC$, regarded as a displayed category over $\MC$ via the forgetful functor.
Note that a definition of $\cred{\mathsf{coslice}} : \bkt{\MC : \cred{\mathsf{Cat}}} \to \cred{\mathsf{ob}}~\MC \to \cred{\mathsf{Cat}}\d~\MC$ at the global level in dTT automatically induces the definition of the dependent coslice $\cred{\mathsf{coslice}}\d$. A similar idea works for bicategories, and any other kind of category for which we can define a displayed (co)slice. \bbox

\subsubsection{Topological singular complexes}
\label{sec:topsing}

In \cref{subsec:sing} we constructed the singular semi-simplicial type associated to the intrinsic $\infty$-groupoid structure of any type.
But we can also construct a more classical singular semi-simplicial set associated to a topological space.
For any type $\cred{\mathsf{Top}}$ of \emph{`topological space'} definable inside of dTT as a record, we have a displayed version $\mathsf{\cred{Top}}\d$.
In some cases, particularly \emph{`nonalgebraic'} ones such as open-set spaces, an element of $\cred{\mathsf{Top}}\d~\X$ is more general than an $\Y:\cred{\mathsf{Top}}$ with a map $\Y\to\X$; but at least \emph{from} such a $\Y$ we can construct its fibers as a displayed space.
Thus, as long as we can construct, for any $\x:\X$, a space of \emph{`continuous paths in $\X$ starting at $\x$'} with an endpoint projection down to $\X$, we can make it a displayed space $\cred{\mathsf{paths}}~\X~\x$ over $\X$, and use this to construct the singular semi-simplicial types:
\begin{lstlisting}
(*$\cred{\mathsf{Sing}}$*) : (*$\cred{\mathsf{Top}}$*) (*$\to$*) (*$\SST$*)
(*$\Z$*) ((*$\cred{\mathsf{Sing}}$*) (*$\X$*)) = (*$\cred{\mathsf{pt}}$*) (*$\X$*)
(*$\S$*) ((*$\cred{\mathsf{Sing}}$*) (*$\X$*)) (*$\x$*) = (*$\cred{\mathsf{Sing}}\d$*) (*$\X$*) ((*$\cred{\mathsf{paths}}$*) (*$\X$*) (*$\x$*))
\end{lstlisting}

\subsubsection{Fibers and higher spans}
\label{sec:fibers}

As we will see in \cref{sec:semantics}, semantically \emph{each} type at mode $\sm$ is already an augmented semi-simplicial type.
We expect that if we fix a particular $(\minusone)$-simplex in an augmented semi-simplicial type, we should obtain an (unaugmented) semi-simplicial type as its \emph{`fibre'}.
And indeed, we can define this operation:
\begin{lstlisting}
(*$\cred{\mathsf{Fib}}$*) : ((*$\X$*) :(*$^\trisq$*) (*$\Type$*)) ((*$\ze$*) : (*$\X$*)) (*$\to$*) (*$\SST$*)
(*$\Z$*) ((*$\cred{\mathsf{Fib}}$*) (*$\X$*) (*$\ze$*)) = (*$\X\d$*) (*$\ze$*)
(*$\S$*) ((*$\cred{\mathsf{Fib}}$*) (*$\X$*) (*$\ze$*)) (*$\x$*) = (*$\cred{\mathsf{Fib}}\d$*) (*$\X$*) (*$\ze$*) (*$\x$*)
\end{lstlisting}
Note that $\X$ is required to be modal so that we can take display of it.
Then we have, for instance, if $\cblu{\ze\sub{00}}:\X$ and $\cblu{x\sub{01}}, \cblu{x\sub{10}} : \X\d~\cblu{\ze\sub{00}}$,
\begin{align*}
  (\cred{\mathsf{Fib}}~\X~\cblu{\ze\sub{00}})_\cred{0} &\equiv \Z~(\cred{\mathsf{Fib}}~\X~\cblu{\ze\sub{00}})\\
  &\equiv \X\d~\cblu{\ze\sub{00}}\\
  (\cred{\mathsf{Fib}}~\X~\cblu{\ze\sub{00}})_\cred{1}~\cblu{x\sub{01}}~\cblu{x\sub{10}} &\equiv \Z\d~(\S~(\cred{\mathsf{Fib}}~\X~\cblu{\ze\sub{00}})~\cblu{x\sub{01}})~\cblu{x\sub{10}}\\
  &\equiv \Z\d~(\cred{\mathsf{Fib}}\d~\X~\cblu{\ze\sub{00}}~\cblu{x\sub{01}})~\cblu{x\sub{10}}\\
  &\equiv \X\d\d~\cblu{\ze\sub{00}}~\cblu{x\sub{01}}~\cblu{x\sub{10}}
\end{align*}
and as a last example
\begin{align*}
(\cred{\mathsf{Fib}}~\X~\cblu{\ze\sub{000}})_\cred{2}~\cblu{x\sub{001}}~\cblu{x\sub{010}}~\cblu{\beta\sub{011}}~\cblu{x\sub{100}}~\cblu{\beta\sub{101}}~\cblu{\beta\sub{110}}
  &\equiv
  \Z\d\d~(\S\d~(\S~(\cred{\mathsf{Fib}}~\X~\cblu{\ze\sub{000}})~\cblu{x\sub{001}})~\cblu{x\sub{010}}~\cblu{\beta\sub{011}})~\cblu{x\sub{100}}~\cblu{\beta\sub{101}}~\cblu{\beta\sub{110}}\\
  &\equiv
  \Z\d\d~(\S\d~(\cred{\mathsf{Fib}}\d~\X~\cblu{\ze\sub{000}}~\cblu{x\sub{001}})~\cblu{x\sub{010}}~\cblu{\beta\sub{011}})~\cblu{x\sub{100}}~\cblu{\beta\sub{101}}~\cblu{\beta\sub{110}}\\
  &\equiv
  \Z\d\d~(\cred{\mathsf{Fib}}\d\d~\X~\cblu{\ze\sub{000}}~\cblu{x\sub{001}}~\cblu{x\sub{010}}~\cblu{\beta\sub{011}})~\cblu{x\sub{100}}~\cblu{\beta\sub{101}}~\cblu{\beta\sub{110}}\\
  &\equiv
  \cred{\mathsf{Fib}}\d\d~\X~\cblu{\ze\sub{000}}~\cblu{x\sub{001}}~\cblu{x\sub{010}}~\cblu{\beta\sub{011}}~\cblu{x\sub{100}}~\cblu{\beta\sub{101}}~\cblu{\beta\sub{110}}.
\end{align*}
In particular, if we let $\X=\Type_{\;\!\ell}$ be a universe and $\ze = \top$ be a unit type, we have
\begin{align*}
  (\cred{\mathsf{Fib}}~\Type_{\;\!\ell}~\top)_\cred{0} &\equiv \Type^{\mathsf{\cred{d}}}_{\;\!\ell}~\top\\
  &\equiv \top \to \Type_{\;\!\ell}\\
  &\cong \Type_{\;\!\ell}\\
  (\cred{\mathsf{Fib}}~\Type_{\;\!\ell}~\top)_\cred{1}~\cblu{X\sub{01}}~\cblu{X\sub{10}} &\equiv \Type^{\mathsf{\cred{dd}}}_{\;\!\ell}~\top~\cblu{X\sub{01}}~\cblu{X\sub{10}}\\
  &\cong \cblu{X\sub{01}} \to \cblu{X\sub{10}} \to \Type_{\;\!\ell}\\
  (\cred{\mathsf{Fib}}~\Type_{\;\!\ell}~\top)_\cred{2}~\cblu{X\sub{001}}~\cblu{X\sub{010}}~\cblu{B\sub{011}}~\cblu{X\sub{100}}~\cblu{B\sub{101}}~\cblu{B\sub{110}}
  &\equiv \Type^{\mathsf{\cred{ddd}}}_{\;\!\ell}~\top~\cblu{X\sub{001}}~\cblu{X\sub{010}}~\cblu{B\sub{011}}~\cblu{X\sub{100}}~\cblu{B\sub{101}}~\cblu{B\sub{110}}\\
  &\cong
  \begin{aligned}[t]
    &(\cblu{x\sub{001}} : \cblu{X\sub{001}}) (\cblu{x\sub{010}} : \cblu{X\sub{010}}) (\cblu{\beta\sub{011}} : \cblu{B\sub{011}}~\cblu{x\sub{001}}~\cblu{x\sub{010}})\\
    &(\cblu{x\sub{100}} : \cblu{X\sub{100}})
    (\cblu{\beta\sub{101}} : \cblu{B\sub{101}}~\cblu{x\sub{001}}~\cblu{x\sub{100}})
    (\cblu{\beta\sub{110}} : \cblu{B\sub{110}}~\cblu{x\sub{010}}~\cblu{x\sub{100}})\to \Type_{\;\!\ell}
  \end{aligned}
\end{align*}
Thus $\cred{\mathsf{Fib}}~\Type_{\;\!\ell}~\top$ is the semi-simplicial type of types, spans, and a sort of simplicial \emph{`higher spans'} that could also be called \emph{`heterogeneous simplices'}.
More generally, $\cred{\mathsf{Fib}}~\Type_{\;\!\ell}~\A$ for any type $\A$ consists of types, spans, and simplicial higher spans indexed by $\A$. \bbox

\subsubsection{Operations on semi-simplicial types}
\label{sec:oper-sst}

We can also use corecursion to define operations on semi-simplicial types that are essentially levelwise.
For instance, any two semi-simplicial types have a product:
\begin{lstlisting}
(*$\cred{\_\!\times\!\_}$*) : (*$\SST$*) (*$\to$*) (*$\SST$*) (*$\to$*) (*$\SST$*)
(*$\Z$*) ((*$\X$*) (*$\cred\times$*) (*$\Y$*)) = (*$\Z$*) (*$\X$*) (*$\cred\times$*) (*$\Z$*) (*$\Y$*)
(*$\S$*) ((*$\X$*) (*$\cred\times$*) (*$\Y$*)) (*$\cred\langle$*) (*$\x$*) (*$\cred{,}$*) (*$\y$*) (*$\cred\rangle$*) = ((*$\S$*) (*$\X$*) (*$\x$*)) (*$\cred{\times\d}$*) ((*$\S$*) (*$\Y$*) (*$\y$*))
\end{lstlisting}
Here in the $\S$ case, we have treated the non-displayed arguments of $\cred{\times\d}$ as implicit: its full type is
\[ \underline{~~}\cred{\times\d}\underline{~~} : \{\X : \SST\} ~(\cblu{X'} : \SST\d~\X)~ \{\Y : \SST\}~ (\cblu{Y'} : \SST\d~\Y) \to \SST\d~(\X\mathbin{\cred{\times}}\Y) \]
There is a similar dependently-typed version, i.e.\ a $\Sigma$-semi-simplicial-type:
\begin{lstlisting}
(*$\Sigma$*) : ((*$\X$*) : (*$\SST$*)) (*$\to$*) (*$\SST\d~\X$*) (*$\to$*) (*$\SST$*)
(*$\Z$*) ((*$\Sigma$*) (*$\X$*) (*$\Y$*)) = (*$\Sigma$*) ((*$\Z$*) (*$\X$*)) ((*$\Z\d$*) (*$\Y$*))
(*$\S$*) ((*$\Sigma$*) (*$\X$*) (*$\Y$*)) (*$\cred{\boldsymbol\langle}$*) (*$\x$*) (*$\cred{,}$*) (*$\y$*) (*$\cred{\boldsymbol\rangle}$*) = (*$\Sigma\d$*) ((*$\S$*) (*$\X$*) (*$\x$*)) ((*$\S$*) (*$\Y$*) (*$\y$*))
\end{lstlisting}
There is an empty semi-simplicial type.
Note that the $\S$ case can be omitted, since one of its arguments would belong to the empty type $\cred{\bot}$.
\begin{lstlisting}
(*$\cred{\varnothing}$*) : (*$\SST$*)
(*$\Z$*) (*$\cred{\varnothing}$*) = (*$\cred{\bot}$*)
\end{lstlisting}
Similarly, there is a trivial one:
\begin{lstlisting}
(*$\top$*) : (*$\SST$*)
(*$\Z$*) (*$\top$*) = (*$\top$*)
(*$\S$*) (*$\top$*) (*$\u$*) = (*$\top\d$*)
\end{lstlisting}
We can also take the product of any family of semi-simplicial types indexed by a \emph{discrete} type.
Note that the discreteness of $\A$ means that it doesn't need a displayed version when we apply $\bt\d$ in the $\S$ case.
\begin{lstlisting}
(*$\bt$*) : ((*$\A$*) :(*$^\tri$*) (*$\Disc$*)) (*$\to$*) (((*$\a$*) :(*$^\tri$*) (*$\A$*))) (*$\to$*) (*$\SST$*)) (*$\to$*) (*$\SST$*)
(*$\Z$*) ((*$\bt$*) (*$\A$*) (*$\X$*)) = (((*$\a$*) :(*$^\tri$*) (*$\A$*)) (*$\to$*) (*$\Z$*) ((*$\X$*) (*$\a$*)))
(*$\S$*) ((*$\bt$*) (*$\A$*) (*$\X$*)) (*$\p$*) = (*$\bt\d$*) (*$\A$*) (*$\X$*) ((*$\lambda$*) (*$\a$*) (*$\to$*) (*$\S$*) ((*$\X$*) (*$\a$*)) ((*$\p$*) (*$\a$*)))
\end{lstlisting}

However, there are some things we would naturally expect to be able to define that do not seem possible with our current theory.
For example, the disjoint union of semi-simplicial types should certainly have the disjoint union of $\cred{0}$-simplices, but the slice over a 0-simplex should come only from one of the two sides.
That is, $\S~(\X~\cred{+}~\Y)~(\cred{\mathsf{inl}}~\x)$ should be morally just $\S~\X~\x$.
However, $\S~\X~\x$ belongs to $\SST\d~\X$, whereas $\S~(\X~\cred{+}~\Y)~(\cred{\mathsf{inl}}~\x)$ must belong to $\SST\d~(\X~\cred{+}~\Y)$; thus we need to take its disjoint union with an empty semi-simplicial type displayed over $\Y$.

We defined a \emph{`global'} empty semi-simplicial type above, and it seems intuitively that we should be able to define a \emph{`constant'} version of this displayed over $\Y$.
But as noted in \cref{subsec:semi-simplicial-types}, without symmetry it does not seem possible to formulate a useful corecursor for $\SST\d$, and without such a thing it is unclear how to define \emph{`constantly displayed'} semi-simplicial types.
This suggests that further work in this direction might require the addition of symmetries. \bbox

\subsection{Displayed coinductive types}
\label{sec:dcoind}

Generalizing the discussion of $\SST$, we now formulate a fairly general notion of \emph{`indexed displayed coinductive type'}.
It depends on a telescope $\Phi$ of \emph{`non-uniform parameters'}, and every element of it has a \emph{`head'} belonging to some specified type family $\A$ and a \emph{`tail'}, depending on a telescope of parameters $\MB$, and belonging to the displayed version of the coinductive type itself. The parameters of this displayed version of the very type being defined are $\Phi\D$, which we can assemble provided that the data of the old parameters $\varphi : \Phi$, the head $\x : \A~\varphi$, and the new dependencies $\mathfrak{\cblu{b}} : \MB~\varphi~\a$, are sufficient to extract a section $\sigma~\varphi~\x~\mathfrak{\cblu{b}} : \Phi\d~\varphi$. The idea is analogous to an \emph{`indexed M-type'}, but with the output of the tail being displayed, and with $\MB$ being a telescope rather than a simple type. The pseudo-\texttt{Agda} corresponding to this would be:
\begin{lstlisting}
(*\textbf{module}*) (*\cred{\_}*) ((*$\Phi$*) :(*$^\trisq$*) (*$\Tel$*)) ((*$\A$*) :(*$^\trisq$*) (*$\Phi$*) (*$\to$*) (*$\Type$*)) ((*$\MB$*) :(*$^\trisq$*) ((*$\varphi$*) : (*$\Phi$*)) ((*$\a$*) : (*$\A$*) (*$\varphi$*)) (*$\to$*) (*$\Tel$*))
        (*\;*)((*$\sigma$*) :(*$^\trisq$*) ((*$\varphi$*) : (*$\Phi$*)) ((*$\a$*) : (*$\A$*) (*$\varphi$*)) ((*$\mathfrak{\b}$*) : (*$\MB$*) (*$\varphi$*) (*$\a$*)) (*$\to$*) (*$\Phi\d$*) (*$\varphi$*)) (*\textbf{where}*)
  (*\textbf{codata}*) (*$\dCoind$*) ((*$\varphi$*) : (*$\Phi$*)) : (*$\Type$*) (*\textbf{where}*)
    (*$\head$*) : (*$\dCoind$*) (*$\varphi$*) (*$\to$*) (*$\A$*) (*$\varphi$*)
    (*$\tail$*) : ((*$\x$*) : (*$\dCoind$*) (*$\varphi$*)) ((*$\mathfrak{\b}$*) : (*$\MB$*) (*$\varphi$*) ((*$\head$*) (*$\x$*))) (*$\to$*) (*$\dCoind\d$*) (*$\cred{\boldsymbol\langle}$*)(*\;*)(*$\varphi$*) (*$\cred{\boldsymbol,}$*) (*$\sigma$*) (*$\varphi$*) ((*$\head$*) (*$\x$*)) (*$\mathfrak{\b}$*)(*\;*)(*$\cred{\boldsymbol\rangle}$*) (*$\x$*)
\end{lstlisting}
We can thus write down the formation and introduction rules for $\dCoind$ as follows:
\begin{mathpar}
  \infer{\Gamma,~\lock_\trisq \vdash_\sm \Phi \tel_{\;\!\ell_\cblu{0}} \\
    \Gamma,~\lock_\trisq \vdash_\sm \A \slfrac{\type_{\;\!\ell_\cblu{1}}}{_{\;\varphi\;:\;\Phi}} \\ \Gamma,~\lock_\trisq  \vdash_\sm \MB \slfrac{\tel_{\;\!\ell_\cblu{2}}}{_{\varphi\;:\;\Phi,~\a\;:\;\A~\varphi}} \\ \Gamma,~\lock_\trisq \vdash_\sm \sigma : \dbkt{\Phi\d~\varphi}_{\;\varphi\;:\;\Phi,~ \a\;:\;\A~\varphi,~ \mathfrak{\b}\;:\;\MB~\varphi~\x}}
  {\Gamma \vdash_\sm \dCoind_{\;\sqbkt{\Phi,\;\A,\;\MB,\;\sigma}} \slfrac{\type_{\;\!\ell_\cblu{1}\;\join\;\ell_\cblu{2}}}{_{\varphi\;:\;\Phi~\sqbkt{\key^{\trisq\le\id_\sm}}}} \\
    \Gamma \vdash_\sm \head_{\;\sqbkt{\Phi,\;\A,\;\MB,\;\sigma}} : \dbkt{\A~\sqbkt{\key^{\trisq\le\id_\sm}}~\varphi}_{\varphi,~ \x\;:\; \dCoind_{\;\sqbkt{\Phi,\;\A,\;\MB,\;\sigma}}~\varphi}\\
    \Gamma\vdash_\sm \tail_{\smsqbkt{\Phi,\A,\MB,\sigma}} : \smdbkt{\dCoind_{\smsqbkt{\Phi,\A,\MB,\sigma}}\d \,\smpair{\varphi}{\sigma\,\varphi\,\bkt{\head\,\x}\,\mathfrak{\b}}\,\x}_{\varphi,\,\x,\,\mathfrak{\b}:\MB\smsqbkt{\key^{\trisq\le\id_\sm}}\,\varphi\,\bkt{\head\,\varphi\,\x}}
}
\end{mathpar}
Note that the universe level of $\dCoind$ is governed by those of $\A$ and $\MB$, but does not depend on the level of the telescope of non-uniform parameters $\Phi$.

Following the example of $\SST$, we will begin by attempting to write down a reasonable template for a coinduction principle. In the same module context, we can attempt to map into a $\dCoind$ type from a length two context as follows:
\begin{lstlisting}
(*$\f$*) : ((*$\t$*) : (*$\X$*)) ((*$\s$*) : (*$\Y$*) (*$\t$*)) (*$\to$*) (*$\dCoind$*) ((*$\phi$*) (*$\t$*) (*$\s$*))
(*$\head$*) ((*$\f$*) (*$\t$*) (*$\s$*)) = ((*\cprime{?$_{\mathsf{h}_1}$}*) : (*$\A$*) ((*$\phi$*) (*$\t$*) (*$\s$*)))
(*$\tail$*) ((*$\f$*) (*$\t$*) (*$\s$*)) (*$\mathfrak{\b}$*) = (*$\f\d$*) (*$\t$*) ((*\cprime{?$_{\mathsf{t}_1}$}*) : (*$\X\d$*) (*$\t$*)) (*$\s$*) ((*\cprime{?$_{\mathsf{t}_2}$}*) : (*$\Y\d$*) (*$\t$*) (*\cprime{?$_{\mathsf{t}_1}$}*) (*$\s$*))
\end{lstlisting}
The types that we have are then:
\begin{align*}
\tail~\bkt{\f~\t~\s}~\mathfrak{\b} &: \dCoind\d~\pair{\phi~\t~\s}{\sigma~\bkt{\phi~\t~\s}~\cprime{?_{\mathsf{h}_1}}~\mathfrak{\b}}~\bkt{\f~\t~\s}\\
\f\d~\t~\cprime{?_{\mathsf{t}_1}}~\s~\cprime{?_{\mathsf{t}_2}} &: \dCoind\d~\pair{\phi~\t~\s}{\phi\d~\t~\cprime{?_{\mathsf{t}_1}}~\s~\cprime{?_{\mathsf{t}_2}}}~\bkt{\f~\t~\s}
\end{align*}
Thus there is a non-trivial condition that needs to be imposed for this definition template to be well typed. Fortunately, unlike in the case of $\SST\d$, we generally have terms lining up in the sense that the terminal $\bkt{\f~\t~\s}$ terms align, which avoids the vicious cycle from before. We get the following rule:
\begin{mathpar}
  \infer{\Gamma,~\lock_\trisq \vdash_\sm \Upsilon \tel_{\;\!\ell^\cblu{'}} \\
    \Gamma,~\lock_\trisq \vdash_\sm \phi : \dbkt{\Phi}{_{\upsilon\;:\;\Upsilon}}\\
    \Gamma,~\lock_\trisq \vdash_\sm \cblu{\widetilde{h}} : \dbkt{\A~\bkt{\phi~\upsilon}}_{\upsilon\;:\;\Upsilon}\\
    \Gamma,~\lock_\trisq \vdash_\sm \cblu{\widetilde{\tau}} : \dbkt{\Upsilon\d~\upsilon}_{\upsilon\;:\;\Upsilon,~ \mathfrak{\b} \;:\; \MB~\bkt{\phi~\upsilon}~(\cblu{\widetilde{h}}~\upsilon)}\\
    \Gamma,~\lock_\trisq \ext \upsilon:\Upsilon,~ \mathfrak{\b} : \MB~(\phi~\upsilon)~ (\cblu{\widetilde{h}}~\upsilon) \vdash
    \phi\d~\pair{\upsilon}{\cblu{\widetilde{\tau}}~\upsilon~\mathfrak{\b}} \equiv \sigma~(\phi~\upsilon)~(\cblu{\widetilde{h}}~\upsilon)~\mathfrak{\b}
  }
  {\Gamma \vdash_\sm \corec_{\;\sqbkt{\Phi,\;\A,\;\MB,\;\sigma}}~\sqbkt{\Upsilon,\;\phi,\;\cblu{\widetilde{h}},\;\cblu{\widetilde{\tau}}} : \dbkt{\dCoind_{\;\sqbkt{\Phi,\;\A,\;\MB,\;\sigma}}~(\phi~\upsilon)}_{\upsilon\;:\;\Upsilon} }
\end{mathpar}
This comes with the following computation rules:
\begin{align*}
\head~(\phi~\upsilon)~\bkt{\corec_{\;\sqbkt{\Phi,\;\A,\;\MB,\;\sigma}} ~\sqbkt{\Upsilon,\;\phi,\;\cblu{\widetilde{h}},\;\cblu{\widetilde{\tau}}}~\upsilon} &\equiv \cblu{\widetilde{h}}~\upsilon \\
\tail~(\phi~\upsilon)~\bkt{\corec_{\;\sqbkt{\Phi,\;\A,\;\MB,\;\sigma}} ~\sqbkt{\Upsilon,\;\phi,\;\cblu{\widetilde{h}},\;\cblu{\widetilde{\tau}}}~\upsilon}~\mathfrak{\b} &\equiv \bkt{\corec_{\;\sqbkt{\Phi,\;\A,\;\MB,\;\sigma}}~\sqbkt{\Upsilon,\;\phi,\;\cblu{\widetilde{h}},\;\cblu{\widetilde{\tau}}}}\d ~\pair{\upsilon}{\cblu{\widetilde{\tau}}~\upsilon~\mathfrak{\b}}
\end{align*}

Now we can define $\SST$ as a particular instance of an indexed displayed coinductive type (which happens to have trivial indexing).
In fact, it is in some sense the \emph{universal} such instance, where the family $\MB$ indexed by $\A$ is the universal family $\El$ indexed by $\Type_{\;\!\ell}$.
This may be compared with the fact that the $W$-type of $\El$ is the type of \emph{`presentations of well-founded sets'}~\cite{aczel:ttint-cst}, while its $M$-type is the type of \emph{`presentations of ill-founded sets'}~\cite{lindstrom:nwfs-mltt}.
\begin{equation*}
  \SST_\ell \equiv \dCoind~\left[\;\ec_\sm,~{\dbkt{\Type_{\;\!\ell}}}_{\varphi\;:\;\ec_\sm}, ~{\dbkt{\El~\X}}_{\X\;:\;\Type_{\;\!\ell}}, ~{\dsqbkt{\esub_\sm}}_{\;\X\;:\;\Type_{\;\!\ell},\;\x\;:\;\El\;\X}\;\right]
\end{equation*}
This is the end of all the primitive rules and definitions we have to give.
From here, we can \emph{deduce} the rules for $\SST_\ell$, defining $\Z \equiv \head$ and $\S \equiv \tail$ and $\R \equiv \corec$:
\begin{mathpar}
  \infer{ }{\Gamma \vdash_\sm \Z : \dbkt{\Type_{\;\!\ell}}_{\X\;:\;\SST_\ell}}
  \and
  \infer{ }{\Gamma \vdash_\sm \S : \dbkt{\SST^{\mathsf{\cred{d}}}_\ell~\X}_{\X\;:\;\SST_\ell,\;\x\;:\;\El\;(\Z\;\X)}}
  \and
  \infer{\Gamma,~\lock_\trisq \vdash_\sm \Upsilon \tel_{\;\!\ell^\cblu{'}} \\
    \Gamma,~\lock_\trisq \vdash_\sm \cblu{\widetilde{Z}} : \dbkt{\Type_{\;\!\ell}}_{\upsilon\;:\;\Upsilon} \\
    \Gamma,~\lock_\trisq \vdash_\sm \cblu{\widetilde{S}} : \dbkt{\Upsilon\d~\upsilon}_{\upsilon\;:\;\Upsilon,\;\x\;:\; \El\;(\cblu{\widetilde{Z}}\;\upsilon)}\\
  }
  {\Gamma \vdash_\sm \R_\Upsilon~\cblu{\widetilde{Z}}~\cblu{\widetilde{S}} : \dbkt{\SST_{\;\!\ell}}_{\upsilon~:~\Upsilon} }
\end{mathpar}

The problem of giving a corecursion rule for $\SST\d$ carries over to the general case in the following way.
Just as $\d$ of a $\Pi$-type is another $\Pi$-type and so on for records and ordinary coinductive types,
We'd like to compute $\d$ of a $\dCoind$ to be another $\dCoind$, with something like the following:
\begin{equation*}
  \text{\cprime{\textquestiondown}}\quad
  \begin{multlined}[c]
    \dCoind_{\;\sqbkt{\Phi,\;\A,\;\MB,\;\sigma}}\d \equiv
    \dCoind
    ~\Big[\bkt{\varphi^\cblu{+} : \Phi\D,~ \c : \dCoind_{\;\sqbkt{\Phi,\;\A,\;\MB,\;\sigma}}~\varphi^\cblu{+}\ev},\\
    \dbkt{\A\d~\varphi^\cblu{+}~\bkt{\head~\c}}_{\;\varphi^\cblu{+},\;\c},~ \dbkt{\MB\D~\varphi^\cblu{+}~\bkt{\head~\c}~\a\cblu{'}}_{\varphi^\cblu{+},\;\x,\;\a\cblu{'}},\\
    \dsqbkt{\sigma\D~\varphi^\cblu{+}~\sqbkt{\head~\c,~\a\cblu{'}}~\mathfrak{\b}^\cblu{+},~\tail~\c~\mathfrak{\b}^\cblu{+}\ev
      }_{\;\varphi^\cblu{+},\;\c,\;\a\cblu{'},\;\mathfrak{\b}^\cblu{+}}
      \Big]
  \end{multlined}
\quad \cprime{?}
\end{equation*}
To see whether this is well-typed, observe that we have
\begin{align*}
  \varphi:\Phi,~\a:\A~\varphi,~\mathfrak{\b}:\MB~\varphi~\a &\vdash_\sm \sigma~\varphi~\a~\mathfrak{\b} : \Phi\d~\varphi\\
  \varphi^\cblu{+}:\Phi\D,~\a:\A,~\a\cblu{'}:\A\d~\a,~\mathfrak{\b}^\cblu{+} :\MB\D~\varphi^\cblu{+}~\sqbkt{\a,~\a\cblu{'}} &\vdash_\sm \sigma\D~\varphi^\cblu{+} ~\sqbkt{\a,~\a\cblu{'}}~\mathfrak{\b}^\cblu{+} : \Phi\d\D~\varphi^\cblu{+}
\end{align*}
whereas the $\sigma$ of the resulting $\dCoind$ must lie in $\bkt{\varphi^\cblu{+} : \Phi\D,~ \c : \dCoind_{\;\sqbkt{\Phi,\;\A,\;\MB,\;\sigma}}~\varphi^\cblu{+}\ev}\d$.
Thus, in particular, we need to compare $\Phi\d\D~\varphi^\cblu{+}$ to $\Phi\D\d~\varphi^\cblu{+}$, where $\varphi^\cblu{+}:\Phi\D$.
In the case of a one-type telescope $\Phi=\bkt{\a:\A}$, this becomes
\begin{align*}
  \Phi\D &\equiv \bkt{\a:\A,~\a\cblu{'}:\A\d~\a}\\
  \Phi\d~\a &\equiv \bkt{\a\cblu{'}:\A\d~\a}\\
  \Phi\d\D~\sqbkt{\a,\a\cblu{'}} &\equiv \bkt{\a\cblu{''} : \A\d~\a, ~ \a\cblu{'''} : \A\d\d~\a~\a\cblu{'}~\a\cblu{''}}\\
  \Phi\D\d~\sqbkt{\a,\a\cblu{'}} &\equiv \bkt{\a\cblu{''} : \A\d~\a, ~ \a\cblu{'''} : \A\d\d~\a~\a\cblu{''}~\a\cblu{'}}
\end{align*}
Unfortunately the last two are \emph{not} the same!
This is not just about ordering the variables in a telescope; although the second and third arguments of $\A\d\d$ both lie in $\A\d~\a$, it need not be symmetrical with respect to those arguments.
So again we see that without adding symmetry to the theory, it seems we can't give a general corecursor for $\dCoind\d$, and hence we can't compute $\corec\d$ to something more primitive. \bbox

\subsection{Examples of displayed coinductive types}
\label{sec:exampl-displ-coind}

We now continue our exploration of the theory of semi-simplicial types from \cref{sec:eg-sst}, now using the general notion of displayed coinductive type.
As in \cref{sec:eg-sst}, we will use Agda-esque codata and copattern-matching definitions, and assume that our type theory has plenty of other structure.

We have already noted that $\SST$ is in some sense the \emph{`universal'} (unparametrised) displayed coinductive type, whose determining family $\x : \A \vdash \MB~\x \type_{\;\!\ell_\cblu{2}}$ is the universal one $\X:\Type_{\;\!\ell} \vdash \El~\X \type_{\;\!\ell}$.
Moreover, it seems likely that in order for an unparametrised displayed coinductive type to be interesting, the types $\A$ and $\MB$ must have nontrivial display structure, i.e. they must not be discrete.
But the simplicial universe $\Type_{\;\!\ell}$ is the primary source of types with nontrivial display, just as the universe in homotopy type theory is a primary source of types with higher homotopy structure.
(In \cref{subsec:conjectural-syntax} we will speculate about a notion of \emph{`display inductive type'} analogous to higher inductive types, which are the other source of higher homotopy structure in homotopy type theory.)
For these reasons, we do not have a lot of interesting examples of other \emph{unparametrised} displayed coinductive types, but there is at least one: augmented semi-simplicial types.

\subsubsection{Augmented semi-simplicial types}
\label{sec:asst}

If we simply omit the $\Z$ input of $\S$ in the definition of $\SST$, we obtain a definition of \emph{augmented} semi-simplicial types.
(Recall from \cref{sec:intro} that these are equivalently \emph{unary semi-cubical types}.)
\begin{lstlisting}
(*\textbf{codata}*) (*$\ASST$*)  : (*$\Type$*) (*\textbf{where}*)
  (*$\Z^\cred{+}$*) : (*$\ASST$*) (*$\to$*) (*$\Type$*)
  (*$\S^\cred{+}$*) : ((*$\A$*) : (*$\ASST$*)) (*$\to$*) (*$\ASST\d$*) (*$\A$*)
\end{lstlisting}
We can convince ourselves of this by extracting types of low-dimensional simplices from an $\X:\ASST$:
\begin{align*}
  &\vdash \X_{\minusone} \equiv \Z^\cred{+}~\X\\
  \ze\sub{\cblu{0}} : \X_{\minusone} &\vdash \X_{\cblu{0}}~\ze\sub{\cblu{0}} \equiv \Z^\cred{+}\d~(\S^\cred{+}~\X)\\
  \ze\sub{\cblu{00}} : \X_{\minusone},~ \x\sub{\cblu{01}} : \X_\cred{\cblu{0}}~\ze\sub{\cblu{00}}, ~ \x\sub{\cblu{10}} : \X_{\cblu{0}}~\ze\sub{\cblu{00}}
  &\vdash
  \X_\cred{\cblu{1}}~\ze\sub{\cblu{00}}~\x\sub{\cblu{01}}~\x\sub{\cblu{10}} \equiv
  \Z^\cred{+}\d\d~(\S^\cred{+}\d~(\S^\cred{+}~\X))~\ze\sub{\cblu{00}}~\x\sub{\cblu{01}}~\x\sub{\cblu{10}}
\end{align*}
and so on.
Now we can observe that the construction $\cred{\mathsf{Fib}}$ of \cref{sec:fibers} factors through $\ASST$ via a pair of maps, both defined by copattern-matching:
\begin{lstlisting}
(*$\cred{\mathsf{Int}}$*) : ((*$\X$*) (*$:^\trisq$*) (*$\Type$*)) (*$\to$*) (*$\ASST$*)
(*$\Z^\cred{+}$*) ((*$\cred{\mathsf{Int}}$*) (*$\X$*)) = (*$\X$*)
(*$\S^\cred{+}$*) ((*$\cred{\mathsf{Int}}$*) (*$\X$*)) = (*$\cred{\mathsf{Int}}\d$*) (*$\X$*)
\end{lstlisting}

\begin{lstlisting}
(*$\cred{\mathsf{Fib}'}$*) : ((*$\X$*) : (*$\ASST$*)) ((*$\x$*) : (*$\Z^\cred{+}$*) (*$\X$*)) (*$\to$*) (*$\SST$*)
(*$\Z$*) ((*$\cred{\mathsf{Fib}'}$*) (*$\X$*) (*$\x$*)) = (*$\Z^\cred{+}\d$*) ((*$\S^\cred{+}$*) (*$\X$*)) (*$\x$*)
(*$\S$*) ((*$\cred{\mathsf{Fib}'}$*) (*$\X$*) (*$\x$*)) (*$\y$*) = (*$\cred{\mathsf{Fib}'}\d$*) (*$\X$*) ((*$\S^\cred{+}$*) (*$\X$*)) (*$\x$*) (*$\y$*)
\end{lstlisting}

\subsubsection{Pointed semi-simplicial types}
\label{sec:psst}

More interesting examples of displayed coinductive types have nontrivial parametrizations, often involving more semi-simplicial types.
For instance, we can define the structure of a \emph{pointing} on a semi-simplicial type displayed-coinductively:
\begin{lstlisting}
(*\textbf{codata}*) (*$\mathsf{\cred{Pt}}$*) ((*$\X$*) : (*$\SST$*)) : (*$\Type$*) (*\textbf{where}*)
  (*$\mathsf{\cred{zp}}$*) : (*$\mathsf{\cred{Pt}}$*) (*$\X$*) (*$\to$*) (*$\Z$*) (*$\X$*)
  (*$\mathsf{\cred{sp}}$*) : ((*$\p$*) : (*$\mathsf{\cred{Pt}}$*) (*$\X$*)) (*$\to$*) (*$\mathsf{\cred{Pt}}\d$*) (*$\X$*) ((*$\S$*) (*$\X$*) ((*$\mathsf{\cred{zp}}$*) (*$\p$*))) (*$\p$*)
\end{lstlisting}
We then have, for $\p : \mathsf{\cred{Pt}}~\X$,
\begin{align*}
  \mathsf{\cred{zp}}~\p &: \Z~\X \equiv \X_\cred{0}\\
  \mathsf{\cred{zp}}\d~(\mathsf{\cred{sp}}~\p) &: \Z\d~(\S~\X~(\mathsf{\cred{zp}}~\p))~(\mathsf{\cred{zp}}~\p) \equiv \X_\cred{1}~(\mathsf{\cred{zp}}~\p)~(\mathsf{\cred{zp}}~\p)\\
  \mathsf{\cred{zp}}\d\d~(\mathsf{\cred{sp}}\d~(\mathsf{\cred{sp}}~\p)) &: \X_\cred{2}~(\mathsf{\cred{zp}}~\p)~(\mathsf{\cred{zp}}~\p)~(\mathsf{\cred{zp}}\d~(\mathsf{\cred{sp}}~\p))~(\mathsf{\cred{zp}}~\p)~(\mathsf{\cred{zp}}\d~(\mathsf{\cred{sp}}~\p))~(\mathsf{\cred{zp}}\d~(\mathsf{\cred{sp}}~\p))
\end{align*}
and so on.
That is, an element of $\mathsf{\cred{Pt}}~\X$ equips $\X$ with a `fat point', i.e.\ a chosen $\cred{0}$-simplex $\cred{zp}$ that comes with all of the higher \emph{`degenerate simplices'} that one would expect to be associated to $\cred{zp}$ if it were in a \emph{simplicial} set rather than a semi-simplicial one. \bbox

\subsubsection{Morphisms of semi-simplicial types}
\label{sec:mor-sst}

With a double parametrization, we can define a type of morphisms of semi-simplicial types.
\begin{lstlisting}
(*\textbf{codata}*) (*$\mathsf{\cred{Hom}}$*) ((*$\X$*) (*$\Y$*) : (*$\SST$*)) : (*$\Type$*) (*\textbf{where}*)
  (*$\mathsf{\cred{zhom}}$*) : (*$\mathsf{\cred{Hom}}$*) (*$\X$*) (*$\Y$*) (*$\to$*) (*$\Z$*) (*$\X$*) (*$\to$*) (*$\Z$*) (*$\Y$*)
  (*$\mathsf{\cred{shom}}$*) : ((*$\f$*) : (*$\mathsf{\cred{Hom}}$*) (*$\X$*) (*$\Y$*)) ((*$\x$*) : (*$\Z$*) (*$\X$*)) (*$\to$*) (*$\mathsf{\cred{Hom}}\d$*) (*$\X$*) ((*$\S$*) (*$\X$*) (*$\x$*)) (*$\Y$*) ((*$\S$*) (*$\Y$*) ((*$\mathsf{\cred{zhom}}$*) (*$\f$*) (*$\x$*))) (*$\f$*)
\end{lstlisting}
As usual, we can unravel this a few steps to see what it looks like.
$\cred{\mathsf{zhom}}~\f$ is a function between types of $\cred{0}$-simplices, which we may denote $\f_\cred{0}$.
At the next dimension we have:
\begin{equation*}
  \cred{\mathsf{zhom}}\d~(\cred{\mathsf{shom}}~\f~\x\sub{\cblu{01}})~\x\sub{\cblu{01}}~\cblu{\beta\sub{11}}
  : \Z\d~(\S~\Y~(\cred{\mathsf{zhom}}~\f~\x\sub{\cblu{01}}))~(\cred{\mathsf{zhom}}~\f~\x\sub{\cblu{01}})
\end{equation*}
which is to say
\[ \cred{\mathsf{zhom}}\d~(\cred{\mathsf{shom}}~\f~\x\sub{\cblu{01}})~\x\sub{\cblu{01}} : \X_\cred{1}~\x\sub{\cblu{01}}~\x\sub{\cblu{10}} \to \Y_\cred{1}~(\f_\cred{0}~\x\sub{\cblu{01}})~(\f_\cred{0}~\x\sub{\cblu{10}}). \]
We may denote this function by $\f_\cred{1}$, and go on to extract a function $\f_\cred{2}$ between types of $\cred{2}$-simplices and so on.
We expect other basic operations on semi-simplicial types to be internalizable in a similar way. \bbox

\section{Semantics}
\label{sec:semantics}

We now discuss the semantics of dTT.
Specifically, we will show that from any model of ordinary dependent type theory with infinite limits, we can construct a model of dTT in which the original model sits as the discrete mode.

\begin{remark}
  Actually we will not quite model all of dTT as presented in \cref{sec:syntax}: we omit the type-former $\tri$ and its associated introduction and elimination rule.
  This is purely for reasons of simplicity and space.
  It should be possible to model $\tri$ as well as long as the starting discrete model has a unit type, but we leave the details for the future.
  We will, however, still model $\tri$-annotated variables and function types such as $(\x :^\tri \A) \to \B~\x$.
\end{remark}

In \cref{sec:categories-with-families} we review the semantics of ordinary dependent type theory, introduce some notation, extend this to our calculus of telescopes and meta-abstractions, and define what it means for such a model to have countable infinite limits.
Then in \cref{sec:simplicial-model} we construct a model of augmented semi-simplicial Reedy diagrams, starting with any model of dependent type theory having countable infinite limits.
This is essentially an instance of the general inverse diagram models constructed in~\cite{shulman15,kl:hoinvdia}, but we give an explicit inductive construction that avoids category-theoretic machinery and builds display and d\'ecalage in from the beginning.
In \cref{sec:modal-semantics} we add modalities to this model, and then in \cref{sec:generic-semantics} we discuss the \emph{general} notion of model of dTT and show that the simplicial model is in fact such.
Finally, in \cref{sec:sem-sst} we construct displayed coinductive types in these models, including the type $\SST$ of semi-simplicial types.

\subsection{The semantics of dependent type theory}
\label{sec:categories-with-families}

We approach semantics from the perspective of \emph{Categories with Families} (CwF)~\cite{CASCWF}. Here we will recount the relevant categorical concepts while providing a translation into language reminiscent of a type theoretic logical framework.

At the most basic level, a category with families is just a category with a terminal object and distinguished substructure of objects and morphisms that behave like \emph{types} and \emph{terms} in a dependent type theory. In the absence of any other structure, the only way in which this behaviour is manifested is through the presence of \emph{substitution}, which categorically corresponds to a choice of definitionally functorial distinguished pullbacks. Here, instead of giving the substructure as a proposition on objects and morphisms, we first give it as presheaves, and then use representability to overlay this structure into the category.
\vspace{0.2cm}

\subsubsection{Categories with Families}

A `CwF with levels' consists of a category $\cC$, along with a chosen terminal object $\One$, and equipped with the data of two families of presheaves, indexed by $\ell \level$:
\begin{mathpar}
    \Tyl : \cC\op \to \Set
    \and
    \Tml : \bkt{\int^\cC~\Tyl}\op \to \Set,
\end{mathpar}
such that for every $\Gamma : \mathsf{ob}_\cC$ and $\A : \Tyl~\Gamma$, there is a chosen representation of the presheaf:
\[\Delta~\mapsto \bkt{\sigma~:~\mathsf{mor}_\cC\bkt{\Delta\;,\;\Gamma}} \times \Tml\bkt{\Delta\;,\;\A^\sigma}.
  \eqno\sectend
\]

\subsubsection{Notation}

The objects of the category $\cC$ are called \emph{contexts} and denoted by $\Gamma,\;\!\Delta$. For $\Gamma : \mathsf{ob}_\cC$ we write:
\[\Gamma \ob\]
\emph{The empty context} is the chosen terminal object $\One$, and is denoted by:
\[\ec \ob\]
The morphisms of $\cC$ are called \emph{substitutions} and denoted by $\sigma,\;\!\tau$. For $\sigma : \mathsf{mor}_{\cC}\bkt{\Delta,\Gamma}$ we write:
\[\sigma : \Delta \to \Gamma\]
The unique substitution into the empty context is denoted by:
\[\esub : \Gamma \to \ec\]
The notations $\Gamma \ob$ and $\sigma : \Delta \to \Gamma$ are examples of \emph{`absolute'} or \emph{`context-less'} \emph{judgments}.
In the notation of this section, each absolute judgment asserts that some element belongs to some set, such as the set of objects or a hom-set of $\cC$.
Note that a set can be \emph{`dependent'} on elements of another set, such as the hom-set $\mathsf{mor}_{\cC}\bkt{\Delta,\Gamma}$ depending on $\Delta,\;\!\Gamma : \mathsf{ob}_{\cC}$; and thus elements of one absolute judgment can appear in another absolute judgment, such as $\sigma : \Delta \to \Gamma$ where $\Delta \ob$ and $\Gamma\ob$.
Operations on sets can be written as rules, such as composition of morphisms:
\[\infer{\sigma: \Delta\to\Gamma \\ \tau : \Gamma \to \Theta}{\tau\circ\sigma : \Delta\to\Theta}.
\]
Equations between these operations can likewise be written as rules, where we use $\equiv$ for equality since it corresponds to definitional equality in syntax:
\[\infer{\sigma: \Delta\to\Gamma \\ \tau : \Gamma \to \Theta \\ \upsilon : \Theta \to \Upsilon}{\upsilon\circ(\tau\circ\sigma) \equiv (\upsilon\circ\tau)\circ \sigma}.
\]

The elements of $\Tyl$ are called \emph{types of level $\ell$} and denoted by $\A,\B$.  For $\A : \Tyl~\Gamma$ we write:
\[\gamma : \Gamma \vdash \A~\gamma \type_{\;\!\ell}\]
Similarly, the elements of $\Tml$ are called \emph{terms} and denoted by $\t,\;\!\s$.  For $\t : \Tml\bkt{\Gamma,\A}$ we write:
\[\gamma : \Gamma \vdash \t~\gamma : \A~\gamma\]
These notations are examples of \emph{`hypothetical'} or \emph{`contextual'} judgments.
In the notation of this section, each hypothetical judgment asserts that some element belongs to a value of some \emph{presheaf}, such as $\Tm_{\;\!\ell}$ or $\Ty_{\;\!\ell}$.
The object of $\cC$ at which the presheaf is evaluated (in these examples, $\Gamma$) is written on the left-hand side of the turnstile $\vdash$.
We annotate it with a formal \emph{`variable'} such as $\gamma$, and \emph{`apply'} all the elements appearing on the right-hand side to that element; at the moment this is just a convention of notation.
As with absolute judgments, one presheaf can be dependent on another, such as $\Tm_{\;\!\ell}(\Gamma,\A)$ depending on $\A : \Ty_{\;\!\ell}~\Gamma$; and thus the elements of one hypothetical judgment can appear in another hypothetical judgment, such as $\gamma : \Gamma \vdash \t~\gamma : \A~\gamma$ where $\gamma : \Gamma \vdash \A~\gamma \type_{\;\!\ell}$.

For a morphism $\sigma : \Delta \to \Gamma$, we denote its functorial action on types $\A$ and terms $\t$ by $\A^\sigma$ and $\t^\sigma$; thus the presheaf actions of $\Ty_{\;\!\ell}$ and $\Tm_{\;\!\ell}$ can be expressed by rules
\begin{mathpar}
  \infer{\sigma : \Delta \to \Gamma \\ \gamma:\Gamma \vdash \A~\gamma \type_{\;\!\ell}}{\delta:\Delta \vdash \A^\sigma~\delta \type_{\;\!\ell}}
  \and
  \infer{\sigma : \Delta \to \Gamma \\ \gamma:\Gamma \vdash \t:\A}{\delta:\Delta \vdash \t^\sigma~\delta : \A^\sigma~\delta}
\end{mathpar}
We allow ourselves to write this alternatively by taking the formal variables $\gamma,\delta$ more seriously and \emph{`applying'} $\sigma$ to them:
\begin{mathpar}
    \delta : \Delta \vdash \A~\bkt{\sigma~\delta} \type_{\;\!\ell}
    \and
    \delta : \Delta \vdash \t~\bkt{\sigma~\delta} : \A~\bkt{\sigma~\delta}
\end{mathpar}
Formally, this is justified by interpretation in the internal type theory of the presheaf category $\Set^{\cC\op}$.
By virtue of functoriality, for $\sigma : \Delta \to \Gamma$ and $\tau : \Omega \to \Delta$ we have that
\begin{mathpar}
    \A~\bkt{\sigma~\bkt{\tau~\omega}} \equiv \A~\bkt{\bkt{\sigma \circ \tau}~\omega}
    \and
    \t~\bkt{\sigma~\bkt{\tau~\omega}} \equiv \t~\bkt{\bkt{\sigma \circ \tau}~\omega}
\end{mathpar}

Now let us consider the representation hypothesis.
In categorical notation, for $\Gamma : \mathsf{ob}_{\cC}$ and $\A : \Ty_{\;\!\ell}~\Gamma$ we have a representing object $\Gamma \ce \A : \mathsf{ob}_\cC$ called the \emph{context extension}.
In type-theoretic notation, we denote the context extension by
\[\infer{\Gamma \ob \\ \gamma:\Gamma\vdash \A~\gamma \type_{\;\!\ell}}{\bkt{\gamma : \Gamma,~\a : \A~\gamma} \ob}\]
Note that this is an operation that takes an element of one set (the absolute judgment $\Gamma\ob$) and an element of one presheaf (the hypothetical judgment $\gamma:\Gamma\vdash \A~\gamma \type_{\;\!\ell}$) and produces an element of a set (the absolute judgment $\bkt{\gamma : \Gamma,~\a : \A~\gamma} \ob$).
As before, at the moment the \emph{`variables'} $\gamma$ and $\a$ are just a convention of notation.

We then have a family of bijections, natural in $\Delta$:
\[\bkt{\Delta \to \bkt{\gamma : \Gamma,~\a : \A~\gamma}} \simeq \bkt{\bkt{\sigma~:~\Delta \to \Gamma} \times \bkt{\delta : \Delta \vdash \t~\delta : \A~\bkt{\sigma~\delta}}}\]
Note that this says that a substitution into an extended context $\bkt{\gamma : \Gamma,~\a : \A~\gamma}$ is precisely a substitution $\sigma$ into $\Gamma$, along with some term $\t$ of type $\A^\sigma$.
By the Yoneda lemma, the left-to-right part of this bijection is determined by setting $\Delta$ to $\bkt{\gamma : \Gamma,~\a : \A~\gamma}$ and evaluating at the identity $\id_{\bkt{\gamma\;:\;\Gamma,\;\a\;:\; \A\;\gamma}}$.
The first component gives us a substitution $\ft^{\A} : \bkt{\gamma : \Gamma,~\a : \A~\gamma} \to \Gamma$, which we call a \emph{fundamental context projection} or \emph{parent map}\footnote{From the perspective of type theoretic fibration categories, an alternative approach to semantics from that of CwFs, the role of fundamental context projections is instead played by \emph{fibrations}.} and denote by:
\[\ft^{\A} : \bkt{\gamma : \Gamma,~\a : \A~\gamma} \rightarrowtriangle \Gamma\]
The second component then gives us a term that we call the \emph{zero variable} and denote by $\zv^{\A} : \Tm(\bkt{\gamma : \Gamma,~ \a : \A~\gamma},\A^{\ft^\A})$, or:
\[\delta : \bkt{\gamma : \Gamma,~ \a : \A~\gamma} \vdash \zv^{\A}~\delta : \A~\bkt{\ft^{\A}~\delta}\]
Note that the forward direction of the bijection sends a substitution $\tau : \Delta \to \bkt{\gamma : \Gamma, ~\a : \A~\gamma}$ to the pair $\bkt{\ft^\A \circ \tau, ~\bkt{\zv^{\A}}^\tau}$. That this map is a bijection is witnessed by the existence of a substitution extension operation:
\begin{mathpar}
    \infer{\sigma : \Delta \to \Gamma \\ \delta : \Delta \vdash \t~\delta : \A~\bkt{\sigma~\delta}}{\sqbkt{\sigma,~\t} : \Delta \to \bkt{\gamma : \Gamma,~\a : \A~\gamma}}
\end{mathpar}
such that for $\sigma : \Delta \to \Gamma$ and $\delta : \Delta \vdash \t~\delta : \A~\bkt{\sigma~\delta}$, we have:
\begin{align}
  \ft^\A \circ \sqbkt{\sigma,~\t} &\equiv \sigma \label{eq:ft-sqbkt}\\
  \bkt{\zv^\A}^{\sqbkt{\sigma,~\t}} &\equiv \t\label{eq:zv-sqbkt}
\end{align}
and, conversely, for $\tau : \Delta \to \bkt{\gamma : \Gamma,~\a : \A~\gamma}$, we have:
\begin{equation}
  \label{eq:sqbkt-ft-zv}
  \sqbkt{ \ft^\A \circ \tau,~ \bkt{\zv^\A}^\tau} \equiv \tau.
\end{equation}
As a corollary of this we have that the following diagram is a pullback:
\[\begin{tikzcd}
\bkt{\delta : \Delta,~\a : \A~\bkt{\sigma~\delta}} \arrow[rrr, "\sqbkt{\sigma ~\circ~ \ft^{\A^\sigma},~\zv^{\A^\sigma}}"] \arrow[d,-{Triangle[open]},"\ft^{\A^\sigma}"']
\arrow[drrr, phantom, "\usebox\pullback" , very near start, color=black] & & &\bkt{\gamma : \Gamma,~\a : \A~\gamma} \arrow[d,-{Triangle[open]},"\ft^\A"] \\
\Delta \arrow[rrr,"\sigma"'] & & & \Gamma
\end{tikzcd}\]
So, in particular, in a CwF we have a distinguished pullback of any parent map along an arbitrary morphism, as another parent map, and this choice of pullbacks is definitionally functorial. We will often omit the superscripts from $\ft$ and $\zv$.

The map constructed from $\sigma$ in the top row of the diagram above will occur frequently in the exposition that follows, and we shall refer to it as the \emph{weakening two} of $\sigma$ by $\A$. Formally, given $\sigma : \Delta \to \Gamma$, and $\gamma : \Gamma \vdash \A~\gamma \type_{\;\!\ell}$, we have:
\begin{align*}
    &\mathsf{\cred{W}}_{\cred{2}}^\A~\sigma : \bkt{\theta : \Theta,~\a : \A~\bkt{\sigma~\theta}} \to \bkt{\gamma : \Gamma,~\a : \A~\gamma} \\
    &\mathsf{\cred{W}}_{\cred{2}}^\A~\sigma = \sqbkt{\sigma \circ \ft,~\zv}
\end{align*}

Finally, when working with hypothetical judgments in an extended context, we can also treat the variables more traditionally.
Instead of $\delta : \bkt{\gamma:\Gamma,~\a:\A~\gamma} \vdash \B~\delta\type_{\;\!\ell_\cblu{1}}$, we write $\gamma:\Gamma,~\a:\A~\gamma \vdash \B~\gamma~\a\type_{\;\!\ell_\cblu{1}}$, and so on.
In particular, the zero variable $\zv$ can be written more simply as
\[\gamma : \Gamma,~ \a : \A~\gamma \vdash \a : \A~\gamma\]
As before, this can be justified formally as an interpretation in the internal type theory of $\Set^{\cC\op}$.\bbox
\vspace{0.2cm}

\subsubsection{$\Pi$-Types}
\label{sec:pi-types}

All the basic type-forming operations in syntax translate into structure on a CwF.
For instance, a \emph{$\Pi$-structure} on a CwF with levels consists of the following structure and properties:
\begin{mathpar}
\infer{\gamma : \Gamma \vdash \A~\gamma \type_{\;\!\ell_\cblu{0}} \\ \gamma : \Gamma, ~\a : \A~\gamma \vdash \B~\gamma~\a \type_{\;\!\ell_\cblu{1}}}{\gamma : \Gamma \vdash \bkt{\cred{\Pi}~\A~\B}~\gamma \type_{\;\!\ell_\cblu{0}\;\join\;\ell_\cblu{1}}}
\\
\infer{\gamma : \Gamma, ~\a : \A~\gamma \vdash \t~\gamma~\a : \B~\gamma~\a}{\gamma : \Gamma \vdash \bkt{\lambda~\t}~\gamma : \bkt{\cred{\Pi}~\A~\B}~\gamma}
\\
\infer{\gamma : \Gamma \vdash \f~\gamma : \bkt{\cred{\Pi}~\A~\B}~\gamma \\ \gamma : \Gamma \vdash \s~\gamma : \A~\gamma}{\gamma : \Gamma \vdash \bkt{\eval~\f~\s}~\gamma : \B^{\sqbkt{\id_\Gamma,\;\s}}~\gamma}
\end{mathpar}
The above notation lets us talk about types in point-free notation, e.g. $\cred{\Pi}~\A~\B : \Ty_{\;\!\ell}~\Gamma$. When the explicit dependence on $\gamma$ is written, we can propagate the notation as follows:
\begin{align*}
\bkt{\cred{\Pi}~\A~\B}~\gamma &\equiv \bkt{\a : \A~\gamma} \to \B~\gamma~\a \\
\bkt{\lambda~\t}~\gamma &\equiv \lambda~\a.~\t~\gamma~\a \\
\bkt{\eval~\f~\s}~\gamma &\equiv \eval~\gamma~\bkt{\f~\gamma}~\bkt{\s~\gamma}
\end{align*}
Such that the $\beta$ and $\eta$ laws hold:
\begin{mathpar}
\infer{\gamma : \Gamma, ~\a : \A~\gamma \vdash \t~\gamma~\a : \B~\gamma~\a \\ \gamma : \Gamma \vdash \s~\gamma : \A~\gamma}{\gamma : \Gamma \vdash \bkt{\eval~\bkt{\lambda~\t} ~\s}~\gamma \equiv \t^{\sqbkt{\id_\Gamma,\;\s}}~\gamma : \B^{\sqbkt{\id_\Gamma,\;\s}}~\gamma}
\\
\infer{\gamma : \Gamma \vdash \f~\gamma: \bkt{\a : \A~\gamma} \to \B~\gamma~\a}{\gamma : \Gamma \vdash \bkt{\lambda~\bkt{\eval~\f^{\;\!\ft}~\zv}}~\gamma \equiv \f~\gamma : \bkt{\cred{\Pi}~\A~\B}~\gamma}
\end{mathpar}
Or, in indexed notation:
\begin{align*}
\gamma : \Gamma \vdash \eval~\gamma~\bkt{\lambda~\a.~\t~\gamma~\a}~\bkt{\s~\gamma} &\equiv  \t~\gamma~\bkt{\s~\gamma} : \B~\gamma~\bkt{\s~\gamma} \\
\gamma : \Gamma \vdash \lambda~\a.~\eval~\sqbkt{\gamma,~\a}~\bkt{\f~\gamma}~\a &\equiv \f~\gamma  : \bkt{\a : \A~\gamma} \to \B~\gamma~\a
\end{align*}
And such that the above constructs are stable under substitution, i.e. for $\sigma : \Delta \to \Gamma$:
\begin{align*}
\bkt{\cred{\Pi}~\A~\B}^\sigma &\equiv \cred{\Pi}~\A^\sigma~\B^{\mathsf{\cred{W}}_{\cred{2}}^\A \sigma} \\
\bkt{\lambda~\t}^\sigma &\equiv \lambda~\t^{\mathsf{\cred{W}}_{\cred{2}}^\A \sigma} \\
\bkt{\eval~\f~\s}^\sigma &\equiv \eval~\f^\sigma~\s^\sigma
\tag*{$\sectend$}
\end{align*}

\subsubsection{Universes}
\label{sec:universes}

A \emph{$\mathcal{U}$-structure} on a CwF with levels consists of the following structure and properties:
\begin{mathpar}
\infer{\ell \level}{\gamma : \Gamma \vdash \Type_{\;\!\ell}~\gamma \type_{\;\!\lsuc\;\ell}}
\and
\infer{\gamma : \Gamma \vdash \A~\gamma \type_{\;\!\ell}}{\gamma : \Gamma \vdash \Code~\A~\gamma : \Type_{\;\!\ell}~\gamma}
\and
\infer{\gamma : \Gamma \vdash \A~\gamma : \Type_{\;\!\ell}~\gamma}{\gamma : \Gamma \vdash \El~\A~\gamma \type_{\;\!\ell}}
\end{mathpar}
Such that $\Code$ and $\El$ are mutual inverses:
\begin{mathpar}
\infer{\gamma : \Gamma \vdash \A~\gamma \type_{\;\!\ell}}{\gamma : \Gamma \vdash \El~\bkt{\Code~\A}~\gamma \equiv \A~\gamma}
\and
\infer{\gamma : \Gamma \vdash \A~\gamma : \Type_{\;\!\ell}~\gamma}{\gamma : \Gamma \vdash \Code~\bkt{\El~\A}~\gamma \equiv \A~\gamma : \Type_{\;\!\ell}~\gamma}
\end{mathpar}
And such that the above constructs are stable under substitution, i.e. for $\sigma : \Delta \to \Gamma$:
\begin{align*}
\Type^{\sigma}_{\;\!\ell} &\equiv \Type_{\;\!\ell} \\
\bkt{\Code~\A}^\sigma &\equiv \Code~\A^\sigma \\
\bkt{\El~\A}^\sigma &\equiv \El~\A^\sigma
\tag*{$\sectend$}
\end{align*}

\subsubsection{Natural models}
\label{sec:natural-models-1}

We recall from~\cite{awodey:natmodels} that a CwF can equivalently be described as a \emph{natural model}: a category $\MC$ together with an (algebraically) representable natural transformation $\pr : \Tm \to \Ty$ of presheaves on $\MC$.
This amounts to representing the dependency of terms on types in \emph{`fibered'} rather than \emph{`indexed'} style, which is different from the usual type-theoretic notation, but it has the advantage that various operations can cleanly be described in terms of $\pr$.
Similarly, of course, a CwF with levels can be described by a family of representable transformations $\prl : \Tml \to \Tyl$.

In particular, like any morphism in a locally cartesian closed category, any such $\pr$ induces a polynomial endofunctor $\P_\pr$ of the presheaf category, where for any presheaf (i.e.\ judgment) $\X$, morphisms $\yon\Gamma \to \P_\pr(\X)$ (i.e.\ elements of the presheaf $\P_\pr(\X)$) are bijectively related to pairs consisting of a type $\A\in \Ty(\Gamma)$ in context $\Gamma$ and a morphism $\yon(\gamma:\Gamma, \a:\A~\gamma) \to \X$ (i.e.\ an element of $\X(\gamma:\Gamma, \a:\A~\gamma)$).
In syntax, this is a bidirectional rule, indicating a bijection between the data above and below the lines:
\begin{mathpar}
  \mprset{fraction={===}}
  \infer{\gamma:\Gamma \vdash \A~\gamma \type \\ \gamma:\Gamma, \a:\A~\gamma \vdash \xi~\gamma~\a : \X}{\gamma:\Gamma \vdash \widetilde{\xi}~\gamma : \P_\pr(\X)}
\end{mathpar}

Thus, for instance, $\P_{\pr_{\;\!\ell_\cblu{0}}}(\Ty_{\;\!\ell_\cblu{1}})$ represents families of types of level $\ell_\cblu{1}$ indexed by a type of level $\ell_\cblu{0}$.
Therefore, formation rules such as those for for $\Pi$-types and $\Sigma$-types:
\begin{mathpar}
  \infer{\gamma : \Gamma \vdash \A~\gamma \type_{\;\!\ell_\cblu{0}} \\ \gamma : \Gamma, ~\a : \A~\gamma \vdash \B~\gamma~\a \type_{\;\!\ell_\cblu{1}}}{\gamma : \Gamma \vdash \bkt{\cred{\Pi}~\A~\B}~\gamma \type_{\;\!\ell_\cblu{0}\;\join\;\ell_\cblu{1}}}
  \and
  \infer{\gamma : \Gamma \vdash \A~\gamma \type_{\;\!\ell_\cblu{0}} \\ \gamma : \Gamma, ~\a : \A~\gamma \vdash \B~\gamma~\a \type_{\;\!\ell_\cblu{1}}}{\gamma : \Gamma \vdash \bkt{\cred{\Sigma}~\A~\B}~\gamma \type_{\;\!\ell_\cblu{0}\;\join\;\ell_\cblu{1}}}
\end{mathpar}
are represented by morphisms $\cred{\Pi},\Sigma : \P_{\pr_{\;\!\ell_\cblu{0}}}(\Ty_{\;\!\ell_\cblu{1}}) \to \Ty_{\;\!\ell_\cblu{0}\join\ell_\cblu{1}}$.

The rules for terms can also be represented in this language.
For instance, a natural model has $\Pi$-types if and only if there is a pullback square
\[
  \begin{tikzcd}
    \P_{\pr_{\;\!\ell_\cblu{0}}}(\Tm_{\;\!\ell_\cblu{1}}) \ar[d,"\P_{\pr_{\;\!\ell_\cblu{0}}}(\pr_{\;\!\ell_\cblu{1}})"']\ar[r]\drpullback & \Tm_{\;\!\ell_\cblu{0}\join\ell_\cblu{1}} \ar[d,"\pr_{\;\!\ell_\cblu{0}\join\ell_\cblu{1}}"] \\
    \P_{\pr_{\;\!\ell_\cblu{0}}}(\Ty_{\;\!\ell_\cblu{1}}) \ar[r,"\cred{\Pi}"'] & \Ty_{\;\!\ell_\cblu{0}\join\ell_\cblu{1}},
  \end{tikzcd}
\]
meaning that there is a bijection
\begin{mathpar}
  \mprset{fraction={===}}
  \infer{\gamma : \Gamma, ~\a : \A~\gamma \vdash \t~\gamma~\a : \B~\gamma~\a}{\gamma:\Gamma \vdash (\lambda~\t)~\gamma : (\cred{\Pi}~\A~\B)~\gamma}
\end{mathpar}

Polynomial functors can also be composed, yielding another polynomial functor.
For instance, in a CwF without levels, $\P_\pr \circ \P_\pr$ is the functor such that $(\P_\pr\circ\P_\pr)(\X)$ represents elements of $\X$ in a doubly-extended context $(\gamma:\Gamma,~\a:\A~\gamma,~\b:\B~\gamma~\a)$, which means that it is the polynomial functor associated to the map $\pr^\cred{2} : \Tm^\cred{2} \to \P_\pr(\Ty)$ where the fiber of $\Tm^\cred{2}(\Gamma)$ over $(\A,\B)$ consists of a pair of terms $\gamma:\Gamma \vdash \a:\A~\gamma$ and $\gamma:\Gamma \vdash \b : \B~\gamma~\a$.
In particular, a CwF $\cC$ has $\Sigma$-types if and only if $\Sigma : \P_\pr(\Ty) \to \Ty$ represents any such pair by a single term, i.e.\ there is a pullback square
\[
  \begin{tikzcd}
    \Tm^\cred{2} \ar[d,"\pr^\cred{2}"']\ar[r]\drpullback & \Tm \ar[d,"\pr"] \\
    \P_\pr(\Ty) \ar[r,"\Sigma"'] & \Ty,
  \end{tikzcd}
\]
This is equivalent to a (cartesian) morphism of polynomial functors $\P_\pr \circ \P_\pr \to \P_\pr$.
Of course, there is an analogous version for a CwF with levels.\bbox

\subsubsection{Telescopes}
\label{sec:semtel}

We will often have finite towers of types:
\begin{align*}
    &\gamma : \Gamma \vdash \A~\gamma \type_{\;\!\ell_\cblu{0}} \\
    &\gamma : \Gamma,~\a : \A \vdash \B~\gamma~\a \type_{\;\!\ell_\cblu{1}} \\\
    &\gamma : \Gamma,~\a : \A,~ \b : \B~\gamma~\a~\b \vdash \C~\gamma~\a~\b \type_{\;\!\ell_\cblu{2}}
\end{align*}
We represent these with a single judgement, defining a \emph{telescope}:
\[\gamma : \Gamma \vdash \bkt{\a : \A~\gamma, ~\b : \B~\gamma~\a, ~\c : \C~\gamma~\a~\b}~\gamma \tel_{\;\!\ell}.\]
Formally speaking, telescopes and their elements (which we call \emph{partial substitutions}) are another CwF structure on the same category of contexts, which are related to the original one by specified operations.
In natural model style, the definition is:

\begin{definition}
  A natural model with levels $\cC$ \textbf{has telescopes} if it is equipped with:
  \begin{itemize}
  \item Another family of (algebraically) representable natural transformations $\tpr\subell : \PSub\subell \to \Tel\subell$.
    This yields two families of judgments for telescopes and partial substitutions:
    \begin{mathpar}
      \gamma:\Gamma \vdash \Upsilon~\gamma \tel_{\;\!\ell} \and
      \Gamma:\Gamma \vdash \upsilon:\Upsilon
    \end{mathpar}
    Their representability yields the extension of a context by a telescope:
    \begin{mathpar}
      \infer{\Gamma \ob \\ \gamma:\Gamma \vdash \Upsilon~\gamma \tel_{\;\!\ell}}{(\gamma :\Gamma \ext \upsilon:\Upsilon~\gamma) \ob.}
    \end{mathpar}
  \item Morphisms of polynomial functors $\id_\cC \to \P_{\tpr\subell}$, i.e.\ pullback squares
    \[
      \begin{tikzcd}
        \id \ar[r] \ar[d] \drpullback & \PSub\subell \ar[d,"\tpr\subell"] \\
        \id \ar[r,"\ec"'] & \Tel\subell.
      \end{tikzcd}
    \]
    This gives `empty telescopes' $\gamma:\Gamma \vdash \ec \tel\subell$ containing exactly one partial substitution $\gamma:\Gamma \vdash \esub : \ec$.
  \item Morphisms of polynomial functors $\P_{\tpr\subell} \circ \P_{\pr\subellp} \to \P_{\tpr\subell}$ whenever $\ell\cblu{'}\le\ell$.
    This says how to extend a telescope by a type:\footnote{Note that $\P_\tpr \circ \P_\pr$ is the polynomial functor associated to a map whose codomain is $\P_\tpr(\Ty)$, which is the presheaf of types in a context extended by a telescope.}
    \begin{mathpar}
      \infer{\gamma:\Gamma \vdash \Upsilon~\gamma \tel_{\;\!\ell} \\ \gamma:\Gamma \ext \upsilon:\Upsilon~\gamma \vdash \A~\gamma~\upsilon \type_{\;\!\ell\cblu{'}} \\ \ell\cblu{'}\le\ell}
      {\gamma:\Gamma \vdash (\upsilon:\Upsilon~\gamma,~\a:\A~\gamma~\upsilon) \tel_{\;\!\ell_\cblu{0}\;\join\;\ell}}
    \end{mathpar}
    such that the partial substitutions in $(\upsilon:\Upsilon~\gamma,~\a:\A~\gamma~\upsilon)$ are exactly pairs of a partial substitution in $\Upsilon$ and a term in $\A$, just as for $\Sigma$-types.
    Thus we get the rules from \cref{sec:part-subst}.
  \item The rules $(\Gamma\ext\ec_\p) \equiv \Gamma$ and $(\Gamma\ext(\Theta,~\x: \A)) \equiv  ((\Gamma\ext\Theta),~\x: \A)$ from \cref{sec:telescopes} hold.
    Note that these are equalities of objects of $\MC$, and in particular only make sense if $\pr$ and $\tpr$ are \emph{algebraically} representable.
  \item A morphism of polynomial functors $\P_{\tpr\sell0} \circ \P_{\tpr\sell1} \to \P_{\tpr\sellj01}$, giving the extension of telescopes by telescopes $(\upsilon:\Upsilon\ext \phi:\Phi~\upsilon)$ from \cref{sec:telesc-conc}, such that the rules from that section hold:
    \begin{mathpar}
      (\Gamma \ext (\Upsilon\ext \Phi)) \equiv ((\Gamma \ext \Upsilon)\ext \Phi)
      \and
      (\Upsilon \ext \ec) \equiv \Upsilon
      \and
      (\Upsilon \ext \bkt{\Phi, \x: \A}) \equiv ((\Upsilon\ext \Phi), \x : \A)
    \end{mathpar}
  \end{itemize}
\end{definition}

Syntactically, this definition represents the rules from \cref{sec:telescopes,sec:part-subst,sec:telesc-conc}, in the non-modal case.
Because it is phrased in terms of presheaves and operations on them, it implicitly includes substitution into telescopes that commute with the other operations.
Some of these commutation properties refer to weakening-two, which can be characterised in terms of them as well, for instance:
\begin{mathpar}
    \infer{\gamma : \Gamma \vdash \Upsilon~\gamma \tel_{\;\!\ell} \\ \sigma : \Delta \to \Gamma}{\delta : \Delta \vdash \Upsilon~\bkt{\sigma~\delta} \tel_{\;\!\ell}}
    \and
    \infer{\sigma : \Delta \to \Gamma \\ \gamma : \Gamma \vdash \Upsilon~\gamma \tel_{\;\!\ell}}{\mathsf{\cred{W}}_{\cred{2}}^\Upsilon~\sigma : \bkt{\delta : \Delta,~\upsilon : \Upsilon~\bkt{\sigma~\delta}} \to \bkt{\gamma : \Gamma, ~\upsilon : \Upsilon~\gamma}}
    \\
    \ec^\sigma \equiv \ec
    \and
    \bkt{\upsilon : \Upsilon~\gamma, ~\a : \A~\gamma~\upsilon}^\sigma~\delta \equiv \bkt{\delta : \Upsilon~\bkt{\sigma~\delta},~\a : \A~\bkt{\bkt{\mathsf{\cred{W}}_{\cred{2}}^\Upsilon~\sigma} ~\sqbkt{\delta, ~\upsilon}}}
    \\
    \mathsf{\cred{W}}_{\cred{2}}^{\ec}~\sigma \equiv \sigma
    \and
    \mathsf{\cred{W}}_{\cred{2}}^{\bkt{\Upsilon,\;\A}}~\sigma \equiv
    \mathsf{\cred{W}}_{\cred{2}}^{\A} ~\bkt{\mathsf{\cred{W}}_{\cred{2}}^\Upsilon~\sigma}
\end{mathpar}
For example, we have:
\[\bkt{\a : \A~\gamma, ~\b : \B~\gamma~\a}^\sigma~\delta \equiv \bkt{\a : \A~\bkt{\sigma~\delta}, ~\b : \B~\bkt{\sigma~\bkt{\ft ~\sqbkt{\delta,~\a}}}~\bkt{\zv~\sqbkt{\delta,~\a}}}\]
When we allow ourselves to use variables in the usual way, justified by the internal type theory of presheaves, we can write this as:
\[\bkt{\a : \A~\gamma, ~\b : \B~\gamma~\a}^\sigma~\delta \equiv \bkt{\a : \A~\bkt{\sigma~\delta}, ~\b : \B~\bkt{\sigma~\delta}~\a}\]
Note that the meaning of this construction is simply iterating the canonical construction of pullbacks over a tower:
\[\begin{tikzcd}
\bkt{\delta : \Delta,~\a : \A~\bkt{\sigma~\delta}, ~\b : \B~\bkt{\sigma~\delta}~\a} \arrow[d,-{Triangle[open]}] \arrow[drrrr, phantom, "\usebox\pullback" , very near start, color=black] \arrow[rrrr, "\sqbkt{\sqbkt{\sigma~\circ~\ft,~\zv}~\circ ~\ft,~\zv}"] &&&& \bkt{\gamma : \Gamma,~\a : \A~\gamma,~\b : \B~\gamma~\a} \arrow[d,-{Triangle[open]}] \\
\bkt{\delta : \Delta,~\a : \A~\bkt{\sigma~\delta}} \arrow[rrrr, "\sqbkt{\sigma ~\circ~ \ft,~\zv}"] \arrow[d,-{Triangle[open]}]
\arrow[drrrr, phantom, "\usebox\pullback" , very near start, color=black] &&&& \bkt{\gamma : \Gamma,~\a : \A~\gamma} \arrow[d,-{Triangle[open]}] \\
\Delta \arrow[rrrr,"\sigma"] &&&& \Gamma
\end{tikzcd}\]

We can also characterise the $\Pi$-telescopes from \cref{sec:pi-telescopes} just like the $\Pi$-types in \cref{sec:pi-types}.

\begin{definition}
  A CwF with levels $\cC$ with telescopes \textbf{has $\Pi$-telescopes} if:
  \begin{itemize}
  \item There is a pullback square
    \[
      \begin{tikzcd}
        \P_{\tpr\sell0}(\PSub\sell1) \ar[d,"\P_{\tpr\sell0}(\tpr\sell1)"']\ar[r]\drpullback & \PSub\sellj01 \ar[d,"\tpr\sellj01"] \\
        \P_{\tpr\sell0}(\Tel\sell1) \ar[r,"\cred{\Pi}"'] & \Tel\sellj01,
      \end{tikzcd}
    \]
  \item The computation rules from \cref{sec:pi-telescopes} hold.
  \end{itemize}
\end{definition}

Now the above rules allow us to build telescopes up from the empty telescope by adding types, just as the rules of a CwF allow us to build contexts from the empty context by adding types.
However, just as in the case of contexts, the rules do not stipulate that every telescope is obtained in that way.
Indeed, there is no way to assert such a thing in a Generalised Algebraic Theory.
However, it holds \emph{`admissibly'} in the initial syntactic model, and any CwF can be extended with telescopes in this way:

\begin{theorem}\label{thm:sctx-as-lists}
  Any CwF with levels can be equipped with telescopes.
  If it has $\Pi$-types, it also has $\Pi$-telescopes.
\end{theorem}
\begin{sectendproof}
  We define $\tpr\subell$ to be the map such that
  \[\P_{\tpr\subell} = \sum_{\n\in\cred{\mathbb{N}}\atop \forall \cblu{i}\le\n, \ell_{\cblu{i}}\le \ell} {\P_{\pr\sell0} \circ \cdots \circ \P_{\pr\sell\n}}.\]
  Thus an element of $\Tel\subell(\Gamma)$ is a tower of $\n$ types over $\Gamma$ of level $\le \ell$, and similarly for terms.
  The two morphisms of polynomial functors are then immediate.
  We define context extension in the obvious way by iterating context extension by types, and the equations hold.
  (This is the \emph{initial} structure of telescopes on $\cC$ in a straightforward sense.)
  Similarly, we define $\Pi$-telescopes by using the rules for computing them on extended telescopes.
\end{sectendproof}

\subsubsection{Meta-abstractions}
\label{sec:meta-abs}

Because meta-abstractions are not \emph{`reified'} in the theory as types, they do not require assuming any structure beyond that which is already present in the presheaf category.
Specifically, the rules for the judgment $\Gamma \vdash \A \slfrac{\type}{_{\delta\;:\;\Upsilon}}$ simply say that it should be (up to isomorphism) the object $\P_\tpr(\Ty)$ that classifies types indexed by a telescope.
Similarly, the rules for the elements of a meta-abstraction simply say that these are the object $\P_\tpr(\Tm)$ that classifies terms indexed by a telescope.
In other words, meta-abstractions of types and their terms are classified by a map (isomorphic to) $\P_\tpr(\pr)$.
Likewise, meta-abstractions of telescopes (the judgment $\Gamma \vdash \Theta \slfrac{\tel}{_{\delta\;:\;\Upsilon}}$) are classified by a map isomorphic to $\P_\tpr(\tpr)$.
This gives all the rules from \cref{sec:meta-abstractions,sec:meta-abstr-telesc}; thus any natural model with telescopes also has meta-abstractions of types and telescopes.

Semantically, we don't ever need to discuss meta-abstractions explicitly, since to judge $\Gamma \vdash \A \slfrac{\type_{\;\!\ell}}{_{\delta\;:\;\Upsilon}}$ is equivalent to judging $\Gamma \ext \Upsilon \vdash \A \type_{\;\!\ell}$, and so on.
Thus we will generally talk only about types and telescopes in contexts. \bbox

\subsubsection{Infinite Telescopes}
\label{sec:inftel}

We now define a new judgement $\gamma : \Gamma \vdash \cblu{\widetilde{\Upsilon}}~\gamma \inftel_{\;\!\ell}$ whose elements are `infinite telescopes'.
As with meta-abstractions, this is not (yet) introducing new structure on a CwF, rather it is a definition that can be made in the presheaf category of any CwF.
The idea is that an infinite telescope consists of an infinite sequence of types each dependent on all the previous ones:
\begin{align*}
  \gamma:\Gamma &\vdash \cblu{\widetilde{\Upsilon}}^\cred{0}~\gamma \type_{\;\!\ell_{\cblu{0}}}\\
  \gamma:\Gamma,~ \upsilon^\cblu{0} :\cblu{\widetilde{\Upsilon}}^\cred{0}~\gamma  &\vdash \cblu{\widetilde{\Upsilon}}^\cred{1}~\gamma~\upsilon^\cblu{0} \type_{\;\!\ell_{\cblu{1}}}\\
  \gamma:\Gamma,~ \upsilon^\cblu{0} :\cblu{\widetilde{\Upsilon}}^\cred{0}~\gamma,~ \upsilon^\cblu{1} : \cblu{\widetilde{\Upsilon}}^\cred{1}~\gamma~\upsilon^\cblu{0}
  &\vdash \cblu{\widetilde{\Upsilon}}^\cred{1}~\gamma~\upsilon^\cblu{0}~\upsilon^\cblu{1} \type_{\;\!\ell_{\cblu{2}}}\\
  &\vdots
\end{align*}
where $\ell_{\n}\le \ell$ for all $\n$.
Formally, we define this along with its approximating finite telescopes:
\begin{align*}
  \cblu{\widetilde{\Upsilon}}^{\cred{\partial}\cred{0}}~\gamma &\equiv \ec\\
  \cblu{\widetilde{\Upsilon}}^{\cred{\partial 1}}~\gamma &\equiv \bkt{\upsilon^\cblu{0} :\cblu{\widetilde{\Upsilon}}^\cred{0}~\gamma}\\
  \cblu{\widetilde{\Upsilon}}^\cred{\partial 2}~\gamma &\equiv \bkt{\upsilon^\cblu{0} :\cblu{\widetilde{\Upsilon}}^\cred{0}~\gamma,~\upsilon^\cblu{1} : \cblu{\widetilde{\Upsilon}}^\cred{1}~\gamma~\upsilon^\cblu{0}}\\
  &\vdots
\end{align*}
so that we can say that in general $\cblu{\widetilde{\Upsilon}}^{\n}$ is a type in context $(\gamma:\Gamma\ext \upsilon: \cblu{\widetilde{\Upsilon}}^{\cred{\partial}\n})$.
In syntax, this means we give the following bidirectional rule with infinitely many premises:
\begin{mathpar}
  \mprset{fraction={===}}
  \infer{\big(\gamma:\Gamma \vdash \cblu{\widetilde{\Upsilon}}^{\cred{\partial}\n}~\gamma \tel_{\;\!\ell}\big)_{\n\in \mathbb{\cred{N}}}\\
    \big(\gamma:\Gamma \ext \cblu{\partial}\upsilon:\cblu{\widetilde{\Upsilon}}^{\cred{\partial}\n}~\gamma \vdash \cblu{\widetilde{\Upsilon}}^\n~\gamma~\cblu{\partial}\upsilon \type_{\;\!\ell_\n}\big)_{\n\in \mathbb{\cred{\cred{N}}}}\\
    (\ell_\n \le \ell)_{\n\in \mathbb{\cred{N}}}\\\\
    \cblu{\widetilde{\Upsilon}}^{\cred{\partial 0}}~\gamma \equiv \ec \\
    \big(\gamma:\Gamma \vdash \cblu{\widetilde{\Upsilon}}^{\cred{\partial}(\n+\cred{1})}~\gamma \equiv (\cblu{\partial}\upsilon: \cblu{\widetilde{\Upsilon}}^{\cred{\partial}\n}~\gamma , ~ \upsilon : \cblu{\widetilde{\Upsilon}}^{\n}~\gamma~\cblu{\partial}\upsilon)\big)_{\n\in \mathbb{\cred{N}}}
  }{\gamma:\Gamma \vdash \cblu{\widetilde{\Upsilon}}~\gamma \inftel_{\;\!\ell}}
\end{mathpar}
(It would also be possible to define infinite contexts coinductively, but for our purposes this concrete definition is easier to work with.)

We have already defined substitution on finite telescopes, and that definition extends level-wise to infinite telescopes. Given $\sigma : \Delta \to \Gamma$ and $\gamma : \Gamma \vdash \cblu{\widetilde{\Upsilon}}~\gamma \inftel_\ell$, we define $\delta : \Delta \vdash \cblu{\widetilde{\Upsilon}}~\bkt{\sigma~\delta} \inftel_\ell$ to consist of the data:
\[\delta : \Delta,~\cblu{\partial}\upsilon : \cblu{\widetilde{\Upsilon}}^\n~\bkt{\sigma~\delta} \vdash \cblu{\widetilde{\Upsilon}}^\n~\bkt{\sigma~\delta}~\cblu{\partial}\upsilon \type_{\;\!\ell_\n}\]

Similarly, we would like to define infinite partial substitutions as infinite lists of terms sectioning an infinite telescope. This is encapsulated by the judgement $\gamma : \Gamma \vdash \cblu{\widetilde{\upsilon}}~\gamma : \cblu{\widetilde{\Upsilon}}~\gamma$, which is characterised by a similar bidirectional rule:
\begin{mathpar}
  \mprset{fraction={===}}
  \infer{(\gamma:\Gamma \vdash \cblu{\widetilde{\upsilon}}^{\cred{\partial}\n}~\gamma : \cblu{\widetilde{\Upsilon}}^{\cred{\partial}\n}~\gamma)_{\n\in \mathbb{\cred{N}}}\\
    (\gamma:\Gamma \vdash \cblu{\widetilde{\upsilon}}^{\n}~\gamma : \cblu{\widetilde{\Upsilon}}^\n~\gamma~(\cblu{\widetilde{\upsilon}}^{\cred{\partial}\n}~\gamma))_{\n\in \mathbb{\cred{N}}}\\\\
    \cblu{\widetilde{\upsilon}}^{\cred{\partial 0}}~\gamma \equiv \esub
    \\
    (\gamma:\Gamma \vdash \cblu{\widetilde{\upsilon}}^{\cred{\partial}(\n+\cred{1})}~\gamma \equiv \sqbkt{\cblu{\widetilde{\upsilon}}^{\cred{\partial}\n}~\gamma , ~ \cblu{\widetilde{\upsilon}}^{\n}~\gamma})_{\n\in \mathbb{\cred{N}}}
  }{\gamma:\Gamma \vdash\cblu{\widetilde{\upsilon}}~\gamma : \cblu{\widetilde{\Upsilon}}~\gamma}
\end{mathpar}
Pullback of infinite partial substitutions is defined, as before, to consist of the data:
\[\delta : \Delta \vdash \cblu{\widetilde{\upsilon}}^\n ~\bkt{\sigma~\delta} :\cblu{\widetilde{\Upsilon}}^\n ~\bkt{\sigma~\delta}~\bkt{\cblu{\widetilde{\upsilon}}^{\cred{\partial}\n} ~\bkt{\sigma~\delta}}\]

Categorically, these rules mean we define the map $\pr^\cred{\infty}\subell : \PSub^\cred{\infty}\subell \to \Tel^\cred{\infty}\subell$ to be the limit of the sequence:
\[ \cdots \to \pr\subell^\n \to \cdots \to \pr\subell^\cred{3} \to \pr\subell^\cred{2} \to \pr\subell \to \One \]
where $\pr\subell^\n$ is the map such that
\[\P_{\pr\subell^\n} = \sum_{\forall \cblu{i}\le\n, \ell_{\cblu{i}}\le \ell} {\P_{\pr\sell0} \circ \cdots \circ \P_{\pr\sell\n}}.\]
In particular, $\P_\One$ is the identity functor and $\One$ is the identity map of the terminal object.
There is only one natural map $\prl^{\n+\cred{1}} \to \prl^\n$, which discards the last type in a telescope of length $\n+\cred{1}$; it is not possible to discard any of the other types and get a telescope of length $\n$. \bbox

\subsubsection{$\oldomega$-Limits}
\label{sec:inflim}

Finally, we define the structure of infinite (sequential, Reedy) limits on a CwF.
These are an \emph{`infinitary rule'} (i.e.\ a non-elementary structure) that is not part of dTT or any implementable type theory, but we will use them to build our intended models of dTT.
Syntactically, they are essentially just a kind of $\Sigma$-type of an infinite telescope.

\begin{definition}
  A CwF with levels has \textbf{$\oldomega$-limits} if it is equipped with pullback squares
  \[
    \begin{tikzcd}
      \PSub\subell^\cred{\infty} \ar[d,"\pr^\cred{\infty}\subell"']\ar[r,"\lim"]\drpullback & \Tml \ar[d,"\prl"] \\
      \Tel\subell^\cred{\infty} \ar[r,"\lim"'] & \Tyl,
    \end{tikzcd}
  \]
\end{definition}

In syntax, this means we have the following structure and properties.
Firstly, having a merely \emph{commutative} square as above gives the following rules:
\begin{mathpar}
\infer{\gamma : \Gamma \vdash \cblu{\widetilde{\Upsilon}}~\gamma \inftel_\ell}{\gamma : \Gamma \vdash \lim~\big(\cblu{\widetilde{\Upsilon}}~\gamma\big) \type_{\;\!\ell}}
\and
\infer{\gamma : \Gamma \vdash \cblu{\widetilde{\upsilon}}~\gamma : \cblu{\widetilde{\Upsilon}}~\gamma}{\gamma : \Gamma \vdash \lim~\big(\cblu{\widetilde{\upsilon}}~\gamma\big) : \lim~\big(\cblu{\widetilde{\Upsilon}}~\gamma\big)}
\end{mathpar}
Secondly,
\begin{mathpar}
\infer{\gamma : \Gamma \vdash \cblu{\mathfrak{u}} : \lim~\big(\cblu{\widetilde{\Upsilon}}~\gamma\big)}{\gamma : \Gamma \vdash \res^{\cred{\partial}\n}~\gamma~\cblu{\mathfrak{u}} : \cblu{\widetilde{\Upsilon}}^{\cred{\partial}\n}~\gamma}
\and
\infer{\gamma : \Gamma \vdash \cblu{\mathfrak{u}} : \lim~\big(\cblu{\widetilde{\Upsilon}}~\gamma\big)}{\gamma : \Gamma \vdash \res^\n~\gamma~\cblu{\mathfrak{u}} : \cblu{\widetilde{\Upsilon}}^\n~\gamma~\big(\res^{\cred{\partial}\n}~\gamma ~\cblu{\mathfrak{u}}\big)}
\end{mathpar}
We require that $\res^{\cred{\partial}\n}$ is derived from $\res^\n$ via:
\begin{align*}
\res^\cred{\partial 0}~\gamma~\cblu{\mathfrak{u}} &\equiv \esub \\
\res^{\cred{\partial}\bkt{\n+\cred{1}}} ~\gamma~\cblu{\mathfrak{u}} &\equiv \sqbkt{\res^{\cred{\partial}\n}~\gamma~\cblu{\mathfrak{u}}, ~\res^\n~\gamma~\cblu{\mathfrak{u}}}
\end{align*}
and that the following computation and uniqueness rules hold:
\begin{align*}
  \res^{\cred{\partial}\n}~\gamma ~\bkt{\lim\bkt{\cblu{\widetilde{a}}~\gamma}} &\equiv \cblu{\widetilde{a}}^{\cred{\partial}\n}~\gamma \\
  \res^\n~\gamma ~\bkt{\lim\bkt{\cblu{\widetilde{a}}~\gamma}} &\equiv \cblu{\widetilde{a}}^\n~\gamma\\
  \cblu{\mathfrak{u}} &\equiv \lim\bkt{\res^\n~\gamma~\cblu{\mathfrak{u}}}_\n
\end{align*}
Of course, all these constructions must also be stable under substitution:
\begin{align*}
  \bkt{\lim\bkt{\cblu{\widetilde{\Upsilon}}~\gamma}}^\sigma &\equiv \lim\bkt{\cblu{\widetilde{\Upsilon}}~\bkt{\sigma~\theta}} \\
  \bkt{\lim\bkt{\cblu{\widetilde{a}}~\gamma}}^\sigma &\equiv \lim\bkt{\cblu{\widetilde{a}}~\bkt{\sigma~\theta}} \\
  \bkt{\res^{\cred{\partial}\n} ~\gamma~\bkt{\cblu{\mathfrak{a}}~\gamma}}^\sigma &\equiv \res^{\cred{\partial}\n}~\theta~\bkt{\cblu{\mathfrak{a}} ~\bkt{\sigma~\theta}} \\
  \bkt{\res^\n ~\gamma~\bkt{\cblu{\mathfrak{a}}~\gamma}}^\sigma &\equiv \res^\n~\theta~\bkt{\cblu{\mathfrak{a}} ~\bkt{\sigma~\theta}}
  \tag*{\bbox}
\end{align*}

\subsection{The simplicial model}
\label{sec:simplicial-model}

In this section we fix a model of dependent type theory with all of the structure described above, which we call the \emph{discrete model} ($\dm$). From it, we will construct a derived model called the \emph{simplicial model} ($\sm$). We will do this, first, by way of constructing the \emph{truncated simplicial models} ($\sm^\n$) for $\n \geq \minustwo$.

\subsubsection{The Augmented Semi-Simplex Category}
\label{sec:asscat}

Let $\bB$ be the type of binary digits, which are $\Zero,~\One : \bB$.  For $\n \geq \m \geq \minusone$, let $\bB^{\angle{\n},\angle{\m}}$ be the type of length $\n+\cred{1}$ binary sequences such that exactly $\m+\cred{1}$ of the digits have value $\One$. When $\b_\cblu{1} : \bB^{\angle{\n},\angle{\m}}$ and $\b_\cblu{0} : \bB^{\angle{\m},\angle{\k}}$, we have a composition $\b_\cblu{1} \circ \b_\cblu{0} : \bB^{\angle{\n},\angle{\k}}$ obtained by replacing the $\One$ digits in $\b_\cblu{1}$ with the digits of $\b_\cblu{0}$. For example $\One\Zero\One\Zero\Zero\One\One \circ \Zero\One\One\Zero = \Zero\Zero\One\Zero\Zero\One\Zero$. The category whose objects are $\angle{\n}$ and whose morphisms $\angle{\m} \to \angle{\n}$ are $\bB^{\angle{\n},\angle{\m}}$ is the augmented semi-simplex category $\cred{\oldDelta^+}$.  Note that each of the representables $\bB^{\angle{\n},\blank}$ only has finitely many elements.  We write $\es$ for the length-zero sequence, which is the unique element of $\bB^{\angle{\minusone},\angle{\minusone}}$.

The identities $\id_{\angle{\n}}$ are given by length $\n + \cred{1}$ sequences of the digit $\One$. Further, for any $\b : \bB^{\angle{\n},\angle{\k}}$, we obtain $\Zero\b : \bB^{\angle{\n+\cred{1}},\angle{\k}}$ and $\One\b : \bB^{\angle{\n+\cred{1}},\angle{\k+\cred{1}}}$ by left appending the indicated digit.  The following identities hold:
\begin{align*}
    &\Zero\b_\cblu{1} \circ \b_\cblu{0} \equiv \Zero\bkt{\b_\cblu{1} \circ \b_\cblu{0}} \\
    &\One\b_\cblu{1} \circ \One\b_\cblu{0} \equiv \One \bkt{\b_\cblu{1} \circ \b_\cblu{0}} \\
    &\One\b_\cblu{1} \circ \Zero\b_\cblu{0} \equiv \Zero\bkt{\b_\cblu{1} \circ \b_\cblu{0}}
\end{align*}
Note that by the second rule, along with the fact that $\One\id_{\angle{\n}} \equiv \id_{\angle{\n+\cred{1}}}$, the assignments $\angle{\n} \mapsto \angle{\n + \cred{1}}$ and $\b \mapsto \One\b$ define an endofunctor of $\cred{\oldDelta^+}$.

Additionally, for every $\n \geq \minustwo$, we have the full subcategory $\cred{\oldDelta^{+}_\n}$ of $\cred{\oldDelta^{+}}$ on those objects $\angle{\k}$ with $\k \leq \n$.
Thus $\cred{\oldDelta^{+}_\minustwo}$ is the empty category, while $\cred{\oldDelta^{+}_{\minusone}}$ is the terminal category. \bbox

\subsubsection{Truncated Simplicial Objects}
The objects of the $\n$-truncated simplicial model $\sm^{\n}$ are $\MC$-valued presheaves on $\cred{\oldDelta^{+}_\n}$, denoted:
\[\Gamma \ob_{\sm^\n}\]
Thus the underlying category of $\sm^\n$ is $\cC^{\cred{\oldDelta^{+}_\n}}$.
For each such presheaf and $\n \geq \m \geq \minustwo$, we have $\Gamma_\m \ob_\dm$, where $\Gamma_\minustwo \equiv \ec_\dm$ is the distinguished terminal object of $\MC$.

Further, if we have $\b : \bB^{\angle{\n},\angle{\m}}$, then $\Gamma^\b : \Gamma_\n \to \Gamma_\m$, and this assignment is contravariantly functorial on the nose. We also sometimes write $\gamma^\b$ for $\Gamma^\b~\gamma$. Morphisms of simplicial objects are natural transformations. The data of $\sigma : \Delta \to \Gamma$ thus consists of a morphism $\sigma_\n : \Delta_\n \to \Gamma_\n$ for each $\n$, such that for any $\b : \bB^{\angle{\n},\angle{\m}}$, we have:
\[\Delta^\b \circ \sigma_\n \equiv \sigma_\m \circ \Gamma^\b\]
There are two additional functors of relevance relating the truncated simplicial models at different dimensions: \emph{truncation} and \emph{d\'ecalage}.
\begin{align*}
&\pi : \MC^{\cred{\oldDelta^{+}_{\n+\cred{1}}}} \to \MC^{\cred{\oldDelta^{+}_\n}}
& &\bkt{\blank}\D : \MC^{\cred{\oldDelta^{+}_{\n+\cred{1}}}} \to \MC^{\cred{\oldDelta^{+}_\n}} \\
&\bkt{\pi\Gamma}_{\m+\cred{1}}  \equiv \Gamma_{\m+\cred{1}}  & &\bkt{\Gamma\D}_{\m+\cred{1}} \equiv \Gamma_{\m+\cred{2}} \\
&\bkt{\pi\Gamma}^\b \equiv \Gamma^\b & &\bkt{\Gamma\D}^\b \equiv \Gamma^{\One\b} \\
&\bkt{\pi\sigma}_{\m+\cred{1}}  \equiv \sigma_{\m+\cred{1}}  & &\bkt{\sigma\D}_{\m+\cred{1}} \equiv \sigma_{\m+\cred{2}}
\end{align*}
There is a natural transformation between them:
\begin{align*}
&\rho : \bkt{\blank}\D \Rightarrow \pi \\
&\bkt{\rho_\Gamma}_{\m+\cred{1}} \equiv \Gamma^{\Zero\id_{\angle{\m+\cred{1}}}}
\end{align*}
Note that $\rho_\Gamma : \Gamma\D \to \pi\Gamma$ is a morphism of presheaves since for $\b : \bB^{\angle{\n+\cred{1}},\angle{\m+\cred{1}}}$, we have:
\begin{align*}
    \big(\pi\Gamma\big)^\b \circ \bkt{\rho_\Gamma}_{\n+\cred{1}} &\equiv \Gamma^{\b} \circ \Gamma^{\Zero\id_{\angle{\n+\cred{1}}}} \equiv \Gamma^{\Zero\id_{\angle{\n+\cred{1}}} \circ \b} \equiv \Gamma^{\Zero\bkt{\id_{\angle{\n+\cred{1}}} \circ \b}} \equiv \Gamma^{\Zero\b} \\
    &\equiv \Gamma^{\Zero\bkt{\b\circ\id_{\angle{\m+\cred{1}}}}} \equiv \Gamma^{\One\b \circ \Zero\id_{\angle{\m+\cred{1}}}} \equiv \Gamma^{\Zero\id_{\angle{\m+\cred{1}}}} \circ \Gamma^{\One\b} \equiv \bkt{\rho_\Gamma}_{\m+\cred{1}} \circ \big(\Gamma\D\big)^\b
\end{align*}
A similar proof shows that $\rho$ is natural, as its components arise from morphisms in $\cred{\oldDelta^{+}_\n}$, and any morphism of presheaves must respect these. \bbox

\subsubsection{Intuition} We will now construct the type-theoretical/fibrant structure of the truncated simplicial model. This will be done concretely through a series of mutually inductive definitions that will require substantially strengthening the inductive hypothesis for the sake of making everything well-typed.

However, before we launch into that, it would be useful to keep in mind where we are headed. At the most basic level, we would like to define the judgement
\[\gamma : \Gamma \vdash_{\sm^{\n+\cred{1}}} \A~\gamma \type_{\;\!\ell}\]
A simplicial type consists entirely of the data of its discrete $\m$-simplex types for $\m \leq \n + \cred{1}$, all of which live at the same level $\ell$:
\begin{align*}
    \gamma_\bminusone : \Gamma_\minusone &\vdash_\dm \A_\minusone~\gamma_\bminusone \type_{\;\!\ell} \\
    \gamma_\cblu{0} : \Gamma_\cred{0}, ~\cblu{\ze\sub{0}} : \A_\minusone~\gamma_\cblu{0}^{\Zero} &\vdash_\dm \A_\cred{0}~\gamma_\cblu{0}~\cblu{\ze\sub{0}} \type_{\;\!\ell} \\
    \gamma_\cblu{1} : \Gamma_\cred{1}, ~\cblu{\ze\sub{00}} : \A_\minusone~\gamma_\cblu{1}^{\Zero\Zero}, ~\cblu{x\sub{01}} : \A_\cred{0}~\gamma_\cblu{1}^{\Zero\One}~\cblu{\ze\sub{00}}, ~\cblu{x\sub{10}}:\A_\cred{0}~\gamma_\cblu{1}^{\One\Zero}~\cblu{\ze\sub{00}} &\vdash_\dm \A_\cred{1}~\gamma_\cblu{1}~\cblu{\ze\sub{00}}~\cblu{x\sub{01}}~\cblu{x\sub{10}} \type_{\;\!\ell} \\
    &\;\vdots
\end{align*}
Note that $\Gamma_{\cred{0}}$ (for example) denotes the $\cred{0}$-component of the presheaf $\Gamma$, which is an object at $\dm$, while $\gamma_{\cblu{0}}$ is an atomic variable name belonging to this object.
As another example, the type annotation on the variable $\cblu{x\sub{01}}$ is well-typed because the outer square of the following diagram is a distinguished pullback:
\[\begin{tikzcd}
\big(\gamma_\cblu{1} : \Gamma_\cred{1}, ~\ze\sub{\cblu{00}} : \A_\minusone~\bkt{\gamma_\cblu{1}^{\Zero\One}}{}^{\Zero}\big) \arrow[r] \arrow[d,-{Triangle[open]}]
\arrow[dr, phantom, "\usebox\pullback" , very near start, color=black] &
\big(\gamma_\cblu{0} : \Gamma_\cred{0},~\ze\sub{\cblu{0}} : \A_\minusone~\gamma_\cblu{0}^{\Zero}\big) \arrow[r] \arrow[d,-{Triangle[open]}]
\arrow[dr, phantom, "\usebox\pullback" , very near start, color=black] & \big(\gamma_\bminusone : \Gamma_\minusone , ~\ze\sub{\cblu{\varnothing}} : \A_\minusone~\gamma_\bminusone\big) \arrow[d,-{Triangle[open]}] \\
\big(\gamma_\cblu{1} : \Gamma_\cred{1}\big) \arrow[r,"\Gamma^{\Zero\One}"]
\arrow[rr, bend right=20,swap, "\Gamma^{\Zero} \circ \Gamma^{\Zero\One} \equiv \Gamma^{\Zero\Zero}"]&
\big(\gamma_\cblu{0} : \Gamma_\cred{0}\big) \arrow[r,"\Gamma^{\Zero}"] & \big(\gamma_\bminusone : \Gamma_\minusone\big)
\end{tikzcd}\]

We will write the type declarations of $\A_\n$ generically as:
\[\gamma_{\cblu{\n+1}} : \Gamma_{\n+\cred{1}}, ~\cblu{\partial}\a : \pi\A_{\cred{\partial}\bkt{\n+\cred{1}}}~\gamma_{\cblu{\n+1}} \vdash_\dm \A_{\n+\cred{1}}~\gamma_{\cblu{\n+1}}~\cblu{\partial}\a \type_{\;\!\ell_\cblu{n+1}}.\]
Here $\pi\A_{\cred{\partial}\bkt{\n+\cred{1}}}$ is a telescope consisting of the \emph{`boundary'} of an $(\n+\cred{1})$-simplex, also known as the Reedy \emph{`matching object'} of an augmented semi-simplicial type.
For example, we will have:
\begin{align*}
  \A_{\cred{\partial}(\minusone)}~\gamma_\minusone &\equiv \ec\\
  \A_{\cred{\partial}\cred{0}}~\gamma_{\cred{0}} &\equiv \bkt{\cblu{\ze\sub{0}} : \A_\minusone~\gamma_\cblu{0}^{\Zero}}\\
  \A_{\cred{\partial}\cred{1}}~\gamma_{\cred{1}} &\equiv \bkt{\cblu{\ze\sub{00}} : \A_\minusone~\gamma_\cblu{1}^{\Zero\Zero}, ~\cblu{x\sub{01}} : \A_\cred{0}~\gamma_\cblu{1}^{\Zero\One}~\cblu{\ze\sub{00}}, ~\cblu{x\sub{10}}:\A_\cred{0}~\gamma_\cblu{1}^{\One\Zero}~\cblu{\ze\sub{00}}}.
\end{align*}

Similarly, we would like to define simplicial terms to consist of the data of their discrete $\m$-simplex terms for $\m \leq \n + \cred{1}$. The judgement
\[\gamma : \Gamma \vdash_{\sm^{\n+\cred{1}}} \t~\gamma : \A~\gamma\]
will be defined to consist of the data:
\begin{align*}
    \gamma_\bminusone : \Gamma_\minusone &\vdash_\dm \t_\minusone~\gamma_\bminusone : \A_\minusone~\gamma_\bminusone \\
    \gamma_\cblu{0} : \Gamma_\cred{0} &\vdash_\dm \t_\cred{0} ~\gamma_\cblu{0} : \A_\cred{0}~\gamma_\cblu{0}~\bkt{\t_\minusone~{\gamma_\cblu{0}}^\Zero} \\
    \gamma_\cblu{1} : \Gamma_\cred{1} &\vdash_\dm \t_\cred{1} ~\gamma_\cblu{1} :
    \A_\cred{1}~\gamma_\cblu{1}~\bkt{\t_\minusone~{\gamma_\cblu{0}}^{\Zero\Zero}} ~\bkt{\t_\cred{0}~{\gamma_\cblu{1}}^{\Zero\One}} ~\bkt{\t_\cred{0}~{\gamma_\cblu{1}}^{\One\Zero}} \\
    &\;\vdots
\end{align*}
Similarly to before, we will write this generically as
\[\gamma_{\cblu{\n+1}} : \Gamma_{\n+\cred{1}} \vdash_\dm \t_{\n+\cred{1}}~\gamma_{\cblu{\n+1}} :  \A_{\n+\cred{1}}~\gamma_{\cblu{\n+1}} ~\bkt{\pi\t_{\cred{\partial}\bkt{\n+\cred{1}}}~\gamma_{\cblu{\n+1}}} \]
where $\pi\t_{\cred{\partial}\bkt{\n+\cred{1}}}~\gamma_{\cblu{\n+1}}$ denotes the action of the lower-dimensional parts of $\t$ on the boundary of $\gamma_{\cblu{\n+1}}$. \bbox

\subsubsection{Fibrant Structure}
\label{sec:fibrant-structure}

As suggested above, the basic structure of the fibrant theory of the models $\sm^{\n}$ will be defined by mutual induction.
In this section our goal is to define the presheaves of types and terms in $\sm^{\n}$, along with the context extension operation (but not yet its universal property).
This requires defining several other notions mutually, including a type-theoretic version of Reedy \emph{`matching objects'} and a truncated version of display that decreases dimension.

\paragraph{Declarations and Simple Cases}
\label{sec:decl-simple-cases}

We start by declaring the type of all these structures and operations, and giving those definitions that are direct.
First, we will have \emph{matching telescopes} and \emph{matching substitutions}:
\begin{mathpar}
\infer{\gamma^\cblu{\blank} : \pi\Gamma \vdash_{\sm^\n} \A~\gamma^\cblu{\blank} \type_{\;\!\ell}}
{\gamma_{\cblu{\n+1}} : \Gamma_{\n+\cred{1}} \vdash_\dm \A_{\cred{\partial}\bkt{\n+\cred{1}}} ~\gamma_{\cblu{\n+1}} \tel_{\;\!\ell}}
\and
\infer{\gamma^\cblu{\blank} : \pi\Gamma \vdash_{\sm^\n} \t~\gamma^\cblu{\blank} : \A~\gamma^\cblu{\blank}}
{\gamma_{\cblu{\n+1}} : \Gamma_{\n+\cred{1}} \vdash_\dm \t_{\cred{\partial}\bkt{\n+\cred{1}}} ~\gamma_{\cblu{\n+1}} : \A_{\cred{\partial}\bkt{\n+\cred{1}}} ~\gamma_{\cblu{\n+1}}}
\end{mathpar}
The inductive definitions of these telescopes and substitutions will be given in \cref{sec:inductive-cases}.
However, in terms of them, we are able to define the types and terms of $\sm^{\n+\cred{1}}$, as pairs of a type or term in $\sm^\n$ with a discrete type or term over its matching object.
We can formulate these definitions type-theoretically as bidirectional rules.
\begin{mathpar}
  \mprset{fraction={===}}
\infer{\quad \gamma^\cblu{\blank} : \pi\Gamma \vdash_{\sm^\n} \pi\A~\gamma^\cblu{\blank} \type_{\;\!\ell} \quad \\ \gamma_{\cblu{\n+1}} : \Gamma_{\n+\cred{1}},~\cblu{\partial}\a : \pi\A_{\cred{\partial}\bkt{\n+\cred{1}}}~\gamma_{\cblu{\n+1}} \vdash_\dm \A_{\n+\cred{1}}~\gamma_{\cblu{\n+1}}~\cblu{\partial}\a \type_{\;\!\ell}}{\gamma : \Gamma \vdash_{\sm^{\n+\cred{1}}} \A~\gamma \type_{\;\!\ell}}
\and
\infer{\quad\gamma^\cblu{\blank} : \pi\Gamma \vdash_{\sm^\n} \pi\t~\gamma^\cblu{\blank} :  \pi\A~\gamma^\cblu{\blank}\quad \\ \gamma_{\cblu{\n+1}} : \Gamma_{\n+\cred{1}} \vdash_\dm \t_{\n+\cred{1}} ~\gamma_{\cblu{\n+1}} : \A_{\n+\cred{1}}~\gamma_{\cblu{\n+1}} ~\bkt{\pi\t_{\cred{\partial}\bkt{\n+\cred{1}}}~\gamma_{\cblu{\n+1}}}}{\gamma : \Gamma \vdash_{\sm^{\n+\cred{1}}} \t~\gamma : \A~\gamma}
\end{mathpar}
Extension of contexts by a type $\gamma : \Gamma \vdash_{\sm^{\n+\cred{1}}} \A~\gamma \type_{\;\!\ell}$, and of a substitution by a term $\gamma : \Gamma \vdash_{\sm^{\n+\cred{1}}} \t~\gamma:\A~\gamma$, are then obtained as follows:
\begin{align*}
\big(\gamma : \Gamma,~\a : \A~\gamma\big)_{\m+\cred{1}} &\equiv \big(\gamma^\cblu{\blank} : \pi\Gamma,~\a^\cblu{\blank} : \pi\A~\gamma^\cblu{\blank}\big)_{\m+\cred{1}} \quad \text{for} \quad \m < \n
\\
\big(\gamma : \Gamma,~\a : \A~\gamma\big)_{\n+\cred{1}} &\equiv \big(\gamma_{\cblu{\n+1}} : \Gamma_{\n+\cred{1}},~\cblu{\partial}\a : \pi\A_{\cred{\partial}\bkt{\n+\cred{1}}}~\gamma_\cblu{\n+1},~\a:\A_{\n+\cred{1}}~\gamma_{\cblu{\n+1}} ~\cblu{\partial}\a\big) \\
\sqbkt{\sigma,~\t}_{\m+\cred{1}} &\equiv \sqbkt{\pi\sigma,~\pi\t}_{\m+\cred{1}} \quad \text{for} \quad \m < \n
\\
\sqbkt{\sigma,~\t}_{\n+\cred{1}} &\equiv \sqbkt{\sigma_{\n+\cred{1}},~\pi\t_{\cred{\partial}\bkt{\n+\cred{1}}},~\t_{\n+\cred{1}}}.
\end{align*}
So far this is just a definition of the family of discrete \emph{objects} underlying $\bkt{\gamma : \Gamma,~\a : \A~\gamma}$; we will enhance it to a diagram in~\eqref{eq:functorial-ext} below.

We will also prove that matching telescopes and substitutions are stable under substitution, such that for $\sigma : \Delta \to \Gamma$ in $\MC^{\cred{\oldDelta^{+}_{\n+\cred{1}}}}$, we have:
\begin{mathpar}
\big(\A^{\pi\sigma}\big)_{\cred{\partial}\bkt{\n+\cred{1}}} \equiv \big(\A_{\cred{\partial}\bkt{\n+\cred{1}}}\big)^{\sigma_{\n+\cred{1}}}
\and
\big(\t^{\pi\sigma}\big)_{\cred{\partial}\bkt{\n+\cred{1}}} \equiv \big(\t_{\cred{\partial}\bkt{\n+\cred{1}}}\big)^{\sigma_{\n+\cred{1}}}
\end{mathpar}
Substitution on types $\gamma : \Gamma \vdash_{\sm^{\n+\cred{1}}} \A~\gamma \type_{\;\!\ell}$ and terms $\gamma : \Gamma \vdash_{\sm^{\n+\cred{1}}} \t~\gamma : \A~\gamma$ can then be defined as:
\begin{align*}
\pi\big(\A^\sigma\big) &\equiv \pi\A^{\pi\sigma} &
\big(\A^\sigma\big)_{\n+\cred{1}} &\equiv \A_{\n+\cred{1}}^{\mathsf{\cred{W}}_{\cred{2}}^{\pi\A_{\cred{\partial}\bkt{\n+\cred{1}}}} \sigma_{\n+\cred{1}}} \\
\pi\big(\t^\sigma\big) &\equiv \pi\t^{\pi\sigma} & \big(\t^\sigma\big)_{\n+\cred{1}} &\equiv \t_{\n+\cred{1}}^{\sigma_{\n+\cred{1}}}.
\end{align*}
Functoriality of substitutions in $\sm^{\n+\cred{1}}$ then follows from that of $\sm^\n$ and $\dm$.

In order to define the matching telescopes and substitutions, we will require the definition of \emph{display} to be part of the mutual induction.
As noted above, when working with truncated diagrams, display takes an $(\n+\cred{1})$-truncated semi-simplicial diagram $\A$ to an $\n$-truncated one that's dependent
on $\pi\A$. Since we have no modal locks available yet, we are also forced to take this version of display to be in a totally d\'ecalaged context; recall that d\'ecalage makes sense for arbitrary (non-fibrant) contexts.
\begin{mathpar}
\infer{\gamma : \Gamma \vdash_{\sm^{\n+\cred{1}}} \A~\gamma \type_{\;\!\ell}}
{\gamma^\cblu{+} : \Gamma\D, ~\a : \pi\A^{\rho_\Gamma} ~\gamma^\cblu{+} \vdash_{\sm^\n} \A\d ~\gamma^\cblu{+}~\a \type_{\;\!\ell}}
\and
\infer{\gamma : \Gamma \vdash_{\sm^{\n+\cred{1}}} \t~\gamma : \A~\gamma}
{\gamma^\cblu{+} : \Gamma\D \vdash_{\sm^\n} \t\d~\gamma^\cblu{+} : \A\d ~\gamma^\cblu{+}~\pi\t^{\rho_\Gamma}}
\end{mathpar}
We will prove that display is stable under substitution by $\sigma : \Delta \to \Gamma$ in $\MC^{\cred{\oldDelta^{+}_{\n+\cred{1}}}}$, and satisfies the expected formulas relating it to d\'ecalage:
\begin{align}
\bkt{\A^\sigma}\d &\equiv \bkt{\A\d}^{\mathsf{\cred{W}}_{\cred{2}}^{\pi\A^{\rho_\Gamma}}\!\sigma\D}\notag
\\
\bkt{\t^\sigma}\d &\equiv \bkt{\t\d}^{\sigma\D}\notag
\\
\bkt{\gamma : \Gamma,~\a : \A~\gamma}\D &\equiv \bkt{\gamma^\cblu{+} : \Gamma\D,~\a : \pi\A^{\rho_\Gamma}~\gamma^\cblu{+}, ~\a\cblu{'} : \A\d~\gamma^\cblu{+}~\a}\label{eq:trunc-Dce}
\\
\sqbkt{\sigma,~\t}\D &\equiv \sqbkt{\sigma\D,~\pi\t^{\rho_\Delta},~\t\d}. \label{eq:trunc-Dsqbkt}
\end{align}

Finally, we will also define substitutions that give an the actions of morphisms in $\cred{\oldDelta^{+}_{\n+\cred{1}}}$ on matching telescopes and on types:
\begin{mathpar}
\infer{\gamma^\cblu{\blank} : \pi\Gamma \vdash_{\sm^\n} \A~\gamma^\cblu{\blank} \type_{\;\!\ell} \\ \b : \bB^{\angle{\n+\cred{1}},\angle{\m+\cred{1}}}}{\gamma_{\cblu{\n+1}} : \Gamma_{\n+\cred{1}}, ~\cblu{\partial}\a : \A_{\cred{\partial}\bkt{\n+\cred{1}}}~\gamma_{\cblu{\n+1}} \vdash_\dm \nat^{\A}_{\cred{\partial}\b}~ \gamma_{\cblu{\n+1}}~\cblu{\partial}\a : \pi^{\n-\m}\A_{\cred{\partial}\bkt{\m+\cred{1}}}~\bkt{\Gamma^\b~\gamma_{\cblu{\n+1}}}}
\\
\infer{\gamma : \Gamma \vdash_{\sm^{\n+\cred{1}}} \A~\gamma \type_{\;\!\ell} \\ \b : \bB^{\angle{\n+\cred{1}},\angle{\m+\cred{1}}}}
{\text{\(
  \begin{aligned}
    \gamma_{\cblu{\n+1}} : \Gamma_{\n+\cred{1}}, ~&\cblu{\partial}\a : \pi\A_{\cred{\partial}\bkt{\n+\cred{1}}}~\gamma_{\cblu{\n+1}},~\a : \A_{\n+\cred{1}}~\gamma_{\cblu{\n+1}}~\cblu{\partial}\a \\ &\vdash_\dm \nat^{\A}_\b ~\gamma_{\cblu{\n+1}}~\cblu{\partial}\a~\a : \pi^{\n-\m}\A_{\m+\cred{1}}~\bkt{\Gamma^\b~\gamma_{\cblu{\n+1}}}~\big(\nat_{\cred{\partial}\b} ~\gamma_{\cblu{\n+1}}~\cblu{\partial}\a\big)
  \end{aligned}\)}
}
\end{mathpar}
We will show that these compute to the identities (i.e. weakening) when $\b = \id_{\angle{\n+\cred{1}}}$, are functorial such that for $\b_\cblu{1} : \bB^{\angle{\n+\cred{1}},\angle{\m+\cred{1}}}$ and $\b_\cblu{0} : \bB^{\angle{\m+\cred{1}},\angle{\k+\cred{1}}}$:
\begin{align}
  \label{eq:act-functorial-1}
    \nat^{\pi^{\n-\m}\A}_{\cred{\partial}\b_\cblu{0}} ~\bkt{\Gamma^{\b_\cblu{1}} ~\gamma_{\cblu{\n+1}}} ~\big(\nat^{\A}_{\cred{\partial}\b_\cblu{1}}~\gamma_{\cblu{\n+1}} ~\cblu{\partial}\a\big) &\equiv \nat^{\A}_{\cred{\partial}\bkt{\b_\cblu{1}\circ\b_{\cblu{0}}}} ~\gamma_{\cblu{\n+1}}~\cblu{\partial}\a \\
  \label{eq:act-functorial-2}
    \nat^{\pi^{\n-\m}\A}_{\b_\cblu{0}}~\bkt{\Gamma^{\b_\cblu{1}} ~\gamma_{\cblu{\n+1}}} ~\big(\nat^{\pi\A}_{\cred{\partial}\b_\cblu{1}}~\gamma_{\cblu{\n+1}} ~\cblu{\partial}\a\big) ~\big(\nat^{\A}_{\b_\cblu{1}}~\gamma_{\cblu{\n+1}} ~\cblu{\partial}\a~\a\big) &\equiv \nat^{\A}_{\b_\cblu{1}\circ\b_{\cblu{0}}} ~\gamma_{\cblu{\n+1}}~\cblu{\partial}\a~\a
\end{align}
and are also stable under substitution in the sense that for $\sigma : \Delta \to \Gamma$ in $\MC^{\cred{\oldDelta^{+}_{\n+\cred{1}}}}$:
\begin{align}
  \label{eq:act-substable-1}
\nat^{\A^{\pi\sigma}}_{\cred{\partial}\b} &\equiv \bkt{\nat^{\A}_{\cred{\partial}\b}}^{\mathsf{\cred{W}}_{\cred{2}}^{\A_{\cred{\partial}\bkt{\n+\cred{1}}}}\!\sigma_{\n+\cred{1}}}
\\\label{eq:act-substable-2}
\nat^{\A^\sigma}_\b &\equiv \bkt{\nat^{\A}_\b}^{\mathsf{\cred{W}}_{\cred{2}}^{\A_{\n+\cred{1}}}\mathsf{\cred{W}}_{\cred{2}}^{\pi\A_{\cred{\partial}\bkt{\n+\cred{1}}}}\!\sigma_{\n+\cred{1}}}.
\end{align}
Given these, we can then define the functorial structure of the putative object $\bkt{\gamma : \Gamma, ~\a : \A~\gamma}$: a morphism $\b_\cblu{1} : \bB^{\angle{\n+\cred{1}},\angle{\m+\cred{1}}}$ acts on it by:
\begin{equation}
  \label{eq:functorial-ext}
  \bkt{\gamma : \Gamma, ~\a : \A~\gamma}^\b~\gamma_{\cblu{\n+1}} ~\cblu{\partial}\a~\a \equiv \sqbkt{\Gamma^\b~\gamma_{\cblu{\n+1}}, ~\nat^{\pi\A}_{\cred{\partial}\b}~\gamma_{\cblu{\n+1}} ~\cblu{\partial}\a, ~\nat^{\A}_\b~\gamma_{\cblu{\n+1}} ~\cblu{\partial}\a~\a}
\end{equation}
\Cref{eq:act-functorial-1,eq:act-functorial-2} tell us that the assignment $\bkt{\gamma : \Gamma, ~\a : \A~\gamma}^\b$ is functorial, while \cref{eq:act-substable-1,eq:act-substable-2} tell us that the extension $\sqbkt{\sigma,~\t}$ is a morphism of presheaves.

The above is the complete list of constructions and theorems that we need in order to inductively define the type and term presheaves in the models $\sm^\n$ and their context extension function.

\paragraph{The Inductive Cases}
\label{sec:inductive-cases}

Now we give the inductive definitions and proofs of the objects and theorems declared previously.
The model $\sm^{\minustwo}$ is the terminal CwF on the terminal category. For $\sm^\minusone$, we have that:
\begin{mathpar}
\A_{\cred{\partial}\bkt{\minusone}} \equiv \ec_\dm
\and
\t_{\cred{\partial}\bkt{\minusone}} \equiv \esub_\dm
\end{mathpar}
from which the rest of the definitions and theorems evidently follow.

Suppose now that the model $\sm^{\n+\cred{1}}$ has been defined with all of the above structure and properties. We first define matching telescopes and substitutions as follows:
\begin{align*}
  \A_{\cred{\partial}\bkt{\n+\cred{2}}}~\gamma_{\cblu{\n+2}} &\equiv
  \begin{aligned}[t]
    \Big(\cblu{\partial}\a : \bkt{\pi\A^{\rho_{\pi\Gamma}}}_{\cred{\partial}\bkt{\n+\cred{1}}}~\gamma_{\cblu{\n+2}},~
    &\a : \bkt{\A^{\rho_\Gamma}}_{\n+\cred{1}}~\gamma_{\cblu{\n+2}}~\cblu{\partial}\a, \\
    &\cblu{\partial}\a\cblu{'} : \bkt{\A\d}_{\cred{\partial}\bkt{\n+\cred{1}}}~\sqbkt{\gamma_{\cblu{\n+2}},~\cblu{\partial}\a,~\a}\Big)
  \end{aligned}
  \\
\t_{\cred{\partial}\bkt{\n+\cred{2}}}~\gamma_{\cblu{\n+2}} &\equiv \sqbkt{\bkt{\pi\t^{\rho_{\pi\Gamma}}}_{\cred{\partial}\bkt{\n+\cred{1}}},~\bkt{\t^{\rho_\Gamma}}_{\n+\cred{1}}, ~\bkt{\t\d}_{\cred{\partial}\bkt{\n+\cred{1}}}}.
\end{align*}
The stability of these under substitution follows from that of the constituent constructions in the previous dimension; for $\sigma : \Delta \to \Gamma$ in $\MC^{\cred{\oldDelta^{+}_{\n+\cred{2}}}}$:
\begin{align*}
\MoveEqLeft\big(\A^{\pi\sigma}\big)_{\cred{\partial}\bkt{\n+\cred{2}}}~\delta_{\cblu{\n+2}} \\
&\equiv
\begin{aligned}[t]
  \Big(\cblu{\partial}\a : \bkt{\pi\A^{\pi\pi\sigma \circ \rho_{\pi\Delta}}}_{\cred{\partial}\bkt{\n+\cred{1}}}~\delta_{\cblu{\n+2}},~
  &\a : \bkt{\A^{\pi\sigma \circ \rho_\Delta}}_{\n+\cred{1}}~\delta_{\cblu{\n+2}}~\cblu{\partial}\a, \\
  &\cblu{\partial}\a\cblu{'} : \bkt{\bkt{\A^{\pi\sigma}}\d}_{\cred{\partial}\bkt{\n+\cred{1}}}~\sqbkt{\delta_{\cblu{\n+2}},~\cblu{\partial}\a,~\a}\Big)
\end{aligned}
\\
&\equiv
\begin{aligned}[t]
  \Big(\cblu{\partial}\a : \bkt{\pi\A^{\rho_{\pi\Gamma} \circ \pi\sigma\D}}_{\cred{\partial}\bkt{\n+\cred{1}}}~\delta_{\cblu{\n+2}},~
  &\a : \bkt{\A^{\rho_\Gamma \circ \sigma\D }}_{\n+\cred{1}}~\delta_{\cblu{\n+2}}~\cblu{\partial}\a,\\
  &\cblu{\partial}\a\cblu{'} : \Big(\bkt{\A\d}^{\pi\mathsf{\cred{W}}_{\cred{2}} ^{\A^{\rho_{\Gamma}}}\!\sigma\D}\Big)_{\cred{\partial}\bkt{\n+\cred{1}}}~\sqbkt{\delta_{\cblu{\n+2}},~\cblu{\partial}\a,~\a}\Big)
\end{aligned}
\\
&\equiv
\begin{aligned}[t]
  \Big(\cblu{\partial}\a : \bkt{\pi\A^{\rho_{\pi\Gamma}}} _{\cred{\partial}\bkt{\n+\cred{1}}} ~\bkt{\sigma^{\mathsf{\cred{D}}}_{\n+\cred{1}}~\delta_{\cblu{\n+2}}},~
  &\a : \bkt{\A^{\rho_\Gamma }}_{\n+\cred{1}} ~\bkt{\sigma^{\mathsf{\cred{D}}}_{\n+\cred{1}}~\delta_{\cblu{\n+2}}}~\cblu{\partial}\a, \\
  &\cblu{\partial}\a\cblu{'} : \bkt{\A\d}_{\cred{\partial}\bkt{\n+\cred{1}}}~\sqbkt{\bkt{\sigma^{\mathsf{\cred{D}}}_{\n+\cred{1}}~\delta_{\cblu{\n+2}}},~\cblu{\partial}\a,~\a}\Big)
\end{aligned}
\\
&\equiv
\begin{aligned}[t]
  \Big(\cblu{\partial}\a : \bkt{\pi\A^{\rho_{\pi\Gamma}}}_{\cred{\partial}\bkt{\n+\cred{1}}} ~\bkt{\sigma_{\n+\cred{2}}~\delta_{\cblu{\n+2}}},~
  &\a : \bkt{\A^{\rho_\Gamma}}_{\n+\cred{1}} ~\bkt{\sigma_{\n+\cred{2}}~\delta_{\cblu{\n+2}}}~\cblu{\partial}\a, \\
  &\cblu{\partial}\a\cblu{'} : \bkt{\A\d}_{\cred{\partial}\bkt{\n+\cred{1}}}~\sqbkt{\bkt{\sigma_{\n+\cred{2}}~\delta_{\cblu{\n+2}}},~\cblu{\partial}\a,~\a}\Big)
\end{aligned}
\\
&\equiv
\begin{aligned}[t]
  \Big(\cblu{\partial}\a : \bkt{\pi\A^{\rho_{\pi\Gamma}}}_{\cred{\partial}\bkt{\n+\cred{1}}} ~\gamma_{\cblu{\n+2}},~
  &\a : \bkt{\A^{\rho_\Gamma}}_{\n+\cred{1}} ~\gamma_{\cblu{\n+2}}~\cblu{\partial}\a, \\
  &\cblu{\partial}\a\cblu{'} : \bkt{\A\d}_{\cred{\partial}\bkt{\n+\cred{1}}}~\sqbkt{\gamma_{\cblu{\n+2}},~\cblu{\partial}\a,~\a}\Big)^{\sigma_{\n+\cred{2}}}
\end{aligned}
\\
&\equiv \A_{\cred{\partial}\bkt{\n+\cred{2}}} ~\bkt{\sigma_{\n+\cred{2}}~\delta_{\cblu{\n+2}}}.
\end{align*}
For display, we define:
\begin{align*}
\pi\big(\A\d\big) &\equiv \pi\A\d & \big(\A\d\big)_{\n+\cred{1}} &\equiv \A_{\n+\cred{2}} \\
\pi\big(\t\d\big) &\equiv \pi\t\d & \big(\t\d\big)_{\n+\cred{1}} &\equiv \t_{\n+\cred{2}}.
\end{align*}
This definition is well typed because the expected typing judgement for $\bkt{\A\d}_{\n+\cred{1}}$ is:
\begin{align*}
&\gamma_\cblu{\n+2} : \Gamma_{\n+\cred{2}},~\cblu{\partial}\a : \bkt{\pi\pi\A^{\rho_{\pi\Gamma}}}_{\cred{\partial}\bkt{\n+\cred{1}}} ~\gamma_\cblu{\n+2}, ~\a : \bkt{\pi\A^{\rho_{\Gamma}}}_{\n+\cred{1}}~\gamma_\cblu{\n+2} ~\cblu{\partial}\a,\\
&\quad\quad\quad\quad\quad\quad\cblu{\partial}\a\cblu{'} : \bkt{\pi\A\d}_{\cred{\partial}\bkt{\n+\cred{1}}} ~\gamma_\cblu{\n+2}~\cblu{\partial}\a~\a \vdash_\dm \bkt{\A\d}_{\n+\cred{1}}~\sqbkt{\gamma_\cblu{\n+2}, ~\cblu{\partial}\a,~\a}~\cblu{\partial}\a\cblu{'} \type_{\;\!\ell}
\end{align*}
and the context in which $\A_{\n+\cred{2}}$ lives expands to this by the definition of matching telescopes:
\begin{align*}
&\gamma_{\cblu{\n+2}} : \Gamma_{\n+\cred{2}},~\cblu{\partial}\a : \pi\A_{\cred{\partial}\bkt{\n+\cred{2}}}~\gamma_{\cblu{\n+2}} \vdash_\dm \A_{\n+\cred{2}} ~\gamma_{\cblu{\n+2}}~\cblu{\partial}\a  \type_{\;\!\ell}.
\end{align*}
We can now check \eqref{eq:trunc-Dce} at the level of $\n+\cred{1}$ simplices:
\begin{align*}
&\Big(\big(\gamma : \Gamma,~\a : \A~\gamma\big)\D\Big)_{\n+\cred{1}} \\
&\quad \equiv \big(\gamma : \Gamma,~\a : \A~\gamma\big)_{\n+\cred{2}} \\
&\quad \equiv \big(\gamma_{\cblu{\n+2}} : \Gamma_{\n+\cred{2}}, ~\cblu{\partial}\a : \pi\A_{\cred{\partial}\bkt{\n+\cred{2}}}~\gamma_{\cblu{\n+2}}, ~\a : \A_{\n+\cred{2}} ~\gamma_{\cblu{\n+2}} ~\cblu{\partial}\a\big) \\
&\quad \equiv
\begin{aligned}[t]
\big(\gamma_{\cblu{\n+2}} : \Gamma_{\n+\cred{2}}, ~&\cblu{\partial}\a : \bkt{\pi\pi\A^{\rho_{\pi\Gamma}}}_{\cred{\partial}\bkt{\n+\cred{1}}}~\gamma_{\cblu{\n+2}},~
\a : \bkt{\pi\A^{\rho_\Gamma}}_{\n+\cred{1}}~\gamma_{\cblu{\n+2}}~\cblu{\partial}\a, \\
&\cblu{\partial}\a\cblu{'} : \bkt{\pi\A\d}_{\cred{\partial}\bkt{\n+\cred{1}}}~\sqbkt{\gamma_{\cblu{\n+2}},~\cblu{\partial}\a,~\a},~\a\cblu{'} : \bkt{\A\d}_{\n+\cred{1}} ~\sqbkt{\gamma_{\cblu{\n+2}},~\cblu{\partial}\a,~\a} ~\cblu{\partial}\a\cblu{'}\big)
\end{aligned} \\
&\quad \equiv \big(\gamma^\cblu{+} : \Gamma\D,~\a : \pi\A^{\rho_\Gamma}~\gamma^\cblu{+}, ~\a\cblu{'} : \A\d~\gamma^\cblu{+}~\a\big)_{\n+\cred{1}},
\end{align*}
where \eqref{eq:trunc-Dsqbkt} follows similarly. Stability under substitutions follows inductively:
\begin{align*}
\pi\Big(\big(\A^\sigma\big)\d\Big)_{\n+\cred{1}}
&\equiv \big(\pi\A^{\pi\sigma}\big)\d \\
&\equiv \big(\pi\A\d\big)^{\mathsf{\cred{W}}_{\cred{2}}^{\pi\pi\A^{\rho_{\pi\Gamma}}}\!\pi\sigma\D} \\
&\equiv \big(\pi\A\d\big)^{\pi\mathsf{\cred{W}}_{\cred{2}}^{\pi\A^{\rho_{\Gamma}}}\!\sigma\D} \\
&\equiv \pi\Big(\big(\A\d\big)^{\mathsf{\cred{W}}_{\cred{2}}^{\pi\A^{\rho_{\Gamma}}}\!\sigma\D}\Big) \\
\Big(\big(\A^\sigma\big)\d\Big)_{\n+\cred{1}} &\equiv \big(\A^\sigma\big)_{\n+\cred{2}} \\
&\equiv \A_{\n+\cred{2}}^{\mathsf{\cred{W}}_{\cred{2}}^{\pi\A_{\cred{\partial}\bkt{\n+\cred{2}}}} \sigma_{\n+\cred{2}}} \\
&\equiv \A_{\n+\cred{2}}^{\mathsf{\cred{W}}_{\cred{2}}^{(\pi\A\d)_{\cred{\partial}\bkt{\n+\cred{1}}}} \mathsf{\cred{W}}_{\cred{2}}^{(\pi\A^{\rho_{\Gamma}})_{\n+\cred{1}}}\mathsf{\cred{W}}_{\cred{2}}^{(\pi\pi\A^{\rho_{\pi\Gamma}})_{\cred{\partial}\bkt{\n+\cred{1}}}}\sigma _{\n+\cred{2}}} \\
&\equiv \Big(\big(\A\d\big)_{\n+\cred{1}}\Big)^{\mathsf{\cred{W}}_{\cred{2}}^{(\pi\A\d)_{\cred{\partial}\bkt{\n+\cred{1}}}} \bkt{\mathsf{\cred{W}}_{\cred{2}}^{\pi\A^{\rho_{\Gamma}}} \sigma\D}_{\n+\cred{1}}} \\
&\equiv \Big(\big(\A\d\big)^{\mathsf{\cred{W}}_{\cred{2}}^{\pi\A^{\rho_{\Gamma}}} \sigma\D}\Big)_{\n+\cred{1}}.
\end{align*}
Lastly, we define the components of the functorial action on presheaves as follows:
\begin{align*}
  \nat^{\A}_{\cred{\partial}\bkt{\Zero\b}}~\gamma_{\cblu{\n+2}} ~\sqbkt{\cblu{\partial}\a,~\a,~\cblu{\partial}\a\cblu{'}}
  &\equiv \nat^{\pi\A^{\rho_{\pi\Gamma}}}_{\cred{\partial}\b}~\gamma_{\cblu{\n+2}} ~\cblu{\partial}\a \\
  \nat^{\A}_{\Zero\b}~\gamma_{\cblu{\n+2}} ~\sqbkt{\cblu{\partial}\a,~\a,~\cblu{\partial}\a\cblu{'}}~\a\cblu{'}
  &\equiv \nat^{\pi\A^{\rho_{\Gamma}}}_\b~\gamma_{\cblu{\n+2}} ~\cblu{\partial}\a~\a \\
  \nat^{\A}_{\cred{\partial}\bkt{\One\b}}~\gamma_{\cblu{\n+2}} ~\sqbkt{\cblu{\partial}\a,~\a,~\cblu{\partial}\a\cblu{'}}
  &\equiv \\
  \MoveEqLeft[4] \sqbkt{\nat^{\pi\A^{\rho_{\pi\Gamma}}}_{\cred{\partial}\b}~\gamma_{\cblu{\n+2}}~\cblu{\partial}\a, ~\nat^{\A^{\rho_{\Gamma}}}_\b~\gamma_{\cblu{\n+2}}~\cblu{\partial}\a~\a, ~\nat^{\A\d}_{\cred{\partial}\b}~\sqbkt{\gamma_{\cblu{\n+2}}, ~\cblu{\partial}\a,~\a}~\cblu{\partial}\a\cblu{'}} \\
  \nat^{\A}_{\One\b}~\gamma_{\cblu{\n+2}} ~\sqbkt{\cblu{\partial}\a,~\a,~\cblu{\partial}\a\cblu{'}}~\a\cblu{'}
  &\equiv \nat^{\A\d}_\b~\sqbkt{\gamma_{\cblu{\n+2}}, ~\cblu{\partial}\a,~\a}~\cblu{\partial}\a\cblu{'}~\a\cblu{'},
\end{align*}
where the last definition is well typed because $\bkt{\A\d}_{\n+\cred{1}} \equiv \A_{\n+\cred{2}}$.
We check functoriality:
\begin{align*}
\MoveEqLeft\nat^{\pi^{\bkt{\n+\cred{1}}-\m}\A}_{\cred{\partial}\b_\cblu{0}} ~\bkt{\Gamma^{\Zero\b_\cblu{1}} ~\gamma_{\cblu{\n+2}}} ~\big(\nat^{\A}_{\cred{\partial}\bkt{\Zero\b_\cblu{1}}}~\gamma_{\cblu{\n+2}} ~\sqbkt{\cblu{\partial}\a,~\a,~\cblu{\partial}\a\cblu{'}}\big) \\
&\equiv \nat^{\pi^{\n-\m}\pi\A}_{\cred{\partial}\b_\cblu{0}} ~\big(\big(\rho_\Gamma\big)_{\m+\cred{1}}~ \big(\big(\Gamma\D\big)^{\b_\cblu{1}} ~\gamma_{\cblu{\n+2}}\big)\big) ~\big(\nat^{\A}_{\cred{\partial}\bkt{\Zero\b_\cblu{1}}}~\gamma_{\cblu{\n+2}} ~\cblu{\partial}\a\big)\\
&\equiv \nat^{\pi^{\n-\m}\pi\A^{\rho_{\pi\Gamma}}}_{\cred{\partial}\b_\cblu{0}} ~\big(\big(\Gamma\D\big)^{\b_\cblu{1}} ~\gamma_{\cblu{\n+2}}\big) ~\big(\nat^{\pi\A^{\rho_{\pi\Gamma}}}_{\cred{\partial}\b_\cblu{1}}~\gamma_{\cblu{\n+2}} ~\cblu{\partial}\a\big) \\
&\equiv \nat^{\pi\A^{\rho_{\pi\Gamma}}}_{\cred{\partial}\bkt{\b_\cblu{1} \circ \b_\cblu{0}}}~\gamma_{\cblu{\n+2}} ~\cblu{\partial}\a \\
&\equiv \nat^{\A}_{\cred{\partial}\bkt{\Zero\b_\cblu{1} \circ \b_\cblu{0}}}~\gamma_{\cblu{\n+2}} ~\sqbkt{\cblu{\partial}\a,~\a,~\cblu{\partial}\a\cblu{'}}
\end{align*}
and stability under substitutions:
\begin{align*}
\nat^{\A^{\pi\sigma}}_{\cred{\partial}\bkt{\Zero\b}}~\delta_{\cblu{\n+2}} ~\sqbkt{\cblu{\partial}\a,~\a,~\cblu{\partial}\a\cblu{'}}
&\equiv \nat^{\pi\A^{\pi\pi\sigma \circ \rho_{\pi\Delta}}}_{\cred{\partial}\b} ~\delta_{\cblu{\n+2}} ~\cblu{\partial}\a \\
&\equiv \nat^{\pi\A^{\rho_{\pi\Gamma} \circ \pi\sigma\D}}_{\cred{\partial}\b} ~\delta_{\cblu{\n+2}} ~\cblu{\partial}\a \\
&\equiv \nat^{\pi\A^{\rho_{\pi\Gamma}}}_{\cred{\partial}\b} ~\bkt{\sigma^{\mathsf{\cred{D}}}_{\n+\cred{1}}~\delta_{\cblu{\n+2}}} ~\cblu{\partial}\a \\
&\equiv \nat^{\A}_{\cred{\partial}\bkt{\Zero\b}}~\bkt{\sigma_{\n+\cred{2}} ~\delta_{\cblu{\n+2}}} ~\sqbkt{\cblu{\partial}\a,~\a,~\cblu{\partial}\a\cblu{'}}.
\end{align*}
All omitted verifications are similar to the cases presented. This completes the construction of the type and term presheaves and their context extension function, plus display, for the truncated simplicial models $\sm^\n$. \bbox

\subsubsection{Variables} To make the models $\sm^\n$ into CwFs, what is missing from the above construction are the fundamental context projections and variables. In this section we will now define these:
\begin{mathpar}
\infer{\gamma : \Gamma \vdash_{\sm^\n} \A~\gamma \type_{\;\!\ell}}{\ft^{\A}_{\sm^\n} : \bkt{\gamma : \Gamma, ~\a : \A~\gamma} \to \Gamma}
\and
\infer{\gamma : \Gamma \vdash_{\sm^\n} \A~\gamma \type_{\;\!\ell}}{\gamma : \Gamma, ~\a : \A~\gamma \vdash_{\sm^\n} \zv^{\A}_{\sm^\n}~\gamma~\a: \A^\ft~\gamma~\a}.
\end{mathpar}
We now construct variables and parent maps in $\sm^\n$ inductively, with all of the hypothesise \cref{eq:ft-sqbkt,eq:zv-sqbkt,eq:sqbkt-ft-zv} outlined before assumed at all prior levels. This construction will be performed such that the following theorems hold inductively:
\begin{align}
\bkt{\ft^{\A}_{\sm^{\n+\cred{1}}}}\D &\equiv \ft^{\pi\A^{\rho_\Gamma}}_{\sm^\n} \circ \ft^{\A\d}_{\sm^\n} \label{eq:D-ft} \\
\bkt{\zv^{\A}_{\sm^{\n+\cred{1}}}}\d &\equiv \zv^{\A\d}_{\sm^\n}\label{eq:d-zv} \\
\bkt{\zv^{\pi\A}_{\sm^\n}}^{\rho_{\bkt{\Gamma,\;\A}}} &\equiv \bkt{\zv^{\pi\A^{\rho_\Gamma}}_{\sm^\n}}^{\ft^{\A\d}_{\sm^\n}}.\label{eq:pt-zv}
\end{align}
Note that the above equations are well typed by way of the formulas for d\'ecalage given in the fibrant construction above. Now for $\sm^\minusone$, we define:
\begin{align*}
\bkt{\ft^{\A}_{\sm^\minusone}}_\minusone &\equiv \ft^{\A_\minusone}_\dm \\
\bkt{\zv^{\A}_{\sm^\minusone}}_\minusone &\equiv \zv^{\A_\minusone}_\dm.
\end{align*}
Then we inductively define:
\begin{align*}
\bkt{\ft^{\A}_{\sm^{\n+\cred{1}}}}_{\m+\cred{1}} &\equiv \bkt{\ft^{\pi\A}_{\sm^{\n}}}_{\m+\cred{1}} \quad \text{for} \quad \m < \n \\
\bkt{\ft^{\A}_{\sm^{\n+\cred{2}}}}_{\n+\cred{2}} &\equiv  \bkt{\ft^{\pi\A^{\rho_\Gamma}}_{\sm^{\n+\cred{1}}}}_{\n+\cred{1}} \circ \big(\ft^{\A\d}_{\sm^{\n+\cred{1}}}\big)_{\n+\cred{1}} \\
\pi\bkt{\zv^{\A}_{\sm^{\n+\cred{2}}}} &\equiv \zv^{\pi\A}_{\sm^\n} \\
\bkt{\zv^{\A}_{\sm^{\n+\cred{2}}}}_{\n+\cred{2}} &\equiv \big(\zv^{\A\d}_{\sm^{\n+\cred{1}}}\big)_{\n+\cred{1}}.
\end{align*}
This says that the constructions are performed level-wise. From this, theorems \cref{eq:D-ft,eq:d-zv} then follow inductively, since the hypothesised d\'ecalage and display formulas were used to define each successive level.

Are these definitions correct? We gave well typed definitions, but to show that they give a notion of parent maps and zero variables, we have to verify that equations \cref{eq:ft-sqbkt,eq:zv-sqbkt,eq:sqbkt-ft-zv} hold. These verification appear in \cref{sec:appendix:variables}. \bbox

\subsubsection{$\Pi$-Types}

We construct $\Pi$-types inductively, with all of the assumptions of a $\Pi$-type structure outlined before assumed at all prior levels. Now note that we have the following two types in the same context:
\begin{align*}
&\gamma^\cblu{+} : \Gamma\D, ~\f : \bkt{\cred{\Pi}^{\sm^\n}~\pi\A~\pi\B}^{\rho_\Gamma}~\gamma^\cblu{+} \vdash_{\sm^\n} \big(\cred{\Pi}^{\sm^{\n+\cred{1}}}~\A~\B\big)\d~\gamma^\cblu{+}~\f \type_{\;\!\ell} \\
&\gamma^\cblu{+} : \Gamma\D, ~\f : \bkt{\cred{\Pi}^{\sm^\n}~\pi\A~\pi\B}^{\rho_\Gamma}~\gamma^\cblu{+} \vdash_{\sm^\n} \\
&\quad\quad\quad\quad \big(\a : \pi\A^{\rho_\Gamma}~\gamma^\cblu{+}\big) \bkt{\a\cblu{'} : \A\d~\gamma^\cblu{+}~\a} \to \B\d~\sqbkt{\gamma^\cblu{+}, ~\a, ~\a\cblu{'}}~\bkt{\eval~\sqbkt{\gamma^\cblu{+}, ~\f, ~\a, ~\a\cblu{'}}~\f~\a} \type_{\;\!\ell}.
\end{align*}
We will prove inductively along with our definition that these two types are equal.
In point-free notation, this means we will have:
\begin{align}
\big(\cred{\Pi}^{\sm^{\n+\cred{1}}}~\A~\B\big)\d &\equiv \cred{\Pi}^{\sm^\n} \bkt{\pi\A^{\rho_\Gamma}}^\ft~ \cred{\Pi}^{\sm^\n} \bkt{\A\d}^{\mathsf{\cred{W}}_{\cred{2}}^{\pi\A^{\rho_\Gamma}}\ft} ~\bkt{\B\d}^{\sqbkt{\mathsf{\cred{W}}_{\cred{2}}^{\A\d} \mathsf{\cred{W}}_{\cred{2}}^{\pi\A^{\rho_\Gamma}}\ft,\;\eval\;\zv^{\ft\circ\ft}\;\zv^\ft}} \label{eq:pi-d} \\
\big(\lambda^{\sm^{\n+\cred{1}}}~\t\big)\d &\equiv \lambda^{\sm^\n}~\bkt{\lambda^{\sm^\n}~\t\d} \label{eq:lambda-d} \\
\big(\eval^{\sm^{\n+\cred{1}}}~\f~\s\big)\d &\equiv \eval^{\sm^\n} ~\bkt{\eval^{\sm^\n}~\f\d~\pi\s^{\rho_\Gamma}}~\s\d. \label{eq:eval-d}
\end{align}
Now to start on the induction, for $\sm^\minusone$ we define:
\begin{align*}
\big(\cred{\Pi}^{\sm^\minusone}~\A~\B\big)_\minusone &\equiv \cred{\Pi}^\dm~\A_\minusone~\B_\minusone \\
\big(\lambda^{\sm^\minusone}~\t\big)_\minusone &\equiv \lambda^\dm~\t_\minusone \\
\big(\eval^{\sm^\minusone}~\f~\s\big)_\minusone &\equiv \eval^\dm~\f_\minusone~\s_\minusone.
\end{align*}
Then we inductively define:
\begin{align*}
\pi\big(\cred{\Pi}^{\sm^{\n+\cred{2}}}~\A~\B\big) &\equiv \cred{\Pi}^{\sm^{\n+\cred{1}}}~\pi\A~\pi\B \\
\big(\cred{\Pi}^{\sm^{\n+\cred{2}}}~\A~\B\big)_{\n+\cred{2}} &\equiv \Big(\cred{\Pi}^{\sm^{\n+\cred{1}}} \bkt{\pi\A^{\rho_\Gamma}}^\ft~ \cred{\Pi}^{\sm^{\n+\cred{1}}} \bkt{\A\d}^{\mathsf{\cred{W}}_{\cred{2}}^{\pi\A^{\rho_\Gamma}}\ft} ~\bkt{\B\d}^{\sqbkt{\mathsf{\cred{W}}_{\cred{2}}^{\A\d} \mathsf{\cred{W}}_{\cred{2}}^{\pi\A^{\rho_\Gamma}}\ft,\;\eval\;\zv^{\ft\circ\ft}\;\zv^\ft}}\Big)_{\n+\cred{1}}  \\
\pi\big(\lambda^{\sm^{\n+\cred{2}}}~\t\big) &\equiv \lambda^{\sm^{\n+\cred{1}}}~\pi\t \\
\big(\lambda^{\sm^{\n+\cred{2}}}~\t\big)_{\n+\cred{2}} &\equiv \Big(\lambda^{\sm^{\n+\cred{1}}}~\bkt{\lambda^{\sm^{\n+\cred{1}}}~\t\d}\Big)_{\n+\cred{1}} \\
\pi\big(\eval^{\sm^{\n+\cred{2}}}~\f~\s\big) &\equiv \eval^{\sm^{\n+\cred{1}}}~\pi\f~\pi\s \\
\big(\eval^{\sm^{\n+\cred{2}}}~\f~\s\big)_{\n+\cred{2}} &\equiv \Big(\eval^{\sm^{\n+\cred{1}}} ~\bkt{\eval^{\sm^{\n+\cred{1}}}~\f\d~\pi\s^{\rho_\Gamma}}~\s\d\Big)_{\n+\cred{1}}.
\end{align*}
As before, this says that the constructions are performed level-wise. From this, theorems \cref{eq:pi-d,eq:lambda-d,eq:eval-d} then follow inductively, since the hypothesised display formulas were used to define each successive level. The correctness of these definitions will follow from verifying the $\beta$ and $\eta$ laws in \cref{sec:appendix:pi-types}. \bbox

\subsubsection{Universes}

The universes of the discrete model are denoted $\Disc_{\;\!\ell}$. We construct universes in $\sm^\n$ inductively, with all of the assumptions of a $\mathcal{U}$-type structure outlined before assumed at all prior levels. We will inductively have that:
\begin{align}
\big(\Type^{\sm^{\n+\cred{1}}}_{\;\!\ell}\big)\d &\equiv \cred{\Pi}^{\sm^\n}~\bkt{\El~\zv}~\Type^{\sm^\n}_{\;\!\ell} \label{eq:type-d} \\
\big(\Code^{\sm^{\n+\cred{1}}}~\A\big)\d &\equiv \lambda^{\sm^\n}~\bkt{\Code^{\sm^\n}~\A\d} \label{eq:code-d} \\
\big(\El^{\sm^{\n+\cred{1}}}~\A\big)\d &\equiv \El^{\sm^\n}~\big(\eval^{\sm^\n} \bkt{\A\d}^\ft~\zv\big). \label{eq:el-d}
\end{align}
For $\sm^\minusone$, we define:
\begin{align*}
\big(\Type^{\sm^\minusone}_{\;\!\ell}\big)_\minusone &\equiv \Disc_{\;\!\ell} \\
\big(\Code^{\sm^\minusone}~\A\big)_\minusone &\equiv \Code^\dm~\A_\minusone \\
\big(\El^{\sm^\minusone}~\A\big)_\minusone &\equiv \El^\dm~\A_\minusone.
\end{align*}
Then we inductively define:
\begin{align*}
\pi\big(\Type^{\sm^{\n+\cred{2}}}_{\;\!\ell}\big) &\equiv \Type^{\sm^{\n+\cred{1}}}_{\;\!\ell} \\
\big(\Type^{\sm^{\n+\cred{2}}}_{\;\!\ell}\big)_{\n+\cred{2}} &\equiv \Big(\cred{\Pi}^{\sm^{\n+\cred{1}}}~\bkt{\El~\zv}~\Type^{\sm^{\n+\cred{1}}}_{\;\!\ell}\Big)_{\n+\cred{1}}  \\
\pi\big(\Code^{\sm^{\n+\cred{2}}}~\A\big) &\equiv \Code^{\sm^{\n+\cred{1}}}~\pi\A\\
\big(\Code^{\sm^{\n+\cred{2}}}~\A\big)_{\n+\cred{2}} &\equiv \Big(\lambda^{\sm^{\n+\cred{1}}}~\big(\Code^{\sm^{\n+\cred{1}}}~\A\d\big)\Big)_{\n+\cred{1}} \\
\pi\big(\El^{\sm^{\n+\cred{2}}}~\A\big) &\equiv \El^{\sm^{\n+\cred{1}}}~\pi\A\\
\big(\El^{\sm^{\n+\cred{2}}}~\A\big)_{\n+\cred{2}} &\equiv \Big(\El^{\sm^{\n+\cred{1}}}~\big(\eval^{\sm^{\n+\cred{1}}}\bkt{\A\d}^\ft~\zv \big)\Big)_{\n+\cred{1}}.
\end{align*}
Again, this says that the constructions are performed level-wise. From this, theorems \cref{eq:type-d,eq:code-d,eq:el-d} then follow inductively, since the hypothesised display formulas were used to define each successive level. The correctness of these definitions will follow from verifying that $\Code$ and $\El$ are mutual inverses in \cref{sec:appendix:universes}. \bbox

\subsubsection{$\oldomega$-Limits}

If $\oldomega$-limits are defined in $\sm^\n$, then given an infinite telescope $\gamma : \Gamma \vdash_{\sm^{\n+\cred{1}}} \widetilde{\Upsilon}~\gamma \inftel_{\;\!\ell}$ or infinite partial substitution $\gamma : \Gamma \vdash_{\sm^{\n+\cred{1}}} \widetilde{\upsilon}~\gamma : \widetilde{\Upsilon}~\gamma$ in $\sm^\n$, we can meaningfully give a type declaration of its display through use of limits:
\begin{mathpar}
\infer{\gamma : \Gamma \vdash_{\sm^{\n+\cred{1}}} \widetilde{\Upsilon}~\gamma \inftel_{\;\!\ell}}{\gamma^\cblu{+} : \Gamma\D, ~\cblu{\mathfrak{u}} : \lim^{\sm^\n}~\pi\widetilde{\Upsilon}^{\rho_\Gamma}~\gamma^\cblu{+} \vdash_{\sm^\n} \cblu{\widetilde{\Upsilon}}\d~\gamma^\cblu{+}~\cblu{\mathfrak{u}} \inftel_{\;\!\ell}}
\and
\infer{\gamma : \Gamma \vdash_{\sm^{\n+\cred{1}}} \widetilde{\upsilon}~\gamma : \widetilde{\Upsilon}~\gamma \inftel_{\;\!\ell}}{\gamma^\cblu{+} : \Gamma\D \vdash_{\sm^\n} \widetilde{\upsilon}\d~\gamma^\cblu{+} : \widetilde{\Upsilon}\d~\gamma^\cblu{+}~\bkt{\lim^{\sm^\n}~\pi\widetilde{\upsilon}^{\rho_\Gamma}}}
\end{mathpar}
We then define these by:
\begin{align*}
\big(\cblu{\widetilde{\Upsilon}}\d~\gamma^\cblu{+}~\cblu{\mathfrak{u}}\big)^{\cred{\partial}\m} &\equiv \big(\cblu{\widetilde{\Upsilon}}^{\cred{\partial}\m}\big)\d~\gamma^\cblu{+}~\bkt{\res^{\cred{\partial}\m}_{\sm^\n}~\gamma^\cblu{+}~\cblu{\mathfrak{u}}} \\
\big(\cblu{\widetilde{\Upsilon}}\d~\gamma^\cblu{+}~\cblu{\mathfrak{u}}\big)^\m &\equiv \big(\cblu{\widetilde{\Upsilon}}^\m\big)\d~\gamma^\cblu{+}~\bkt{\res^{\m}_{\sm^\n}~\gamma^\cblu{+}~\cblu{\mathfrak{u}}} \\
\big(\cblu{\widetilde{\upsilon}}\d\big)^{\cred{\partial}\m} &\equiv \big(\cblu{\widetilde{\upsilon}}^{\cred{\partial}\m}\big)\d\\
\big(\cblu{\widetilde{\upsilon}}\d\big)^\m &\equiv \big(\cblu{\widetilde{\upsilon}}^\m\big)\d.
\end{align*}
The third declaration, for example, is well typed because its expected type is:
\begin{align*}
&\big(\widetilde{\Upsilon}\d~\gamma^\cblu{+}~\bkt{\lim^{\sm^\n}~\pi\widetilde{\upsilon}^{\rho_\Gamma}}\big)^{\cred{\partial}\m} \\
&\quad \equiv \big(\cblu{\widetilde{\Upsilon}}^{\cred{\partial}\m}\big)\d~\gamma^\cblu{+}~\bkt{\res^{\cred{\partial}\m}_{\sm^\n}~\gamma^\cblu{+}~\bkt{\lim^{\sm^\n}~\pi\widetilde{\upsilon}^{\rho_\Gamma}}} \\
&\quad \equiv \big(\cblu{\widetilde{\Upsilon}}^{\cred{\partial}\m}\big)\d~\gamma^\cblu{+}~\big(\pi\big(\widetilde{\upsilon}^{\cred{\partial}\m}\big)^{\rho_\Gamma}\big).
\end{align*}
We now construct $\oldomega$-limits in $\sm^\n$ inductively, with all of the assumptions of a $\oldomega$-structure outlined before assumed at all prior levels. This construction will be performed such that the following theorems hold inductively:
\begin{align}
\big(\lim^{\sm^{\n+\cred{1}}}~\cblu{\widetilde{\Upsilon}}\big)\d &\equiv \lim^{\sm^\n}~\cblu{\widetilde{\Upsilon}}\d \label{eq:lim-type-d} \\
\big(\lim^{\sm^{\n+\cred{1}}}~\cblu{\widetilde{\upsilon}}\big)\d &\equiv \lim^{\sm^\n}~\cblu{\widetilde{\upsilon}}\d \label{eq:lim-term-d} \\
\big(\res^{\cred{\partial}\m}_{\sm^{\n+\cred{1}}}~\cblu{\mathfrak{u}}\big)\d &\equiv \res^{\cred{\partial}\m}_{\sm^\n}~\cblu{\mathfrak{u}}\d \label{eq:res-type-d}\\
\big(\res^{\m}_{\sm^{\n+\cred{1}}}~\cblu{\mathfrak{u}}\big)\d &\equiv \res^{\m}_{\sm^\n}~\cblu{\mathfrak{u}}\d. \label{eq:res-term-d}
\end{align}
For $\sm^\minusone$, we define:
\begin{align*}
\big(\lim^{\sm^\minusone}~\cblu{\widetilde{\Upsilon}}\big)_\minusone &\equiv \lim^\dm~\cblu{\widetilde{\Upsilon}}_\minusone \\
\big(\lim^{\sm^\minusone}~\cblu{\widetilde{\upsilon}}\big)_\minusone &\equiv \lim^\dm~\cblu{\widetilde{\upsilon}}_\minusone \\
\big(\res^{\cred{\partial}\m}_{\sm^\minusone}~\cblu{\mathfrak{u}}\big)_\minusone &\equiv \res^{\cred{\partial}\m}_\dm~\cblu{\mathfrak{u}}_\minusone \\
\big(\res^\m_{\sm^\minusone}~\cblu{\mathfrak{u}}\big)_\minusone &\equiv \res^{\m}_\dm~\cblu{\mathfrak{u}}_\minusone.
\end{align*}
We then inductively define:
\begin{align*}
\pi\big(\lim^{\sm^{\n+\cred{2}}}~\cblu{\widetilde{\Upsilon}}\big) &\equiv \lim^{\sm^{\n+\cred{1}}}~\pi\cblu{\widetilde{\Upsilon}}\\
\big(\lim^{\sm^{\n+\cred{2}}}~\cblu{\widetilde{\Upsilon}}\big)_{\n+\cred{2}} &\equiv \big(\lim^{\sm^\n}~\cblu{\widetilde{\Upsilon}}\d\big)_{\n+\cred{1}} \\
\pi\big(\lim^{\sm^{\n+\cred{2}}}~\cblu{\widetilde{\upsilon}}\big) &\equiv \lim^{\sm^{\n+\cred{1}}}~\pi\cblu{\widetilde{\upsilon}}\\
\big(\lim^{\sm^{\n+\cred{2}}}~\cblu{\widetilde{\upsilon}}\big)_{\n+\cred{2}} &\equiv \big(\lim^{\sm^\n}~\cblu{\widetilde{\upsilon}}\d\big)_{\n+\cred{1}} \\
\pi\big(\res^{\cred{\partial}\m}_{\sm^{\n+\cred{2}}}~\cblu{\mathfrak{u}}\big) &\equiv \res^{\cred{\partial}\m}_{\sm^{\n+\cred{1}}}~\pi\cblu{\mathfrak{u}} \\
\big(\res^{\cred{\partial}\m}_{\sm^{\n+\cred{2}}}~\cblu{\mathfrak{u}}\big)_{\n+\cred{2}} &\equiv \big(\res^{\cred{\partial}\m}_{\sm^\n}~\cblu{\mathfrak{u}}\d\big)_{\n+\cred{1}} \\
\pi\big(\res^\m_{\sm^{\n+\cred{2}}}~\cblu{\mathfrak{u}}\big) &\equiv \res^{\m}_{\sm^{\n+\cred{1}}}~\pi\cblu{\mathfrak{u}} \\
\big(\res^{\m}_{\sm^{\n+\cred{2}}}~\cblu{\mathfrak{u}}\big)_{\n+\cred{2}} &\equiv \big(\res^{\m}_{\sm^\n}~\cblu{\mathfrak{u}}\d\big)_{\n+\cred{1}}.
\end{align*}
As always, this says that the constructions are performed level-wise. From this, theorems \cref{eq:lim-type-d,eq:lim-term-d,eq:res-type-d,eq:res-term-d} then follow inductively, since the hypothesised display formulas were used to define each successive level. The correctness of these definitions will follow from verifying laws in \cref{sec:appendix:omega-limits}. \bbox

\subsubsection{The Simplicial Model} Having constructed the truncated simplicial models $\sm^\n$, we obtain the \emph{simplicial model} fairly directly by taking a limit. In order to state this, we first define a \emph{tail-cutting truncation functor} and extend \emph{d\'ecalage} to an endofunctor:
\begin{align*}
&\pi_\n : \MC^{\cred{\oldDelta^{+}}} \to \MC^{\cred{\oldDelta^{+}_\n}}
& &\bkt{\blank}\D : \MC^{\cred{\oldDelta^{+}}} \to \MC^{\cred{\oldDelta^{+}}} \\
&\bkt{\pi_\n\Gamma}_{\m+\cred{1}}  \equiv \Gamma_{\m+\cred{1}}  & &\bkt{\Gamma\D}_{\m+\cred{1}} \equiv \Gamma_{\m+\cred{2}} \\
&\bkt{\pi_\n\Gamma}^\b \equiv \Gamma^\b & &\bkt{\Gamma\D}^\b \equiv \Gamma^{\One\b} \\
&\bkt{\pi_\n\sigma}_{\m+\cred{1}}  \equiv \sigma_{\m+\cred{1}}  & &\bkt{\sigma\D}_{\m+\cred{1}} \equiv \sigma_{\m+\cred{2}}
\end{align*}
Since d\'ecalage is now an endofunctor, $\rho$ no longer involves truncation:
\begin{align*}
&\rho : \bkt{\blank}\D \Rightarrow \id_{\MC^{\cred{\oldDelta^{+}}}} \\
&\bkt{\rho_\Gamma}_{\m+\cred{1}} \equiv \Gamma^{\Zero\id_{\angle{\m+\cred{1}}}}
\end{align*}
Now we define the types and terms in $\sm$ to be compatible towers of types and terms in the truncated models $\sm^\n$.
In syntax this can be expressed by the following infinitary bidirectional rules:
\begin{mathpar}
\mprset{fraction={===}}
\infer{\bkt{\gamma : \pi_\n \Gamma \vdash_{\sm_\n} \pi_\n\A ~\gamma \type_{\;\!\ell}}_{\n \geq \minustwo} \\
  (\pi\bkt{\pi_{\n+\cred{1}}\A} \equiv \pi_\n\A)_{\n \geq \minustwo}}{\gamma : \Gamma \vdash_{\sm} \A~\gamma \type_{\;\!\ell}}
\\
\infer{\bkt{\gamma : \pi_\n \Gamma \vdash_{\sm_\n} \pi_\n\t~\gamma :  \pi_\n\A~\gamma}_{\n \geq\minustwo} \\ (\pi\bkt{\pi_{\n+\cred{1}}\t} \equiv \pi_\n\t)_{\n \geq\minustwo}}{\gamma : \Gamma \vdash_{\sm} \t~\gamma : \A~\gamma}
\end{mathpar}
We also define:
\begin{mathpar}
\infer{\gamma : \Gamma \vdash_{\sm} \A~\gamma \type_{\;\!\ell}}{\A_{\n+\cred{1}} \equiv \bkt{\pi_{\n+\cred{1}}\A}_{\n+\cred{1}}}
\and
\infer{\gamma : \Gamma \vdash_{\sm} \t~\gamma : \A~\gamma}{\t_{\n+\cred{1}} \equiv \bkt{\pi_{\n+\cred{1}}\t}_{\n+\cred{1}}}.
\end{mathpar}
The the above introduction rules equivalently then say that $\A$ and $\t$ in $\sm$ are defined by the data of each of the simplex levels $\A_{\n+\cred{1}}$ and $\t_{\n+\cred{1}}$. At this point, every single construction in $\sm^\n$ performed above extends to $\sm$ levelwise, since it is preserved strictly by all the finite truncation functors. In lieu of listing all of them, we will only give the case of \emph{display}, which is slightly modified in the absence of truncation:
\begin{mathpar}
\infer{\gamma : \Gamma \vdash_{\sm} \A~\gamma \type_{\;\!\ell}}
{\gamma^\cblu{+} : \Gamma\D, ~\a : \A^{\rho_\Gamma} ~\gamma^\cblu{+} \vdash_{\sm} \A\d ~\gamma^\cblu{+}~\a \type_{\;\!\ell}}
\and
\infer{\gamma : \Gamma \vdash_{\sm} \t~\gamma : \A~\gamma}
{\gamma^\cblu{+} : \Gamma\D \vdash_{\sm} \t\d~\gamma^\cblu{+} : \A\d ~\gamma^\cblu{+}~\t^{\rho_\Gamma}}.
\end{mathpar}
The computation rules for display on variables, $\Pi$-types, universes, and $\oldomega$-limits similarly hold in $\sm$ when modified to exclude $\pi$. \bbox

\subsection{Modalities}
\label{sec:modal-semantics}

We now relate the discrete ($\dm$) and simplicial ($\sm$) models by way of modalities, and introduce modal variants of the structural operations of a CwF.

The interesting facet of our approach is our treatment of $\bkt{\Gamma, ~\lock_\dia}$ and $\bkt{\gamma : \Gamma, ~\a :^\tri \A~\gamma}$. These examples concern the passage from $\dm$ to $\sm$. Both examples construct a context in $\sm$, but where (part of) the starting data is discrete --- $\Gamma$ in the first example and $\A$ in the second example. One naive approach to this construction would be to convert the discrete data to simplicial data (fibrantly so in the second case) --- in the first example we would set values of the presheaf at $\m+\cred{2}$ to be zeros, and in the second example we would set the simplex types at levels $\m+\cred{2}$ to be units. However, this would require us to assume that the starting CwF has, respectively, an initial object and unit types. The approach that we take avoids these assumptions, and also ensures that all computation laws have definitionally strict interpretations.

\subsubsection{Pieces of the triangle modality}
\label{sec:tri}

We begin by dealing with $\tri$.
The modality $\tri$ is supposed to construct a constant (augmented \mbox{semi-)}simplicial diagram, while its left adjoint $\lock_\tri$ picks out the object of $(\minusone)$-simplices.
Both of these operations are determined levelwise by their behavior on truncated diagrams, which is where most of the work is.
Recalling that we will not be modeling the modality $\tri$ on \emph{types} itself, since it would require assuming the existence of unit types in $\dm$, in this section we describe the other aspects of $\tri$ in the models $\sm^{\n+\cred{1}}$ and how they fit together on $\sm$.

We begin by defining a functor $\bkt{\blank,~\lock_{\tri_{\n+\cred{1}}}} : \MC^{\cred{\oldDelta}^{\cred{+}}_{\n+\cred{1}}}\to \MC$ via:
\begin{align*}
\bkt{\gamma : \Gamma,~\lock_{\tri_{\n+\cred{1}}}} &\equiv \Gamma_\minusone \\
\sqbkt{\sigma,~\lock_{\tri_{\n+\cred{1}}}} &\equiv \sigma_\minusone
\end{align*}
Then we construct modal extension for $\tri_{\n+\cred{1}}$ in $\sm^{\n+\cred{1}}$:
\begin{mathpar}
\infer{\Gamma \ob_{\sm^{\n+\cred{1}}} \\ \gamma  : \Gamma, ~\lock_{\tri_{\n+\cred{1}}} \vdash_\dm \A~\gamma \type_{\;\!\ell}}{\bkt{\gamma : \Gamma,~\a :^{\tri_{\n+\cred{1}}} \A~\gamma} \ob_{\sm^{\n+\cred{1}}}}
\and
\infer{\sigma : \Delta \to_{\sm^{\n+\cred{1}}} \Gamma \\ \gamma  : \Gamma, ~\lock_{\tri_{\n+\cred{1}}} \vdash_\dm \t~\gamma : \A~\gamma}{\sqbkt{\sigma,~\t}_{\tri_{\n+\cred{1}}} : \Delta \to_{\sm^{\n+\cred{1}}} \bkt{\gamma : \Gamma,~\a :^{\tri_{\n+\cred{1}}} \A~\gamma}}
\end{mathpar}
These are defined such that:
\begin{align*}
\bkt{\gamma : \Gamma,~\a :^{\tri_{\n+\cred{2}}} \A~\gamma}\D &\equiv \bkt{\gamma^\cblu{+} : \Gamma\D,~\a :^{\tri_{\n+\cred{1}}} \A^{[\rho_\Gamma,\;\lock_{\tri_{\n+\cred{1}}}]}~\gamma^\cblu{+}} \\
\bkt{\sqbkt{\sigma,~\t}_{\tri_{\n+\cred{2}}}}\D &\equiv \sqbkt{\sigma\D, ~\t^{[\rho_\Gamma,\;\lock_{\tri_{\n+\cred{1}}}]}}_{\tri_{\n+\cred{1}}}
\end{align*}
For dimension $\minusone$, we define:
\begin{align*}
\big(\gamma : \Gamma,~\a :^{\tri_\minusone} \A~\gamma\big)_{\minusone} &\equiv \big(\gamma_\bminusone : \Gamma_\minusone, ~\a : \A~\gamma_\bminusone\big) \\
\big(\sqbkt{\sigma,~\t}_{\tri_\minusone}\big)_{\minusone} &\equiv \sqbkt{\sigma_{\minusone}, ~\t}
\end{align*}
Then we inductively define:
\begin{align*}
\big(\gamma : \Gamma,~\a :^{\tri_{\n+\cred{2}}} \A~\gamma\big)_{\m+\cred{1}} &\equiv \big(\gamma^\cblu{\blank} : \pi\Gamma,~\a :^{\tri_{\n+\cred{1}}} \A~\gamma^\cblu{\blank}\big)_{\m+\cred{1}} \quad\text{for}\quad \m \leq \n \\
\big(\gamma : \Gamma,~\a :^{\tri_{\n+\cred{2}}} \A~\gamma\big)_{\n+\cred{2}} &\equiv \big(\gamma^\cblu{+} : \Gamma\D,~\a :^{\tri_{\n+\cred{1}}} \A^{[\rho_\Gamma,\;\lock_{\tri_{\n+\cred{1}}}]}~\gamma^\cblu{+}\big)_{\n+\cred{1}} \\
\big(\sqbkt{\sigma,~\t}_{\tri_{\n+\cred{2}}}\big)_{\m+\cred{1}} &\equiv \big(\sqbkt{\pi\sigma,~\t}_{\tri_{\n+\cred{1}}}\big)_{\m+\cred{1}} \quad\text{for}\quad \m \leq \n \\
\big(\sqbkt{\sigma,~\t}_{\tri_{\n+\cred{2}}}\big)_{\n+\cred{2}} &\equiv \big(\sqbkt{\sigma\D, ~\t^{[\rho_\Gamma,\;\lock_{\tri_{\n+\cred{1}}}]}}_{\tri_{\n+\cred{1}}}\big)_{\n+\cred{1}}
\end{align*}
We next have fundamental context projections and zero variables:
\begin{mathpar}
\infer{\gamma  : \Gamma, ~\lock_{\tri_{\n+\cred{1}}} \vdash_\dm \A~\gamma \type_{\;\!\ell}}{\ft^{\A}_{\tri_{\n+\cred{1}}} : \bkt{\gamma : \Gamma,~\a :^{\tri_{\n+\cred{1}}} \A} \to_{\sm^{\n+\cred{1}}} \Gamma}
\and
\infer{\gamma  : \Gamma, ~\lock_{\tri_{\n+\cred{1}}} \vdash_\dm \A~\gamma \type_{\;\!\ell}}{\gamma  : \Gamma, ~\a:^{\tri_{\n+\cred{1}}} \A~\gamma, ~\lock_{\tri_{\n+\cred{1}}} \vdash_\dm \zv^{\A}_{\tri_{\n+\cred{1}}}~\gamma~\a : \A^{[\ft,\;\lock_{\tri_{\n+\cred{1}}}]}~\gamma ~\a}
\end{mathpar}
These are defined such that:
\[\big(\ft^{\A}_{\tri_{\n+\cred{2}}}\big)\D \equiv \ft^{\A^{[\rho_\Gamma,\;\lock_{\tri_{\n+\cred{1}}}]}}_{\tri_{\n+\cred{1}}}\]
For dimension $\minusone$, we define:
\begin{align*}
\big(\ft^{\A}_{\tri_\minusone}\big)_\minusone &\equiv \ft^{\A}_\dm \\
\zv^{\A}_{\tri_\minusone} &\equiv \zv^{\A}_\dm
\end{align*}
Then we inductively define:
\begin{align*}
\pi\big(\ft^{\A}_{\tri_{\n+\cred{2}}}\big) &\equiv \ft^{\A}_{\tri_{\n+\cred{1}}} \\
\big(\ft^{\A}_{\tri_{\n+\cred{2}}}\big)_{\n+\cred{2}} &\equiv \big(\ft^{\A^{[\rho_\Gamma,\;\lock_{\tri_{\n+\cred{1}}}]}}_{\tri_{\n+\cred{1}}}\big)_{\n+\cred{1}} \\
\zv^{\A}_{\tri_{\n+\cred{2}}} &\equiv \zv^{\A}_{\tri_{\n+\cred{1}}}
\end{align*}
Finally, we construct modal $\Pi$-types:
\begin{mathpar}
\infer{\gamma : \Gamma, ~\lock_\tri \vdash_\dm \A~\gamma \type_{\;\!\ell_\cblu{0}} \\ \gamma : \Gamma,~\a :^{\tri_{\n+\cred{1}}} \A~\gamma \vdash_{\sm^{\n+\cred{1}}} \B~\gamma~\a \type_{\;\!\ell_\cblu{1}}}{\gamma : \Gamma \vdash_{\sm^{\n+\cred{1}}} \cred{\Pi}^{\sm^{\n+\cred{1}}}_{\tri}~\A~\B~\gamma \type_{\;\!\ell_\cblu{0}\;\join\;\ell_\cblu{1}}}
\and
\infer{\gamma : \Gamma, ~\a :^{\tri_{\n+\cred{1}}} \A~\gamma \vdash_{\sm^\n} \t~\gamma~\a : \B~\gamma~\a}{\gamma : \Gamma \vdash_\sm \lambda^{\sm^{\n+\cred{1}}}_\tri~\t~\gamma : \cred{\Pi}^{\sm^{\n+\cred{1}}}_\tri~\A~\B~\gamma}
\and
\infer{\gamma : \Gamma \vdash_{\sm^\n} \f~\gamma :  \cred{\Pi}^{\sm^\n}_\tri~\A~\B~\gamma \\ \gamma : \Gamma, ~\lock_\tri \vdash_\dm \s~\gamma : \A~\gamma}{\gamma : \Gamma \vdash_\sm \eval^{\sm^\n}_\tri~\f~\s~\gamma : \B^{\sqbkt{\id_\Gamma,\;\s}_\tri}~\gamma}
\end{mathpar}
This construction follows the pattern of the non-modal truncated case and is performed level-wise. We will inductively assert the following formulas for display:
\begin{align*}
\big(\cred{\Pi}^{\sm^{\n+\cred{2}}}_\tri~\A~\B\big)\d &\equiv \cred{\Pi}^{\sm^{\n+\cred{1}}}_\tri \big(\A^{[\rho_\Gamma,\;\lock_\tri]}\big)^\ft~\bkt{\B\d}^{\sqbkt{\mathsf{\cred{W}}_{\cred{2}}^{\A^{[\rho_\Gamma,\;\lock_\tri]}}\ft,\;\eval^{\sm^{\n+\cred{1}}}_\tri\;\zv^{\ft\circ\ft}\;\zv^{\ft}_{\tri}}} \\
\big(\lambda^{\sm^{\n+\cred{2}}}_\tri~\t\big)\d &\equiv \lambda^{\sm^{\n+\cred{1}}}_\tri~\t\d \\
\big(\eval^{\sm^{\n+\cred{2}}}_\tri~\f~\s\big)\d &\equiv \eval^{\sm^{\n+\cred{1}}}_\tri~\f\d~\s.
\end{align*}
In dimension $\minusone$ we set:
\begin{align*}
\big(\cred{\Pi}^{\sm^\minusone}_\tri~\A~\B\big)_\minusone &\equiv \cred{\Pi}^\dm~\A~\B_\minusone \\
\big(\lambda^{\sm^\minusone}_\tri~\t\big)_\minusone &\equiv \lambda^\dm~\t_\minusone \\
\big(\eval^{\sm^\minusone}_\tri~\f~\s\big)_\minusone &\equiv \eval^\dm~\f_\minusone~\s.
\end{align*}
Then we inductively define:
\begin{align*}
\pi\big(\cred{\Pi}^{\sm^{\n+\cred{2}}}_\tri~\A~\B\big) &\equiv \cred{\Pi}^{\sm^{\n+\cred{1}}}_\tri~\A~\pi\B \\
\big(\cred{\Pi}^{\sm^{\n+\cred{2}}}_\tri~\A~\B\big)_{\n+\cred{2}} &\equiv \Big(\cred{\Pi}^{\sm^{\n+\cred{1}}}_\tri \big(\A^{[\rho_\Gamma,\;\lock_\tri]}\big)^\ft~\bkt{\B\d}^{\sqbkt{\mathsf{\cred{W}}_{\cred{2}}^{\A^{[\rho_\Gamma,\;\lock_\tri]}}\ft,\;\eval^{\sm^{\n+\cred{1}}}_\tri\;\zv^{\ft\circ\ft}\;\zv^{\ft}_{\tri}}})_{\n+\cred{1}} \\
\pi\big(\lambda^{\sm^{\n+\cred{2}}}_\tri~\t\big) &\equiv \lambda^{\sm^{\n+\cred{1}}}_\tri~\pi\t \\
\big(\lambda^{\sm^{\n+\cred{2}}}_\tri~\t\big)_{\n+\cred{2}} &\equiv \lambda^{\sm^{\n+\cred{1}}}_\tri~\t\d \\
\pi\big(\eval^{\sm^{\n+\cred{2}}}_\tri~\f~\s\big) &\equiv \eval^{\sm^{\n+\cred{1}}}_\tri~\pi\f~\s\\
\big(\eval^{\sm^{\n+\cred{2}}}_\tri~\f~\s\big)_{\n+\cred{2}} &\equiv \eval^{\sm^{\n+\cred{1}}}_\tri~\f\d~\s.
\end{align*}
The verification of many identities has been omitted.

Finally, we check that $\pi$ preserves all the operations defined above.
Therefore, we can define the untruncated operations $\tri$ and $\lock_\tri$ on $\sm$, with modal context extension $x :^\tri \A$ and modal $\Pi$-types, simply by acting levelwise on each $\sm^{\n+\cred{1}}$. \bbox

\subsubsection{Pieces of the Box Modality}
\label{sec:box}

The box modality is more subtle because it is not determined levelwise by operations on truncated diagrams.
However, we can still construct it in terms of truncated data.
We start with a truncated lock functor $\bkt{\blank,~\lock_{\sq_\n}} : \MC \to \MC^{\cred{\oldDelta}^{\cred{+}}_\n}$ that constructs a constant simplicial diagram:
\begin{align*}
\big(\gamma : \Gamma,~\lock_{\sq_{\n+\cred{1}}}\big)_{\m+\cred{1}} &\equiv \Gamma \\
\big(\gamma : \Gamma,~\lock_{\sq_{\n+\cred{1}}}\big)^\b &\equiv \id_\Gamma \\
\sqbkt{\sigma,~\lock_{\sq_{\n+\cred{1}}}}_{\m+\cred{1}} &\equiv \sigma.
\end{align*}
We the define the following four new pieces of syntax.
The operation $\A_{\subsq\bkt{\n+\cred{1}}}$ is like a truncated version of $\sq$, in that it takes the limit of a truncated diagram, but yielding a finite telescope rather than a type.
\begin{mathpar}
\infer{\gamma : \Gamma,~\lock_{\sq_\n} \vdash_{\sm^\n} \A~\gamma \type_{\;\!\ell}}{\gamma : \Gamma \vdash_\dm \A_{\subsq\bkt{\n+\cred{1}}}~\gamma \tel_{\;\!\ell}}
\and
\infer{\gamma : \Gamma,~\lock_{\sq_\n} \vdash_{\sm^\n} \t~\gamma : \A~\gamma}{\gamma : \Gamma \vdash_\dm \t_{\subsq\bkt{\n+\cred{1}}} : \A_{\subsq\bkt{\n+\cred{1}}}~\gamma}
\and
\infer{\gamma : \Gamma,~\lock_{\sq_\n} \vdash_{\sm^\n} \A~\gamma \type_{\;\!\ell}}{\ft^{\A}_{\sq_\n} : \bkt{\gamma : \Gamma,~\smsq\a : \A_{\subsq\bkt{\n+\cred{1}}}~\gamma} \to \Gamma}
\and
\infer{\gamma : \Gamma,~\lock_{\sq_\n} \vdash_{\sm^\n} \A~\gamma \type_{\;\!\ell}}{\gamma : \Gamma,~\smsq\a : \A_{\subsq\bkt{\n+\cred{1}}}~\gamma,~\lock_{\sq_\n} \vdash_{\sm^\n} \zv^{\A}_{\sq_\n}~\gamma~\smsq\a : \A^{\sqbkt{\ft^{\A}_{\sq_\n},\; \lock_{\sq_\n}}}~\gamma~\smsq\a}
\end{mathpar}
These will satisfy the inductively proven property that for $\gamma : \Gamma,~\lock_{\sq_\n} \vdash_{\sm^\n} \t~\gamma : \A~\gamma$:
\[\gamma : \Gamma, ~\lock_{\sq_\n} \vdash_{\sm^\n} \zv^{\A}_{\sq_\n}~\gamma ~\bkt{\t_{\subsq\bkt{\n+\cred{1}}}~\gamma} \equiv \t~\gamma.\]
For $\sm^\minustwo$, the term $\zv^{\A}_{\sq_\minustwo}$ is trivial, since it lives in the terminal CwF structure. We also set:
\begin{align*}
\A_{\subsq\bkt{\minusone}}~\gamma &\equiv \ec_\dm \\
\t_{\subsq\bkt{\minusone}}~\gamma &\equiv \esub_\dm \\
\ft^{\A}_{\sq_\minustwo} &\equiv \id_\Gamma.
\end{align*}
Note that, in general, since $\zv^{\A}_{\sq_\n}$ is a simplicial term, we may form its matching substitution:
\[\infer{\gamma : \Gamma,~\lock_{\sq_\n} \vdash_{\sm^\n} \A~\gamma \type_{\;\!\ell}}{\gamma : \Gamma,~\smsq\a : \A_{\subsq\bkt{\n+\cred{1}}}~\gamma \vdash_\dm \big(\zv^{\A}_{\sq_\n}\big)_{\cred{\partial}\bkt{\n+\cred{1}}}~\gamma~\smsq\a : \A_{\cred{\partial}\bkt{\n+\cred{1}}}~\gamma}\]
For $\sm^{\n+\cred{1}}$ we then inductively set:
\begin{align*}
\A_{\subsq\bkt{\n+\cred{2}}} &\equiv \big(\a : \pi\A_{\subsq\bkt{\n+\cred{1}}} ~\gamma, ~\a\cblu{'} : \A_{\n+\cred{1}}~\gamma ~\big(\big(\zv^{\pi\A}_{\sq_\n}\big)_{\cred{\partial}\bkt{\n+\cred{1}}}~\gamma~\a\big) \\
\t_{\subsq\bkt{\n+\cred{2}}} &\equiv \sqbkt{\t_{\subsq\bkt{\n+\cred{1}}}, ~\t_{\n+\cred{1}}}\\
\ft^{\A}_{\sq_{\n+\cred{1}}} &\equiv \ft^{\pi\A}_{\sq_\n} \circ \ft^{\A_{\n+\cred{1}}}_\dm \\
\zv^{\A}_{\sq_{\n+\cred{1}}} &\equiv \pair{\bkt{\zv^{\pi\A}_{\sq_\n}}^{\sqbkt{\ft^{\A_{\n+\cred{1}}}_\dm,\;\lock_\sq}}}{\big(\zv^{\A}_{\sq_{\n+\cred{1}}}\big)_{\n+\cred{1}}} \\
\big(\zv^{\A}_{\sq_{\n+\cred{1}}}\big)_{\n+\cred{1}}~\gamma~\sqbkt{\a,~\a\cblu{'}} &\equiv \a\cblu{'}.
\end{align*}
The second line is well typed by the inductive hypothesis and makes the next case of the hypothesis clear. Also, note the $\ft$ substitution in the fourth line. The basic idea is that, at the top dimension, the $\bkt{\n+\cred{1}}$-st simplicial value of a boxed variables access the last component of the modal context extension, whereas lower dimensional simplicial values search further back in the linear context.

We now move on to the untruncated model.
The functor $\lock_\sq$ in $\sm$ similarly constructs a constant presheaf.
Note that $\bkt{\Gamma,~\lock_\sq}\D \equiv \bkt{\Gamma,~\lock_\sq}$, and $\rho_{\Gamma,\;\lock_\sq}$ is an identity; we will omit writing these whenever the previous rules say that a $\D$ or $\rho$ is necessary.
We now define a key natural transformation:
\begin{align*}
&\key^{\trisq \leq \id_\sm} : \id_\sm \Rightarrow \bkt{\blank,~\lock_\trisq} \\
&\bkt{\key^{\trisq \leq \id_\sm}_\Gamma}_{\m+\cred{1}} \equiv \Gamma^{\Zero^{\m+\cred{1}}}.
\end{align*}
We then use the constructions above to construct a modal type former:
\begin{mathpar}
\infer{\gamma : \Gamma,~\lock_\sq \vdash_{\sm} \A~\gamma \type_{\;\!\ell}}{\gamma : \Gamma \vdash_\dm \sq_\sm~\A~\gamma \type_{\;\!\ell}}
\and
\infer{\gamma : \Gamma, ~\lock_\sq \vdash_{\sm} \t~\gamma : \A~\gamma}{\gamma : \Gamma \vdash_\dm \sq_\sm~\t~\gamma : \sq_\sm~\A~\gamma}
\and
\infer{\gamma_\bminusone : \lock_\tri \Gamma \vdash_\dm \t~\gamma_\bminusone : \sq~\A~\gamma_\bminusone}{\gamma : \Gamma \vdash_{\sm} \unsq^{\A}_\sm~\t~\gamma : \A~\big(\key^{\trisq \leq \id_{\sm}}_{\Gamma} ~\gamma\big)}
\end{mathpar}
(Recall from \cref{sec:tri} that $\lock_\tri \Gamma \equiv \Gamma_\minusone$.)
In order to form these, we will take an $\omega$-limit of sequences $\A_\sq$ or $\t_\sq$ obtained from the $\m$-simplex levels of $\A$ or $\t$:
\begin{mathpar}
\infer{\gamma : \Gamma,~\lock_\sq \vdash_{\sm} \A~\gamma \type_{\;\!\ell}}{\gamma : \Gamma \vdash_\dm \A_\sq \inftel_{\;\!\ell}}
\and
\infer{\gamma : \Gamma,~\lock_\sq \vdash_{\sm} \t~\gamma : \A~\gamma \type_{\;\!\ell}}{\gamma : \Gamma \vdash_\dm \t_\sq~\gamma : \A_\sq~\gamma}
\end{mathpar}
These are defined as follows:
\begin{align*}
\A^{\cred{\partial}\bkt{\m+\cred{1}}}_\sq~\gamma &\equiv \bkt{\pi_\m\A}_{\subsq\bkt{\m+\cred{1}}}~\gamma \\
\A^{\m+\cred{1}}_\sq~\gamma~\smsq\a &\equiv \A_{\m+\cred{1}}~\gamma ~\big(\big(\zv^{\pi_{\m}\A}_{\sq_\m}\big)_{\cred{\partial}\bkt{\m+\cred{1}}}~\gamma~\smsq\a\big) \\
\t^{\cred{\partial}\bkt{\m+\cred{1}}}_\sq~\gamma &\equiv \bkt{\pi_\m\t}_{\subsq\bkt{\m+\cred{1}}}~\gamma \\
\t^{\m+\cred{1}}_\sq~\gamma &\equiv \t_{\m+\cred{1}}~\gamma.
\end{align*}
We then define:
\begin{align*}
\sq_\sm~\A &\equiv \lim~\A_\sq \\
\sq_\sm~\t &\equiv \lim~\t_\sq.
\end{align*}
We define the eliminator by:
\begin{align*}
\pi_{\n+\cred{1}}\bkt{\unsq^{\A}_\sm~\cblu{\mathfrak{a}}} ~\gamma_{\cblu{\n+1}} \equiv \zv^{\pi_{\n+\cred{1}}\A}_{\sq_{\n+\cred{1}}} ~{\gamma_{\cblu{\n+1}}}^{\Zero^{\n+\cred{1}}}~\sqbkt{\res^{\cred{\partial}\bkt{\n+\cred{1}}}~{\gamma_{\cblu{\n+1}}}^{\Zero^{\n+\cred{1}}}~\cblu{\mathfrak{a}}, ~\res^{\n+\cred{1}}~{\gamma_{\cblu{\n+1}}}^{\Zero^{\n+\cred{1}}} ~\cblu{\mathfrak{a}}}.
\end{align*}
One then checks the computation laws. \bbox

\subsubsection{The Extended Simplicial Model}
\label{sec:smplus}

So far, we have equipped the simplicial model $\sm$ with the locks $\lock_\tri$ and $\lock_\sq$, modal extension and modal $\Pi$-types for $\tri$, and a modality $\sq_{\sm}$ with Fitch-style introduction and elimination rules.
(Because $\sq_{\sm}$ satisfies an $\eta$-rule, we could then derive modal extension and modal $\Pi$-types for $\sq_\sm$ by simply extending and mapping out of $\sq_\sm \A$, as we will do in \cref{sec:modal-variables,sec:modal-pi-types} for our eventual model.)

The modality $\dia$ presents a different problem: in syntax, for $\Gamma \ob_\dm$, the context $\bkt{\Gamma,~\lock_\dia}$ is \emph{flat}. This creates an issue of how we store such contexts semantically.  Our solution is to extend the simplicial model constructed in \cref{sec:simplicial-model} to what we call the \emph{extended simplicial model}, $\sm_\cred{+}$, built out of a copy of $\dm$ (representing the flat contexts) and the original $\sm$ (representing the non-flat contexts).
We start with the non-modal aspects of this model.

\paragraph{Contexts and substitutions.}

We have the following introduction rules:
\begin{mathpar}
\infer{\Gamma \ob_\dm}{\inf~\Gamma \ob_{\sm_\cred{+}}}
\and
\infer{\Gamma \ob_\sm}{\ins~\Gamma \ob_{\sm_\cred{+}}}
\\
\infer{\sigma : \Delta \to_\sm \Gamma}{\ins~\sigma : \ins~\Delta \to_{\sm_\cred{+}} \ins~\Gamma}
\and
\infer{\sigma : \Delta \to_\dm \Gamma}{\inf~\sigma : \inf~\Delta \to_{\sm_\cred{+}} \inf~\Gamma}
\and
\infer{\sigma : \Delta \to_\dm \Gamma_\minusone}{\inc~\sigma : \inf~\Delta \to_{\sm_\cred{+}} \ins~\Gamma}
\end{mathpar}
Equivalently, we can say that the underlying category of $\sm_\cred{+}$, which we denote $\MC^{\cred{\oldDelta^{+}}}_\cred{+}$, is defined as follows:
\begin{align*}
\ob_{\MC^{\cred{\oldDelta^{+}}}_\cred{+}} &\cong \ob_{\MC} \sqcup \ob_{\MC^{\cred{\oldDelta^{+}}}} \\
\mathsf{mor}_{\MC^{\cred{\oldDelta^{+}}}_\cred{+}}\bkt{\inf~\Delta~,~\inf~\Gamma} &\cong \mathsf{mor}_{\MC}\bkt{\Delta~,~\Gamma} \\
\mathsf{mor}_{\MC^{\cred{\oldDelta^{+}}}_\cred{+}}\bkt{\ins~\Delta~,~\ins~\Gamma} &\cong \mathsf{mor}_{\MC^{\cred{\oldDelta^{+}}}}\bkt{\Delta~,~\Gamma} \\
\mathsf{mor}_{\MC^{\cred{\oldDelta^{+}}}_\cred{+}}\bkt{\inf~\Delta~,~\ins~\Gamma} &\cong \mathsf{mor}_{\MC}\bkt{\Delta~,~\Gamma_\minusone} \\
\mathsf{mor}_{\MC^{\cred{\oldDelta^{+}}}_\cred{+}}\bkt{\ins~\Delta~,~\inf~\Gamma} &\cong \emptyset.
\end{align*}
This makes sense, because we intuitively think of $\inf~\Delta$ as having been extended by zeroes, thus it is easy to map out of. A substitution of the form $\inc~\sigma$ is known as \emph{flat}. \bbox

\paragraph{Types and Terms.}
We have the following introduction forms for types and terms in $\sm_\cred{+}$:
\begin{mathpar}
\infer{\gamma : \Gamma \vdash_\dm \A~\gamma \type_{\;\!\ell}}{\gamma : \inf~\Gamma \vdash_{\sm_\cred{+}} \inf~\A~\gamma \type_{\;\!\ell}}
\and
\infer{\gamma : \Gamma \vdash_\sm \A~\gamma \type_{\;\!\ell}}{\gamma : \ins~\Gamma \vdash_{\sm_\cred{+}} \ins~\A~\gamma \type_{\;\!\ell}}
\and
\infer{\gamma : \Gamma \vdash_\dm \t~\gamma : \A~\gamma}{\gamma : \inf~\Gamma \vdash_{\sm_\cred{+}} \inf~\t~\gamma :  \inf~\A~\gamma}
\and
\infer{\gamma : \Gamma \vdash_\sm \t~\gamma : \A~\gamma}{\gamma : \ins~\Gamma \vdash_{\sm_\cred{+}} \ins~\t~\gamma :  \ins~\A~\gamma}.
\end{mathpar}
Formally, we set the following, depending on whether on not $\Gamma$ is flat:
\begin{align*}
\Ty_{\sm_\cred{+}}\bkt{\inf~\Gamma} &\cong \Ty_\dm~\Gamma &
\Tm_{\sm_\cred{+}}\bkt{\inf~\Gamma}~\bkt{\inf~\A} &\cong \Tm_\dm~\Gamma~\A \\
\Ty_{\sm_\cred{+}}\bkt{\ins~\Gamma} &\cong \Ty_\sm~\Gamma &
\Tm_{\sm_\cred{+}}\bkt{\ins~\Gamma}~\bkt{\ins~\A} &\cong \Tm_\sm~\Gamma~\A.
\end{align*}
Note that, in the following definition of the functorial action of substitutions, the flat case discards higher data:
\begin{align*}
\bkt{\inf~\A}^{\inf~\sigma} &\equiv \inf~\A^\sigma &
\bkt{\inf~\t}^{\inf~\sigma} &\equiv \inf~\A^\sigma \\
\bkt{\ins~\A}^{\ins~\sigma} &\equiv \ins~\A^\sigma &
\bkt{\ins~\t}^{\ins~\sigma} &\equiv \ins~\A^\sigma \\
\bkt{\ins~\A}^{\inc~\sigma} &\equiv \inf~\bkt{\A_{\minusone}}^\sigma &
\bkt{\ins~\t}^{\inc~\sigma} &\equiv \inf~\bkt{\A_{\minusone}}^\sigma.
\end{align*}
Extension of contexts operates by passing under the inclusion:
\begin{align*}
\bkt{\inf~\Gamma,~\inf~\A} &\equiv \inf~\bkt{\Gamma,~\A} \\
\bkt{\ins~\Gamma,~\ins~\A} &\equiv \ins~\bkt{\Gamma,~\A} \\
\sqbkt{\inf~\sigma,~\inf~\t} &\equiv \inf~\sqbkt{\sigma,~\t} \\
\sqbkt{\ins~\sigma,~\ins~\t} &\equiv \ins~\sqbkt{\sigma,~\t} \\
\sqbkt{\inc~\sigma,~\inf~\t} &\equiv \inc~\sqbkt{\sigma,~\t}.
\end{align*}
Note that in the last case, if we have $\inc~\sigma : \inf~\Delta \to \ins~\Gamma$ then $\sigma : \Delta \to \Gamma_\minusone$. In order to form the extension $\sqbkt{\inc~\sigma,~\s} : \inf~\Delta \to \ins~\bkt{\gamma : \Gamma, ~\a : \A~\gamma}$ , we must give $\delta : \inf~\Delta \vdash_{\sm_\cred{+}} \s~\delta : \bkt{\ins~\A}^{\inc~\sigma}~\delta$. We see then that such an $\s$ has type $\inf~\bkt{\A_{\minusone}}^\sigma$ and must be of the form $\inf~\t$. \bbox

\paragraph{$\Pi$-Types and Universes.}

We define (non-modal) $\Pi$-types and universes in $\sm_{\cred{+}}$ by reducing to the respective constructs in $\dm$ and $\sm$, depending on whether or not the context is flat:
\begin{align*}
\cred{\Pi}^{\sm_\cred{+}}~\bkt{\inf~\A}~\bkt{\inf~\B} &\equiv \inf~\bkt{\cred{\Pi}^\dm~\A~\B} \\
\cred{\Pi}^{\sm_\cred{+}}~\bkt{\ins~\A}~\bkt{\ins~\B} &\equiv \ins~\bkt{\cred{\Pi}^\sm~\A~\B} \\
\Type^{\sm_\cred{+}}_{\;\!\ell} &\equiv
\begin{cases}
\inf~\Disc_{\;\!\ell} &\text{for} \quad \inf~\Gamma \\
\ins~\Type^{\sm}_{\;\!\ell}  &\text{for} \quad \ins~\Gamma.
\end{cases}
\end{align*}
The definitions of $\lambda^{\sm_\cred{+}}$, $\eval^{\sm_\cred{+}}$, $\Code^{\sm_\cred{+}}$, and $\El^{\sm_\cred{+}}$ are similar.

Note that stability under substitution is a more general property in $\sm_\cred{+}$ since we have to additionally consider flat substitutions; if $\inc~\sigma : \inf~\Delta \to \ins~\Gamma$, then we have:
\begin{align*}
&\Big(\cred{\Pi}^{\sm_\cred{+}}~\bkt{\ins~\A}~\bkt{\ins~\B}\Big)^{\inc\;\sigma} \\
&\quad\equiv \Big(\ins~\big(\cred{\Pi}^\sm~\A~\B\big)\Big)^{\inc\;\sigma} \\
&\quad \equiv
\inf~\Big(\big(\cred{\Pi}^\sm~\A~\B\big)_\minusone\Big)^\sigma \\
&\quad \equiv
\inf~\Big(\cred{\Pi}^\dm~\A_\minusone~\B_\minusone\Big)^\sigma \\
&\quad \equiv
\inf~\Big(\cred{\Pi}^\dm~\bkt{\A_\minusone}^\sigma~\bkt{\B_\minusone}^{\mathsf{\cred{W}}_{\cred{2}}^{\A_\minusone}\sigma}\Big) \\
&\quad \equiv
\cred{\Pi}^{\sm_\cred{+}}~\big(\inf~\bkt{\A_\minusone}^\sigma\big)~\big(\inf~\bkt{\B_\minusone}^{\mathsf{\cred{W}}_{\cred{2}}^{\A_\minusone}\sigma}\big) \\
&\quad \equiv
\cred{\Pi}^{\sm_\cred{+}}~\big(\ins~\A\big)^{\inc\;\sigma}~\big(\ins~\B\big)^{\mathsf{\cred{W}}_{\cred{2}}^{\A}\;\bkt{\inc\;\sigma}}
\end{align*}
Similarly for universes:
\begin{align*}
\Big(\Type^{\sm_\cred{+}}_{\;\!\ell}\Big)^{\inc\;\sigma} &\equiv \Big(\ins~\Type^{\sm}_{\;\!\ell}\Big)^{\inc\;\sigma} \\
&\quad \equiv \inf~\Big(\bkt{\Type^{\sm}_{\;\!\ell}}_\minusone\Big)^\sigma \\
&\quad \equiv \inf~\Disc^{\sigma}_{\;\!\ell} \\
&\quad \equiv \inf~\Disc_{\;\!\ell} \\
&\quad \equiv \Type^{\sm_\cred{+}}_{\;\!\ell}
\end{align*}
What makes these calculations work is the relevant constructs have been defined to agree with their discrete counterparts in dimension $\minusone$.
In the rest of this section, we show how $\dm$ and $\sm^{\cred{+}}$ can be made into a model of all of dTT (except for the type-former $\tri$).\bbox

\subsubsection{Locks and Keys}
\label{sec:locks-keys}

The definition of $\sm^{\cred{+}}$ is tailored to allow us to define $\lock_\dia$.
Putting this together with the $\lock_\tri$ and $\lock_\sq$ defined on $\sm$ in \cref{sec:tri,sec:box}, we now define a $\cred{2}$-functor $\dsqbkt{\blank} : \M\coop \to \Cat$, where $\M\coop$ denotes the $\cred{2}$-category obtained by reversing both $\cred{1}$ and $\cred{2}$ cells. On modes, we have:
\begin{align*}
&\dsqbkt{\dm} \equiv \MC \\
&\dsqbkt{\sm} \equiv \MC^{\cred{\oldDelta^{+}}}_\cred{+}.
\end{align*}
To define this $\cred{2}$-functor on modalities, we extend the prior definitions of locks to $\sm_\cred{+}$:
\begin{alignat*}{3}
\bkt{\blank,~\lock^{\cred{+}}_\tri} &\;:\; \MC^{\cred{\oldDelta^{+}}}_\cred{+} \to \MC
&\hspace{1cm} \bkt{\blank,~\lock^{\cred{+}}_\sq} &\;:\; \MC \to \MC^{\cred{\oldDelta^{+}}}_\cred{+}
&\hspace{1cm} \bkt{\blank,~\lock^{\cred{+}}_\dia} &\;:\; \MC \to \MC^{\cred{\oldDelta^{+}}}_\cred{+}
\\
\bkt{\ins~\Gamma,~\lock^{\cred{+}}_\tri} &\equiv \bkt{\Gamma,~\lock_\tri}
&\bkt{\Gamma,~\lock^{\cred{+}}_\sq} &\equiv \ins~\bkt{\Gamma,~\lock_\sq}
&\bkt{\Gamma,~\lock^{\cred{+}}_\dia} &\equiv \inf~\Gamma \\
\bkt{\inf~\Gamma,~\lock^{\cred{+}}_\tri} &\equiv \Gamma \\
\sqbkt{\ins~\sigma,~\lock^{\cred{+}}_\tri} &\equiv \sigma_\minusone
& \sqbkt{\sigma,~\lock^{\cred{+}}_\sq}_{\m+\cred{1}} &\equiv \ins~\sqbkt{\sigma,~\lock_\sq}
&\sqbkt{\sigma,~\lock_\dia^\cred{+}} &\equiv \inf~\sigma. \\
\sqbkt{\inf~\sigma,~\lock^{\cred{+}}_\tri} &\equiv \sigma \\
\sqbkt{\inc~\sigma,~\lock^{\cred{+}}_\tri} &\equiv \sigma
\end{alignat*}
We then define the evident composites:
\begin{align*}
\bkt{\blank,~\lock^{\cred{+}}_\trisq} &\equiv \bkt{\blank, ~\lock^{\cred{+}}_\tri, ~\lock^{\cred{+}}_\sq} \\
\bkt{\blank,~\lock^{\cred{+}}_\tridia} &\equiv \bkt{\blank, ~\lock^{\cred{+}}_\tri, ~\lock^{\cred{+}}_\dia}.
\end{align*}
Finally, it is easy to check that $\bkt{\blank,~\lock^{\cred{+}}_\sq,~\lock^{\cred{+}}_\tri}$ and $\bkt{\blank,~\lock^{\cred{+}}_\dia,~\lock^{\cred{+}}_\tri}$ define identity functors. It follows that we have a contravariantly functorial assignment:
\[\infer{\mu : \p \to \q}{\dsqbkt{\mu} \equiv \bkt{\blank,~\lock^{\cred{+}}_\mu} : \dsqbkt{\q} \to \dsqbkt{\p}}\]
Next, to define this $\cred{2}$-functor on $\cred{2}$-cells, we define the \emph{key} natural transformations. We have $\sq \leq \dia$, $\trisq \leq \id_{\sm_\cred{+}}$, and $\id_{\sm_\cred{+}} \leq \tridia$, which corresponds to the following natural transformations:
\begin{align*}
&\key^{\sq \leq \dia} : \bkt{\blank,~\lock^{\cred{+}}_\dia} \Rightarrow \bkt{\blank,~\lock^{\cred{+}}_\sq}
&&\key^{\trisq \leq \id_{\sm_\cred{+}}} : \id_{\sm_\cred{+}} \Rightarrow \bkt{\blank,~\lock^{\cred{+}}_\trisq}
&&\key^{\id_{\sm_\cred{+}} \leq \tridia} :  \bkt{\blank,~\lock^{\cred{+}}_\tridia} \Rightarrow \id_{\sm_\cred{+}} \\
&\key^{\sq \leq \dia}_{\Gamma} \equiv \inc~\id_\Gamma
&&\key^{\trisq \leq \id_{\sm_\cred{+}}}_{\ins~\Gamma} \equiv \ins~\key^{\trisq \leq \id_\sm}_{\Gamma}
&&\key^{\id_\sm \leq \tridia}_{\ins_\Gamma} \equiv \inc~\id_{\Gamma_\minusone} \\
& &&\key^{\trisq \leq \id_\sm}_{\inf~\Gamma} \equiv \inc~\id_\Gamma
&&\key^{\id_\sm \leq \tridia}_{\inf_\Gamma} \equiv \inf~\id_\Gamma.
\end{align*}
We also have $\key^{\trisq \leq \tridia} \equiv \key^{\trisq \leq \id_{\sm_\cred{+}}} \circ \key^{\id_{\sm_\cred{+}} \leq \tridia}$. The keys assemble into a contravariantly functorial assignment:
\[\infer{\alpha : \mu \leq \nu}{\dsqbkt{\alpha} \equiv \key^\alpha : \dsqbkt{\nu} \Rightarrow \dsqbkt{\mu}}.\]
One checks whiskering identities to verify that $\dsqbkt{\blank}$ defines a $\cred{2}$-functor $ \M\coop \to \Cat$. \bbox

\subsubsection{Modal Types}
\label{sec:modal-types}

Displayed Type Theory has two modal type formers that we need to model (recall that we omit $\tri$ at present):
\begin{mathpar}
\infer{\gamma : \Gamma,~\lock^{\cred{+}}_\dia \vdash_{\sm_\cred{+}} \A~\gamma \type_{\;\!\ell}}{\gamma : \Gamma \vdash_\dm \dia~\A~\gamma \type_{\;\!\ell}}
\and
\infer{\gamma : \Gamma,~\lock^{\cred{+}}_\sq \vdash_{\sm_\cred{+}} \A~\gamma \type_{\;\!\ell}}{\gamma : \Gamma \vdash_\dm \sq~\A~\gamma \type_{\;\!\ell}}.
\end{mathpar}
These come with the following intro and elimination forms:
\begin{mathpar}
\infer{\gamma : \Gamma, ~\lock^{\cred{+}}_\dia \vdash_{\sm_\cred{+}} \t~\gamma : \A~\gamma}{\gamma : \Gamma \vdash_\dm \dia~\t~\gamma : \dia~\A~\gamma}
\and
\infer{\gamma : \Gamma \vdash_\dm \t~\gamma : \dia~\A~\gamma}{\gamma : \inf~\Gamma \vdash_{\sm_\cred{+}} \undia^{\A}~\t~\gamma : \A~\gamma}
\\
\infer{\gamma : \Gamma, ~\lock^{\cred{+}}_\sq \vdash_{\sm_\cred{+}} \t~\gamma : \A~\gamma}{\gamma : \Gamma \vdash_\dm \sq~\t~\gamma : \sq~\A~\gamma}
\and
\infer{\gamma : \Gamma, ~\lock^{\cred{+}}_\tri \vdash_\dm \t~\gamma : \sq~\A~\gamma}{\gamma : \Gamma \vdash_{\sm_\cred{+}} \unsq^\A~\t~\gamma : \A~\sqbkt{\key^{\trisq \leq \id_{\sm_\cred{+}}}_{\Gamma} ~\gamma}}.
\end{mathpar}
Note that there is an asymmetry between the statements of laws for $\undia$ and $\unsq$. To clear up this confusion, we could have instead written:
\[\infer{\Gamma \flat \\ \gamma : \Gamma,~\lock^{\cred{+}}_\tri \vdash_\dm \t~\gamma : \dia~\A~\gamma}{\gamma : \Gamma \vdash_{\sm^\cred{+}} \undia^{\A}~\t~\gamma : \A~\sqbkt{\key^{\tridia \geq \id_{\sm_\cred{+}}}_{\Gamma} ~\gamma}}.\]
But this is \emph{entirely equivalent}, because the semantic-side definition of the predicate $\Gamma \flat$ is that $\Gamma$ is of the form $\inf~\Delta$, in which case we have $\bkt{(\inf~\Delta),~\lock^{\cred{+}}_\tri}\equiv \Delta$ by definition.
Note that the key $\key^{\tridia \geq \id_{\sm_\cred{+}}}_{\Gamma}$ does not arise from a natural transformation and is only defined when $\Gamma \equiv \inf~\Delta$, in which case we simply have $\key^{\tridia \geq \id_{\sm_\cred{+}}}_{\inf~\Delta} \equiv \inf~\id_\Delta$. The first definition can thus be seen as a proof-relevant pattern match along the $\mathsf{flat}$ predicate.

The definition of $\dia$ and its introduction and elimination rules are done by shuffling around discrete information and inclusions:
\begin{align*}
\dia~\bkt{\inf~\A} &\equiv \A \\
\dia~\bkt{\inf~\t} &\equiv \t \\
\undia^{\A}~\t &\equiv \inf~\t.
\end{align*}
For $\sq$, we fall back to our prior construction for the type former and intro rule:
\begin{align*}
\sq~\bkt{\ins~\A} &\equiv \sq_\sm~\A \\
\sq~\bkt{\ins~\t} &\equiv \sq_\sm~\t
\end{align*}
For the eliminator, we split on whether or not $\Gamma$ is flat:
\[
\unsq^{\;\ins\;\A}~\t \equiv
\begin{cases}
\inf~\bkt{\unsq^{\A}_\dm~\t} &\text{for} \quad \inf~\Gamma \\
\ins~\bkt{\unsq^{\A}_\sm~\t}  &\text{for} \quad \ins~\Gamma
\end{cases}
\]
where the discrete case above is as follows:
\begin{mathpar}
\infer{\Gamma \ob_\dm \\ \gamma : \Gamma \vdash_\dm \t~\gamma : \lim~\A_\sq~\gamma}{\gamma : \Gamma \vdash_\dm \unsq^{\A}_\dm~\t~\gamma : \A_\minusone~\gamma}
\end{mathpar}
It is defined by:
\[\unsq^{\;\ins\;\A}~\t \equiv \res^\minusone~\gamma~\cblu{\mathfrak{a}} \eqno{\bbox}\]

\subsubsection{Modal Variables}
\label{sec:modal-variables}

We have the following rules for extending a context and substitution modally:
\begin{mathpar}
\infer{\mu: \p \to \q \\ \Gamma \ob_{\;{\small\newllbracket}\q{\small\newrrbracket}} \\ \gamma : \Gamma, ~\lock^{\cred{+}}_\mu \vdash_{\;{\small\newllbracket}\p{\small\newrrbracket}} \A~\gamma \type_{\;\!\ell}}{\bkt{\gamma : \Gamma, ~\a :^{\mu_\cred{+}} \A~\gamma} \ob_{\;{\small\newllbracket}\q{\small\newrrbracket}}}
\and
\infer{\sigma : \Delta \to_{{\small\newllbracket}\q{\small\newrrbracket}} \Gamma \\ \gamma : \Gamma, ~\lock^{\cred{+}}_\mu \vdash_{\;{\small\newllbracket}\p{\small\newrrbracket}} \t~\gamma : \A~\gamma}{\sqbkt{\sigma,~\t}_{\mu_\cred{+}} : \Delta \to_{{\small\newllbracket}\q{\small\newrrbracket}} \bkt{\gamma : \Gamma, ~\a :^{\mu_\cred{+}} \A~\gamma}}
\end{mathpar}
The case of $\tri_\cred{+}$ is defined as follows, splitting on whether or not $\Gamma$ is flat:
\begin{align*}
\bkt{\gamma : \inf~\Gamma,~\a :^{\tri_\cred{+}} \A~\gamma} &\equiv \inf~\bkt{\gamma : \Gamma,~\a : \A~\gamma}\\
\sqbkt{\inf~\sigma,~\t}_{\tri_\cred{+}} &\equiv \inf~\sqbkt{\sigma, ~\t} \\
\bkt{\gamma : \ins~\Gamma,~\a :^{\tri_\cred{+}} \A~\gamma} &\equiv \ins~\bkt{\gamma : \Gamma,~\a :^\tri \A~\gamma} \\
\sqbkt{\ins~\sigma,~\t}_{\tri_\cred{+}} &\equiv \ins~\sqbkt{\sigma, ~\t}_\tri.
\end{align*}
Following this, the rest of the definitions say that the case of modal extension reduces to extension by a variable or term of modal type:
\begin{align*}
\bkt{\gamma : \Gamma,~\a :^{\dia_\cred{+}} \A~\gamma} &\equiv \bkt{\gamma : \Gamma,~\a : \dia~\A~\gamma} \\
\sqbkt{\sigma,~\t}_{\dia_\cred{+}} &\equiv \sqbkt{\sigma,~\t} \\
\bkt{\gamma : \Gamma,~\a :^{\tridia_\cred{+}} \A~\gamma} &\equiv \bkt{\gamma : \Gamma,~\a :^{\tri_\cred{+}} \dia~\A~\gamma} \\
\sqbkt{\sigma,~\t}_{\tridia_\cred{+}} &\equiv \sqbkt{\sigma,~\t}_{\tri_\cred{+}} \\
\big(\gamma : \Gamma,~\a :^{\sq_\cred{+}} \A~\gamma\big) &\equiv \bkt{\gamma : \Gamma,~\a : \sq~\A~\gamma} \\
\sqbkt{\sigma,~\t}_{\sq_\cred{+}} &\equiv \sqbkt{\sigma,~\t} \\
\big(\gamma : \Gamma,~\a :^{\trisq_\cred{+}} \A~\gamma\big) &\equiv \bkt{\gamma : \Gamma,~\a :^{\tri_\cred{+}} \sq~\A~\gamma} \\
\sqbkt{\sigma,~\t}_{\trisq_\cred{+}} &\equiv \sqbkt{\sigma,~\t}_{\tri_\cred{+}}.
\end{align*}
Each of the context extension operations comes with a notion of parent maps and variables:
\begin{mathpar}
\infer{\gamma : \Gamma,~\lock^{\cred{+}}_\mu \vdash_{\;{\small\newllbracket}\p{\small\newrrbracket}} \A~\gamma \type_{\;\!\ell}}{\ft^{\A}_{\mu_\cred{+}} : \big(\gamma : \Gamma,~\a :^{\mu_\cred{+}} \A~\gamma\big) \to \Gamma}
\and
\infer{\gamma : \Gamma,~\lock^{\cred{+}}_\mu \vdash_{\;{\small\newllbracket}\p{\small\newrrbracket}} \A~\gamma \type_{\;\!\ell}}{\gamma : \Gamma,~\a :^{\mu_\cred{+}} \A~\gamma,~\lock^{\cred{+}}_\mu \vdash_{\;{\small\newllbracket}\p{\small\newrrbracket}} \zv^{\A}_{\mu_\cred{+}} ~ \gamma~\a : \A^{\sqbkt{\ft^{\A}_{\mu_\cred{+}},\;\lock^{\cred{+}}_\mu}}~\gamma~\a}
\end{mathpar}
For the parent maps $\ft^{\A}_{\tri_\cred{+}}$, we make a definition by cases on whether or not $\Gamma$ is flat:
\[
\ft^{\A}_{\tri_\cred{+}} \equiv
\begin{cases}
\inf~\ft^{\A}_\dm &\text{for} \quad \inf~\Gamma \\
\ins~\ft^{\A}_\tri  &\text{for} \quad \ins~\Gamma.
\end{cases}
\]
The parent maps for $\dia_\cred{+}$ and $\sq_\cred{+}$ reduce to discrete parent maps of variables of modal type:
\begin{align*}
&\ft^{\A}_{\dia_\cred{+}} \equiv \ft^{\dia\A}_\dm
&\ft^{\A}_{\sq_\cred{+}} &\equiv \ft^{\sq\A}_\dm.
\end{align*}
Then, for $\tridia_\cred{+}$ and $\trisq_\cred{+}$, we combine this with the substitution above:
\begin{align*}
&\ft^{\A}_{\tridia_\cred{+}} \equiv \ft^{\dia\A}_{\tri_\cred{+}}
&\ft^{\A}_{\trisq_\cred{+}} \equiv \ft^{\sq\A}_{\tri_\cred{+}}.
\end{align*}
For the zero variables for $\dia_\cred{+}$ and $\sq_\cred{+}$, one checks that the following are well-typed:
\begin{align*}
&\zv^{\A}_{\dia_\cred{+}} \equiv \undia^\A~\zv^{\dia\A}_\dm
&\zv^{\A}_{\sq_\cred{+}} \equiv \unsq^\A~\zv^{\sq\A}_\dm.
\end{align*}
For the zero variables $\zv^{\A}_{\tri_\cred{+}}$, we once again case split:
\[
\zv^{\A}_{\tri_\cred{+}} \equiv
\begin{cases}
\zv^{\A}_\dm &\text{for} \quad \inf~\Gamma \\
\zv^{\A}_\tri  &\text{for} \quad \ins~\Gamma.
\end{cases}
\]
Then, for $\tridia_\cred{+}$ and $\trisq_\cred{+}$, we use $\zv_{\tri_\cred{+}}$:
\begin{align*}
&\zv^{\A}_{\tridia_\cred{+}} \equiv \undia^\A~\zv^{\dia\A}_{\tri_\cred{+}}
&\zv^{\A}_{\trisq_\cred{+}} \equiv \unsq^\A~\zv^{\sq\A}_{\tri_\cred{+}}.
\tag*{\bbox}
\end{align*}

\subsubsection{Modal $\Pi$-Types}
\label{sec:modal-pi-types}

The last remaining modal construct that we must address is modal $\Pi$-types. These behave according to the following rules:
\begin{mathpar}
\infer{\mu : \p \to \q \\ \gamma : \Gamma, ~\lock^{\cred{+}}_\mu \vdash_{{\small\newllbracket}\p{\small\newrrbracket}} \A~\gamma \type_{\;\!\ell_\cblu{0}} \\ \gamma : \Gamma,~\a :^{\mu_\cred{+}} \A~\gamma \vdash_{{\small\newllbracket}\q{\small\newrrbracket}} \B~\gamma~\a \type_{\;\!\ell_\cblu{1}}}{\gamma : \Gamma \vdash_{{\small\newllbracket}\q{\small\newrrbracket}} \cred{\Pi}^{\sm_\cred{+}}_{\mu}~\A~\B~\gamma \type_{\;\!\ell_\cblu{0}\;\sqcup\;\ell_\cblu{1}}}
\and
\infer{\gamma : \Gamma, ~\a :^{\mu_\cred{+}} \A~\gamma \vdash_{{\small\newllbracket}\q{\small\newrrbracket}} \t~\gamma~\a : \B~\gamma~\a}{\gamma : \Gamma \vdash_{{\small\newllbracket}\q{\small\newrrbracket}} \lambda^{\sm_\cred{+}}_{\mu}~\t~\gamma : \cred{\Pi}^{\sm_\cred{+}}_{\mu}~\A~\B~\gamma}
\and
\infer{\gamma : \Gamma \vdash_{{\small\newllbracket}\q{\small\newrrbracket}} \f~\gamma :  \cred{\Pi}^{\sm_\cred{+}}_{\mu}~\A~\B~\gamma \\ \gamma : \Gamma, ~\lock^{\cred{+}}_\mu \vdash_{{\small\newllbracket}\p{\small\newrrbracket}} \s~\gamma : \A~\gamma}{\gamma : \Gamma \vdash_{{\small\newllbracket}\q{\small\newrrbracket}} \eval^{\sm_\cred{+}}_{\mu}~\f~\s~\gamma : \B^{\sqbkt{\id_\Gamma,\;\s}_{\mu}}~\gamma}
\end{mathpar}
For $\tri_\cred{+}$ we define:
\begin{align*}
\cred{\Pi}^{\sm_\cred{+}}_{\tri}~\A~\bkt{\ins~\B} &\equiv \ins~\big(\cred{\Pi}^{\sm}_{\tri}~\A~\B\big) \\
\cred{\Pi}^{\sm_\cred{+}}_{\tri}~\A~\bkt{\inf~\B} &\equiv \inf~\big(\cred{\Pi}^{\dm}~\A~\B\big) \\
\lambda^{\sm_\cred{+}}_{\tri}~\bkt{\ins~\t} &\equiv \ins~\big(\lambda^{\sm}_{\tri}~\t\big) \\
\lambda^{\sm_\cred{+}}_{\tri}~\bkt{\inf~\t} &\equiv \inf~\big(\lambda^{\dm}~\t\big) \\
\eval^{\sm_\cred{+}}_{\tri}~\bkt{\ins~\f}~\s &\equiv \ins~\big(\eval^{\sm}_{\tri}~\f~\s\big) \\
\eval^{\sm_\cred{+}}_{\tri}~\bkt{\inf~\f}~\s &\equiv
\inf~\big(\eval^{\dm}~\f~\s\big).
\end{align*}
The other cases reduce to functions of a modal variable:
\begin{align*}
\cred{\Pi}^{\sm_\cred{+}}_{\dia}~\A~\B &\equiv \cred{\Pi}^{\sm_\cred{+}}~\bkt{\dia~\A}~\B \\
\cred{\Pi}^{\sm_\cred{+}}_{\tridia}~\A~\B &\equiv \cred{\Pi}^{\sm_\cred{+}}_{\tri}~\bkt{\dia~\A}~\B \\
\cred{\Pi}^{\sm_\cred{+}}_{\sq}~\A~\B &\equiv \cred{\Pi}^{\sm_\cred{+}}~\bkt{\sq~\A}~\B \\
\cred{\Pi}^{\sm_\cred{+}}_{\trisq}~\A~\B &\equiv \cred{\Pi}^{\sm_\cred{+}}_{\tri}~\bkt{\sq~\A}~\B
\end{align*}
The cases of $\lambda$ and $\eval$ are similar. \bbox

\subsection{Semantics of dTT}
\label{sec:generic-semantics}

Having reviewed the general notion of model for dependent type theory in \cref{sec:categories-with-families}, and constructed our intended model of dTT in \cref{sec:simplicial-model,sec:modal-semantics}, we now describe the general notion of model for dTT.
Of course, since our syntax was presented as a Generalised Algebraic Theory, there is an immediate notion of \emph{model}, namely an algebra for that theory (with algebraic syntax being the initial model).
The point is to reformulate this in more familiar category-theoretic terms.

\subsubsection{Modal structure}
\label{sec:modal-structure}

The general multimodal type theory MTT was presented algebraically in~\cite{gknb:mtt}.
We therefore start from this as a baseline and add structure corresponding to our particular theory.

For a general mode 2-category $\M$, the starting point is a \emph{modal context structure}, which is a 2-functor $\MC:\M\coop \to \Cat$, where $\M\coop$ denotes reversal of both 1-cells and 2-cells.
The image of a mode $\p$ is the category $\MC_\p$ of contexts and substitutions at that mode, and the image of a morphism $\mu:\p\to \q$ is the lock functor $\lock_\mu : \MC_\q \to \MC_\p$, which we write postfix, $\Gamma \mapsto \Gamma\ce \lock_\mu$.
It is a \emph{modal natural model} if each $\MC_\p$ is equipped with a morphism $\pr_\p : \Tm_\p \to \Ty_\p$ in its presheaf category such that for any $\mu:\p\to \q$ the morphism $(\lock_\mu)^*(\pr_\p)$ is representable.
(In particular, taking $\mu=\id_\p$, we see that each $\MC_\p$ is an ordinary natural model, hence a CwF.)
This notion encapsulates all the rules for building contexts and substitutions from \cref{subsec:modal-type-theory} except for those that refer to flatness of contexts.

Now, specializing to our mode 2-category from \cref{subsec:mode-theory}, the rules of \cref{subsec:modal-type-theory} say that the flat contexts form a full subcategory of $\MC_{\sm}$ that contains the image of $\lock_{\dia}$.
In our algebraic theory we take as primitive the derived rules that a context is flat if and only if it admits a substitution to $(\ec, \lock_{\dia})$, and in that case the latter substitution is unique.
Thus, semantically, $1\ce\lock_{\dia}$ is subterminal and the flat contexts are the slice category $\MC_{\sm}/1\ce\lock_{\dia}$.
Since the unit of the adjunction $\lock_{\dia} \dashv \lock_{\tri}$ is an identity, $\lock_\dia$ is fully faithful, and on its image it has $\lock_{\dia}$ as an inverse and therefore also a left adjoint.
In addition, since flat contexts are fixed points (up to isomorphism) of $\lock_{\tridia}$, when $\Gamma$ is flat there is a bijection $\Ty(\Gamma) \cong \Ty(\Gamma\ce\lock_{\tridia})$, and so the modal comprehension $(\Gamma, \x:^{\tridia} \A)$ is isomorphic to an ordinary one $(\Gamma, \x:\A\cblu{'})$.
This justifies the special variable rule and key substitution.
Thus, the following definition encapsulates the judgmental structure of our modal type theory.

\begin{definition}
  A \textbf{dTT context structure} is a modal context structure $\MC:\M\coop \to \Cat$ in the sense of~\cite{gknb:mtt}, where $\M$ is as in \cref{subsec:mode-theory}, such that $1\ce\lock_{\dia}$ is subterminal and the slice category $\MC_{\sm}/1\ce \lock_{\dia}$ is the replete image of the fully faithful functor $\lock_{\dia} : \MC_\dm \to \MC_\sm$.
  A \textbf{dTT natural model} is a dTT context structure that is also a modal natural model.
\end{definition}

Next, since our modalities are Fitch-style, their semantics follows~\cite{gckgb:fitchtt}.
This requires each functor $\lock_{\mu}$ to \emph{be} a parametric right adjoint and \emph{have} a dependent right adjoint.
However, since each safe $\mu$ is already a right adjoint in $\M$, such $\lock_{\mu}$ are also ordinary (hence parametric) right adjoints.
And since $\lock_{\dia}$ is an equivalence onto a slice category, the inverse of that equivalence is a parametric left adjoint of it.
Thus, to justify our Fitch-style rules for modalities, it suffices to assume the following.

\begin{definition}
  A \textbf{dTT modal model} is a dTT natural model such that the functors $\lock_{\dia}$ and $\lock_{\sq}$ have dependent right adjoints.
\end{definition}

Our primary example is, of course, the following.

\begin{theorem}
  The simplicial model of \cref{sec:simplicial-model,sec:modal-semantics} is a dTT natural model.
\end{theorem}
\begin{sectendproof}
  Of course, we use the extended simplicial model along with the discrete model.
  We showed explicitly in \cref{sec:modal-semantics} that this yields a modal context structure for our $\M$.
  In the notation of that section, the object $1\ce\lock_{\dia}$ referred to above is $\bkt{\ec,~\lock^{\cred{+}}_\dia} \equiv \inf~\ec$.
  The definition of the category $\sm_\cred{+}$ implies immediately that this object is subterminal and its slice category is the replete image (and even the literal image) of $\lock^{\cred{+}}_\dia$.
  Finally, in \cref{sec:modal-types} we verified the rules of~\cite{gckgb:fitchtt} for $\dia$ and $\sq$, which as shown in \textit{loc.~cit.} are equivalent to their being dependent right adjoints of $\lock^{\cred{+}}_\dia$ and $\lock^{\cred{+}}_\sq$.
\end{sectendproof}

\subsubsection{Telescopes}
\label{sec:modal-tel}

Modal telescopes generalise ordinary telescopes to modal natural models in a straightforward way.
Let $\M$ be a 2-category.

\begin{definition}
  A modal natural model $\MC : \M\coop\to\Cat$ \textbf{has telescopes} if it is equipped with:
  \begin{itemize}
  \item For each $\p\in \M$, a representable natural transformation $\tpr_\p : \PSub_\p \to \Tel_\p$, whose comprehensions we write as $(\gamma:\Gamma \ext \theta: \Theta~\gamma)$.
  \item For each $\p$, a morphism of polynomial functors $\ec_\p : \id_{\MC_\p} \to \P_{\tpr_\p}$.
  \item For any $\mu : \p\to \q$, a morphism of polynomial functors $\P_{\tpr_\q} \circ \P_{\lock_\mu^*  \pr_\p} \to \P_{\tpr_\q}$ that we write as $(\theta:\Theta,~\x:^\mu \A~\theta)$.
  \item The rules $(\gamma:\Gamma\ext \ec) = \Gamma$ and $(\gamma:\Gamma\ext(\upsilon:\Upsilon~\gamma,~ \x:^\mu \A~\gamma~\upsilon)) = ((\gamma:\Gamma\ext\upsilon:\Upsilon),~ \x:^\mu \A~\gamma~\upsilon)$ from \cref{sec:telescopes} hold.
  \item A morphism of polynomial functors $\P_\tpr \circ \P_\tpr \to \P_\tpr$, which we write as $\Upsilon\ext\Phi$.
    (This says how to concatenate telescopes.)
  \item The rules $(\gamma:\Gamma \ext (\upsilon:\Upsilon~\gamma \ext \phi:\Phi~\gamma~\upsilon)) = ((\gamma:\Gamma\ext \upsilon:\Upsilon~\gamma) \ext \phi:\Phi~\gamma~\upsilon)$ and $(\upsilon:\Upsilon \ext \ec) = \Upsilon$ and $(\upsilon:\Upsilon \ext (\phi:\Phi~\upsilon,~\x:^\mu \A~\upsilon~\phi)) = ((\upsilon:\Upsilon\ext \phi:\Phi~\upsilon),~ \x:^\mu \A~\upsilon~\phi)$ from \cref{sec:telesc-conc} hold.
  \end{itemize}
  In addition, we say $\MC$ has \textbf{$\Pi$-telescopes} if for each $\p$ there is a pullback square
    \[
      \begin{tikzcd}
        \P_{\tpr_\p}(\PSub_\p) \ar[d,"\P_{\tpr_\p}(\tpr_\p)"']\ar[r]\drpullback & \PSub_\p \ar[d,"\tpr_\p"] \\
        \P_{\tpr_\p}(\Tel_\p) \ar[r,"\cred{\Pi}"'] & \Tel_\p,
      \end{tikzcd}
    \]
    such that the computation rules from \cref{sec:pi-telescopes} hold.
\end{definition}

As in \cref{sec:semtel,sec:meta-abs}, we can equip any modal natural model with telescopes and $\Pi$-telescopes, and interpret meta-abstracted types and telescopes automatically, without needing to discuss them explicitly. \bbox

\subsubsection{Display and d\'ecalage}
\label{sec:sem-d}

If $\MC$ is a CwF with telescopes and $\Gamma\in \MC$, we write $\Telover\Gamma$ for the category whose objects are telescopes $\Theta \in \Tel(\Gamma)$ and whose morphisms are morphisms $(\Gamma\ext \Theta_\cblu{1}) \to (\Gamma\ext \Theta_\cblu{2})$ in $\MC/\Gamma$.
Thus, it is equivalent to the full subcategory of $\MC/\Gamma$ on objects of the form $\Gamma\ext \Theta$.
We call this \emph{the category of telescopes of $\Gamma$}.
Note that by substitution, it is strictly functorial in $\Gamma$, i.e.\ we have a functor $\Telover- : \MC\op \to \Cat$.
Equivalently, therefore, we can regard this as an internal category in the presheaf category $\Set^{\MC\op}$.

In fact the category $\Telover\Gamma$ is actually itself a CwF.
Its \emph{`types'} in \emph{`context'} $\Upsilon \in \Telover\Gamma$ are meta-abstracted telescopes $\Gamma \vdash \Phi \slfrac{\tel_{\;\!\ell}}{_{\upsilon\;:\;\Upsilon}}$.
These are equivalent (but not equal) to telescopes in an extended context, i.e.\ the elements of $\Tel(\Gamma \ext \Upsilon)$.
Similarly, meta-abstracted partial substitutions $\Gamma \vdash \phi : \Phi$ are equivalent to terms $\Gamma\ext(\upsilon:\Upsilon) \vdash \phi:\Phi~\upsilon$, and semantically to sections of such a projection, $(\Gamma\ext \Upsilon) \to (\Gamma\ext \Upsilon\ext\Phi)$ over $(\Gamma\ext \Upsilon)$.
Comprehension is by telescope concatenation.
Because $\Gamma \ext \ec \equiv \Gamma$, the CwF $\Telover\Gamma$ has the following special property.

\begin{definition}
  A CwF $\MC$ is \textbf{strongly democratic} if every context is the comprehension of a unique type in the empty context.
\end{definition}

Since all of the structure of $\Telover\Gamma$ is strictly stable under substitution in $\Gamma$, these categories form an internal CwF in the presheaf category $\Set^{\MC\op}$.
We call this the \emph{internal telescope model} of $\MC$ and denote it $\bTel$.
Recalling \cref{rmk:acks} and comparing to the discussion of the local theory in~\cite{acks:iparam-noint}, we find that the syntactic structure of d\'ecalage from \cref{sec:telescope-decalage,sec:metaabsdec} can be described precisely as an internal CwF morphism.

\begin{definition}\label{def:decalage}
  Let $\MC : \M\coop\to\Cat$ be a dTT natural model with telescopes.
  We say it \textbf{has d\'ecalage} if it is equipped with a strict morphism of internal strongly democratic CwFs
  \[(\blank)\D : (\lock_\trisq)^*\bTel_\sm \to \bTel_\sm \]
  in the category of presheaves over $\MC_\sm$, together with an internal natural transformation $\evens$ from $(\blank)\D$ to $(\blank)\sqbkt{\key^{\trisq\le \id_\sm}} : (\lock_\trisq)^*\bTel_\sm \to \bTel_\sm$.
\end{definition}

In particular, therefore, this structure ordinary includes functors
\((\blank)\D : \Telover(\Gamma\ce\lock_\trisq) \to \Telover\Gamma\)
and natural transformations $\evens$ consisting of maps
\( (\Gamma\ext \Upsilon\D) \to (\Gamma\ce\lock_\trisq\ext \Upsilon), \)
all strictly stable under pullback.
This, together with the strict preservation of empty contexts by the functor $(\blank)\D$, is what the rules of \cref{sec:telescope-decalage} say.
The corresponding action on \emph{`types'} (meta-abstracted telescopes) and \emph{`terms'} (meta-abstracted partial substitutions) assembles into the rules of \cref{sec:metaabsdec}.

We can phrase display in a similar way by defining an auxiliary internal CwF.
We start with telescope display (\cref{sec:dtel}).
Recall from~\cite{shulman15,kl:hoinvdia} that from any CwF $\MC$ we can construct a \emph{`Sierpinski'} model $\MC^\cred{2}$.
Its objects (contexts) are arbitrary morphisms $\gamma_\cblu{01}: \Gamma_\cred{1} \to \Gamma_\cred{0}$ in $\MC$, but its types are pairs of $\A_\cred{0} \in \Ty(\Gamma_\cred{0})$ and $\A_\cred{1} \in \Ty(\Gamma_\cred{1} \ce \A_\cred{0}\sqbkt{\gamma_\cblu{01}})$.
Moreover, there is a strict CwF morphism $(\blank)_\cred{0} : \MC^\cred{2} \to \MC$, which preserves all the type-formers, in particular $\Sigma$-types if $\MC$ has them.
If $\MC$ has $\Sigma$-types there is also a functor $(\blank)_\cred{1} : \MC^\cred{2} \to \MC$ that sends a type $(\A_\cred{0},\A_\cred{1})$ to $\Sigma(\A_\cred{0}\sqbkt{\gamma_\cblu{01}}, \A_\cred{1})\in \Ty(\Gamma_\cred{1})$, but this only preserves comprehension and $\Sigma$-types up to isomorphism.
Thus it is a pseudo CwF morphism in the sense of~\cite[Definition 10]{cd:lccc-tt}, although it preserves substitution and the empty context strictly.

Now let $\MC$ be a CwF, and apply this construction internally in the category of presheaves on $\MC$ to the internal telescope model $\bTel$.
We thus obtain another CwF $\bTel^\cred{2}$ internal to presheaves, in which the \emph{`contexts'} over $\Gamma$ are morphisms of telescopes $\theta_\cblu{01} : \Theta_\cred{1} \to \Theta_\cred{0}$ over $\Gamma$, and the \emph{`types'} in such a \emph{`context'} over $\Gamma$ are pairs of two telescopes $\Upsilon_\cred{0} \in \Tel(\Gamma\ext\Theta_\cred{0})$ and $\Upsilon_\cred{1} \in \Tel(\Gamma \ext \Theta_\cred{1} \ext \Upsilon_\cred{0}\sqbkt{\theta_\cblu{01}})$.
This is no longer strongly democratic, but as always we have the strict CwF morphism $(\blank)_\cred{0} : \bTel^\cred{2} \to \bTel$ that preserves $\Sigma$-types (i.e.\ telescope concatenation), and the pseudo CwF-morphism $(\blank)_\cred{1} : \bTel^\cred{2} \to \bTel$.
Moreover, in this case the latter is actually a strict CwF-morphism, because telescope concatenation is strictly associative.

\begin{definition}\label{def:tel-disp}
  Let $\MC : \M\coop\to\Cat$ be a dTT natural model with telescopes.
  We say $\MC$ \textbf{has telescope display} if it is equipped with
  \begin{enumerate}
  \item An internal pseudo CwF morphism that preserves substitution, the empty context, and $\Sigma$-types strictly:
    \[ (\blank)\d : (\lock_\trisq)^*\bTel_\sm \to \bTel_\sm^\cred{2} \]
  \item An equality between the composite $(\lock_\trisq)^*\bTel_\sm \xrightarrow{(\blank)\d} \bTel_\sm^\cred{2} \xrightarrow{(\blank)_\cred{0}} \bTel_\sm$ and the key transformation $(\blank)\sqbkt{\key^{\trisq\le \id_\sm}} : (\lock_\trisq)^*\bTel_\sm \to \bTel_\sm$.
    Since the latter is a strict morphism, that means that so is the former.
  \item A strict internal CwF morphism
    \[(\blank)\D : (\lock_\trisq)^*\bTel_\sm \to \bTel_\sm \]
    and an isomorphism of pseudo CwF morphisms between $(\blank)\D$ and the composite morphism $(\lock_\trisq)^*\bTel_\sm \xrightarrow{(\blank)\d} \bTel_\sm^\cred{2} \xrightarrow{(\blank)_\cred{1}} \bTel$, that is the identity on underlying functors.
  \end{enumerate}
\end{definition}

We have not assumed \textit{a priori} in \cref{def:tel-disp} that $\MC$ has d\'ecalage, but it is actually included: the morphism $(\blank)\D$ is of course the same as in \cref{def:decalage}, and the transformation $\evens : \Theta\D \to \Theta$ from \cref{def:decalage} arises in \cref{def:tel-disp} as the image of $\Theta$ under the underlying functor of $(\blank)\d$.
The additional data in \cref{def:tel-disp} beyond this is the $1$-part of the action of $(\blank)\d$ (the $\cred{0}$-part is determined by the composition with $(\blank)_\cred{0}$ equaling $(\blank)\sqbkt{\key^{\trisq\le \id_\sm}}$) making it a pseudo CwF morphism preserving substitution and $\Sigma$-types strictly, and the isomorphism on \emph{`types'} (dependent telescopes) between $(\blank)\D$ and the composite of $(\blank)\d$ with $(\blank)_\cred{1}$.
But since the latter is to be a pseudo CwF transformation (see~\cite[Appendix B]{ccd:undeceq-flccc}), and since $(\blank)\D$ is strict and the underlying functor is the identity this just means that this isomorphism must coincide with the $\cred{1}$-part of the comprehension coherence isomorphism of $(\blank)\d$ (the $\cred{0}$-part being the identity).

So all that remains is the $\cred{1}$-part of the action of $(\blank)\d$ on meta-abstracted telescopes, preserving substitution and telescope concatenation, and coherence isomorphisms relating it to comprehension.
This gives the rules of \cref{sec:meta-abstr-displ}, which have \cref{sec:teldisp} as a special case.
In particular, since the $\cred{1}$-part of the comprehension of $(\Upsilon, \Upsilon\d)$ in $\bTel^\cred{2}$ is $(\Upsilon\ext\Upsilon\d)$, the comprehension isomorphisms are the pairing $\pair{\blank}{\blank}$ together with $\ev$ and $\od$.

Finally, we consider display of types, as in \cref{sec:displ-meta-abstr}, and its relation to d\'ecalage as in \cref{sec:comp-tel-dec,sec:comp-metaabsdec}.
In some ways this is simpler, since we don't have to worry about rearranging between display and d\'ecalage; but in other ways it is more complicated, since we have to take account of extending dependent telescopes by types.

To start with, note that the internal telescope model $\bTel$ of any CwF has a \emph{`sub-model'} $\bTel_\cred{1}$ whose internal category of \emph{`contexts'} is the same (telescopes), but whose internal presheaf of \emph{`types'} consists of the \emph{length-$\cred{1}$ telescopes}, i.e.\ single types annotated by a modality.
Note that unlike $\bTel$, it does not automatically have $\Sigma$-types.

\begin{definition}\label{def:ty-disp}
  Let $\MC : \M\coop\to\Cat$ be a dTT natural model with telescopes.
  We say $\MC$ \textbf{has type display} if it is equipped with
  \begin{enumerate}
  \item An internal strict CwF morphism:
    \[ (\blank)\d : (\lock_\trisq)^*\bTel_{\sm,\cred{1}} \to \bTel_\sm^\cred{2} \]
  \item An equality between the composite $(\lock_\trisq)^*\bTel_\sm \xrightarrow{(\blank)\d} \bTel_\sm^\cred{2} \xrightarrow{(\blank)_\cred{0}} \bTel_\sm$ and the key transformation $(\blank)\sqbkt{\key^{\trisq\le \id_\sm}} : (\lock_\trisq)^*\bTel_\sm \to \bTel_\sm$.
  \item If a length-1 telescope $\A$ is non-modal, then the telescope $\A\d$ is a single non-modal type.\label{item:l1tnm}
  \item If a length-1 telescope $\A$ is nontrivially modal, then the telescope $\A\d$ is empty.\label{item:l1tm}
  \end{enumerate}
\end{definition}

Note that, like \cref{def:tel-disp}, this definition includes d\'ecalage.
It represents the rules for display from \cref{sec:disp-ty,sec:displ-meta-abstr}, and the rules for computing d\'ecalage on a telescope extended by a type from \cref{sec:comp-tel-dec}.
Of course, with both telescope display and type display we want them to be compatible.

\begin{definition}
  Let $\MC : \M\coop\to\Cat$ be a dTT natural model with telescope display.
  We say it has \textbf{complete display} if the restriction of
  $(\blank)\d : (\lock_\trisq)^*\bTel_\sm \to \bTel_\sm^\cred{2}$ to $(\lock_\trisq)^*\bTel_{\sm,\cred{1}}$ is an internal strict CwF morphism such that
  \begin{itemize}
  \item \Cref{item:l1tnm,item:l1tm} of \cref{def:ty-disp} hold.
  \item The rules in \cref{sec:comp-metaabsdec} for computing meta-abstracted d\'ecalage in terms of type display hold.
  \item The rules in \cref{sec:comp-matd} for computing meta-abstracted telescope display in terms of type display hold.
  \end{itemize}
\end{definition}

Finally, we add the compatibility conditions with type-formers:

\begin{definition}
  Let $\MC : \M\coop\to\Cat$ be a dTT natural model with telescopes, d\'ecalage, telescope display, and type display.
  We say that \textbf{display respects $\Pi$-types} (respectively \textbf{universes}) if the rules in \cref{sec:comp-telesc-displ} hold.
\end{definition}

This completes the description of the abstract categorical semantics of the theory of \cref{sec:syntax}: it is a dTT natural model with telescopes and complete display that respects $\Pi$-types and universes.
However, as noted in \cref{sec:syntax}, when telescopes are lists of types, as they almost always are, much of this structure can be deduced from the rest.

\begin{theorem}
  Let $\MC$ be a dTT natural model, with telescopes defined from types as in \cref{thm:sctx-as-lists}, and with type display defined relative to these telescopes.
  Then there is a unique way to extend this type display on $\MC$ to complete display.
\end{theorem}
\begin{proof}
  The rules in \cref{sec:dtel} for computing telescope display and d\'ecalage in terms of type display uniquely determine those operations when telescopes are defined as lists of types.
\end{proof}

Using this, we can verify that our intended model is indeed a model.

\begin{theorem}\label{thm:sm-has-disp}
  The simplicial model of \cref{sec:simplicial-model,sec:modal-semantics} has type display, and hence complete display, which respects $\Pi$-types and universes.
\end{theorem}
\begin{sectendproof}
  We constructed a display operation for the simplicial model in \cref{sec:simplicial-model}, but it does not yet have exactly the needed form.
  What we have so far is a \emph{`global'} operation that d\'ecalages the whole context:
  \[ \infer{\gamma : \Gamma \vdash_{\sm} \A~\gamma \type}
    {\gamma^\cblu{+} : \Gamma\D, ~\a : \A^{\rho_\Gamma} ~\gamma^\cblu{+} \vdash_{\sm} \A\d ~\gamma^\cblu{+}~\a \type}.
  \]
  (Note that this is only defined on the original simplicial model $\sm$, not the extended one $\sm_\cred{+}$: indeed, d\'ecalage is not even defined on flat contexts.)
  But the (meta-abstracted version of the) operation we specified in the syntax of \cref{sec:syntax} is a \emph{`local'} one that only d\'ecalages part of the context, keeping the rest of it modally locked away:
  \[  \infer{\gamma:\Gamma,~\lock_\trisq,~\upsilon:\Upsilon \vdash_\sm \A~\gamma~\upsilon \type}{
      \gamma:\Gamma,~\upsilon^\cblu{+}:\Upsilon\D,~\a:\A^{\key^{\trisq\le \id_\sm}}~\gamma~(\upsilon^\cblu{+})\ev \vdash_\sm \A\d~\gamma~\upsilon^\cblu{+}~\a \type}
  \]
  However, it is straightforward to obtain the latter from the former.
  In $\sm_+$, a context of the form $\bkt{\gamma:\Gamma,~\lock_\trisq,~\upsilon:\Upsilon}$ is not flat, hence lies essentially in $\sm$ so that d\'ecalage is defined on it.
  Furthermore, we already observed that $\bkt{\Gamma,~\lock_\sq}\D \equiv \bkt{\Gamma,~\lock_\sq}$ since $\lock_\sq$ lands in constant presheaves.
  Thus, when $\Upsilon$ is a telescope built out of types, we have
  \[\bkt{\gamma:\Gamma,~\lock_\trisq,~\upsilon:\Upsilon}\D
    \equiv \bkt{\gamma:\Gamma,~\lock_\trisq,~\upsilon^\cblu{+}:\Upsilon\D}
  \]
  and so the global operation yields as a special case
  \[ \infer{\gamma:\Gamma,~\lock_\trisq,~\upsilon:\Upsilon \vdash_{\sm_+} \A~\gamma~\upsilon \type}
    {\gamma:\Gamma,~\lock_\trisq,~\upsilon^\cblu{+}:\Upsilon\D,~\a : \A^{\rho} ~\gamma~\upsilon^{\cblu{+}} \vdash_{\sm_+} \A\d~\gamma~\upsilon^\cblu{+}~\a \type}.
  \]
  Now we simply substitute along $\key^{\trisq\le \id_\sm}$ to obtain the desired local rule.
  The necessary computation rules for d\'ecalage, $\Pi$-types, and universes follow immediately from the rules we proved for the global operation in \cref{sec:simplicial-model}.
\end{sectendproof}

\subsubsection{Display of $\oldomega$-limits}
\label{sec:displim}

Finally, when we have both display and also $\oldomega$-limits, it is reasonable to require the former to compute on the latter, in the following way.
Suppose that $\Gamma,~\lock_{\trisq} \ext \phi:\Phi \vdash_\sm \cblu{\widetilde{\Upsilon}}~\phi \inftel$, and we want to compute $\Gamma,~\phi:\Phi\D,~\mathfrak{\cblu{u}}:\lim~\cblu{\widetilde{\Upsilon}} \vdash \lim~\cblu{\widetilde{\Upsilon}}\d~\phi~\mathfrak{\cblu{u}}$.
Then by definition, we have
\begin{align*}
  \Gamma,~\lock_{\trisq} &\vdash_\sm \cblu{\Upsilon}^{\cred{\partial}\n} \slfrac{\tel}{_{\phi\;:\;\Phi}}
  \\
  \Gamma,~\lock_{\trisq} &\vdash_\sm \cblu{\Upsilon}^{\n} \slfrac{\type}{_{\phi\;:\;\Phi,\;\cblu{\partial}\upsilon\;:\;\cblu{\Upsilon}^{\cred{\partial}\n}\;\phi}}
\end{align*}
and therefore
\begin{align*}
  \Gamma &\vdash_\sm (\cblu{\Upsilon}^{\cred{\partial}\n})\d \slfrac{\tel_{\ell}}{_{\phi\;:\;\Phi\D,\;\cblu{\partial}\upsilon \;:\; \cblu{\Upsilon}^{\cred{\partial}\n}\;\phi\ev}}
  \\
  \Gamma &\vdash_\sm (\cblu{\Upsilon}^{\n})\d \slfrac{\type_{\;\!\ell}}{_{\phi\;:\;\Phi\D,\;\cblu{\partial}\upsilon\;:\;(\cblu{\Upsilon}^{\cred{\partial}\n})\D\;\phi,\;\upsilon \;:\; \cblu{\Upsilon}^{\n}\;\phi\ev\;\cblu{\partial}\upsilon\ev}}
\end{align*}
Weakening and substituting to the needed context $\Gamma,~\phi:\Phi\D,~\mathfrak{\cblu{u}}:\lim~\cblu{\widetilde{\Upsilon}}$, we have
\begin{align*}
  \Gamma,~\phi:\Phi\D,~\mathfrak{\cblu{u}}:\lim~\cblu{\widetilde{\Upsilon}}
  &\vdash_\sm (\cblu{\Upsilon}^{\cred{\partial}\n})\d~\phi~(\res^{\cred{\partial}\n}~\mathfrak{\cblu{u}}) \tel
  \\
  \Gamma,~\phi:\Phi\D,~\mathfrak{\cblu{u}}:\lim~\cblu{\widetilde{\Upsilon}}, ~ \cblu{\partial}\upsilon : (\cblu{\Upsilon}^{\cred{\partial}\n})\d~\phi~(\res^{\cred{\partial}\n}~\mathfrak{\cblu{u}})
  &\vdash_\sm (\cblu{\Upsilon}^{\n})\d~\phi~\pair{\res^{\cred{\partial}\n}~\mathfrak{\cblu{u}}}{\cblu{\partial}\upsilon}~(\res^{\n}~\mathfrak{\cblu{u}}) \type
\end{align*}
such that
\begin{align*}
(\cblu{\Upsilon}^{\cred{\partial}(\n+\cred{1})})\d~\phi~(\res^{\cred{\partial}(\n+\cred{1})}~\mathfrak{\cblu{u}})
  &\equiv
  \bkt{\cblu{\partial}\upsilon : \cblu{\widetilde{\Upsilon}}^{\cred{\partial}\n}, \upsilon : \cblu{\widetilde{\Upsilon}}^\n~\cblu{\partial}\upsilon}
  \d~\phi~\sqbkt{\res^{\cred{\partial}\n}~\cblu{\mathfrak{\cblu{u}}}, ~\res^\n~\cblu{\mathfrak{\cblu{u}}}}\\
  &\equiv \bkt{\cblu{\partial}\upsilon : (\cblu{\widetilde{\Upsilon}}^{\cred{\partial}\n})\d~\phi~(\res^{\cred{\partial}\n}~\cblu{\mathfrak{\cblu{u}}}),~
  \upsilon : (\cblu{\widetilde{\Upsilon}}^\n)\d~\phi~\pair{\res^{\cred{\partial}\n}~\cblu{\mathfrak{\cblu{u}}}}{\cblu{\partial}\upsilon}~(\res^\n~\cblu{\mathfrak{\cblu{u}}})}.
\end{align*}
Thus, these data form another infinite telescope, which we denote
\begin{align*}
  \Gamma,~\phi:\Phi\D,~\mathfrak{\cblu{u}}:\lim~\cblu{\widetilde{\Upsilon}}
  &\vdash_\sm \widetilde{\Upsilon}\d~\phi~\mathfrak{\cblu{u}} \inftel\\
  (\Upsilon\d)^{\cred{\partial}\n}~\phi~\mathfrak{\cblu{u}}
  &\equiv
  (\cblu{\Upsilon}^{\cred{\partial}\n})\d~\phi~(\res^{\cred{\partial}\n}~\mathfrak{\cblu{u}})\\
  (\Upsilon\d)^{\n}~\phi~\mathfrak{\cblu{u}}~\cblu{\partial}\upsilon
  &\equiv
  (\cblu{\Upsilon}^{\n})\d~\phi~\pair{\res^{\cred{\partial}\n}~\mathfrak{\cblu{u}}}{\cblu{\partial}\upsilon}~(\res^{\n}~\mathfrak{\cblu{u}})
\end{align*}
We say that \textbf{display respects $\oldomega$-limits} if
\begin{align*}
  \Gamma,~\phi:\Phi\D, ~ \mathfrak{\cblu{u}}:\lim~\cblu{\widetilde{\Upsilon}}
  &\vdash_\sm
  \lim~\cblu{\widetilde{\Upsilon}}\d~\phi~\mathfrak{\cblu{u}} \equiv
  \lim(\widetilde{\Upsilon}\d~\phi~\mathfrak{\cblu{u}})\\
  \Gamma,~\phi:\Phi\D, ~ \mathfrak{\cblu{u}}:\lim~\cblu{\widetilde{\Upsilon}}
  &\vdash_\sm (\res^{\cred{\partial}\n})\d~\phi~\mathfrak{\cblu{u}} \equiv \res^{\cred{\partial}\n}~\phi~\mathfrak{\cblu{u}}\\
  \Gamma,~\phi:\Phi\D, ~ \mathfrak{\cblu{u}}:\lim~\cblu{\widetilde{\Upsilon}}
  &\vdash_\sm (\res^{\n})\d~\phi~\mathfrak{\cblu{u}} \equiv \res^{\n}~\phi~\mathfrak{\cblu{u}}.
\end{align*}
where in the last two equations, the left-hand side is a restriction relative to $\widetilde{\Upsilon}$, and on the right-hand side it is relative to $\widetilde{\Upsilon}\d$.

\begin{theorem}
  Display respects $\oldomega$-limits in the simplicial model.
\end{theorem}
\begin{sectendproof}
  This holds essentially by construction of $\oldomega$-limits therein, plus passing across the translation between different forms of display from \cref{thm:sm-has-disp}.
\end{sectendproof}

\subsection{Semantics of semi-simplicial types}
\label{sec:sem-sst}

Finally, we construct semantics for the displayed coinductive types of \cref{sec:dcoind}, in particular including $\SST$.
As with most kinds of coinductive definitions, they are terminal coalgebras of some sort, but in this case they are terminal coalgebras for a \emph{copointed} endofunctor.
We will construct such a terminal coalgebra by a sequential limit construction, assuming that the base (discrete) model admits such limits.

\subsubsection{Terminal coalgebras for copointed endofunctors}
\label{sec:term-coalg}

\begin{definition}
  A \textbf{copointed endofunctor} of a category $\MC$ is a functor $\F:\MC\to \MC$ together with a natural transformation $\epsilon : \F \to \id_\MC$.
  A \textbf{coalgebra} for a copointed endofunctor is an object $\X$ with a morphism $\x : \X\to \F\X$ such that the composite $\X \xrightarrow{\x} \F\X \xrightarrow{\epsilon_\X} \X$ is the identity.
  A \textbf{terminal coalgebra} is a terminal object of the category of coalgebras.
\end{definition}

Note that a coalgebra for a \emph{copointed} endofunctor is not just a coalgebra for its underlying ordinary endofunctor, but satisfies the equation $\epsilon_\X \circ \x = \id_\X$.

As usual, we can obtain terminal coalgebras by a sequential limit construction, when such limits exist.
However, in the copointed case it does not suffice to simply consider the limit of the tower $\cdots \to \F^\cred{3}\One \to \F^\cred{2}\One \to \F\One \to \One$; we have to incorporate the transformation $\epsilon$ in some way.
The classical way to do this (e.g.\ the dual of~\cite{kelly:transfinite}) is to take \emph{equalisers} at each step.
However, equalisers are difficult to understand homotopy-theoretically, so we replace them by a pullback.
The following definition is a partial dual of~\cite[Definition 8.6]{shulman:univinj}.

\begin{definition}
  Given a natural transformation $\epsilon : \F \to \G$ and a morphism $\f :\X\to \Y$ in the domain of $\F$ and $\G$, we write $\leibhom\epsilon \f$ for the gap map in the following pullback, assuming that the pullback exists.
  
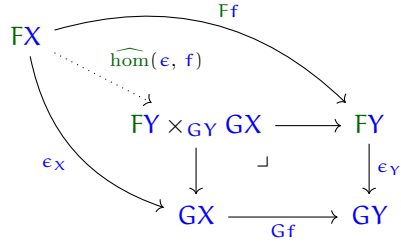
\[
    \begin{tikzcd}
      \F\X \ar[dr,dotted,"\leibhom{\epsilon}{\f}" ] \ar[drr,bend left,"\F\f"] \ar[ddr,bend right,"\epsilon_\X"'] \\
      & \F\Y \times_{\G\Y} \G\X \ar[r] \ar[d] \drpullback & \F\Y \ar[d,"\epsilon_\Y"] \\
      & \G\X \ar[r,"\G\f"'] & \G\Y
    \end{tikzcd}
  \]
  If the domain and codomain of $F$ and $G$ have a notion of \emph{`fibration'}, we say that $\epsilon$ is a \textbf{Quillen pre-fibration} if whenever $\f$ is a fibration, so is $\leibhom{\epsilon}{\f}$.
\end{definition}

For example, we have:

\begin{lemma}\label{thm:evens-qpf}
  In a category of telescopes, consider the fibrations to be the morphisms isomorphic to a dependent projection of some telescope.
  Then the transformation $\evens : (\blank)\D \to (\blank)\sqbkt{\key^{\trisq\le \id_\sm}}$ from \cref{def:decalage} is a Quillen pre-fibration.
\end{lemma}
\begin{proof}
  Given a dependent projection $(\Theta\ext\Upsilon) \to \Theta$ in $\Telover(\Gamma\ce\lock_\trisq)$, the gap map is isomorphic to the dependent projection of $\Upsilon\d$:
  \[ (\Theta\D\ext \Upsilon\sqbkt{\key^{\trisq\le 1}} \ext \Upsilon\d) \to (\Theta\D\ext \Upsilon\sqbkt{\key^{\trisq\le 1}}).\qedhere \]
\end{proof}

\begin{theorem}\label{thm:termcoalg}
  Suppose $\MC$ is a category with a terminal object and a notion of fibration that is stable under pullback, and that $\F$ is a copointed endofunctor of $\MC$ such that $\epsilon :\F\to \id_\MC$ is a Quillen pre-fibration.
  Suppose also that $\MC$ has limits of inverse $\omega$-sequences of fibrations, and that $\F$ preserves these limits.
  Then there is a terminal $\F$-coalgebra.
\end{theorem}
\begin{sectendproof}
  We define inductively a sequence of objects $\X_\n$ with morphisms $\g_{\n+\cred{1}}: \X_{\n+\cred{1}} \to \X_\n$, of which the terminal $\F$-coalgebra will be the limit $\X_{\cred{\infty}}$.
  We can think of each $\X_\n$ as an approximation to the terminal coalgebra, with $\X_{\n+\cred{1}}$ extending $\X_\n$ with additional data making it a better approximation; thus each $\g_\n$ should be a fibration.
  Since $\X_{\n+\cred{1}}$ will be constructed inductively from $\X_\n$, we can expect it to contain all the data that $\X_{\cred{\infty}}$ \emph{should} contain that relates to $\X_\n$, and thus we can expect to have a map $\x_{\n+\cred{1}} : \X_{\n+\cred{1}} \to \F \X_{\n}$ (but not yet to $\F\X_{\n+\cred{1}}$).
  To achieve the copointedness condition in the limit for these data, we should demand $\epsilon_{\X_\n} \circ \x_{\n+\cred{1}} = \g_{\n+\cred{1}}$.
  And to ensure that the successive approximations are consistent with each other, we should ask that $\F\g_\n \circ \x_{\n+\cred{1}} = \x_\n \circ \g_{\n+\cred{1}}$.

  In sum, therefore, we will inductively construct a sequence of objects $\X_\n$, with fibrations $\g_{\n+\cred{1}}: \X_{\n+\cred{1}} \to \X_\n$ and morphisms $\x_{\n+\cred{1}} : \X_{\n+\cred{1}} \to \F \X_{\n}$, such that $\epsilon_{\X_\n} \circ \x_{\n+\cred{1}} = \g_{\n+\cred{1}}$ and $\F\g_\n \circ \x_{\n+\cred{1}} = \x_\n \circ \g_{\n+\cred{1}}$.

  To start with, let $\X_\cred{0}=\One$, the terminal object, and let $\X_\cred{1} = \F \X_\cred{0} = \F \One$, with $\x_{\cred{1}}$ the identity.
  Now, assume the data constructed up to level $\n>\cred{0}$.
  The idea is to define $\X_{\n+\cred{1}}$ to be the \emph{universal} object equipped with $\g_{\n+\cred{1}}$ and $\x_{\n+\cred{1}}$ satisfying the desired equations.
  This means it is a limit of some diagram.
  The usual way to write that diagram is as the equalizer of the two maps
  \[
    \begin{tikzcd}
      & \X_\n \ar[dr,"\x_\n"]\\
      \F \X_\n \ar[ur,"\epsilon"] \ar[rr, "\F \g_\n"] & & \F \X_{\n-1},
    \end{tikzcd}
  \]
  but this does not make it evident that $\g_{\n+\cred{1}}$ is a fibration.
  Instead, we can express this same limit as the following pullback:
  \begin{equation}
    \label{eq:xnplusone}
    \begin{tikzcd}
      \X_{\n+\cred{1}} \ar[r,"\x_{\n+\cred{1}}"] \ar[d,"\g_{\n+\cred{1}}"'] \drpullback & \F \X_\n \ar[d,"\leibhom{\epsilon}{\g_\n}"]\\
      \X_\n \ar[r,"{(\id,\x_\n)}"'] & \X_\n \times_{\X_{\n-\cred{1}}} \F \X_{\n-\cred{1}}.
    \end{tikzcd}
  \end{equation}
  The commutativity of this square says that $\F\g_\n \circ \x_{\n+\cred{1}} = \x_\n \circ \g_{\n+\cred{1}}$ and $\epsilon_{\X_\n} \circ \x_{\n+\cred{1}} = \g_{\n+\cred{1}}$.
  And by assumption, $\leibhom{\epsilon}{\g_\n}$ is a fibration, hence so is its pullback $\g_{\n+\cred{1}}$.

  Now let $\X_\cred{\infty}$ be the limit of the $\omega$-sequence of fibrations
  \[ \X_\cred{\infty} \cdots \xrightarrow{\g_{\n+\cred{1}}} \X_\n \xrightarrow{\g_\n} \cdots \xrightarrow{\g_\cred{2}} \X_\cred{1} \xrightarrow{\g_\cred{1}}\X_\cred{0} = \One.\]
  Since $\F$ preserves limits of inverse $\omega$-sequences, $\F \X_\cred{\infty}$ is the limit of the corresponding sequence
  \[ \F\X_\infty \cdots \xrightarrow{\F \g_{\n+\cred{1}}} \F\X_\n \xrightarrow{\F \g_\n} \cdots \xrightarrow{\F \g_\cred{2}} \F \X_\cred{1} \xrightarrow{\F \g_\cred{1}} \F \X_\cred{0} = \One.\]
  The morphisms $\x_\n$ and $\epsilon_{\X_\n}$ form fence diagrams:
  \[
    \begin{tikzcd}
      \X_\cred{\infty}  \ar[d,"\x_\cred{\infty}"'] \ar[r] & \cdots \ar[r] &
      \X_\cred{3} \ar[r,"\g_\cred{3}"'] \ar[dr,"\x_\cred{3}"'] &
      \X_\cred{2} \ar[r,"\g_\cred{2}"'] \ar[dr,"\x_\cred{2}"'] &
      \X_\cred{1} \ar[r,"\g_\cred{1}"'] \ar[dr,"\x_\cred{1}"'] &
      \X_\cred{0}\\
      \F \X_\cred{\infty} \ar[d,"\epsilon_{\X_\cred{\infty}}"'] \ar[r] & \cdots \ar[r] &
      \F\X_\cred{3} \ar[r,"\F\g_\cred{3}"'] \ar[d,"\epsilon_{\X_\cred{3}}"] &
      \F\X_\cred{2} \ar[r,"\F\g_\cred{2}"'] \ar[d,"\epsilon_{\X_\cred{2}}"] &
      \F\X_\cred{1} \ar[r,"\F\g_\cred{1}"'] \ar[d,"\epsilon_{\X_\cred{1}}"] &
      \F\X_\cred{0} \ar[d,"\epsilon_{\X_\cred{0}}"] \\
      \X_\cred{\infty} \ar[r] & \cdots \ar[r] &
      \X_\cred{3} \ar[r,"\g_\cred{3}"'] &
      \X_\cred{2} \ar[r,"\g_\cred{2}"'] &
      \X_1 \ar[r,"\g_\cred{1}"'] &
      \X_\cred{0}
    \end{tikzcd}
  \]
  composed of the parallelograms $\F\g_\n \circ \x_{\n+\cred{1}} = \x_\n \circ \g_{\n+\cred{1}}$ from our construction, and naturality squares $\epsilon_{\X_\n} \circ \F \g_{\n+\cred{1}} = \g_{\n+\cred{1}} \circ \epsilon_{\X_{\n+\cred{1}}}$.
  The former induces a map of limits $\x_\cred{\infty} : \X_\cred{\infty} \to \F \X_\cred{\infty}$, while by naturality the latter induces $\epsilon_{\X_\cred{\infty}}$.
  The universal property of limits implies that $\epsilon_{\X_\cred{\infty}}\circ \x_\cred{\infty}$ is induced by the composite fence, and since $\g_{\n+\cred{1}} = \epsilon_{\X_\n}\circ \x_{\n+\cred{1}}$ this is
  \[
    \begin{tikzcd}
      \X_\cred{\infty}  \ar[d,"\id_{\X_\cred{\infty}}"'] \ar[r] & \cdots \ar[r] &
      \X_\cred{3} \ar[r,"\g_\cred{3}"'] \ar[dr,"\g_\cred{3}"'] &
      \X_\cred{2} \ar[r,"\g_\cred{2}"'] \ar[dr,"\g_\cred{2}"'] &
      \X_\cred{1} \ar[r,"\g_\cred{1}"'] \ar[dr,"\g_\cred{1}"'] &
      \X_\cred{0}\\
      \X_\cred{\infty} \ar[r] & \cdots \ar[r] &
      \X_\cred{3} \ar[r,"\g_\cred{3}"'] &
      \X_\cred{2} \ar[r,"\g_\cred{2}"'] &
      \X_\cred{1} \ar[r,"\g_\cred{1}"'] &
      \X_\cred{0}
    \end{tikzcd}
  \]
  which induces the identity $\id_{\X_\cred{\infty}}$.
  Thus, $\X_\cred{\infty}$ is an $\F$-coalgebra.

  Now suppose $\y:\Y \to \F\Y$ is another $\F$-coalgebra.
  We construct inductively maps $\h_\n : \Y \to \X_\n$ such that $\x_{\n+\cred{1}} \circ \h_{\n+\cred{1}} = \F \h_\n \circ \y$ and $\g_{\n+\cred{1}}\circ \h_{\n+\cred{1}} = \h_\n$.
  We start with $\h_\cred{0} : \Y \to \X_\cred{0} = \One$ the unique morphism, and $\h_\cred{1} : \Y \to \X_\cred{1} = \F\X_\cred{0}$ the composite $\Y \xrightarrow{\y} \F\Y \xrightarrow{\F\h_\cred{0}} \F\X_\cred{0}$.
  Then we induce $\h_{\n+\cred{1}}$ by the universal property of the pullback defining $\X_{\n+\cred{1}}$:
  \[
    \begin{tikzcd}
      \Y \ar[ddr,bend right,"\h_\n"'] \ar[drr,bend left,"\F\h_\n\circ \y"] \ar[dr,dashed,"\h_{\n+\cred{1}}"] \\
      &\X_{\n+\cred{1}} \ar[r,"\x_{\n+\cred{1}}"] \ar[d,"\g_{\n+\cred{1}}"] \drpullback & \F \X_\n \ar[d,"\leibhom{\epsilon}{\g_\n}"]\\
      &\X_\n \ar[r,"{\bkt{\id,\;\x_\n}}"'] & \X_\n \times_{\X_{\n-\cred{1}}} \F \X_{\n-\cred{1}}.
    \end{tikzcd}
  \]
  This is valid because using the inductive assumptions about $\h_\n$ and $\h_{\n-\cred{1}}$, we have
  \begin{align*}
    \epsilon_{\X_\n} \circ \F \h_\n \circ \y
    &= \h_\n \circ \epsilon_{\Y} \circ \y\\
    &= \h_\n
  \end{align*}
  and
  \begin{align*}
    \F\g_{\n} \circ \F\h_\n \circ \y
    &= \F\bkt{\g_\n \circ \h_\n} \circ \y\\
    &= \F\h_{\n-\cred{1}} \circ \y\\
    &= \x_\n \circ \h_\n,
  \end{align*}
  and the two triangles relating to $\h_{\n+\cred{1}}$ show that it has the necessary properties.

  Now, the equations $\g_{\n+\cred{1}}\circ \h_{\n+\cred{1}} = \h_\n$ imply there is an induced map $\h_\cred{\infty} : \Y \to \X_\cred{\infty}$, such that $\x_\cred{\infty} \circ \h_\cred{\infty}$ is induced by the composites $\x_{\n+\cred{1}}\circ \h_{\n+\cred{1}}$.
  But $\x_{\n+\cred{1}} \circ \h_{\n+\cred{1}} = \F \h_\n \circ \y$, and the morphisms $\F \h_\n$ induce the limit map $\F \h_\cred{\infty}$.
  Thus, $\h_\cred{\infty}$ is an $\F$-coalgebra morphism.

  Finally, suppose $\k:\Y\to \X_\cred{\infty}$ is any $\F$-coalgebra morphism, so we have $\x_\cred{\infty} \circ \k = \F \k \circ \y$.
  Then $\k$ is uniquely determined by the maps $\k_\n : \Y \to \X_\n$, and we have $\x_{\n+\cred{1}} \circ \k_{\n+\cred{1}} = \F \k_\n \circ \y$.
  But this equation implies by induction that $\k_\n = \h_\n$ for all $\n$, hence $\k = \h_\cred{\infty}$.
\end{sectendproof}

\subsubsection{Displayed coinductive types}
\label{sec:displ-coind-types}

Let $\MC$ be a dTT natural model with levels, telescopes, d\'ecalage, telescope display, type display respecting $\Pi$-types and universes, and $\Pi$-telescopes.
We will apply \cref{thm:termcoalg} in $\Telover(\Gamma\ce\lock_\trisq\ext \Phi)$, in which the fibrations are the morphisms isomorphic to the dependent projection from some telescope,
\[(\Gamma\ce\lock_\trisq\ext \Phi\ext \Theta\ext \Upsilon) \to (\Gamma\ce\lock_\trisq\ext \Phi\ext\Theta).\]
In the presence of levels, the objects of $\Telover(\Gamma\ce\lock_\trisq\ext \Phi)$ are telescopes at \emph{any} level.

Suppose given the input data for a displayed coinductive type, consisting of
\begin{mathpar}
  \Gamma\in \MC
  \and
  \Phi\in\Tel_{\;\!\ell_{\cblu{0}}}(\Gamma\ce\lock_\trisq)
  \and
  \A \in \Ty_{\;\!\ell_{\cblu{1}}}(\Gamma\ce\lock_\trisq\ext\Phi)
  \and
  \MB \in \Tel_{\;\!\ell_{\cblu{2}}}(\Gamma\ce\lock_\trisq\ext\Phi\ce\A)
  \and
  \sigma\in \PSub_{\;\!\ell_{\cblu{0}}}\left((\Gamma\ce\lock_\trisq\ext\Phi\ce\A\ext \MB),~ \Phi\d\right).
\end{mathpar}
Categorically, this yields the data of a sort of \emph{`display polynomial'}, where we indicate fibrations with $\twoheadrightarrow$:
\[
  \begin{tikzcd}
    (\Gamma\ce\lock_\trisq\ext\Phi\D) \ar[d,->>] &
    (\Gamma\ce\lock_\trisq\ext\Phi\ce \A\ext \MB) \ar[r,->>,"\pi"] \ar[dl,->>,"\tau\pi"'] \ar[l,"\sigma"'] &
    (\Gamma\ce\lock_\trisq\ext\Phi\ce \A) \ar[dr,->>,"\tau"] \\
    (\Gamma\ce\lock_\trisq\ext\Phi) \ar[rrr,equals] &&& (\Gamma\ce\lock_\trisq\ext\Phi)
  \end{tikzcd}
\]
The left vertical map is a fibration because it is isomorphic to the dependent projection
\[ (\Gamma\ce\lock_\trisq\ext\Phi\ext\Phi\d) \to (\Gamma\ce\lock_\trisq\ext\Phi). \]
Since everything is in the category of telescopes over $\Gamma\ce\lock_\trisq$, we will omit it from the notation for conciseness, so that the above display polynomial becomes
\[
  \begin{tikzcd}
    \Phi\D \ar[d,->>] &
    (\Phi\ce \A\ext \MB) \ar[r,->>,"\pi"] \ar[dl,->>,"\tau\pi"'] \ar[l,"\sigma"'] &
    (\Phi\ce \A) \ar[dr,->>,"\tau"] \\
    \Phi \ar[rrr,equals] &&& \Phi
  \end{tikzcd}
\]

This display polynomial then defines a copointed endofunctor of $\Telover(\Gamma\ce\lock_\trisq\ext\Phi)$ as follows:
\begin{mathpar}
  \infer{\Gamma,~\lock_\trisq \vdash_\sm \X \slfrac{\tel_{\;\!\ell}}{_{\phi:\Phi}}}{\Gamma,~\lock_\trisq\vdash_\sm \F\X \slfrac{\tel_{\;\!\ell\;\join\;\ell_{\cblu{1}}\;\join\;\ell_{\cblu{2}}}}{_{\phi\;:\;\Phi,~ \x\;:\;\X~\phi}}}
  \and
  \F\X \equiv \dbkt{ \a:\A~\phi,~ \xp : (\b:\MB~\phi~\a) \to \X\d~\pair{\phi}{\sigma~\a~\b}~\x }_{\;\phi\;:\;\Phi,~\x\;:\;\X\;\phi}
\end{mathpar}
Here the $\to$ denotes a $\Pi$-telescope (\cref{sec:pi-telescopes}), and $\X\d$ denotes meta-abstracted telescope display (\cref{sec:meta-abstr-displ}); this is why we wanted those in the syntax.
Note that $\F\X$ is meta-abstracted over $\Phi$ extended by $\X$, so it lies in $\Telover(\Gamma\ce\lock_\trisq\ext\Phi\ext \X)$ rather than $\Telover(\Gamma\ce\lock_\trisq\ext\Phi)$.
The actual \emph{endo}functor of $\Telover(\Gamma\ce\lock_\trisq\ext\Phi)$ is thus
\[\Ftil(\X) \equiv (\X \ext \F\X).\]
The weakening projection $(\X\ext \F\X) \to \X$ is then a copointing $\epsilon : \Ftil \to \id$ of this endofunctor, which is evidently a fibration.
More than this, we have:

\begin{lemma}
  The copointing $\epsilon : \Ftil \to \id$ is a Quillen pre-fibration.
\end{lemma}
\begin{proof}
  Suppose given a fibration over $\X \in \Tel_{\ell_{\cblu{0}}}(\Gamma,~\lock_\trisq)$, meaning a dependent telescope
  \[\Gamma,~\lock_\trisq \vdash_\sm \Y \slfrac{\tel_{\;\!\ell_{\cblu{1}}}}{_{\phi\;:\;\Phi,~\x\;:\;\X~\phi}}.\]
  Then we have
  \[\Gamma,~\lock_\trisq \vdash_\sm \dbkt{\x\;:\;\X~\phi\ext \y\;:\;\Y~\phi~\x}_{\;\phi\;:\;\Phi} \slfrac{\tel_{\;\!\ell_{\cblu{0}}\;\join\;\ell_{\cblu{1}}}}{_{\phi\;:\;\Phi}} \]
  which we write as $\X\Y$ for conciseness.
  Then by definition, we have
  \begin{equation*}
    \F(\X\Y)
    \equiv
    \dbkt{ \a:\A~\phi,~ \z\cblu{'} : (\b:\MB~\phi~\a) \to (\X\Y)\d~\pair{\phi}{\sigma~\a~\b}~\sqbkt{\x\ext \y} }_{\;\phi\;:\;\Phi,~\x\;:\;\X~\phi,~\y\;:\;\Y~\phi~\x}.
  \end{equation*}
  To simplify this, note that by the rules in \cref{sec:meta-abstr-displ}
  \begin{equation*}
    (\X\Y)\d~\pair{\phi}{\sigma~\a~\b}~\sqbkt{\x\ext \y}
    \equiv \left(\x\cblu{'}:\X\d~\pair{\phi}{\sigma~\a~\b}~\x \ext \Y\d~\pair{\phi}{\sigma~\a~\b}~\pair{\x}{\cblu{x'}}~\y\right)
  \end{equation*}
  and therefore by the rules in \cref{sec:pi-telescopes}
  \begin{multline*}
    (\b:\MB~\phi~\a) \to (\X\Y)\d~\pair{\phi}{\sigma~\a~\b}~\sqbkt{\x\ext \y}\\
    \equiv
    \left(\delta: (\b:\MB~\phi~\a) \to \X\d~\pair{\phi}{\sigma~\a~\b}~\x ~\ext~
      \epsilon : (\b:\MB~\phi~\a) \to \Y\d~\pair{\phi}{\sigma~\a~\b}~\pair{\x}{\delta~\b}~\y\right).
  \end{multline*}
  Now when $\delta$ is paired with $\a:\A~\phi$, it yields $\F\X$.
  Thus, the relevant gap map
  
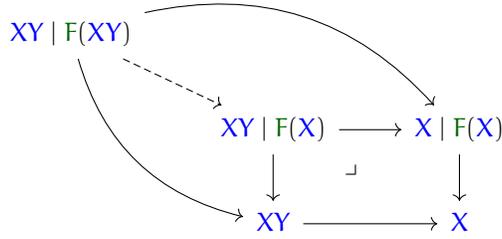
\[
    \begin{tikzcd}
      \X\Y \ext \F(\X\Y) \ar[dr,dashed] \ar[drr,bend left] \ar[ddr,bend right] \\
      & \X\Y\ext \F(\X) \ar[r] \ar[d] \drpullback & \X\ext \F(\X) \ar[d] \\
      & \X\Y \ar[r] & \X
    \end{tikzcd}
  \]
  is the dependent projection from the telescope
  \begin{multline}\label{eq:fgap}
    \dbkt{(\b:\MB\;\phi\;\a) \to \Y\d\;\pair{\phi}{\sigma\;\a\;\b}\;\pair{\x}{\delta\;\b}\;\y}_{\;\phi\;:\;\Phi,\;\x\;:\;\X\;\phi,\;\y\;:\;\Y\;\phi \;\x,\;\a\;:\;\A\;\phi,}\\_{\delta\;:\;(\b\;:\;\MB\;\phi\;\a)\;\to\;\X\d\; \pair{\phi}{\sigma\;\a\;\b}\;\x}
  \end{multline}
  and thus a fibration.
\end{proof}

\begin{lemma}
  The endofunctor $\Ftil$ preserves inverse limits of $\omega$-sequences of fibrations.
\end{lemma}
\begin{proof}[Sketch of proof]
  This follows from the $\eta$-rules for inverse limits, together with the fact that display also preserves inverse limits.
\end{proof}

Therefore, by \cref{thm:termcoalg}, there exists a terminal $\Ftil$-coalgebra.
Moreover, since we have assumed that inverse limits are representable by single types, this coalgebra is a type and not just a telescope.
This type is our candidate for the displayed coinductive type; we can unpack its definition as follows.
The construction produces a tower of fibrations $g_n$, which is to say a sequence of finite telescopes dependent on the previous ones:
\begin{align*}
  \phi:\Phi &\vdash_\sm \X^{\cred{\partial}\n}~\phi \tel_{\;\!\ell} \\
  \phi:\Phi , ~\cblu{\partial}\x : \X^{\cred{\partial}\n}~\phi &\vdash_\sm \X^\n~\phi~\cblu{\partial}\x \tel_{\;\!\ell}\\
  \phi:\Phi &\vdash_\sm \X^{\cred{\partial}\cred{0}}~\phi \equiv \ec\\
  \phi:\Phi &\vdash_\sm \X^{\cred{\partial}(\n+\cred{1})}~\phi \equiv (\cblu{\partial}\x:\X^{\cred{\partial}\n}~\phi ,~ \x : \X^\n~\phi~\cblu{\partial}\x)
\end{align*}
We choose the level of the empty telescope $\X^{\cred{\partial}\cred{0}}$ to be $\ell \equiv \ell_{\cblu{0}} \;\join\;\ell_{\cblu{1}}$; the explicit description given below then implies that all the other telescopes $\X^{\cred{\partial}\n}$ and $\X^\n$ are also at level $\ell$.

The object $\X_{\n+\cred{1}}$ in \cref{sec:term-coalg} corresponds to the telescope
\[(\phi:\Phi,~\cblu{\partial}\x: \X^{\cred{\partial} (\n+\cred{1})}~\phi ) = (\phi:\Phi,~\cblu{\partial}\x: \X^{\cred{\partial}\n}~\phi,~ \x : \X^\n~\phi~\cblu{\partial}\x ).\]
Each morphism $\x_{\n+\cred{1}} : \X_{\n+\cred{1}} \to \Ftil \X_{\n}$ such that $\epsilon \circ \x_{\n+\cred{1}} = \g_{\n+\cred{1}}$ then corresponds to a term
\[ \phi:\Phi,~\cblu{\partial}\x : \X^{\cred{\partial}\n}~\phi,~ \x : \X^\n~\phi~\cblu{\partial}\x \vdash_\sm \xi_\n : \F (\X^{\cred{\partial}\n})~\phi~\cblu{\partial}\x.
\]
By definition of $\F$, $\xi_n$ is equivalent to two terms
\begin{align*}
  \phi:\Phi,~\cblu{\partial}\x : \X^{\cred{\partial}\n}~\phi,~ \x : \X^\n~\phi~\cblu{\partial}\x
  &\vdash_\sm \h_\n~\phi~\cblu{\partial}\x~\x : \A~\phi\\
  \phi:\Phi,~\cblu{\partial}\x : \X^{\cred{\partial}\n}~\phi,~ \x : \X^\n~\phi~\cblu{\partial}\x, ~ \b :\MB~\phi~(\h_\n~\phi~\cblu{\partial}\x~\x)
  &\vdash_\sm
  \begin{multlined}[t]
    \t_\n~\phi~\cblu{\partial}\x~\x~\b : \\
    (\X^{\cred{\partial}\n})\d~\pair{\phi}{\sigma~(\h_\n~\phi~\cblu{\partial}\x~\x)~\b}~\cblu{\partial}\x.
  \end{multlined}
\end{align*}
The equation $\F\g_{\n+\cred{1}} \circ \x_{\n+\cred{2}} = \x_{\n+\cred{1}} \circ \g_{\n+\cred{2}}$ means that
\[\phi:\Phi,~\cblu{\partial}\x : \X^{\cred{\partial}\n}~\phi,~ \x : \X^\n~\phi~\cblu{\partial}\x,~\x\cblu{'} : \X^{\n+\cred{1}}~\phi~\sqbkt{\cblu{\partial}\x,\x}
\vdash_\sm
\h_{\n+\cred{1}}~\phi~\sqbkt{\cblu{\partial}\x, \x}~\x\cblu{'} \equiv \h_\n~\phi~\cblu{\partial}\x~\x\]
and
\begin{multline*}
  \phi:\Phi,~\cblu{\partial}\x : \X^{\cred{\partial}\n}~\phi,~ \x : \X^\n~\phi~\cblu{\partial}\x,~\x\cblu{'} : \X^{\n+\cred{1}}~\phi~\sqbkt{\cblu{\partial}\x,\x},~\b :\MB~\phi~(\h_\n~\phi~\cblu{\partial}\x~\x)\\
  \vdash_\sm
  \t_{\n+\cred{1}}~\phi~\sqbkt{\cblu{\partial}\x,\x}~\x\cblu{'}~\b
  \equiv
  \sqbkt{\t_\n~\phi~\cblu{\partial}\x~\x~\b,~ \s_\n~\phi~\cblu{\partial}\x~\x~\x\cblu{'}~\b }
\end{multline*}
for some term
\begin{multline*}
  \phi:\Phi,~\cblu{\partial}\x : \X^{\cred{\partial}\n}~\phi,~ \x : \X^\n~\phi~\cblu{\partial}\x,~\x\cblu{'} : \X^{\n+\cred{1}}~\phi~\sqbkt{\cblu{\partial}\x,\x},~\b :\MB~\phi~(\h_\n~\phi~\cblu{\partial}\x~\x)\\
  \vdash_\sm
  \begin{multlined}[t]
    \s_\n~\phi~\cblu{\partial}\x~\x~\x\cblu{'}~\b : (\X^\n)\d~\pair{\phi}{\sigma~(\h_\n~\phi~\cblu{\partial}\x~\x)~\b}~\pair{\cblu{\partial}\x}{\t_\n~\phi~\cblu{\partial}\x~\x~\b}~\x\\
    \equiv(\X^\n)\d~\pair{\phi}{\sigma~(\h_{\n-\cred{1}}~\phi~\cblu{\partial}\x)~\b}~\pair{\cblu{\partial}\x}{\t_\n~\phi~\cblu{\partial}\x~\x~\b}~\x
  \end{multlined}
\end{multline*}

Now inspecting the actual construction, we start with $\X_\cred{0} = \X^\cred{\partial 0} = \ec$, the empty telescope, and $\X^\cred{0} = \F(\X^\cred{\partial 0}) = \F\ec = \A$.
It is easy to see by induction that the functions $\h_\n$ then all just project to $\X^\cred{0}$.
For the rest, combining \cref{eq:xnplusone,eq:fgap}, we find that $\X^{\n+\cred{1}}$ is defined by
\begin{multline*}
  \phi:\Phi , ~\cblu{\partial}\x : \X^{\cred{\partial}\n}~\phi,~\x:\X^\n~\phi~\cblu{\partial}\x \vdash_\sm\\
  \begin{multlined}[t]
  \X^{\n+\cred{1}}~\phi~\sqbkt{\cblu{\partial}\x,~\x} \equiv\\
    (\b : \MB~\phi~(\h_{\n-\cred{1}}~\phi~\cblu{\partial}\x)) \to  (\X^\n)\d~\pair{\phi}{\sigma~(\h_{\n-\cred{1}}~\phi~\cblu{\partial}\x)~\b}~\pair{\cblu{\partial}\x}{\t_\n~\phi~\cblu{\partial}\x~\x~\b}~\x
  \end{multlined}
\end{multline*}
and we have tautologically
\[ \s_\n~\phi~\cblu{\partial}\x~\x~\x\cblu{'}~\b \equiv \x\cblu{'}~\b.\]
In particular, by induction we find that in fact each $\X^\n$, though by construction a telescope, is actually just a single type for all $\n\ge \cred{0}$.
Therefore, the tower $(\X^{\cred{\partial}\n},\X^\n)$ is precisely an infinite telescope as defined in \cref{sec:inftel}, which we denote
\[  \cblu{\Xtil} = (\X^{\cred{\partial}\n},\X^\n) .\]
Our displayed coinductive type is therefore precisely the limit of this infinite telescope as in \cref{sec:inflim}:
\[\dCoind~\sqbkt{\Phi,\A,\MB,\sigma}~\phi \equiv
  \lim\bkt{ \cblu{\Xtil}~\phi}.  \]

\begin{example}
  Recall that for semi-simplicial types $\SST$, we have $\Phi\equiv\ec$, with $\A\equiv\Type$ and $\MB~\a \equiv \bkt{x:\El~\a}$.
  Therefore, in this case we have
  \begin{align*}
    \X^\cred{0} &\equiv \Type\\
    \X^\cred{0}~\A_\cblu{0} &\equiv \El~\A_\cblu{0} \to \Type\d~\A_\cblu{0}\\
    &\equiv \El~\A_\cblu{0} \to \El~\A_\cblu{0}\to \Type\\
    \X^\cred{2}~\A_\cblu{0}~\A_\cblu{1} &\equiv
    (\a_\cblu{1}:\A_\cblu{0}) \to \dbkt{\A \to \A\to \Type}_{\A\;:\;\Type}\d~\A_\cblu{0}~(\A_\cblu{1}~\a_\cblu{1})~\A_\cblu{1}\\
    &\equiv
    \begin{multlined}[t]
      (\a_\cblu{1}:\A_\cblu{0}) \to \dbkt{(\x:\A)(\x\cblu{'}:\A\cblu{'}~\x)(\y:\A)(\y\cblu{'}:\A\cblu{'}~\y) \to \Type\d~(\A\cblu{''}~\x~\y)}\\_{\A\;:\;\Type, \;\A\cblu{'}\;:\;\Type\d\;\A, \; \A\cblu{''} \;: \;\A\;\to\;\A\;\to\;\Type}~\A_\cblu{0}~(\A_\cblu{1}~\a_\cblu{1})~\A_\cblu{1}
    \end{multlined}
    \\
    &\equiv
    \begin{multlined}[t][0.8\linewidth]
      (\a\sub{\cblu{001}}:\A_\cblu{0})(\a\sub{\cblu{010}}:\A_\cblu{0})(\a\sub{\cblu{011}}:\A_\cblu{1}~\a\sub{\cblu{001}}~\a\sub{\cblu{010}})\\ (\a\sub{\cblu{100}}:\A_\cblu{0}) (\a\sub{\cblu{101}}:\A_\cblu{1}~\a\sub{\cblu{001}}~\a\sub{\cblu{100}}) (\a\sub{\cblu{110}}:\A_\cblu{1}~\a\sub{\cblu{010}}~\a\sub{\cblu{100}}) \to \Type
    \end{multlined}
  \end{align*}
  This suggests that in general, $\X^{\cred{\partial}\n}$ will be the type of $(\n-\cred{1})$-truncated semi-simplicial types, while $\X^\n~\A$ will be the type of ways to extend such an $\A$ to an $\n$-truncated one, i.e.\ the types of indexed families of $\n$-simplices.
  We will prove this formally in \cref{sec:correctness-sst}.
\end{example}

It remains to show that this construction of $\dCoind$ has the structure stipulated in \cref{sec:dcoind}.
To this end, we first unpack what it means for a type $\C\in \Ty(\Gamma\ce\lock_\trisq\ext\Phi)$ to be an $\Ftil$-coalgebra.
This means it is equipped with a section of the projection $\Ftil(\C) = (\C\ext \F(\C)) \to \C$, which syntactically is to say a partial substitution
\[ \Gamma,~\lock_\trisq\ext(\phi:\Phi), \x:\C \vdash_\sm \c : (\a:\A~\phi,~ \xp : (\b:\MB~\phi~\a) \to \X\d~\pair{\phi}{\sigma~\a~\b}~\x).
\]
But this is equivalent to giving its components, which we abstract over $x$ to emphasise their dependence on it:
\begin{align*}
  \Gamma,~\lock_\trisq\ext(\phi:\Phi) &\vdash_\sm \h : \C \to \A~\phi\\
  \Gamma,~\lock_\trisq\ext(\phi:\Phi) &\vdash_\sm \t : (\x:\C) (\b:\MB~\phi~(\h~\x)) \to \X\d~\pair{\phi}{\sigma~(\h~\x)~\b}~\x
\end{align*}
This is evidently precisely the structure of $\head$ and $\tail$ from \cref{sec:dcoind}.
Thus, our terminal $\Ftil$-coalgebra admits these destructors.

Furthermore, to give some other telescope $\Theta \in \Tel(\Gamma\ce\lock_\trisq\ext\Phi)$ an $\Ftil$-coalgebra structure is equivalent to equipping $\Upsilon \equiv (\Phi\ext\Theta)$ with the premises of the corecursor, where the indices-assigning map $\zeta : (\Phi\ext\Theta) \to \Phi$ is the dependent projection.
Thus, terminality of our terminal $\Ftil$-coalgebra implies that it admits the corecursor for telescopes $\Upsilon$ of this form.

In the models arising from type-theoretic model toposes, the underlying category is actually locally cartesian closed, and thus the functor $\Ftil$ can be extended from $\Telover(\Gamma\ce\lock_\trisq\Phi)$ to the larger slice category $(\Telover(\Gamma\ce\lock_\trisq))/\Phi$, with the same terminal coalgebra in this larger category.
This directly implies the full corecursion principle, since $\zeta$ in that rule equips $\Upsilon$ with the structure of an object of this slice.

In fact, the same is true for arbitrary models: the premises of the corecursor equip $\Upsilon$ with \emph{`enough of an $\Ftil$-coalgebra structure'} to deduce the existence of a unique compatible map to the terminal coalgebra.
In the next section we prove this in a more general abstract context.

\subsubsection{Terminal generalised coalgebras}
\label{sec:gener-term-coalg}

Let $\F$ be a copointed endofunctor of a category $\MC$ as in \cref{sec:term-coalg}, where $\MC$ is a full subcategory of some larger category $\E$.

\begin{definition}
  An object $\Y\in \E$ is a \textbf{generalised $\F$-coalgebra} if it is equipped with:
  \begin{itemize}
  \item For any $\X\in \MC$ and morphism $\h:\Y\to \X$, a specified morphism $\overline{\h}: \Y \to \F\X$, such that
  \item $\epsilon_\X \circ \overline{\h} = \h$.
  \item For any $\g : \X \to \Z$ in $\MC$, we have $\F\g \circ \overline{\h} = \overline{\g\circ \h}$.
  \end{itemize}
\end{definition}

In more abstract language, we can say that $\F$ induces a copointed endofunctor $\F^\cred{*}$ of the functor category $\Set^\MC$ by precomposition, and $\Y$ is a generalised $\F$-coalgebra if the functor $\E(\Y,-) : \MC \to \Set$ is an $\F^\cred{*}$-coalgebra.
The following observation is then a consequence of the Yoneda lemma, but we write it out explicitly.

\begin{lemma}
  If $\Y\in\MC$, then generalised $\F$-coalgebra structures on $\Y$ are bijective to ordinary $\F$-coalgebra structures.
\end{lemma}
\begin{proof}
  In one direction, if $\y:\Y\to \F\Y$ is an $\F$-coalgebra structure, then given $\h:\Y\to \X$ define $\overline{\h} = \F\h \circ \y$.
  Then we have
  \[\epsilon_\X \circ \overline{\h} = \epsilon_\X \circ \F\h \circ \y = \h\circ \epsilon_\Y \circ \y = \h \]
  and
  \[ \F\g \circ \overline{\h} = \F\g \circ \F\h \circ \y = \F(\g\circ \h) \circ \y = \overline{\g\circ \h}. \]

  In the other direction, given a generalised $\F$-coalgebra structure, let $\y = \overline{\id_\Y}  : \Y \to \F\Y$.
  Then $\epsilon_\Y \circ \y = \id_\Y$ by assumption, so $\y$ is an $\F$-coalgebra structure.
  Moreover, the other axiom implies that for any $\g : \Y \to \Z$ we have $\overline{\g} = \overline{\g\circ \id_\Y} = \F\g \circ \overline{\id_\Y} = \F\g\circ \y$.
  Thus one round-trip composite is the identity.
  The other round-trip composite simply sends $\y :\Y \to \F\Y$ to $\overline{\id_\Y} = \F(\id_\Y) \circ \y = \id_{\F\Y}\circ \y = \y$.
\end{proof}

Of course, a \textbf{morphism of generalised $\F$-coalgebras} is a morphism $\f:\Y\to \Z$ such that for any $\h:\Z\to \X\in \MC$ we have $\overline{\h}\circ \f = \overline{\h\circ \f}$.

\begin{lemma}\label{thm:mor-gen-coalg}
  If $\Y\in \E$ is a generalised $\F$-coalgebra and $\x:\X\to \F\X$ is an $\F$-coalgebra in $\MC$, then a morphism $\f:\Y\to \X$ is a generalised $\F$-coalgebra morphism if and only if $\overline{\f} = \x\circ \f$.
\end{lemma}
\begin{proof}
  If it is a generalised $\F$-coalgebra map, then taking $\h = \id_\X$ in $\overline{\h}\circ \f = \overline{\h\circ \f}$ we get $\x\circ \f = \overline{\f}$.
  On the other hand, if $\x\circ \f = \overline{\f}$ then for any $\h:\X\to \cblu{X'}$ in $\MC$ we have $\overline{\h} \circ \f = \cblu{x'} \circ \h \circ \f = \F\h \circ \x \circ \f = \F\h \circ \overline{\f} = \overline{\h\circ \f}$, as desired.
\end{proof}

\begin{theorem}\label{thm:termgencoalg}
  Let $\MC$ and $\F$ be as in \cref{thm:termcoalg}, and let $\MC$ be a full subcategory of $\E$ such that the embedding preserves the terminal object and the inverse limits of $\omega$-sequences of fibrations.
  Then the terminal $\F$-coalgebra constructed in \cref{thm:termcoalg} is also a terminal generalised $F$-coalgebra.
\end{theorem}
\begin{sectendproof}
  Indeed, the proof of terminality in \cref{thm:termcoalg} really only uses the generalised $\F$-coalgebra structure, which we can see clearly by repeating it in that language.
  Let $\Y\in \E$ be a generalised $\F$-coalgebra.
  We construct inductively maps $\h_\n : \Y \to \X_\n$ such that $\x_{\n+\cred{1}} \circ \h_{\n+\cred{1}} = \overline{\h_\cred{n}}$ and $\g_{\n+\cred{1}}\circ \h_{\n+\cred{1}} = \h_\n$.
  We start with $\h_\cred{0} : \Y \to \X_\cred{0} = \One$ the unique morphism (since $1\in\MC$ is also terminal in $\E$), and $\h_\cred{1} = \overline{\h_\cred{0}} : \Y \to \X_\cred{1} = \F\X_\cred{0}$.
  Then we induce $\h_{\n+\cred{1}}$ by the universal property of the pullback defining $\X_{\n+\cred{1}}$:
  \[
    \begin{tikzcd}
      \Y \ar[ddr,bend right,"\h_\n"'] \ar[drr,bend left,"\overline{\h_\n}"] \ar[dr,dashed,"\h_{\n+\cred{1}}"] \\
      &\X_{\n+\cred{1}} \ar[r,"\x_{\n+\cred{1}}"] \ar[d,"\g_{\n+\cred{1}}"] \drpullback & \F \X_\n \ar[d,"\leibhom{\epsilon}{\g_\n}"]\\
      &\X_\n \ar[r,"{(\id,\;\x_\n)}"'] & \X_\n \times_{\X_{\n-\cred{1}}} \F \X_{\n-\cred{1}}.
    \end{tikzcd}
  \]
  This is valid because using the inductive assumptions about $\h_\n$ and $\h_{\n-\cred{1}}$ and the properties of generalised coalgebras, we have
  \begin{align*}
    \epsilon_{\X_\n} \circ \overline{\h_\n} = \h_\n
  \end{align*}
  and
  \begin{align*}
    \F\g_{\n} \circ \overline{\h_\n}
    &= \overline{\g_\n \circ \h_\n}\\
    &= \overline{\h_{\n-\cred{1}}}\\
    &= \x_\n \circ \h_\n,
  \end{align*}
  and the two triangles relating to $\h_{\n+\cred{1}}$ show that it has the necessary properties.

  Now, the equations $\g_{\n+\cred{1}}\circ \h_{\n+\cred{1}} = \h_\n$ imply there is an induced map $\h_\cred{\infty} : \Y \to \X_\cred{\infty}$, such that $\x_\cred{\infty} \circ \h_\cred{\infty}$ is induced by the composites $\x_{\n+\cred{1}}\circ \h_{\n+\cred{1}}$.
  But $\x_{\n+\cred{1}} \circ \h_{\n+\cred{1}} = \overline{\h_\n}$, and the morphisms $\F \h_\n$ induce the limit map $\F \h_\cred{\infty}$, so $\x_\cred{\infty} \circ \h_\cred{\infty} = \overline{\h_\cred{\infty}}$.
  Thus, by \cref{thm:mor-gen-coalg}, $\h_\cred{\infty}$ is an $\F$-coalgebra morphism.

  Finally, suppose $\k:\Y\to \X_\cred{\infty}$ is any $\F$-coalgebra morphism, so we have $\x_\cred{\infty} \circ \k = \overline{\k}$.
  Then $\k$ is uniquely determined by the maps $\k_\n : \Y \to \X_\n$, and we have $\x_{\n+\cred{1}} \circ \k_{\n+\cred{1}} = \overline{\k_\n}$.
  But this equation implies by induction that $\k_\n = \h_\n$ for all $\n$, hence $\k = \h_\cred{\infty}$.
\end{sectendproof}

\subsubsection{The general corecursor}
\label{sec:general-corecursor}

Suppose $\Upsilon\in \Tel (\Gamma\ce\lock_\trisq)$ has the structure of the premises of the corecursor from \cref{sec:dcoind}:
\begin{align*}
  \Gamma,~\lock_\trisq\ext \upsilon:\Upsilon &\vdash_\sm \zeta~\upsilon : \Phi\\
  \Gamma,~\lock_\trisq\ext\upsilon:\Upsilon &\vdash_\sm \h~\upsilon : \A~{(\zeta~\upsilon)}\\
  \Gamma,~\lock_\trisq\ext \upsilon:\Upsilon\ext \y : \MB{(\zeta\upsilon, \h)} &\vdash_\sm \tau~\upsilon~\y : {\Upsilon\d~\upsilon}\\
  \Gamma,~\lock_\trisq \ext \upsilon:\Upsilon \ext \y : \MB{(\zeta\upsilon, \h)} &\vdash_\sm
  \zeta\d~\pair{\upsilon}{\tau~\upsilon~\y} = \sigma~(\zeta~\upsilon)~(\h~\upsilon)~\y
\end{align*}
Then $\zeta$ makes it an object of the slice category $(\Telover(\Gamma\ce\lock_\trisq))/\Phi$.
We will apply \cref{thm:termgencoalg} to the full subcategory $\Telover (\Gamma\ce\lock_\trisq\ext\Phi) \subseteq (\Telover(\Gamma\ce\lock_\trisq))/\Phi$.
To that end, we give $\Upsilon$ the structure of a generalised $\Ftil$-coalgebra as follows.

Suppose $\X\in \Tel(\Gamma\ce\lock_\trisq\ext\Phi)$, and suppose we have a map $\g : \Upsilon \to \X$ in $(\Telover(\Gamma\ce\lock_\trisq))/\Phi$, which is to say
\[ \Gamma,~\lock_\trisq\ext \upsilon:\Upsilon \vdash_\sm \g~\upsilon : \X~(\zeta~\upsilon). \]
We want to lift $\g$ to $\F \X$, which is to say we want to give
\begin{align*}
  \Gamma,~\lock_\trisq\ext(\upsilon:\Upsilon) &\vdash_\sm \h~\upsilon : \A~(\zeta~\upsilon)\\
  \Gamma,~\lock_\trisq\ext(\upsilon:\Upsilon),(\b:\MB~\phi~(\h~\x)) &\vdash_\sm \t~\upsilon~\b : \X\d~\pair{\zeta~\upsilon}{\sigma~(\h~(\zeta~\upsilon)~(\g~\upsilon))~\b}~(\g~\upsilon)
\end{align*}
But such an $\h$ is exactly part of the structure of $\Upsilon$, while we can define
\[ \t~\upsilon~\b \equiv \g\d~\upsilon~(\tau~\upsilon~\b).
\]
The final equation in the structure of $\Upsilon$ is precisely what is necessary to make this well-typed.
The functoriality condition is immediate from the functoriality of $\d$.

Thus $\Upsilon$ is a generalised $\Ftil$-coalgebra, and hence it admits a unique generalised $\Ftil$-coalgebra morphism to the terminal $\Ftil$-coalgebra $\C$.
This is a map $\Upsilon \to (\Phi\ext \X)$ over $\zeta$, which is precisely the right type of $\corec$.
And by \cref{thm:mor-gen-coalg}, the fact that it is a generalised $\Ftil$-coalgebra map precisely gives it the correct computation rules. \bbox

\subsubsection{Correctness of semi-simplicial types}
\label{sec:correctness-sst}

Finally, we will justify our universal characterization of $\SST$ semantically.
Specifically, we will show that when $\SST$ is constructed as a displayed coinductive type as in \cref{sec:displ-coind-types}, in a model with $\oldomega$-limits, it does in fact yield a `classifier' of Reedy fibrant semi-simplicial types in the classical sense.

We begin by constructing such a classifier category-theoretically, and then show that this construction coincides with the one obtained from \cref{sec:displ-coind-types}.
We will assume some familiarity with the classical notions of Reedy fibrant diagrams as in~\cite{kl:hoinvdia}.
For all of this section, we fix a particular universe level $\ell$.

\paragraph{Ordered direct categories.}

Our category-theoretic construction of diagram classifiers works for presheaves over any `direct category' (i.e.\ diagrams on any `inverse category').

\begin{definition}
  A \textbf{direct category} is a category such that the relation `there is a nonidentity arrow from $\x$ to $\y$' on its objects is well-founded.
  A \textbf{sieve} in a (direct) category is a full subcategory $\J$ such that if $\f:\y \to \x$ and $\x\in \J$, then $\y\in\J$.
  An \textbf{ordered direct category} is a \emph{finite} direct category together with (1) a total ordering on its objects such that if $\f:\x\to\y$ then $\x\le\y$, and (2) such that for all objects $\x$, the set of arrows with codomain $\x$ has a linear order such that $\f\circ \g \le \f$ for any composable $\f,\g$ (hence in particular $\id_\x$ is the greatest element).

  An \textbf{ordered presheaf} on a direct category is a finite presheaf together with a linear order on the finite set $\sum_{\x\in \I} \H(\x)$ such that $\H(\f)(\h) < \h$ whenever the left-hand side makes sense.
\end{definition}

An ordered direct category is equivalently the opposite of a (finite) `ordered inverse category' in the sense of~\cite[Definition 3.17]{kl:hoinvdia}, together with a suitable total ordering on its objects (we require this so that the order of variables in the classifying context is specified).
Similarly, an ordered presheaf is a `finite extension` $\cred{\emptyset} \hookrightarrow \H$ in the sense of~\cite[Definition 3.10]{kl:hoinvdia}.

\begin{example}
  Let $\cred{\oldDelta_\n}$ be the subcategory of the category $\cred{\oldDelta^{+}}$ from \cref{sec:asscat} containing the objects $\angle{\k}$ with $\cred{0}\le \k\le \n$.
  Thus $\cred{\oldDelta_\n}(\angle{\k},\angle{\l})$ is the set of length $\l+\cred{1}$ binary sequences containing exactly $\k+\cred{1}$ $\One$s.
  For fixed $\l$ we give these morphisms Campion's ordering, namely the usual ordering of binary numbers.
  Then $\cred{\oldDelta_\n}$ is an ordered direct category.
\end{example}

\def\yonp{\cred{\partial\!\yon}}
For $\x\in\I$ we write $\yonp_\x$ for the sub-presheaf of the representable $\yon_\x$ consisting of nonidentity morphisms, i.e.\ $\yonp_\x(\y) = \{ \f\in \I(\y,\x) \mid \f\neq \id_\x \} = \{ \f\in \I(\y,\x) \mid \y\prec\x \}$.

If $\I$ is a finite direct category and $\H$ is a finite presheaf on it, there is a new finite direct category $\I \oplus \H$, called the \textbf{collage} of $\H$, which contains $\I$ as a full subcategory, together with one new object $\cred{\star}$ such that $\I(\x,\cred{\star}) = \H(\x)$ for all $\x\in \I$.
Note that $\yonp_{\cred{\star}}$ restricted to $\I$ coincides with $\H$.
Moreover, $\I$ and $\H$ are ordered if and only if $\I \oplus \H$ is.
Moreover, if $\I$ is an ordered direct category of finite height with $\x$ its object of greatest rank, then $\I \cong (\I\setminus \{\x\}) \oplus \yonp_\x$.
Thus, we can treat this as an induction principle for ordered direct categories.

\paragraph{Classifying contexts.}

As our first use of this sort of induction, we construct for each ordered direct category $\I$ a `classifying context' for Reedy fibrant $\I$-presheaves.
Specifically, we construct by simultaneous induction:
\def\rG{\cred{\oldGamma}}
\def\rg{\cred{\oldgamma}}
\def\rT{\cred{\oldTheta}}
\def\rt{\cred{\oldtheta}}
\def\rB{\cred{B}}
\def\rb{\cred{b}}
\def\rbe{\cred{\beta}}
\begin{enumerate}
\item For each ordered direct category $\I$, a context $\rG^\I$.
  This will be the classifying context of Reedy fibrant $\I$-types at level $\ell$.
  \label{item:ods1}
\item For each ordered presheaf $\H$ on $\I$, a telescope $\rG^\I \vdash_\sm \rT^\H \tel_{\ell}$.
  \label{item:ods2}
\item For each map of ordered presheaves $\alpha:\H \to \H\cblu{'}$ (not necessarily order-preserving) on $\I$, a partial substitution $\rG^\I \vdash_\sm \rt^\alpha :  \rT^{\H\cblu{'}} \to \rT^\H$, varying functorially.\label{item:ods3}
\item For each object $\x\in \I$, a type $\rG^\I\ext \rT^{\yonp_\x} \vdash_\sm \rB^\x \type_\ell$.\label{item:ods4}
\item For each $\h \in \H(\x)$, inducing by the Yoneda lemma a map $\rbe_\h :\yonp_\x \subseteq \yon_\x \to \H$, a term $\rG^\I\ext \rT^\H \vdash_\sm \rb^\h : \rB^\x \sqbkt{\rt^{\rbe_\h}}$, such that $\rb^\h\sqbkt{\rt^\alpha} = \rb^{\alpha(\h)}$ for any $\alpha:\H \to \H\cblu{'}$.\label{item:ods5}
\item For each sieve $\J\subseteq \I$, a telescope $\rG^\J \vdash_\sm \rG^{\J,\I} \tel_{\lsuc\;\ell}$ and an isomorphism $\rG^\I\cong \bkt{\rG^\J\ext \rG^{\J,\I}}$.
  Morevoer, for all the structure in~\ref{item:ods2}--\ref{item:ods5}, the action of the weakening substitution $\rG^\I\cong \bkt{\rG^\J\ext \rG^{\J,\I}} \to \rG^\J$ corresponds to left Kan extension along the inclusion $\J\hookrightarrow \I$.\label{item:ods6}
\end{enumerate}

For~\ref{item:ods1}, we inductively use~\ref{item:ods2} and set
\begin{align*}
  \rG^{\cred{\emptyset}} &\equiv \ec \\
  \rG^{\I\oplus \H} &\equiv \bkt{\rG^\I,~ \A_{\cred{\star}} : \rT^\H \to \Type_\ell}.
\end{align*}

For~\ref{item:ods2}, we argue inductively on the linear ordering of $\H$.
If $\H$ is empty, we set
\[\rT^{\cred{\emptyset}} \equiv \ec.\]
Otherwise, $\H = (\H\setminus \{\h\}) \cup \{\h\}$ where $\h\in \H(\x)$ is the last element in the ordering; the condition on the ordering ensures that $\H\setminus \{\h\}$ is still an (ordered) presheaf.
By the Yoneda lemma, $\h$ induces a map $\rbe_\h: \yonp_\x \subseteq \yon_\x \to \H\setminus \{\h\}$, hence by \ref{item:ods3} a substitution $\rG^\I \vdash_\sm \rt^{\rbe_\h} : \rT^{\H\setminus\{\h\}} \to \rT^{\yonp_\x}$.
Thus, inductively using~\ref{item:ods4} as well, we can define
\[ \rT^{\H} = \bkt{\rT^{\H\setminus\{\h\}}, ~\a_\h : \rB^{\x}\sqbkt{\rt^{\rbe_\h}} }. \]

We similarly construct~\ref{item:ods3} by induction on $\H$ (the domain of $\alpha$).
The case when $\H$ is empty is trivial.
Otherwise, we inductively have $\rt^{\alpha\setminus \{\h\}} : \rT^{\H\cblu{'}} \to \rT^{\H\setminus\{\h\}}$, and to extend the codomain to $\rT^\H$ it suffices to give a term in context $\rG^\I \ext \rT^{\H\cblu{'}}$ of type
\begin{align*}
  \rB^{\x}\sqbkt{\rt^{\rbe_\h}} \sqbkt{\rt^{\alpha\setminus \{\h\}}}
  &\equiv \rB^{\x} \sqbkt{\rt^{\rbe_\h} \circ \rt^{\alpha\setminus \{\h\}}}\\
  &\equiv \rB^{\x} \sqbkt{\rt^{\rbe_{\alpha(\h)}}}.
\end{align*}
For this we can pick $\rb^{\alpha(\h)}$, using~\ref{item:ods5} inductively.
Functoriality follows from the inductive assumption of functoriality in~\ref{item:ods5}.

For~\ref{item:ods4}, note that the slice category $\I/\x$ is a sieve in $\I$ containing $\x$.
Then it suffices to define $\rB^\x$ in the case of $\I/\x$, since it can then be weakened to $\I$ using \ref{item:ods6}.
In this case we have $\I/\x = (\I/\x\setminus \{\x\}) \oplus \yonp_\x$, so the last variable in $\rG^{\I/\x}$ is $\A_{\x} : \rT^{\yonp_\x}\to \Type_\ell$.
Thus, we can define $\rG^{\I/\x} \ext \y : \rT^{\yonp_x} \vdash_\sm \rB^\x$ to be $\rG^{\I/\x \setminus \{\x\}},\A_{\x} : \rT^{\yonp_\x}\to \Type_\ell \ext \y : \rT^{\yonp_x} \vdash_\sm \A_{\x}~\y$.

For~\ref{item:ods5}, it suffices to deal with the case when $\h$ is the last element in the ordering of $\H$, since otherwise we can weaken from the sub-presheaf of all elements $\le \h$ to all of $\H$, using \ref{item:ods3} for the inclusion of this sub-presheaf.
But in this case, the last variable in $\rT^\H$ is $\a_\h : \rB^{\x}\sqbkt{\rt^\h}$, so we can take $\rb^\h \equiv \a_\h$.
Functoriality follows immediately, as does stability under weakening from initial segments for all the data.

Finally, for~\ref{item:ods6} we induct on $\I$.
For a sieve in $\I\oplus \H$ there are two possibilities: it could be $\J$ or $\J\oplus \H$ for some sieve $\J$ in $\I$, depending on whether it contains the new object $\cred{\star}$.
(Of course, if it contains $\cred{\star}$, it must also contain all objects $\y$ such that $\H(\y)\neq\cred\emptyset$, which is to say that $\H$ must be left Kan extended from $\J$.)
In these two cases, we define
\begin{align*}
  \rG^{\J,~\I\oplus\H} &\equiv \bkt{\rG^{\J,\I},~\A_{\cred{\star}} : \rT^\H \to \Type_\ell}\\
  \rG^{\J\oplus\H,~\I\oplus\H} &\equiv \rG^{\J,\I} \quad\text{ weakened to }\rG^{\J\oplus \H}.
\end{align*}

This completes the construction of the classifying context.
Note in particular that a consequence of \ref{item:ods3} is that re-ordering the elements of a presheaf $\H$ modifies $\rT^\H$ only up to isomorphism.

\paragraph{The classifying context is classifying.}

To show this, we first construct a `universal' diagram over $\rG^\I$.
Specifically, in any category with families, we construct simultaneously:
\begin{enumerate}
\item For each ordered direct category $\I$, a Reedy type $\rB$ of shape $\I$ and level $\ell$ over $\rG^\I$ in the sense of~\cite[Definition 3.22]{kl:hoinvdia}.\label{item:os1}
\item For each ordered presheaf $\H$ on $\I$, the object $\rT^\H$ is the canonical $\H$-weighted limit of $\rB$ constructed by the `master lemma` of~\cite[Lemma 3.11]{kl:hoinvdia}.\label{item:os2}
  (In particular, therefore, $\rT^{\yonp_\x}$ is the matching object of $\rB$ at $\x$.)
\item The maps $\rt^\alpha$ are the functorial action of these limits.\label{item:os3}
\item The type $\rG^\I \ext \rT^{\yon_x} \vdash_\sm \rB^\x$ is the object $\rB(\x)$ with its fibration to the matching object $\cred{M}_\x \rB = \rT^{\yonp_x}$.\label{item:os4}
\item The elements $\rb^\x$ are the projections from the weighted limit $\rT^\H$.\label{item:os5}
\end{enumerate}
The interesting case is \ref{item:os1}, where we weaken the Reedy $\I$-type $\rB$ over $\rG^\I$ to $\rG^{\I\oplus\H}$ and then must extend it to a Reedy $(\I\oplus\H)$-type by giving a type over the matching object $\rT^{\yonp_{\cred{\star}}}$.
But $\yon_{\cred{\star}}=\H$ (weakened to $\I\oplus\H$), and so we can use $\El~\A^\H$ where $\A^\H$ is the newly added variable in $\rG^{\I\oplus\H}$.
The other parts follow essentially tautologically.

Lastly, suppose $\C$ is a Reedy $\I$-type over any context $\Delta$, and suppose that it is `$\ell$-small' in the sense that each type $\C(\x)$ over the matching object $\cred{M}_\x \C$ is classified by a specified map into the universe $\cred{M}_\x\C \to \Type_\ell$.
We show by induction on $\I$ that it is classified by a unique map $\c:\Delta \to \rG^\I$ such that $\rB\sqbkt{\c} \equiv \C$.
This is trivial when $\I$ is empty.
Assuming it to be true for $\I$, if $\C$ is a Reedy $(\I\oplus \H)$-type over $\Delta$, and its restriction to an $\I$-type is classified by a map $\c : \Delta \to \rG^\I$, then to extend this to a map into $\rG^{\I\oplus \H}$ we must give a term in context $\Delta$ of type $\rT^\H\sqbkt{\c}\to \Type_\ell$.
But matching objects are preserved by substitution, so $\rT^\H\sqbkt{\c}$ is the matching object $\cred{M}_{\cred{\star}}\C$, and thus this is exactly the data extending $\C$ to a Reedy $(\I\oplus \H)$-type.

\paragraph{Display and d\'ecalage of classifying contexts.}

\def\Two{\cred{\mathbb{2}}}
Since the data $\rG^\I$, $\rT^\H$, $\rB^\x$, and so on are concrete finite syntactic objects (for any fixed $\I,\H,\x$ and so on), the rules in \cref{sec:comp-tel-dec,sec:comp-telesc-displ,sec:comp-metaabsdec,sec:comp-matd} suffice to completely compute display and d\'ecalage on them.
We can characterize the results as follows.

\def\rxi{\cred{\oldxi}}
Let $\Two$ denote the interval category $(\cred{0}\xrightarrow{\rxi} \cred{1})$, with two objects and one nonidentity morphism $\rxi$ between them.
Then if $\I$ is a direct category, so is $\Two\times \I$ with the product well-ordering where $\x\prec \y$ yields $(\cred{0},\x) \prec (\cred{1},\x) \prec (\cred{0},\y) \prec (\cred{1},\y)$.
If $\I$ is ordered, we make $\Two\times \I$ ordered as follows.
The morphisms in $\Two\times \I$ with codomain $(\cred{0},\x)$ are bijective to those in $\I$ with codomain $\x$, so we inherit that ordering.
And the morphisms in $\Two\times \I$ with codomain $(\cred{1},\x)$ are two copies of those in $\I$ with codomain $\x$, one copy indexed by $\id_\cred{1}$ and one by $\rxi$, so we give them the product well-ordering where $\rxi\prec \id_\cred{1}$: thus from $\g\prec \f$ we have $(\rxi,\g) \prec (\id_\cred{1},\g) \prec (\rxi,\f) \prec (\id_\cred{1},\f)$.

\def\rp{\cred{p}}
There is an evident projection $\rp: \Two\times \I \to \I$, inducing by precomposition from any presheaf $\H$ on $\I$ a presheaf $\rp^*\H$ on $\Two\times \I$.
Each element of $\H$, say $\h\in \H(\x)$, then induces two elements $\rp^*\H$ in $\H((\cred{0},\x))$ and $\H((\cred{1},\x))$; we denote these $(\cred{0},\h)$ and $(\cred{1},\h)$ respectively for clarity.
If $\H$ is ordered, we induce an ordering on $\rp^*\H$ by ordering each element of $\H((\cred{0},\x))$ before the corresponding element of $\H((\cred{1},\x))$ (which is necessary, since $\xi$ maps the latter to the former).

We also have an inclusion $\cred{i_0} : \I \to \Two\times \I$ defined by $\cred{i_0}(\x) = (\cred{0},\x)$, which is a sieve.
Left Kan extending along this inclusion takes a presheaf $\H$ on $\I$ to a presheaf $(\cred{i_0})_!\H$ on $\Two\times \I$ that is supported only on the objects of the form $(\cred{0},\x)$ and has exactly the same elements, hence inherits an ordering as well.

Now we prove by simultaneous induction:
\begin{enumerate}
\item For any $\I$, we have $(\rG^\I)\D \equiv \rG^{\Two\times\I}$.\label{item:dcl1}
\item For any presheaf $\H$ on $\I$, we have $(\rT^\H)\D \equiv \rT^{\rp^*\H}$ and $\gamma\cblu{'} : \rG^{\Two\times \I} \vdash_\sm \rT^{(\cred{i_0})_!\H}~\gamma\cblu{'} \equiv \rT^\H~\gamma\cblu{'}\ev$.\label{item:dcl2}
\item This identification is functorial in maps $\alpha : \H\to \H\cblu{'}$.\label{item:dcl3}
\item For each $\x\in \I$, we have:\label{item:dcl4}
  \begin{alignat*}{2}
    \gamma\cblu{'} : \rG^{\Two\times \I} &\vdash_\sm&\; \rT^{\yonp_{(\cred{0},\x)}}~\gamma\cblu{'} &\equiv \rT^{\yonp_\x}~\gamma\cblu{'}\ev\\
    \gamma\cblu{'} : \rG^{\Two\times \I},~ \y : \rT^{\yonp_{(\cred{0},\x)}}~\gamma\cblu{'} &\vdash_\sm& \rB^{(\cred{0},\x)}~\gamma\cblu{'}~\y &\equiv \rB^\x~\gamma\cblu{'}\ev~\y\\
    && \yonp_{(\cred{1},\x)} &= \rp^*\yonp_\x \cup \{ (\xi,\id_\x) \}\\
    \gamma\cblu{'} : \rG^{\Two\times \I} &\vdash_\sm & \rT^{\yonp_{(\cred{1},\x)}}~\gamma\cblu{'} &\equiv \bkt{\y\cblu{'}:(\rT^{\yonp_\x})\D~\gamma\cblu{'}, \a_{(\cred{0},\x)} : \rB^{\x}~\gamma\cblu{'}\ev~\y\cblu{'}\ev}\\
    \gamma\cblu{'} : \rG^{\Two\times \I},~\y\cblu{'} : \rT^{\yonp_{(\cred{1},\x)}}~\gamma\cblu{'} &\vdash_\sm&\; \rB^{(\cred{1},\x)}~\gamma\cblu{'}~\y\cblu{'} &\equiv (\rB^\x)\d~\gamma\cblu{'}~\y\cblu{'}
  \end{alignat*}
\item For each $\h\in \H(\x)$, we have:\label{item:dcl5}
  \begin{alignat*}{2}
    \gamma\cblu{'} : \rG^{\Two\times \I},~ \y\cblu{'} : \rT^{\rp^*\H} &\vdash_\sm&\; \rb^{(\cred{0},\h)}~\gamma\cblu{'}~\y\cblu{'} &\equiv \rb^\h~\gamma\cblu{'}\ev~\y\cblu{'}\ev\\
    \gamma\cblu{'} : \rG^{\Two\times \I},~ \y\cblu{'} : \rT^{\rp^*\H} &\vdash_\sm&\; \rb^{(\cred{1},\h)}~\gamma\cblu{'}~\y\cblu{'} &\equiv (\rb^\h)\d~\gamma\cblu{'}~\y\cblu{'}
  \end{alignat*}
\end{enumerate}
For the inductive step of~\ref{item:dcl1}, we have
\[\Two\times (\I \oplus \H) = (\Two\times \I) \oplus (\cred{i_0})_!\H \oplus (\rp^*\H \cup \{(\xi,\id_\star)\}).\]
Thus, using~\ref{item:dcl2}, we have
\begin{align*}
  \MoveEqLeft \rG^{\Two\times (\I\oplus \H)}\\
  &\equiv
  \bkt{\gamma\cblu{'} : \rG^{(\Two\times \I)},~
    \A_{(\cred{0},\star)} : \rT^{(\cred{i_0})_!\H} \to \Type_\ell,~
    \A_{(\cred{1},\star)} : \rT^{(\rp^*\H \cup \{(\xi,\id_\star)\})} \to \Type_\ell}\\
  &\equiv
  \bkt{\gamma\cblu{'} : (\rG^{\I})\D,~
    \A_{(\cred{0},\star)} : \rT^\H~\gamma\cblu{'}\ev \to \Type_\ell,~
    \A_{(\cred{1},\star)} : (\y\cblu{'} : (\rT^{\H})\D~\gamma\cblu{'}) \to \A_{(\cred{0},\star)}~\y\cblu{'}\ev \to \Type_\ell}\\
  &\equiv
  \bkt{\gamma : \rG^\I,~
    \A_\star : \rT^\H~\gamma \to \Type_\ell}\D.
\end{align*}
The other cases are similar.
We can likewise show that
\[ \rG^{\I, \Two\times \I} \equiv (\rG^\I)\d, \]
with the isomorphism $\rG^{\Two\times \I} \cong \bkt{\rG^\I\ext \rG^{\I, \Two\times \I}}$ coinciding with the evens/odds pairing isomorphism $(\rG^{\I})\D \cong \bkt{\rG^\I\ext (\rG^\I)\d}$.

\paragraph{Discrete fibrations.}

The isomorphism $\rG^\I\cong \bkt{\rG^\J\ext \rG^{\J,\I}}$ ensures that if $\J\subseteq \I$ is a sieve, we have a weakening substitution $\rG^\I \to \rG^\J$.
But more generally, we can expect to induce a context substitution from any discrete fibration.
Even more generally, we can get a \emph{partial} substitution from a `dependent' discrete fibration, in the following sense.

\begin{definition}
  If $\cblu{i} : \J\hookrightarrow \I$ is the inclusion of a sieve in a direct category, a \textbf{co-section} of it is a discrete fibration $\cblu{p} : \I \to \J$ such that $\cblu{p}\circ \cblu{i} = \id_\J$.
  In this case, if $\H$ is a presheaf on $\I$ and $\K$ a presheaf on $\J$, a morphism $\H \to \K$ over $\cblu{p}$ is a \textbf{relative isomorphism} if it induces a bijection $\sum_{\y\in \I} \H(\y) \to \sum_{\y\in \J} \K(\y)$.
\end{definition}

Note that the projection $\rp: \Two\times \I \to \I$ above is \emph{not} a co-section of the sieve $\cred{i_0} : \I \hookrightarrow \Two\times \I$, since it is not a discrete fibration.
The prototypical example of a relative isomorphism is $\yonp_\x \to \yonp_{\cblu{p}(\x)}$ for any $\x\in \I$ (this is essentially the definition of a discrete fibration).

Now we define and prove inductively:
\begin{enumerate}
\item For any co-section $\cblu{p} : \I \to \J$ of a sieve $\cblu{i} : \J\hookrightarrow \I$ in an ordered direct category, a partial substitution $\rG^\J \vdash_\sm \rg^{\cblu{p}} : \rG^{\J,\I}$.\label{item:cs1}
\item In addition, for any order-preserving relative isomorphism $\H \to \K$ between ordered presheaves, we have $\rT^\H\sqbkt{\rg^{\cblu{p}}} \equiv \rT^\K$.\label{item:cs2}
\item For $\x\in \I$, we have $\rB^\x\sqbkt{\rg^{\cblu{p}}} \equiv \rB^{\cblu{p}(\x)}$.\label{item:cs4}
\item For $\alpha : \H \to \K$ an order-preserving relative isomorphism and $\h \in \H(x)$, we have $\rb^{\x}\sqbkt{\rg^{\cblu{p}}} \equiv \rb^{\alpha(\x)}$.\label{item:cs5}
\end{enumerate}

To construct \ref{item:cs1}, note that as before there are two possibilities for a sieve in $\I\oplus\H$: it can be $\J$ or $\J\oplus\H$ for a sieve $\J$ in $\I$.
In the latter case, we have $\rG^{\J\oplus\H,~\I\oplus\H} = \rG^{\J,\I}$ weakened to $\rG^{\J\oplus \H}$, and a co-section of $\J\oplus \H \hookrightarrow \I\oplus \H$ is determined by a co-section $\cblu{p}$ of $\J\subseteq \I$; thus we can similarly weaken $\rg^{\cblu{p}}$.

In the former case, a co-section $\I\oplus\H \to \J$ is determined by a co-section $\cblu{p} : \I\to\J$ together with an object $\x\in\J$ and a relative isomorphism $\H \to \yonp_\x$.
Since $\rG^{\J,~\I\oplus\H} = \bkt{\rG^{\J,\I},~\A_{\cred{\star}} : \rT^\H \to \Type_\ell}$ in this case, to extend $\rg^{\cblu{p}}$ as desired it suffices to give a term of type $\rG^\J \vdash_\sm \rT^\H\sqbkt{\rg^{\cblu{p}}} \to \Type_\ell$.
But using \ref{item:cs2} inductively, this is equal to $\rG^\J \vdash_\sm \rT^{\yonp_\x} \to \Type_\ell$, so we can use the variable $\A_\x$ in $\rG^\J$.

Now to prove~\ref{item:cs2}, we induct on the ordering of $\H$ and $\K$, inductively using~\ref{item:cs4}.
The inductive arguments for \ref{item:cs4}--\ref{item:cs5} are similar.

\paragraph{Categorical coning.}

\def\rz{\cred{\oldzeta}}
Our last generic construction is a category-theoretic notion of `coning' a direct category.
Let $\J\subseteq \I$ be a sieve in a direct category that contains the bottom object, which we presciently denote $\angle{\cred0}$.
Let $\I^{\cred{+}}$ denote the direct category $\I$ augmented by an additional morphism $\rz_\x : \angle{\cred0} \to \x$ for all objects $\x\in \I\setminus \J$.\footnote{The notation is somewhat abusive, since the construction depends on $\J$ as well as $\I$.}
We define $\f\circ \rz_\x = \rz_\y$ for all $\f : \x\to\y$; note that $\x\in \I\setminus \J$ implies $\y\in \I\setminus \J$ since $\J$ is a sieve.
If $\I$ is ordered, we order $\I^{\cred{+}}$ by placing $\rz_\x$ \emph{before} all other morphisms with codomain $\x$; this is actually the only possibility given our definition of composition.
Note that $\J$ is still a sieve in $\I^{\cred+}$.

Similarly, for a presheaf $\H$ on $\I$, let $\H^{\cred+}$ denote the presheaf on $\I^{\cred+}$ consisting of $\H$ augmented by a new element $\rz_\H \in \H(\angle{\cred{0}})$, such that $\H^{\cred+}(\rz_\x)(\h) = \rz_\H$ for all $\h\in\H(\x)$.
If $\H$ is ordered, we order $\H^{\cred+}$ by putting $\rz_\H$ first.

We now inductively prove:
\begin{enumerate}
\item For any sieve $\J\subseteq \I$ in an ordered direct category, we have $\rG^{\J,\I^{\cred+}} \equiv (\z:\rB^{\angle{\cred0}}) \to \rG^{\J,\I}$ (meaning a $\Pi$-telescope).
\item In addition, for any $\H$ on $\I$, if we transfer $\rT^\H$ and $\rT^{\H^{\cred+}}$ across the isomorphisms
  \begin{align*}
    \rG^{\I} &\cong \bkt{\gamma : \rG^\J \ext \delta: \rG^{\J,\I}~\gamma}\\
    \rG^{\I^{\cred+}} &\cong \bkt{\gamma : \rG^\J \ext \delta: \rG^{\J,\I^{\cred+}}~\gamma}
    \equiv \bkt{\gamma : \rG^\J \ext \delta: (\z:\rB^{\angle{\cred0}}~\gamma) \to \rG^{\J,\I}~\gamma}
  \end{align*}
  to get $\smash{\cred{\tilde{\rT}}}^\H$ and $\smash{\cred{\tilde{\rT}}}^{\H^{\cred+}}$, then we have
  \[\smash{\cred{\tilde{\rT}}}^{\H^{\cred+}}~\gamma~\delta \equiv \bkt{\z:\rB^{\angle{\cred0}}~\gamma,~\smash{\cred{\tilde{\rT}}}^\H~\gamma~\bkt{\delta~\z}} \]
\end{enumerate}
Both proofs are entirely straightforward, using the inductive definition of $\Pi$-telescopes as well as $\rG^\I$ and $\rT^\H$.

\paragraph{Correctness of semi-simplicial types.}

Recall that our definition of semi-simplicial types $\SST_{\ell}$ is as a displayed coinductive type with $\Phi\equiv\ec$, $\A\equiv\Type_\ell$, and $\MB~\a \equiv \bkt{x:\El~\a}$.
Therefore, the construction in \cref{sec:displ-coind-types} simplifies as follows:
\begin{mathpar}
  \vdash_\sm \X^{\cred{\partial}\n} \tel
  \and
  \cblu{\partial}\x : \X^{\cred{\partial}\n} \vdash_\sm \X^\n~\cblu{\partial}\x \tel
  \and
  \vdash_\sm \X^{\cred{\partial0}} \equiv \ec
  \and
  \vdash_\sm \X^{\cred{\partial}(\n+\cred{1})} \equiv (\cblu{\partial}\x:\X^{\cred{\partial}\n} ,~ \x : \X^\n~\cblu{\partial}\x)
  \and
  \vdash_\sm \X^\cred{0} \equiv \Type_\ell
  \and
  \cblu{\partial}\x : \X^{\cred{\partial}\n},~ \x : \X^\n~\cblu{\partial}\x
  \vdash_\sm \h_\n~\cblu{\partial}\x~\x : \Type_\ell
  \and
  \x : \X^{\cred{0}} \vdash_\sm \h_{\cred{0}}~\x \equiv \x
  \and
  \cblu{\partial}\x : \X^{\cred{\partial}\n},~ \x : \X^\n~\cblu{\partial}\x,~\x\cblu{'} : \X^{\n+\cred{1}}~\sqbkt{\cblu{\partial}\x,\x}
  \vdash_\sm \h_{\n+\cred{1}}~\sqbkt{\cblu{\partial}\x,~\x}~\x\cblu{'} \equiv \h_\n~\cblu{\partial}\x~\x
  \and
  \cblu{\partial}\x : \X^{\cred{\partial}\n},~ \x : \X^\n~\cblu{\partial}\x, ~ \b :\El~(\h_\n~\cblu{\partial}\x~\x)
  \vdash_\sm \t_\n~\cblu{\partial}\x~\x~\b : (\X^{\cred{\partial}\n})\d~\cblu{\partial}\x
  \and
  \x : \X^{\cred{0}}, ~ \b :\El~(\h_{\cred{0}}~\x)
  \vdash_\sm \t_{\cred{0}}~\x~\b \equiv \sqbkt{}
\end{mathpar}
\begin{equation*}
  \cblu{\partial}\x : \X^{\cred{\partial}\n},~ \x : \X^\n~\cblu{\partial}\x,~\x\cblu{'} : \X^{\n+\cred{1}}~\sqbkt{\cblu{\partial}\x,\x},~\b :\El~(\h_\n~\cblu{\partial}\x~\x)
  \vdash_\sm
  \t_{\n+\cred{1}}~\sqbkt{\cblu{\partial}\x,\x}~\x\cblu{'}~\b
  \equiv
  \sqbkt{\t_\n~\cblu{\partial}\x~\x~\b,~\x\cblu{'}~\b }
\end{equation*}
\begin{align*}
\cblu{\partial}\x : \X^{\cred{\partial}\n},~\x:\X^\n~\cblu{\partial}\x &\vdash_\sm
  \X^{\n+\cred{1}}~\sqbkt{\cblu{\partial}\x,~\x} \equiv
    (\b : \El~(\h_{\n-\cred{1}}~\cblu{\partial}\x)) \to  (\X^\n)\d~\pair{\cblu{\partial}\x}{\t_\n~\cblu{\partial}\x~\x~\b}~\x
\end{align*}
We will prove inductively that
\begin{equation*}
  \X^{\cred{\partial}\n} \equiv \rG^{\cred{\oldDelta_{\n\minusone}}}
  \quad\text{and}\quad
  \X^\n \equiv \rT^{\yonp_{\angle{\n}}} \to \Type_\ell.
\end{equation*}
This will imply that $\SST = \lim_\n \X^{\cred{\partial}\n}$ is a classifying context for all of $\cred{\oldDelta}$.
The claim about $\X^\n$ clearly inductively implies the claim about $\X^{\cred{\partial}\n}$.
Also it is easy to show inductively that $\h_\n \equiv \rB^{\angle{\cred0}}$.
So it remains to say something useful about $\t_\n$.

\def\rI{\cred{I}}
\def\rJ{\cred{J}}
Let $\rI_\n$ be the subcategory of $\Two \times \cred{\oldDelta_{\n}}$ containing all the objects except $(\cred{1},\angle{\n})$, and let $\rJ_\n = \{\cred{0}\} \times \cred{\oldDelta_{\n}}$ regarded as a sieve in $\rI_\n$.
The central fact is the following.

\def\rq{\cred{q}}
\begin{lemma}
  For any $\n$, there is a co-section $\rq_\n : \rI_\n^{\cred+} \to \rJ_\n$.
\end{lemma}
\begin{proof}
  On objects, let $\rq_\n((\cred{1},\angle{\k})) = (\cred{0},\angle{\k+\cred1})$ for $\cred{0}\le \k< \n$.
  A morphism $(\cred{1},\angle{\k}) \to (\cred{1},\angle{\l})$ is a length $\l+\cred{1}$ sequence with $\k+\cred{1}$ $\One$s, and we augment it by another $\One$ on the \emph{right} to get a length $\l+\cred{2}$ sequence with $\k+\cred{2}$ $\One$s, hence a morphism $(\cred{0},\angle{\k+1}) \to (\cred{0},\angle{\l+1})$.
  A morphism $(\cred{0},\angle{\k}) \to (\cred{1},\angle{\l})$ is also a length $\l+\cred{1}$ sequence with $\k+\cred{1}$ $\One$s, but this time we augment it by a $\Zero$ on the right to get a length $\l+\cred{2}$ sequence with $\k+\cred{1}$ $\One$s, hence a morphism $(\cred{0},\angle{\k}) \to (\cred{0},\angle{\l+1})$.
  Finally, we send the new morphism $\rz_{(\cred{1},\angle{\l})}$ to the sequence of $\l+\cred{1}$ $\Zero$s followed by one $\One$.
  Functoriality is easy to check.
  And to see that it is a discrete fibration, we observe that any binary sequence of length $\l+\cred{2}$ with a positive number of $\One$s must be of exactly one of these three forms: a positive number of $\One$s followed by a $\One$, a positive number of $\One$s followed by a $\Zero$, or a sequence of $\Zero$s followed by a $\One$.
\end{proof}

Evidently $\rq_{\n+\cred{1}}$ restricts to $\rq_\n$ as we shrink the categories.
Thus, we also get a relative isomorphism $\yonp_{(\cred{1},\angle{\n})}^{\cred+} \to \yonp_{(\cred{0},\angle{\n+\cred{1}})}$ over $\rq_\n$.

Now note that if we abstract over $\b$, the type of $\t_\n$ matches that of $\rg^{\rq_\n}$.
Thus, we can now prove by simultaneous induction that:
\begin{enumerate}
\item $\X^{\cred{\partial}\n} \equiv \rG^{\cred{\oldDelta_{\n\minusone}}}$.\label{item:css1}
\item $\X^\n \equiv \rT^{\yonp_{\angle{\n}}} \to \Type_\ell$.\label{item:css2}
\item $\h_\n \equiv \rB^{\angle{\cred{0}}}$.\label{item:css3}
\item $\t_\n \equiv \rg^{\rq_\n}$.\label{item:css4}
\end{enumerate}
We have already remarked that \ref{item:css1} and~\ref{item:css3} are easy, and the base cases of \ref{item:css2} and~\ref{item:css4} are likewise trivial.
For the induction step of~\ref{item:css2}, we have
\begin{align*}
  \X^{\n+\cred{1}}~\sqbkt{\cblu{\partial}\x,~\x}
  &\equiv
  (\b : \El~(\h_{\n-\cred{1}}~\cblu{\partial}\x)) \to  (\X^\n)\d~\pair{\cblu{\partial}\x}{\t_\n~\cblu{\partial}\x~\x~\b}~\x\\
  &\equiv (\b : \rB^{\angle{\cred{0}}}~\cblu{\partial}\x) \to \rT^{\yonp_{(\cred{1},\angle{\n})}}~\pair{\cblu{\partial}\x}{\t_\n~\cblu{\partial}\x~\x~\b}~\x \to \Type_\ell\\
  &\equiv \rT^{\yonp_{(\cred{1},\angle{\n})}^{\cred+}}~\sqbkt{\rg^{\rq_\n}} \to \Type_\ell\\
  &\equiv \rT^{\yonp_{(\cred{1},\angle{\n+\cred{1}})}}~\sqbkt{\cblu{\partial}\x,~\x} \to \Type_\ell.
\end{align*}
Finally, the induction step of~\ref{item:css4} follows immediately from the definition of $\rg^\p$ and the inductive hypothesis of~\ref{item:css2}.
This completes the proof of the correctness of our construction of semi-simplicial types.\bbox

\section{Conclusion and Future Work}
\label{sec:conclusion}

In this paper we have made two main contributions.
First, we have described \emph{Displayed Type Theory (dTT)}, a new kind of type theory that incorporates (unary) internal parametricity but guarded by a modality, and showed that any model of dependent type theory with countable Reedy limits can be lifted to a model of dTT using augmented semi-simplicial diagrams.
Because the latter are diagrams on an \emph{inverse} category, their type theory is more closely related to that of the original model, and indeed the original model sits inside our model of dTT at the discrete mode.
In particular, unlike other internally parametric type theories, dTT is compatible with classical axioms such as excluded middle and choice, as long as they are formulated at the discrete mode (or under the modality $\dia$), and can be used as an internal logic to reason about arbitrary $(\infty,1)$-toposes.

Secondly, inside dTT we have introduced a notion of \emph{displayed coinductive type}, where the output of a destructor can be a parametricity `computability witness' of the input, and showed that as a particular case of this notion we can define a type of \emph{semi-simplicial types}.
This yields a new approach to the long-standing open problem of representing infinitely coherent higher structures in type theory.
Relative to other approaches, ours has the advantage that semi-simplicial types are defined (not postulated) as a simple instance of a type-former with natural introduction and elimination rules, i.e.\ a categorical universal property.
While it remains to be seen how much can actually be done in practice with our definition, early indications of its utility are promising.

There are a number of directions for future work suggested by our results; here we survey a few of them briefly.

\vspace{-0.2cm}

\paragraph{Computation and implementation.}

We conjecture that dTT satisfies canonicity and normalization, and should therefore be possible to implement in a proof assistant.

\vspace{-0.2cm}

\paragraph{Modal internal parametricity.}

We expect that most applications of ordinary internal parametricity have modal versions that can be proven in dTT (or a higher-ary version of it, as discussed below).
In addition to the traditional `free theorems' such as those mentioned in the introduction, it would be especially interesting to investigate this for the proof of the pentagon identity for the smash product in~\cite{cavallo:thesis}, since such a proof would then apply internally to any $(\infty,1)$-topos.

\vspace{-0.2cm}

\paragraph{Computing diamond.} Our intended semantic model strictly validates computation rules for diamond, such as $\dia~\bkt{\A \to \B} \equiv \dia~\A \to \dia~\B$, because type formers in the simplicial model extend type formers in the discrete model at the level of $(\minusone)$-simplices. The syntax of dTT may thus be augmented with computation rules for diamond. Formally, this would be accomplished in a manner similar to display, by extending the definition of diamond to meta-abstractions, such as to handle open terms under binders. One would introduce an operation, \emph{troncature} (meaning \emph{`truncation'} in French), to model the action of diamond on telescopes. One then has:
\begin{mathpar}
    \infer{\Gamma,~\lock_\dia \vdash_\sm \Upsilon \tel_{\;\!\ell}}{\Gamma \vdash_\dm \ddia~\Upsilon \tel_{\;\!\ell}}
    \and
    \infer{\Gamma,~\lock_\dia \vdash_\sm \cA \slfrac{\type_{\;\!\ell_\cblu{1}}}{_{\upsilon\;:\;\Upsilon}}}{\Gamma \vdash_\dm \dia~\cA \slfrac{\type_{\;\!\ell_\cblu{1}}}{_{\upsilon^\cblu{\oldflat}\;:\;\ddia\;\Upsilon}}}
\end{mathpar}

\vspace{-0.6cm}

\paragraph{Higher category theory.}

We have \emph{defined} a type of semi-simplicial types in dTT, but such a definition is not an end in itself; it is intended as a tool for developing a theory of higher categories and other higher structures.
We hope that our corecursion principle of $\SST$, and the availability of other displayed coinductive types for predicates and structures on them, should make the development of such a theory feasible in dTT.
We sketched some initial ideas in \cref{sec:eg-sst,sec:exampl-displ-coind}, but much remains to be done.

\vspace{-0.2cm}

\paragraph{Elementary models.}

In \cref{sec:semantics} we showed that any model of type theory with $\oldomega$-limits can be enhanced to a model of dTT including $\SST$, but this is probably not the only way to construct models of dTT.
In particular, we conjecture that there are `realizability' models of dTT in which $\SST$ is a classifier for `uniform' semi-simplicial types.
This suggests that perhaps displayed coinductive types might be useful to include in a definition of elementary $(\infty,1)$-topos.

\vspace{-0.2cm}

\paragraph{Higher-ary dTT.}
\label{sec:higher-ary-dtt}

The parametricity of dTT is \emph{unary}, meaning that $\A\d$ depends on one copy of $\A$; but parametricity in general can be $\n$-ary for any natural number $\n$, and some applications require higher arities.
We expect that higher-ary versions of dTT can be defined and modeled by a straightforward modification of the constructions in this paper, using higher-ary semi-cubical types in place of augmented semi-simplicial types (recall that augmented semi-simplicial types are the same as unary semi-cubical types).
In this case the binary numbers described in \cref{sec:asscat} become base-$(\n+\cred{1})$ numbers.

\vspace{-0.2cm}

\paragraph{Symmetries.}
\label{sec:symmetries}

In theories with non-modal internal parametricity, and also in Higher Observational Type Theory, it appears to be necessary to include a \emph{`symmetry'} operation on higher-dimensional types.
For instance, in our notation a symmetry operation would have the type $\A\d\d~\cblu{a\sub{00}}~\cblu{a\sub{01}}~\cblu{a\sub{10}}\to \A\d\d~\cblu{a\sub{00}}~\cblu{a\sub{10}}~\cblu{a\sub{01}}$.
The absence of symmetry in dTT is a significant simplification; for instance, it means that $\cred{\oldDelta^\cred{+}}\op$ is a strict inverse category, making possible the explicit syntactic model construction in \cref{sec:simplicial-model}.
However, we have also seen that it leads to certain limitations, e.g.\ without symmetry it is unclear how to give a corecursion principle for $\SST\d$.

It should be possible to formulate a version of dTT (unary or higher-ary) with symmetry, but in the presence of symmetry it is unclear whether it is possible for display to compute definitionally on type-formers.
However, it should work to use either the interval-based style of~\cite{bm:parametricity,bcm:pshf-parametric,moulin2016internalizing} or the `observational' style of~\cite{acks:iparam-noint}.

\vspace{-0.2cm}

\paragraph{Unimode dTT.}

We have formulated dTT with two modes, but intuitively the discrete mode is unnecessary, as the $\dm$-types are embedded in the $\sm$-types by the modality $\tri$.
Thus, it should be possible to formulate a version of dTT in which there is only one mode.
This is similar to other situations such as spatial/cohesive type theory~\cite{shulman:bfp-realcohesion} and synthetic guarded domain theory~\cite{gknb:mtt} that have both unimodal and bimodal versions.

\vspace{-0.2cm}

\paragraph{Conjectural syntax.}
\label{subsec:conjectural-syntax}

In addition to displayed coinductive types, one may consider other kinds of generalized inductive and coinductive types.
These are especially useful when taking a more `synthetic' approach to higher structures in dTT, using the $\sm$-types as augmented semi-simplicial objects rather than working with the internally defined type $\SST$ of semi-simplicial types.

Firstly, regarding display as analogous to paths in homotopy type theory suggests \emph{displayed inductive types} as analogues of higher inductive types.
Here the constructors generate displayed elements rather than ordinary ones.
%
As an example, we can construct the simplicial cone of any type:\\
\begin{minipage}{\textwidth}
\begin{lstlisting}
(*\textbf{data}*) (*$\cred{\mathsf{C}}$*) ((*$\A$*) :(*$^\trisq$*) (*$\Type$*)) : (*$\Type$*) (*\textbf{where}*)
  (*$\cred{\iota}$*) : (*$\A$*) (*$\to$*) (*$\cred{\mathsf{C}}$*) (*$\A$*)
  (*$\cred{\oldsigma}$*) : ((*$\x$*) : (*$\A$*)) (*$\to$*) ((*$\cred{\mathsf{C}}$*) (*$\A$*))(*$\d$*) ((*$\cred{\iota}$*) (*$\x$*))

(*$\f$*) : (*$\cred{\mathsf{C}}$*) (*$\A$*) (*$\to$*) (*$\B$*)
(*$\f$*) ((*$\cred{\iota}$*) (*$\x$*)) = (*\cprime{?$_\iota$}*) : (*$\B$*)
(*$\f\d$*) ((*$\cred{\iota}$*) (*$\x$*)) ((*$\cred{\oldsigma}$*) (*$\x$*)) = (*\cprime{?$_\oldsigma$}*) : (*$\B\d$*) (*\cprime{?$_\iota$}*)
\end{lstlisting}
\end{minipage}

Secondly, regarding both display and paths as a kind of modality suggests considering more general \emph{modal inductive types}, whose constructors can land in modal versions of the type.
For instance, since $\dia\A$ is the $(\minusone)$-simplices of $\A$, a $\dia$-modal constructor adds a $(\minusone)$-simplex without any higher simplices above it.
In this way we can construct the free-living $(\minusone)$-simplex, and then all the higher simplices by coning:
\begin{lstlisting}
(*\textbf{data}*) (*$\cred{\Delta^{\minusone}}$*) : (*$\Type$*) (*\textbf{where}*)
  (*$\cred{\star}$*) : (*$\dia$*) (*$\cred{\Delta^{\minusone}}$*)

(*$\f$*) : (*$\cred{\Delta^{\minusone}}$*) (*$\to$*) (*$\A$*)
(*$\dia$*) (*$\f$*) (*$\cred{\star}$*) = (*\cprime{?$_\star$}*) : (*$\dia$*) (*$\A$*)

(*$\cred{\Delta}$*) : (*$\cred{\mathbb{N}}$*) (*$\to$*) (*$\Type$*)
(*$\cred{\Delta}$*) (*$\cred{\mathsf{zero}}$*) = (*$\cred{\mathsf{C}}$*) (*$\cred{\Delta^{\minusone}}$*)
(*$\cred{\Delta}$*) ((*$\cred{\mathsf{suc}}$*) (*$\n$*)) = (*$\cred{\mathsf{C}}$*) ((*$\cred{\Delta}$*) (*$\n$*))
\end{lstlisting}

Note that in both cases, we rely on the computation behaviour of $\d$ and $\dia$ in order to directly give induction principles.\footnote{The case splits for defining a function $\f$ valued out of an inductive type in this hypothetical extension of \texttt{Agda}, as above, would result automatically by writing $\f~\x~=~?$ and pressing \texttt{C-c C-c} in the context of the hole.} For example, the pattern match $\dia~\f~\a$ requires that $\dia$ of a function type compute to a function type. This is an improvement on the treatment of higher inductive types in~\cite{hottbook}, since their natural elimination principle use $\mathsf{\cred{ap}}$, which is not there a primitive constant but a compound expression defined by path induction.
In particular, we conjecture that in dTT, these displayed and modal inductive types can be made fully computational.

\bibliographystyle{alpha}
\bibliography{main}

\clearpage

\appendix
\section{Verifications for the Simplicial Model}
\label{sec:simplicial-proofs}

\subsection{Variables}
\label{sec:appendix:variables}

We first check that $\ft$ is a morphism of presheaves:
\begin{align*}
\MoveEqLeft\Big(\Gamma^{\b} \circ \bkt{\ft^{\A}_{\sm^{\n+\cred{1}}}}_{\n+\cred{1}}\Big) ~\gamma_{\cblu{\n+1}} ~\cblu{\partial}\a~\a \\
&\equiv \Gamma^\b~\gamma_{\cblu{\n+1}} \\
&\equiv \bkt{\ft^{\A}_{\sm^{\n+\cred{1}}}}_{\m+\cred{1}} ~\big(\Gamma^\b~\gamma_{\cblu{\n+1}}\big) ~\big(\nat^{\pi\A}_{\cred{\partial}\b}~\gamma_{\cblu{\n+1}} ~\cblu{\partial}\a\big) ~\big(\nat^{\A}_\b~\gamma_{\cblu{\n+1}} ~\cblu{\partial}\a~\a\big) \\
&\equiv \Big(\bkt{\ft^{\A}_{\sm^{\n+\cred{1}}}}_{\m+\cred{1}} \circ \bkt{\gamma : \Gamma, ~\a : \A~\gamma}^{\b}\Big) ~\gamma_{\cblu{\n+1}}~\cblu{\partial}\a~\a.
\end{align*}
We now verify \cref{eq:pt-zv} at the level of $\bkt{\n+\cred{1}}$-simplices:
\begin{align*}
\MoveEqLeft\Big(\bkt{\zv^{\pi\A}_{\sm^{\n+\cred{1}}}}^{\rho_{\bkt{\Gamma,\;\A}}}~\gamma^\cblu{+}~\a~\a\cblu{'}\Big)_{\n+\cred{1}} \\
& \equiv \bkt{\zv^{\pi\A}_{\sm^{\n+\cred{1}}}}_{\n+\cred{1}} ^{\bkt{\Gamma,\;\A}^{\Zero\id_{\angle{\n+\cred{1}}}}}~\gamma_{\cblu{\n+2}} ~\a_{\cred{\partial}\bkt{\n+\cred{1}}}~\a_{\n+\cred{1}}~\a\cblu{'}_{\cred{\partial}\bkt{\n+\cred{1}}}~\a\cblu{'}_{\n+\cred{1}} \\
& \equiv
\begin{aligned}[t]
  \bkt{\zv^{\pi\A}_{\sm^{\n+\cred{1}}}}_{\n+\cred{1}}~ &\big(\Gamma^{\Zero\id_{\angle{\n+\cred{1}}}}~\gamma_{\cblu{\n+2}}\big)~ ~\big(\nat^{\pi\A}_{\cred{\partial}(\Zero\id_{\angle{\n+\cred{1}}})}~\gamma_{\cblu{\n+2}} ~\a_{\cred{\partial}\bkt{\n+\cred{1}}} ~\a_{\n+\cred{1}}~\a\cblu{'}_{\cred{\partial}\bkt{\n+\cred{1}}}\big) \\
  &\big(\nat^{\A}_{\Zero\id_{\angle{\n+\cred{1}}}}~\gamma_{\cblu{\n+2}} ~\a_{\cred{\partial}\bkt{\n+\cred{1}}}~\a_{\n+\cred{1}}~\a\cblu{'}_{\cred{\partial}\bkt{\n+\cred{1}}}~\a\cblu{'}_{\n+\cred{1}}\big)
\end{aligned}
\\
&\equiv \bkt{\zv^{\pi\A}_{\sm^{\n+\cred{1}}}}_{\n+\cred{1}} ~\big(\big(\rho_{\Gamma}\big)_{\n+\cred{1}}~\gamma_{\cblu{\n+2}}\big) ~\a_{\cred{\partial}\bkt{\n+\cred{1}}}~\a_{\n+\cred{1}} \\
&\equiv \bkt{\zv^{\pi\A^{\rho_\Gamma}}_{\sm^{\n+\cred{1}}}}_{\n+\cred{1}} ~\gamma_{\cblu{\n+2}}~\a_{\cred{\partial}\bkt{\n+\cred{1}}}~\a_{\n+\cred{1}} \\
&\equiv \bkt{\zv^{\pi\A^{\rho_\Gamma}}_{\sm^{\n+\cred{1}}}}_{\n+\cred{1}}  ~\big(\big(\ft^{\A\d}_{\sm^{\n+\cred{1}}}\big)_{\n+\cred{1}} ~\gamma_{\cblu{\n+2}}~\a_{\cred{\partial}\bkt{\n+\cred{1}}}~\a_{\n+\cred{1}}, ~\a\cblu{'}_{\cred{\partial}\bkt{\n+\cred{1}}}~\a\cblu{'}_{\n+\cred{1}}\big) \\
&\equiv\big(\bkt{\zv^{\pi\A^{\rho_\Gamma}}_{\sm^{\n+\cred{1}}}}^{\ft^{\A\d}_{\sm^{\n+\cred{1}}}}\big)_{\n+\cred{1}} ~\gamma_{\cblu{\n+2}}~\a_{\cred{\partial}\bkt{\n+\cred{1}}}~\a_{\n+\cred{1}} ~\a\cblu{'}_{\cred{\partial}\bkt{\n+\cred{1}}}~\a\cblu{'}_{\n+\cred{1}} \\
&\equiv \Big(\bkt{\zv^{\pi\A^{\rho_\Gamma}}_{\sm^{\n+\cred{1}}}}^{\ft^{\A\d}_{\sm^{\n+\cred{1}}}}~\gamma^\cblu{+}~\a~\a\cblu{'}\Big)_{\n+\cred{1}}.
\end{align*}
We now need to verify \cref{eq:ft-sqbkt,eq:zv-sqbkt,eq:sqbkt-ft-zv} at the level of $\bkt{\n+\cred{2}}$-simplices. For the first two of these, for $\sigma : \Delta \to \Gamma$ and $\delta : \Delta \vdash_{\sm^{\n+\cred{2}}} \t~\delta : \A~\bkt{\sigma~\delta}$, we have that:
\begin{align*}
\Big(\ft^{\A}_{\sm^{\n+\cred{2}}}\circ \sqbkt{\sigma,~\t}\Big)_{\n+\cred{2}}
&\equiv \Big(\big(\ft^{\A}_{\sm^{\n+\cred{2}}}\circ \sqbkt{\sigma,~\t}\big)\D\Big)_{\n+\cred{1}} \\
&\equiv \Big(\big(\ft^{\A}_{\sm^{\n+\cred{2}}}\big)\D \circ \sqbkt{\sigma,~\t}\D\Big)_{\n+\cred{1}} \\
&\equiv \Big(\ft^{\pi\A^{\rho_\Gamma}}_{\sm^{\n+\cred{1}}} \circ\ft^{\A\d}_{\sm^{\n+\cred{1}}} \circ \sqbkt{\sigma\D,~\pi\t^{\rho_\Delta},~\t\d}\Big)_{\n+\cred{1}} \\
&\equiv \Big(\ft^{\pi\A^{\rho_\Gamma}}_{\sm^{\n+\cred{1}}} \circ \sqbkt{\sigma\D,~\pi\t^{\rho_\Delta}}\Big)_{\n+\cred{1}} \\
&\equiv \Big(\sigma\D\Big)_{\n+\cred{1}} \\
&\equiv \sigma_{\n+\cred{2}}.
\end{align*}
For the second:
\begin{align*}
\Big(\big(\zv^{\A}_{\sm^{\n+\cred{2}}}\big)^{\sqbkt{\sigma,\;\t}}\Big)_{\n+\cred{2}}
&\equiv \Big(\big(\big(\zv^{\A}_{\sm^{\n+\cred{2}}}\big)^{\sqbkt{\sigma,\;\t}} \big)\d\Big)_{\n+\cred{1}} \\
&\equiv \Big(\big(\big(\zv^{\A}_{\sm^{\n+\cred{2}}}\big)\d\big)^{\sqbkt{\sigma,\;\t}\D}\Big)_{\n+\cred{1}} \\
&\equiv \Big(\big(\zv^{\A\d}_{\sm^{\n+\cred{1}}}\big)^{\sqbkt{\sigma\D,\; \pi\t^{\rho_\Delta},\;\t\d}}\Big)_{\n+\cred{1}} \\
&\equiv \Big(\t\d\Big)_{\n+\cred{1}} \\
&\equiv \t_{\n+\cred{2}}.
\end{align*}
For the third, we will use \cref{eq:pt-zv}. For $\tau : \Delta \to \bkt{\gamma : \Gamma, ~\a : \A~\gamma}$ in $\sm^{\n+\cred{2}}$ we have:
\begin{align*}
\MoveEqLeft\big[\;\ft^{\A}_{\sm^{\n+\cred{2}}} \circ \tau, ~\big(\zv^{\A}_{\sm^{\n+\cred{2}}}\big)^\tau\;\big]_{\n+\cred{2}} \\
&\equiv \big(\big[\;\ft^{\A}_{\sm^{\n+\cred{2}}} \circ \tau, ~\big(\zv^{\A}_{\sm^{\n+\cred{2}}}\big)^\tau\;\big]\D\big)_{\n+\cred{1}} \\
&\equiv \big[\;\big(\ft^{\A}_{\sm^{\n+\cred{2}}}\big)\D \circ \tau\D, ~\big(\zv^{\pi\A}_{\sm^{\n+\cred{1}}}\big)^{\pi\tau \circ \rho_\Delta}, ~\big(\big(\zv^{\A}_{\sm^{\n+\cred{2}}}\big)^\tau\big)\d\;\big]_{\n+\cred{1}} \\
&\equiv \big[\;\ft^{\pi\A^{\rho_\Gamma}}_{\sm^{\n+\cred{1}}} \circ\ft^{\A\d}_{\sm^{\n+\cred{1}}} \circ \tau\D, ~\big(\zv^{\pi\A}_{\sm^{\n+\cred{1}}}\big)^{\rho_{\bkt{\Gamma,\;\A}} \circ \tau\D}, ~\big(\big(\zv^{\A}_{\sm^{\n+\cred{2}}}\big)\d\big)^{\tau\D} \;\big]_{\n+\cred{1}} \\
&\equiv \big[\;\ft^{\pi\A^{\rho_\Gamma}}_{\sm^{\n+\cred{1}}} \circ\ft^{\A\d}_{\sm^{\n+\cred{1}}} \circ \tau\D, ~\bkt{\zv^{\pi\A^{\rho_\Gamma}}_{\sm^{\n+\cred{1}}}}^{\ft^{\A\d}_{\sm^{\n+\cred{1}}} \circ \tau\D}, ~\big(\zv^{\A\d}_{\sm^{\n+\cred{1}}}\big)^{\tau\D} \;\big]_{\n+\cred{1}} \\
&\equiv \big[\;\ft^{\A\d}_{\sm^{\n+\cred{1}}} \circ \tau\D, ~\big(\zv^{\A\d}_{\sm^{\n+\cred{1}}}\big)^{\tau\D} \;\big]_{\n+\cred{1}} \\
&\equiv \big[\; \tau\D \;\big]_{\n+\cred{1}} \\
&\equiv \tau_{\n+\cred{2}}.
\tag*{\bbox}
\end{align*}

\subsection{$\Pi$-types}
\label{sec:appendix:pi-types}

We verify the $\beta$-law at the level of $\bkt{\n+\cred{2}}$-simplices:
\begin{align*}
\MoveEqLeft\bkt{\eval^{\sm^{\n+\cred{2}}}~\big(\lambda^{\sm^{\n+\cred{2}}}~\t\big)~\s}_{\n+\cred{2}} \\
&\equiv \bkt{\big(\eval^{\sm^{\n+\cred{2}}}~\big(\lambda^{\sm^{\n+\cred{2}}}~\t\big)~\s\big)\d}_{\n+\cred{1}} \\
&\equiv \bkt{\eval^{\sm^{\n+\cred{1}}} ~\big(\eval^{\sm^{\n+\cred{1}}}~\big(\lambda^{\sm^{\n+\cred{2}}}~\t\big)\d~\pi\s^{\rho_\Gamma}\big)~\s\d}_{\n+\cred{1}} \\
&\equiv \bkt{\eval^{\sm^{\n+\cred{1}}} ~\big(\eval^{\sm^{\n+\cred{1}}}~\big(\lambda^{\sm^{\n+\cred{1}}}~\big(\lambda^{\sm^{\n+\cred{1}}}~\t\d\big)\big)~\pi\s^{\rho_\Gamma}\big)~\s\d}_{\n+\cred{1}} \\
&\equiv \bkt{\eval^{\sm^{\n+\cred{1}}}~\big(\lambda^{\sm^{\n+\cred{1}}}~\t\d\big) ^{\sqbkt{\id_{\Gamma\D},\;\pi\s^{\rho_\Gamma}}}~\s\d}_{\n+\cred{1}} \\
&\equiv \bkt{\eval^{\sm^{\n+\cred{1}}}~\big(\lambda^{\sm^{\n+\cred{1}}} ~\big(\t\d\big)^{\mathsf{\cred{W_{2}}}^{\A\d}\sqbkt{\id_{\Gamma\D},\;\pi\s^{\rho_\Gamma}}}\big)~\s\d}_{\n+\cred{1}} \\
&\equiv \bkt{\bkt{\t\d}^{\sqbkt{\id_\Gamma\D,\;\pi\s^{\rho_\Gamma},\; \s\d}}}_{\n+\cred{1}} \\
&\equiv \bkt{\big(\t^{\sqbkt{\id_\Gamma,\;\s}}\big)\d}_{\n+\cred{1}} \\
&\equiv \bkt{\t^{\sqbkt{\id_\Gamma,\;\s}}}_{\n+\cred{2}}.
\end{align*}
We also verify the $\eta$-law:
\begin{align*}
\MoveEqLeft\Big(\lambda^{\sm^{\n+\cred{2}}}~\big(\eval^{\sm^{\n+\cred{2}}} \f^\ft~\zv\big)\Big)_{\n+\cred{2}} \\
&\equiv \Big(\big(\lambda^{\sm^{\n+\cred{2}}}~\big(\eval^{\sm^{\n+\cred{2}}} \f^\ft~\zv\big)\big)\d\Big)_{\n+\cred{1}} \\
&\equiv \Big(\lambda^{\sm^{\n+\cred{1}}}~\big(\lambda^{\sm^{\n+\cred{1}}}~ \big(\eval^{\sm^{\n+\cred{2}}} \f^\ft~\zv\big)\d\big) \Big)_{\n+\cred{1}} \\
&\equiv \Big(\lambda^{\sm^{\n+\cred{1}}}~\big(\lambda^{\sm^{\n+\cred{1}}}~ \big(\eval^{\sm^{\n+\cred{1}}}~\big(\eval^{\sm^{\n+\cred{1}}}~ \big(\f^\ft\big)\d~\pi\zv^{\rho_{\bkt{\Gamma,\;\A}}}\big)~\zv\d\big)\big) \Big)_{\n+\cred{1}} \\
&\equiv \Big(\lambda^{\sm^{\n+\cred{1}}}~\big(\lambda^{\sm^{\n+\cred{1}}}~ \big(\eval^{\sm^{\n+\cred{1}}}~\big(\eval^{\sm^{\n+\cred{1}}}~ \big(\f\d\big)^{\ft\D}~\zv^\ft\big)~\zv\big)\big)\Big)_{\n+\cred{1}} \\
&\equiv \Big(\lambda^{\sm^{\n+\cred{1}}}~\big(\lambda^{\sm^{\n+\cred{1}}}~ \big(\eval^{\sm^{\n+\cred{1}}}~\big(\eval^{\sm^{\n+\cred{1}}}~ \big(\f\d\big)^{\ft \circ \ft}~\zv^\ft\big)~\zv\big)\big)\Big)_{\n+\cred{1}} \\
&\equiv\Big(\lambda^{\sm^{\n+\cred{1}}}~\big(\lambda^{\sm^{\n+\cred{1}}}~ \big(\eval^{\sm^{\n+\cred{1}}}~\big(\eval^{\sm^{\n+\cred{1}}}~ \big(\f\d\big)^\ft~\zv\big)^\ft~\zv\big)\big)\Big)_{\n+\cred{1}} \\
&\equiv\Big(\lambda^{\sm^{\n+\cred{1}}}~\big(\eval^{\sm^{\n+\cred{1}}}~ \big(\f\d\big)^\ft~\zv\big)\Big)_{\n+\cred{1}} \\
&\equiv \Big(\f\d\Big)_{\n+\cred{1}} \\
&\equiv \f_{\n+\cred{2}}.
\tag*{\bbox}
\end{align*}

\subsection{Universes}
\label{sec:appendix:universes}

We verify that $\Code$ and $\El$ are mutually inverse at the level of $\bkt{\n+\cred{2}}$-simplices:
\begin{align*}
\MoveEqLeft\Big(\El^{\sm^{\n+\cred{2}}}~\big(\Code^{\sm^{\n+\cred{2}}}~\A\big)\Big)_{\n+\cred{2}} \\
&\equiv \Big(\big(\El^{\sm^{\n+\cred{2}}}~\big(\Code^{\sm^{\n+\cred{2}}}~\A\big)\big)\d\Big)_{\n+\cred{1}} \\
&\equiv \Big(\El^{\sm^{\n+\cred{1}}}~\big(\eval^{\sm^{\n+\cred{1}}}~\big(\big(\Code^{\sm^{\n+\cred{2}}}~\A\big)\d\big)^\ft~\zv\big)\Big)_{\n+\cred{1}}
\\
&\equiv \Big(\El^{\sm^{\n+\cred{1}}}~\big(\eval^{\sm^{\n+\cred{1}}}~\big(\lambda^{\sm^{\n+\cred{1}}}~\big(\Code^{\sm^{\n+\cred{1}}}~\A\d\big)\big)^\ft~\zv\big)\Big)_{\n+\cred{1}} \\
&\equiv \Big(\El^{\sm^{\n+\cred{1}}}~\big(\eval^{\sm^{\n+\cred{1}}}~\big(\lambda^{\sm^{\n+\cred{1}}}~\big(\Code^{\sm^{\n+\cred{1}}}~\A\d\big)^{\mathsf{\cred{W_{2}}}^{\pi\A^{\rho_\Gamma}}\ft}\big)~\zv\big)\Big)_{\n+\cred{1}} \\
&\equiv \Big(\El^{\sm^{\n+\cred{1}}}~\big(\Code^{\sm^{\n+\cred{1}}}~\A\d\big)^{\sqbkt{\ft,\;\zv}}\Big)_{\n+\cred{1}} \\
&\equiv \Big(\El^{\sm^{\n+\cred{1}}}~\big(\Code^{\sm^{\n+\cred{1}}}~\A\d\big)\Big)_{\n+\cred{1}} \\
&\equiv \Big(\A\d\Big)_{\n+\cred{1}} \\
&\equiv \A_{\n+\cred{2}}.
\end{align*}
In the other direction:
\begin{align*}
\MoveEqLeft\Big(\Code^{\sm^{\n+\cred{2}}}~\big(\El^{\sm^{\n+\cred{2}}}~\A\big)\Big)_{\n+\cred{2}} \\
&\equiv \Big(\big(\Code^{\sm^{\n+\cred{2}}}~\big(\El^{\sm^{\n+\cred{2}}}~\A\big)\big)\d\Big)_{\n+\cred{1}} \\
&\equiv \Big(\lambda^{\sm^{\n+\cred{1}}}~\big(\Code^{\sm^{\n+\cred{1}}}~\big(\El^{\sm^{\n+\cred{2}}}~\A\big)\d\big)\Big)_{\n+\cred{1}} \\
&\equiv \Big(\lambda^{\sm^{\n+\cred{1}}}~\big(\Code^{\sm^{\n+\cred{1}}}~\big(\El^{\sm^{\n+\cred{1}}}~\big(\eval^{\sm^{\n+\cred{1}}}~\big(\A\d\big)^\ft~\zv\big)\big)\big)\Big)_{\n+\cred{1}} \\
&\equiv \Big(\lambda^{\sm^{\n+\cred{1}}}~\big(\eval^{\sm^{\n+\cred{1}}}~\big(\A\d\big)^\ft~\zv\big)\Big)_{\n+\cred{1}} \\
&\equiv\Big(\A\d\Big)_{\n+\cred{1}} \\
&\equiv \A_{\n+\cred{2}}.
\tag*{\bbox}
\end{align*}

\subsection{$\oldomega$-limits}
\label{sec:appendix:omega-limits}

We mutually verify the identity between $\res$ and $\lim$ at the $\m$-th stage and on the boundary:
\begin{align*}
&\Big(\res^{\cred{\partial}\m}_{\sm^{\n+\cred{2}}}~\big(\lim^{\sm^{\n+\cred{2}}}\cblu{\widetilde{\mathfrak{a}}}\big)\Big)_{\n+\cred{2}}
&&\Big(\res^{\m}_{\sm^{\n+\cred{2}}}~\big(\lim^{\sm^{\n+\cred{2}}}\cblu{\widetilde{\mathfrak{a}}}\big)\Big)_{\n+\cred{2}} \\
&\quad\equiv\Big(\big(\res^{\cred{\partial}\m}_{\sm^{\n+\cred{2}}}~\big(\lim^{\sm^{\n+\cred{2}}}~\cblu{\widetilde{\mathfrak{a}}}\big)\big)\d\Big)_{\n+\cred{1}}
&&\quad\equiv\Big(\big(\res^{\m}_{\sm^{\n+\cred{2}}}~\big(\lim^{\sm^{\n+\cred{2}}}~\cblu{\widetilde{\mathfrak{a}}}\big)\big)\d\Big)_{\n+\cred{1}} \\
&\quad \equiv \Big(\res^{\cred{\partial}\m}_{\sm^{\n+\cred{1}}}~\big(\lim^{\sm^{\n+\cred{2}}}~\cblu{\widetilde{\mathfrak{a}}}\big)\d\Big)_{\n+\cred{1}}
&&\quad \equiv \Big(\res^{\m}_{\sm^{\n+\cred{1}}}~\big(\lim^{\sm^{\n+\cred{2}}}~\cblu{\widetilde{\mathfrak{a}}}\big)\d\Big)_{\n+\cred{1}} \\
&\quad \equiv \Big(\res^{\cred{\partial}\m}_{\sm^{\n+\cred{1}}}~\big(\lim^{\sm^{\n+\cred{1}}}~\cblu{\widetilde{\mathfrak{a}}}\d\big)\Big)_{\n+\cred{1}}
&&\quad \equiv \Big(\res^{\m}_{\sm^{\n+\cred{1}}}~\big(\lim^{\sm^{\n+\cred{1}}}~\cblu{\widetilde{\mathfrak{a}}}\d\big)\Big)_{\n+\cred{1}} \\
&\quad \equiv \Big(\big(\cblu{\widetilde{\mathfrak{a}}}\d\big)^{\cred{\partial}\m}\Big)_{\n+\cred{1}}
&&\quad \equiv \Big(\big(\cblu{\widetilde{\mathfrak{a}}}\d\big)^{\m}\Big)_{\n+\cred{1}} \\
&\quad \equiv \Big(\big(\cblu{\widetilde{\mathfrak{a}}}^{\cred{\partial}\m}\big)\d\Big)_{\n+\cred{1}}
&&\quad \equiv \Big(\big(\cblu{\widetilde{\mathfrak{a}}}^{\m}\big)\d\Big)_{\n+\cred{1}}\\
&\quad \equiv \Big(\cblu{\widetilde{\mathfrak{a}}}^{\cred{\partial}\m}\Big)_{\n+\cred{2}}
&&\quad \equiv \Big(\cblu{\widetilde{\mathfrak{a}}}^{\m}\Big)_{\n+\cred{2}}.
\end{align*}
We also verify the $\eta$-law:
\begin{align*}
&\Big(\lim^{\sm^{\n+\cred{2}}}~\big(\res^{\m}_{\sm^{\n+\cred{2}}}~\cblu{\widetilde{\mathfrak{u}}}\big)_\m\Big)_{\n+\cred{2}} \\
&\quad \equiv \Big(\big(\lim^{\sm^{\n+\cred{2}}}~\big(\res^{\m}_{\sm^{\n+\cred{2}}}~\cblu{\widetilde{\mathfrak{u}}}\big)_\m\big)\d\Big)_{\n+\cred{1}} \\
&\quad \equiv \Big(\lim^{\sm^{\n+\cred{1}}}~\big(\res^{\m}_{\sm^{\n+\cred{2}}}~\cblu{\widetilde{\mathfrak{u}}}\big)^{\cred{\mathsf{d}}}_\m\Big)_{\n+\cred{1}} \\
&\quad \equiv \Big(\lim^{\sm^{\n+\cred{1}}}~\big(\res^{\m}_{\sm^{\n+\cred{1}}}~\cblu{\widetilde{\mathfrak{u}}}\d\big)_\m\Big)_{\n+\cred{1}} \\
&\quad \equiv \Big(\cblu{\widetilde{\mathfrak{u}}}\d\Big)_{\n+\cred{1}} \\
&\quad \equiv \Big(\cblu{\widetilde{\mathfrak{u}}}\Big)_{\n+\cred{2}}.
\tag*{\bbox}
\end{align*}

\end{document}